\documentclass[10pt,reqno,a4paper]{amsart}
\usepackage{graphicx}
\usepackage{subcaption}
\usepackage{enumerate} 
\usepackage[colorlinks=true, pdfstartview=FitV,
linkcolor=blue, citecolor=red, urlcolor=blue]{hyperref}
\usepackage{amsmath,amsfonts,latexsym,amssymb}
\usepackage[utf8]{inputenc} 
\usepackage[T1]{fontenc}
\usepackage{braket}
\usepackage{ae,aecompl}
\usepackage{dsfont} 
\usepackage{pxfonts}
\usepackage{microtype}
\usepackage{ae,aecompl}
\usepackage{marginnote}
\usepackage{accents}
\newtheorem{theorem}{Theorem}[subsection]
\newtheorem{lemma}[theorem]{Lemma}
\newtheorem{proposition}[theorem]{Proposition}
\newtheorem{corollary}[theorem]{Corollary}

\newtheorem{definition}[theorem]{Definition}

\newcommand{\rmks}{\vskip 0.1 truecm\noindent{\sc Remarks: }}

\newcommand{\T}{\ms T}
\renewcommand{\leq}{\leqslant}
\renewcommand{\geq}{\geqslant}
\newcommand{\lga}{\Gamma\backslash}
\newcommand{\diam}{\operatorname{diam}}

\newcommand{\vect}{\underline}
\renewcommand{\epsilon}{\varepsilon}
\newcommand*{\dt}[1]{\accentset{\mbox{\large\bfseries .}}{#1}}

\makeatletter
\newcommand{\oset}[3][0ex]{
  \mathrel{\mathop{#3}\limits^{
    \vbox to#1{\kern-2\ex@
    \hbox{$\scriptstyle#2$}\vss}}}}
\makeatother
\newcommand{\G}{\ms G}
\newcommand{\sld}{\mathsf{SL}_2(\mathbb R)}
\newcommand{\psld}{\mathsf{PSL}_2(\mathbb R)}
\newcommand{\hh}{{\bf H}^2}

\newcommand{\gp}{{\bf F}}
\newcommand{\ms}{\mathsf}
\newcommand{\mk}{\mathfrak}
\newcommand{\msl}{{\skd}}
\newcommand{\sg}{\operatorname{Sym}({\G})}
\newcommand{\mc}{\mathcal}
\newcommand{\id}{\operatorname{Id}}

\newcommand{\bA}{{\rm A}}
\newcommand{\bk}{\boldsymbol{\kappa}}
\newcommand{\bK}{\boldsymbol{K}}
\newcommand{\bI}{{{\bf I}}}
\newcommand{\bJ}{{{\bf J}}}
\newcommand{\bM}{{{\bf M}}}
\newcommand{\bC}{{\rm M}}
\newcommand{\car}{{\boldsymbol i}}
\newcommand{\bQ}{{\rm q}}

\newcommand{\bpsi}{{\boldsymbol\Psi}}
\newcommand{\bmu}{{\boldsymbol\mu}}
\newcommand{\bnu}{{\boldsymbol\nu}}
\newcommand{\defeq}{\coloneqq}
\newcommand{\eqdef}{\eqqcolon}
\newcommand{\slt}{\mk s}
\newcommand{\skd}{\mk{sl}_2}
\newcommand{\wa}{{\rm A}}
\newcommand{\wb}{{\rm B}}
\newcommand{\wD}{{\rm D}}
\newcommand{\eR}{\frac{\epsilon}{R}}
\newcommand{\KMT}{{\bf{T}}}
\newcommand{\Rp}{{\bf P}^1(\mathbb R)}
\renewcommand{\d}{{\rm d}}

\newcommand{\Ad}{\operatorname{Ad}}

\newcommand{\bN}{{\bf N}}
\newcommand{\seR}{_{\epsilon,R}}
\newcommand{\Sc}{\operatorname{Suc}}
\newcommand{\Op}{\operatorname{Opp}}
\newcommand{\Aut}{\operatorname{Aut}}

\newcommand{\chem}[1]{\{{#1}_t\}_{t\in[0,1]}}

\newcommand{\seq}[1]{ \{{#1}_m\}_{m\in\mathbb N}}

\title{Surface groups in uniform  lattices of  some semi-simple groups} 
\author[Kahn]{Jeremy Kahn}
\address{Brown University, Department of Mathematics, 
Providence, RI 02912, U.S.A.}
\author[Labourie]{Fran\c cois Labourie}
\address{Universit\'e C\^ote d'Azur, LJAD, Nice F-06000; France}
\author[Mozes]{Shahar Mozes}
\address{The Hebrew University of Jerusalem, Einstein Institute of Mathematics, Givat Ram. Jerusalem, 9190401, Israel}
\thanks{J.K and F.L acknowledge support from U.S. National Science Foundation grants DMS 1107452, 1107263, 1107367 ``RNMS: Geometric structures And Representation varieties'' (the GEAR Network). J.K acknowledges support by the National Science Foundation under Grant ${\rm n}^{\rm o}$ DMS 1352721
F.L was partially supported by the European Research Council under the {\em European Community}'s seventh Framework Programme (FP7/2007-2013)/ERC {\em grant agreement} ${\rm n}^{\rm o}$ FP7-246918 as well as by  the ANR grant DynGeo ANR-11-BS01-013. 
S.M acknowledges support by ISF grant 1003/11 and ISF-Moked grant 2095/15. }

\makeindex

\begin{document}
\renewcommand{\theenumi}{(\roman{enumi})}
\begin{abstract}
We show that uniform lattices in some semi-simple groups (notably complex ones) admit  Anosov surface subgroups. This result has a quantitative version: we introduce a notion, called $K$-Sullivan maps, which generalizes the notion of $K$-quasi-circles in hyperbolic geometry, and show in particular that Sullivan maps are H{\"o}lder. Using this notion, we show a quantitative version of our surface subgroup theorem and in particular that one can obtain $K$-Sullivan limit maps, as close as one wants to smooth round circles. All these results use the coarse geometry of  ``path of triangles''  in a certain flag manifold and we prove an analogue to the Morse Lemma for quasi-geodesics in that context.
\end{abstract}
\maketitle

\section{Introduction}
\newcounter{foo}
\newtheorem{theo}[foo]{Theorem}
\renewcommand*{\thefoo}{\Alph{foo}}

As a corollary of our main Theorem, we obtain the following easily stated result

\begin{theo}
	Let $\ms G$ be a center free, complex semisimple Lie group and $\Gamma$ a uniform lattice in $\ms G$. Then $\Gamma$ contains a surface group.
\end{theo}
However our main result is a quantitative version of this result.

By a surface group, we mean the fundamental group of a closed connected oriented surface of genus at least 2. We shall see later on that the restriction that $\ms G$ is complex can be relaxed : the theorem holds for a wider class of groups, for instance $\ms{PU}(p,q)$ with $q>p>0$. This theorem is a generalization of the celebrated Kahn--Markovic Theorem \cite{Kahn:2009wh,Bergeron2013} which deals with the case of $\mathsf{PSL}(2,\mathbb C)$ and its proof follows a similar  scheme: building pairs of pants, gluing them and showing the group is injective, however the details vary greatly, notably in the injectivity part.  Let us note that Hamenstädt \cite{Hamenstadt:2015wa}
 had followed a similar proof to show the existence of surface subgroups of all rank 1 groups, except $\mathsf{SO}(2n,1)$, while Kahn and Markovic  essentially deals with the case $\mathsf G=\mathsf{SL}(2,\mathbb R)\times \mathsf{SL}(2,\mathbb R)$ in their Ehrenpreis paper \cite{Kahn:2011tt}.
 
 Finally, let us recall that Kahn--Markovic paper was preceded in the context of hyperbolic 3-manifolds by (non quantitative) results of Lackenby \cite{Lackenby:2010wk} for lattices with torsion  and  Cooper, Long and Reid \cite{Cooper:1997ug} in the non uniform case, both papers  using very different techniques.

Kahn--Markovic Theorem has a quantitative version: the surface group obtained is $K$-quasi-symmetric where $K$ can be chosen  arbitrarily close to 1. Our theorem also has  a quantitative version that needs some preparation and definitions to be stated properly: in particular, we need to define in this higher rank context what is the analog of a quasi-symmetric (or rather almost-symmetric) map. 

\subsection{Sullivan maps}
We make the choice of an $\skd$ triple in $\ms G$, that is an embedding of the Lie algebra of $\sld$ with its standard generators $(a,x,y)$ into the Lie algebra of $\ms G$. For the sake of simplification, in this introduction, we suppose that this triple has a compact centralizer. Such an $\skd$ triple defines a flag manifold $\gp$: a compact $\ms G$-transitive space on which the hyperbolic element $a$ acts with a unique attractive fixed point (see section \ref{sec:prel} for details).

Most of the results and techniques of the proof involves the study of the following geometric objects in $\gp$:
\begin{enumerate}
	\item  {\em circles in $\gp$} which are maps from $\Rp$ to $\gp$ equivariant under a representation of $\sld$ conjugate to the one defined by the $\skd$ triple chosen above.
	\item {\em tripods} which are triple of distinct point on a circle. Such a tripod $\tau$ defines -- in a $\ms G$-equivariant way -- a metric $d_\tau$ on $\gp$.
\end{enumerate}
We can now define what is the generalization of a $K$-quasi-symmetric map, for $K$-close to 1. Let $\zeta$ be a positive number. A {\em $\zeta$-Sullivan} map is a map $\xi$ from $\Rp$ to $\gp$, so that for every triple of pairwise distinct points $T$ in $\Rp$, there is a circle $\eta_T:\Rp\to\gp$ so that
$$
\forall x\in \Rp\ , \ \ \ d_{\eta_T(T)}(\eta_T(x),\xi(x))\leq \zeta\ .
$$
We remark that circles are $0$-Sullivan map. Also, we insist that this notion is relative to the choice of some $\skd$ triple, or more precisely of a conjugacy class of $\skd$-triple. This notion is discussed more deeply in Section \ref{sec:sull}.

Obviously, for this definition to make sense,  $\zeta$ has to be  small. We do not require any regularity nor continuity of the map $\xi$. Our first result actually guarantees some regularity:
\begin{theo}{\sc[H{\"o}lder property]}\label{theo:sull-hold0}
	 There exists some positive numbers $\zeta$ and $\alpha$,  so that any  $\zeta$-Sullivan map  is $\alpha$-H{\"o}lder.
\end{theo}
If we furthermore assume that the map $\xi$ is equivariant under some representation $\rho$ of a Fuchsian group $\Gamma$ acting on $\Rp$, we have 

\begin{theo} {\sc [Sullivan implies Anosov]}\label{theo:sull-anos0}
 There exists a positive number $\zeta$ such that if $\Gamma$ is a cocompact Fuchsian group, $\rho$ a representation of $\Gamma$ in $\G$ so that there exists a $\rho$ equivariant $\zeta$-Sullivan map $\xi$ from $\Rp$ to $\gp$, 
	 then  $\rho$ is $\gp$-Anosov and $\xi$ is its limit curve.
\end{theo}

When $\G=\mathsf{PSL}(2,\mathbb C)$, $\gp={\bf P}^1(\mathbb C)=\partial_\infty{\bf H}^3$, circles are boundaries at infinity of hyperbolic planes, and the theorems above translate into classical properties of quasi-symmetric maps. We refer to \cite{Labourie:2006,Guichard:2012eg} for reference on Anosov representations and give a short introduction in paragraph \ref{subsec:introAnos}. In particular recall that Anosov representations are faithful

\subsection{A quantitative surface subgroup theorem}
We can now state what is our quantitative version of the existence of surface subgroup in higher rank lattices.  

\begin{theo}
	Let $\ms G$ be a center free,  semisimple Lie group without compact factor and $\Gamma$ a uniform lattice in $\ms G$. 
	Let us choose an $\skd$-triple in $\ms G$ with a compact centralizer and satisfying the {\em  flip assumption} (See below) with associated flag manifold $\gp$. 
	
	Let $\zeta$ be a positive number. Then there exists a cocompact Fuchsian group $\Gamma_0$ and a $\gp$-Anosov representation $\rho$ of $\Gamma_0$ in $\ms G$ with values in $\Gamma$ and  whose limit map is $\zeta$-Sullivan.
\end{theo}

The flip assumption is satisfied for all complex groups, all rank 1 group except $\mathsf{SO}(1,2n)$, but not for real split groups. The precise statement is the following. Let $(a,x,y)$ be an $\skd$-triple and $\zeta_0$ the smallest real positive number so that $\exp(2i\zeta_0\cdotp a)=1$. We say the $(a,x,y)$ satisfies the {\em flip assumption} if the automorphism of $\ms G$, $\bJ_0\defeq\exp(i\zeta_0\cdotp a)$ belongs to the connected component of a compact factor of the centralizer of $a$. Ursula Hamenstädt also used the flip assumption in \cite{Hamenstadt:2015wa}.

We do hope the flip assumption is unnecessary. However removing it is beyond the scope of the present article: it  would involve in particular incorporating generalized arguments from \cite{Kahn:2011tt} which deal with the (non flip) case of $\ms G=\mathsf{SL}(2,\mathbb R)\times \mathsf{SL}(2,\mathbb R)$.

Finally let us notice that Kahn and Wright have announced a quantitative version of the surface subgroup theorem for non uniform lattice in the case of $\mathsf{PSL}(2,\mathbb C)$, leaving thus open the possibility to extend our theorem also for non uniform lattices.

\subsection{A tool: coarse geometry in flag manifolds} A classical tool for Gromov hyperbolic spaces is the Morse Lemma: quasi-geodesics are at uniform distance to geodesics. Higher rank symmetric spaces are not Gromov hyperbolic but they do carry a version of the Morse Lemma: see Kapovich--Leeb--Porti \cite{Kapovich:2014ue} and Bochi--Potrie--Sambarino \cite{Bochi:2016wl}

Our approach in this paper is however to avoid as much as possible dealing with the (too rich) geometry of the symmetric space. We will only use the geometry of the flag manifolds that we defined above: circles, tripods and metrics assigned to tripods. In this new point of view, the analogs of geodesics will be coplanar path of triangles: roughly speaking  a coplanar path of triangles corresponds to a sequence of non overlapping ideal triangles in some hyperbolic space so that two consecutive triangles are adjacent -- see figure \ref{fig:2a}. We now have to describe a coarse version of that. First we need to define {\em quasi-tripods} which are deformation of tripods: roughly speaking these are tripods with deformed vertices (See Definition 
\ref{def:qua-trip} for precisions). Then we want to define almost coplanar sequence of quasi-tripods (See Definition \ref{def:shear-qt}). Finally our main theorem \ref{theo:exislimi} guarantees some circumstances under which these  ``quasi-paths'' converge ``at infinity", that is shrink to a point in $\gp$. 

The Morse Lemma by itself is not enough to conclude in the hyperbolic case and we need a refined version. Our Theorem \ref{theo:exislimi} is used at several points in the paper: to prove the main theorem and to prove the theorems around Sullivan maps. Although, this theorem requires too many definitions to be stated in the introduction, it is one of the main and new contributions of this paper.
\vskip 0.1 truecm  
While this paper was in its last stage, we learned that Ursula Hamenstädt has announced existence results for lattices in higher rank group, without the quantitative part of our results, but with other very interesting features. 

We thank Bachir Bekka,  Yves Benoist, Nicolas Bergeron, Marc Burger,   Mahan Mj, Dennis Sullivan for their help and interest while we were completing this project. Fanny Kassel, Pierre-Louis Blayac and Olivier Glorieux deserve special and extremely warm thanks for comments and crucial corrections on the writing up. Their input was essential in improving the readability of the final version as well as pointing out some crucial mistakes.

\subsection{A description of the content of this article}
What follows is meant to be a reading guide of this article, while introducing informally the essential ideas. In order to improve readability, an index is produced at the end of this paper.
\begin{enumerate}
	\item Section \ref{sec:prel} sets up the Lie theory background: it describes in more details $\skd$-triples, the flip assumption, and the associated parabolic subgroups and flag manifolds. 
	\item Section \ref{sec:tripods} introduces the main tools of our paper: tripods. In the simplest case (for instance principal $\skd$-triples in complex simple groups), tripods are just preferred triples of points in the associated flag manifold. In the general case, tripods come with some extra decoration. They may be thought of as generalizations of ideal triangles in hyperbolic planar geometry and they reflect our choice of a preferred $\skd$-triple.  The space of tripods admits several actions that are introduced here and notably a {\em shearing flow}. Moreover each tripod defines a metric on the flag manifold itself and we explore the relationships between the shearing flow and these metric assignments.
	\item For the hyperbolic plane, (nice) sequences of non overlapping ideal triangles, where two successive ones have a common edge, converges at infinity. This corresponds in our picture to {\em coplanar paths of tripods}. Section \ref{sec:quas-trip} deals with ``coarse deformations'' of these paths. First we introduce quasi-tripods, which are deformation of tripods: in the simplest case these are triples of points in the flag manifold which are not far from a tripod, with respect to the metric induced by the tripod. Then we introduce {\em paths of quasi-tripods} that we see as deformation of coplanar paths of tripods. Our goal will be in a later section to show that this deformed paths converge under some nice hypotheses.
	\item  For coplanar paths of tripods (which are sequences of ideal triangles), one see the convergence to infinity as a result of nesting of intervals in the boundary at infinity. This however is the consequence of the order structure on $\partial_\infty\hh$ and very specific to planar geometry. In our case, we need to introduce ``coarse deformations'' of these intervals, that we call {\em slivers} and introduce quantitative versions of the nesting property of intervals called {\em squeezing} and {\em controlling}. In section \ref{sec:cone} and section \ref{sec:conf-lem}, we define all these objects and prove the confinement Lemma. This lemma tells us  that certain deformations of coplanar paths still satisfy our coarse nesting properties. These two sections are preliminary to the next one.
	\item In section \ref{sec:Morse}, we prove one of the main results of the papers, the Limit Point Theorem that gives a condition under which a deformed sequence of quasi-tripods converges to a point in the flag manifold as well as some quantitative estimates on the rate of convergence. This theorem will be used several times in the sequel. Special instances of this theorem may be thought of as higher rank versions of the Morse Lemma. Our motto is to use the coarse geometry of path of quasi-tripods in the flag manifolds rather than quasi-geodesics in the symmetric space.
	\item In section \ref{sec:sull}, we introduce {\em Sullivan curves} which are analogs of quasi circles. We show extensions of two classical results for Kleinian groups and quasi-circles: Sullivan curves are H{\"o}lder and if a Sullivan curve is equivariant under the representation of of a surface group, this surface group is Anosov -- the analog of quasi-fuchsian. In the case of deformation of equivariant curves, we prove an Improvement Theorem that needs a Sullivan curve to be only defined on a smaller set.
	\item So far, the previous sections were about the geometry of the flag manifolds and did not make use of a lattice or discrete subgroups of $\ms G$. We now move to the proof of existence of surface groups, that we shall build by gluing pairs of pants together. The next two sections deals with pairs of pants: section \ref{sec:pairpants} introduces the concept of a {\em almost closing pair of pants} that generalizes the idea of building a pair of pants out of two ideal triangles. We describe the structure of these pairs of pants in a Structure Theorem  using a partially hyperbolic Closing Lemma. 
 In Kahn--Markovic original paper a central role is played by ``triconnected pair of tripods'' which are (roughly speaking) three homotopy classes of paths joining two points. In section \ref{sec:triconn}, we introduce here the analog in our case (under the same name), then  describe  {\em weight functions}. A triconnected pair of tripods on which the weight function is positive, gives rise to a nearby almost closing pair of pants. We also study an orientation inverting symmetry.
	\item We study in the next two sections the boundary data that is needed to describe the gluing of pair of pants. After having introduced {\em biconnected pair of tripods} which amounts to forget one of the paths in our triple of paths. In section \ref{sec:biconn}, we introduce spaces and measures for both triconnected and biconnected pairs of tripods and show that the forgetting map almost preserve the measure using the mixing property of our mixing flow. Then in section \ref{sec:feet}, we move more closely to study the boundary data: we introduce the {\em feet spaces and projections} which is the higher rank analog to the normal bundle to geodesics and we prove a Theorem that describes under which circumstances a measure is not perturbed too much by  a {\em Kahn--Markovic} twist.
	\item  In section \ref{sec:equi}, we wrap up the previous two sections in proving the Even Distribution Theorem which essentially  roughly says that there are  the same number  pairs of pants coming from ``opposite sides'' in the feet space. This makes use of the flip assumption which is  discussed there with more details (with examples and counter examples).
	\item As in Kahn--Markovic original paper, we use the Measured Marriage Theorem  in section \ref{sec:straight} to produce {\em straight surface groups} which are pair of pants glued nicely along their boundaries. It now remains to prove that these straight surface groups injects and are Sullivan.
	\item Before starting that proof, we need to describe in section \ref{sec:lam} a little further the {\em $R$-perfect lamination} and more importantly the {\em accessible points} in the boundary at infinity, which are roughly speaking those points which are limits of nice path of ideal triangles with respect to the lamination. This section is purely hyperbolic planar geometry.
	\item  We finally make a connexion with the first part of the paper which leads to the Limit Point Theorem. In section  \ref{sec:ss-lm}, we consider the nice paths of tripods converging to accessible points described in the previous section, and show that a straight surface (or more generally an {\em equivariant straight surface}) gives rise to a deformation of these paths of tripods into paths of quasi-tripods, these latter paths being well behaved enough to have limit points according to the Limit Point Theorem.  Then using the Improvement Theorem of section \ref{sec:sull}, we show that this gives rive to a Sullivan limit map for our surface.
	\item The last section is a wrap-up of the previous results and finally in an Appendix, we present results and constructions dealing with the Levy--Prokhorov distance between measures.   \end{enumerate}.
\tableofcontents

\section{Preliminaries:  \texorpdfstring{$\mathfrak{sl}_2$}{sl2}-triples}\label{sec:prel}
In this preliminary section, we recall some facts about $\skd$-triples in Lie groups, the hyperbolic plane and discuss the flip assumption that we need to state our result. We also recall the construction of parabolic groups and the flag manifold whose geometry is going to play a fundamental role in this paper.
\subsection{ \texorpdfstring{$\mathfrak{sl}_2$}{sl2}-triples and the flip assumption}\label{sec:sld}
Let ${\G}$ be a semisimple center free Lie group without compact factors. 

\begin{definition}
An {\em ${\skd}$-triple}\index{${\skd}$-triple, even, regular}  \cite{Kostant:1959wi} is $\slt\defeq(a,x,y)\in\mk g^3$ so that $[a,x]=2x$,  $[a,y]=-2y$ and $[x,y]=a$. 

An ${\skd}$-triple $(a,x,y)$ is {\em regular}, if $a$ is a regular element. The centralizer of  a regular  ${\skd}$-triple is compact.  

An  ${\skd}$-triple $(a,x,y)$ is {\em even} if all the eigenvalues of $a$ by the adjoint representation are even.
	\end{definition}
 An $\msl$-triple $(a,x,y)$ generates a Lie algebra $\mk a$ isomorphic  to $\mk{sl}(2,\mathbb R)$ so that
\begin{eqnarray}
a=\left(\begin{array}{rl}1&0\cr0&-1\end{array}\right),\, \,
x=\left(\begin{array}{rl}0&1\cr0&0\end{array}\right),\, \,
y=\left(\begin{array}{rl}0&0\cr1&0\end{array}\right)\ .\label{def:matrix}	\end{eqnarray}
For an even triple, the group whose Lie algebra is  $\mk a$ is isomorphic to $\psld$.
 
 Say an element $\bJ_0$ of $\Aut(\G)$ is a {\em reflexion} for the $\skd$-triple $(a,x,y)$, if
 \begin{itemize}
 	\item $\bJ_0$ is an involution and belongs to ${\ms Z}({\ms Z}(a))$
 	\item $\bJ_0(a,x,y)=(a,-x,-y)$ and in particular $\bJ_0$ normalizes the group generated by $\skd$ isomorphic to $\psld$, and acts by conjugation by the matrix
 $\left(\begin{matrix}1&0\\ 0&-1\end{matrix}\right)$
 \end{itemize}
 
 An example of a reflexion in the case of complex group is  $\bJ_0\defeq\exp\left(\frac{i\zeta a}{2}\right)\in {\G}$, where $\zeta$ be the smallest non zero real number so that $\exp(i\zeta a)=1$. It follows a reflexion always exists (by passing to the complexified group) but is not necessarily an element of $\mathcal G$.

\begin{definition}{\sc [Flip assumption]}\label{def:flip}
We say that that  the ${\skd}$-triple $\slt=(a,x,y)$ in ${\G}$ satisfies the {\em flip assumption} if  $\slt$ is even and   there exists  a reflexion $\bJ_0$ which is an inner automorphism, which belongs to the connected component of the identity of ${\ms Z}({\ms Z}(a))$ of $a$ in $\ms G$.
\end{definition}\index{Flip assumption}
In the regular case, we have a weaker assumption:
\begin{definition}{\sc [Regular flip assumption]}\label{def:flipr}
If the even ${\skd}$-triple $\slt$ is regular, we say that $\slt$ satisfies the {\em  regular flip assumption} if  $\slt$ is even and   there exists  a reflexion $\bJ_0$ which belongs to the connected  component of the identity of ${\ms Z}(a)$ in $\G$.
\end{definition}\index{Flip assumption}
The flip assumption for the  ${\skd}$-triple $(a^0,x^0,y^0)$ in $\mk g$ will only be assumed in order to prove the Even Distribution Theorem  \ref{theo:even}.

In paragraph \ref{sec:flip-ex}, we shall give examples of groups and $\slt$-triples satisfying the flip assumption. 

\subsection{Parabolic subgroups  and the flag manifold}\label{sec:para}
We recall standard facts about parabolic subgroups in real semi-simple Lie groups, for references see \cite[Chapter VIII, Â§3, paragraphs 4 and 5]{Bourbaki:owAvyv1m}
\subsubsection{Parabolic subgroups, flag manifolds, transverse flags} 
Let  $\slt=(a,x,y)$ be an $\skd$-triple. Let $\mk g^\lambda$ be the eigenspace associated to the eigenvalue $\lambda$ for the adjoint action of $a$ and let  
$
\mk p=\bigoplus_{\lambda\geq 0}\mk g^\lambda
$.
Let $\ms P$ be the normalizer of $\mk p$\index{$\ms P$}\index{Parabolic subgroup}. By construction,  $\ms P$ is a {\em parabolic subgroup} and its Lie algebra is $\mk p$. 

The  associated {\em flag manifold} is the set $\gp$ \index{$\gp$}\index{Flag Manifold} of all Lie subalgebras of $\mk g$ conjugate to $\mk p$. By construction, the choice of an element of $\gp$ identifies $\gp$ with ${\G}/\ms P$. The group ${\G}$ acts transitively on $\gp$ and the stabilizer of a point -- or {\em flag} -- $x$ (denoted by 
$\ms P_x$) is a  parabolic subgroup. 

Given $a$, let now $\mk q=\bigoplus_{\lambda\leq 0}\mk g^\lambda$. 
By definition, the normalizer $\ms Q$ of $\mk q$ is the {\em opposite parabolic} to $\ms P$ with respect to $a$. Since in $\sld$, $a$ is conjugate to $-a$, it follows that in this special case opposite parabolic subgroups are conjugate.

Two points $x$ and $y$ of $\gp$ are {\em transverse}\index{Transverse flags} if their stabilizers are opposite parabolic subgroups. Then the stabilizer $\ms L$ of the transverse pair of points is the intersection of two opposite parabolic subgroups, in particular its Lie algebra is $\mk g_{\lambda_0}$, for the eigenvalue $\lambda_0=0$. Moreover, $\ms L$ is the Levi part of $\ms P$. 

\begin{proposition}\label{pro:Levi}
The group $\ms L$ is the centralizer of $a$.	
\end{proposition}

\begin{proof} Obviously  $Z(a)$ and $L$ have the same Lie algebra and $\ms Z(a)\subset L$. When $\ms G=\ms{SL}(m,\mathbb R)$ the result follows from the explicit description of $\ms L$ as block diagonal group. In general, it is enough to consider a faithful linear representation of $\ms G$ to get the result.
\end{proof}

\subsubsection{Loxodromic elements}
We say that an element in ${\G}$ is {\em $\ms P$-loxodromic}\index{Loxodromic}, if it has one attractive fixed point and one repulsive fixed point in $\gp$ and these two points are transverse.
We will denote by  $\lambda^-$ the repulsive fixed point of the 
  loxodromic element $\lambda$ and by $\lambda^+$ its attractive fixed point in $\gp$. By construction, for any non zero real number $s$, $\exp(sa)$ is a loxodromic element.

\subsubsection{Weyl chamber}
Let  $\ms C=\ms Z(\ms L)$ be the centralizer of $\ms L$. Since the 1-parameter subgroup generated by $a$ belongs to $\ms L=\ms Z(a)$, it follows that $\ms C\subset \ms L$ and $\ms C$ is an abelian group. Let $\ms A$ be the (connected) split torus in $\ms C$. We now decompose $\mk p^+$ and $\mk p^-$  under the adjoint action of $\ms A$ as 
$
\mk p^\pm  =\bigoplus_{\lambda \in R^\pm }\mk p^\lambda,
$
where $R^+,R^-\subset \ms A^*$, and $\ms A$ acts on $\mk p^\lambda$ by the weight $\lambda$.
The {\em positive Weyl chamber} is 
$$
W=\{b\in \ms A\mid \lambda(b)>0 \hbox{ if }\lambda\in R^+\}\subset \ms A\,.
$$
Observe that $W$ is an open cone that contains $a$.

\section{Tripods and  perfect triangles}\label{sec:tripods}

We define here {\em tripods} which are going to be one of the main tools of the proof. The first definition is not very geometric but we will give more flesh to it.

Namely, we will associate to a tripod a {\em perfect triangle} that is a  certain type of triple of points in $\gp$. We will define various actions and dynamics on the space of tripods. We will also associate to every tripod two important objects in $\gp$: a {\em circle} (a certain class of embedding of $\bf P^1(\mathbb R)$ in $\gp$) as well as a metric on $\gp$.

\subsection{Tripods}\label{subsec:tripods}

Let ${\G}$ be a semi-simple Lie group with trivial center and  Lie algebra $\mk g$. 
Let us fix  a group ${\G}_0$ isomorphic to ${\G}$.

\begin{definition}{\sc [Tripod]}
A {\em tripod} is an isomorphism from ${\G}_0$ to ${\G}$.
\end{definition}\index{Tripod}

So far the terminology  ``tripod'' is baffling. We will explain in the next section how tripods are related to triples of points in a flag manifolds.

We denote by $\mc G$\index{$\mc G$} the {\em space of tripods}. To be more concrete, when one chooses $G_0:=\ms {SL}_n(\mathbb R)$ in the case  of $G=\ms {SL}(V)$, the space of tripods is exactly the set of frames. The space of tripods $\mc G$ is a left principal  ${\Aut (\G)}$-torsor as well a right principal  $\Aut (\G_0)$-torsor where the actions are defined respectively by post-composition and pre-composition. These two actions commute.

\subsubsection{Connected components}

Let us fix a tripod $\xi_0\in\mc G$, that is an isomorphism $\xi_0:\ms  G_0\to {\G}$. Then the map  defined from ${\G}$ to $\mc G$ defined by $g\mapsto g\cdotp \xi_0\cdotp g^{-1}$, realizes an isomorphism from ${\G}$ to the connected component of $\mc G$ containing $\xi_0$. Obviously $\Aut ({\G})$ acts transitively on $\mc G$. We thus obtain

\begin{proposition}
	Every connected component of $\mc G$ is identified (as a ${\G}$-torsor) with ${\G}$. Moreover, the number of connected components of $\mc G$ is equal to the cardinality of $\operatorname{Out}({\G})$.
\end{proposition}

\subsubsection{Correct $\msl$-triples and circles} \label{def:groups}

Throughout this paper, we fix an ${\skd}$-triple $\slt_0=(a_0,x_0,y_0)$ in $\mk g_0$. \index{$\slt_0$} Let $\car_0$ be a  Cartan involution  that extends the standard Cartan involution of $\sld$, that is so that 
\begin{eqnarray}
	\car_0 (a_0,x_0,y_0)&=&(-a_0,y_0,x_0)\label{eq:xychev}\ .
\end{eqnarray}
 Let then 
 \begin{itemize}
 \item $\ms S_0$ be the connected subgroup of ${\G}_0$ whose Lie algebra is generated by $\slt_0$.\index{$\ms S_0$
} The group $\ms S_0$ is isomorphic either to $\ms{SL}_2(\mathbb R)$ or $\ms{PSL}_2(\mathbb R)$. \item   $\ms Z_0$ be  the centralizer of $(a_0,x_0,y_0)$ in ${\G}_0$,\index{$\ms Z_0$}
 \item $\ms L_0$\index{$\ms L_0$}  be the centralizer of $a_0$,
 \item   $\ms P^+_0$ be  the parabolic subgroup associated to $a_0$ 
 in ${\G}_0$ and $\ms P^-_0$ the opposite parabolic 
 \item  $\ms N_0^\pm$ be the respective unipotent radicals of  $\ms P^\pm_0$.
 \end{itemize}

 \begin{definition}\label{def:corr-triple}{\sc [Correct $\msl$-triples]}
 	A {\em correct $\msl$-triple}\index{Correct $\msl$-triple} --with respect to the choice of $\slt_0$ -- is the image of $\slt_0$ by a tripod $\tau$. The space of correct $\msl$-triple forms an orbit under the action of $\Aut (\G)$ on the space of conjugacy classes of $\msl$-triples.

\end{definition}
 
 A correct $\msl$-triple $\slt$ is thus identified with an embedding $\xi^\slt$\index{$\xi^\slt$} of $\slt_0$ in ${\G}$ in a given orbit of $\Aut (\G)$.  
 \begin{definition}{\sc [Circles]}
 The {\em circle map}\index{Circle map, c
 Circle} associated to the correct $\msl$-triple $\slt$ is the unique  $\xi^\slt$-equivariant map $\phi^\slt$ from $\bf{P}^1(\mathbb R)$ to $\gp$. The image of a circle map is a {\em circle}.\label{def:circle}	
 \end{definition}
Since we can associate a correct $\skd$-triple to a tripod, we can associate a circle map to a tripod.

We define a right $\sld$-action on $\mc G$  by restricting the ${\G}_0$ action to $\ms S_0$.

\begin{definition}{\sc [Coplanar]}
	Two tripods are {\em coplanar} if they belong to the same $\sld$-orbit.
\end{definition}\index{Coplanar tripods}

\subsection{Tripods and perfect triangles of flags}
This paragraph will justify our terminology. We introduce perfect triangles which generalize ideal triangles in the hyperbolic plane and relate them to tripods.
\begin{definition}{\sc [Perfect triangle]}
Let $\slt=(a,x,y)$ be a correct $\msl$-triple.
The associated {\em perfect triangle} is 	the triple of flags $t_\slt\defeq(t^-,t^+,t^0)$  which are  the attractive fixed points of the 1-parameter subgroups generated respectively by  $-a$,  $a$ and  $a+2y$.
We denote by $\mc T$ the space of perfect triangles.\index{$\mc T$}
\end{definition}\index{Perfect triangle}

We represent in Figure
(\ref{fig:Trip}) graphically a perfect triangle $(t^-,t^+,t^0)$ as a triangle
whose vertices are $(t^-,t^+,t^0)$ with an arrow from $t^-$ to
$t^+$.

\begin{figure}[h] 
  \begin{center}
    \includegraphics[width=1.5in]{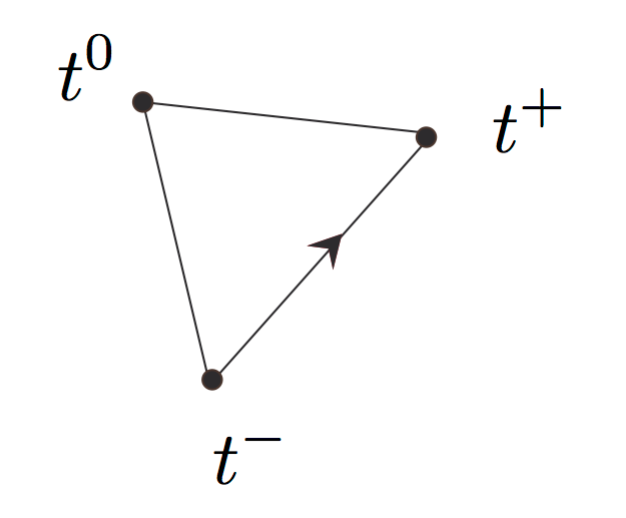}
    \caption{A perfect triangle}\label{fig:Trip}
  \end{center}
 
\end{figure} 
If ${\G}=\sld$, then the perfect triangle associated to the standard $\msl$-triple $(a_0,x_0,y_0)$ described in equation \eqref{def:matrix}  is $(0,\infty,1)$, the perfect triangle associated to $(a_0,-x_0,-y_0)$ is $(0,\infty,-1)$. As a consequence 

\begin{definition}{\sc [Vertices of a tripod]}
 Let $\phi_\tau$ be the circle map associated to a tripod. The {\em set of vertices}   associated to $\tau$  is the perfect triangle  $\partial\tau\defeq\phi_\tau(0,\infty,1)$.\index{Vertices of a tripod}\index{$\partial\tau$}
 \end{definition}
 Observe that any triple of distinct points in a circle is a perfect triangle and that, if two tripods are coplanar, their vertices lie in the same circle.

\subsubsection{Space of perfect triangles}

The group  ${\G}$ acts on the space of tripods, the space of $\msl$-triples and the space of perfect triangles.  

\begin{proposition}{\sc [Stabilizer of a perfect triangle]} 
Let $t=(u,v,w)$ be a  perfect triangle associated to a correct  $\skd$-triple $\slt$. Then the stabilizer of $t$ in ${\G}$ is  the centralizer $\ms Z_\mk s$ of $\mk s$.
\end{proposition}

\begin{proof} Let $\xi$, $u$, $v$ and $w$ be as above. Denote by $L_{x,y}$ the stabilizer of a pair of  transverse points $(x,y)$ in $\gp$. Let also $A_{x,y}=L_{x,y}\cap \ms S$, where $\ms S$ is the group generated by $\skd$ . Observe that $A_{x,y}$ is a 1-parameter subgroup. By Proposition \ref{pro:Levi}, $L_{x,y}$ is the centralizer of $A_{x,y}$. Now given three distinct points in the projective line, the group generated by the three diagonal subgroups $A_{u,v}$, $A_{v,w}$ and $A_{u,w}$ is $\sld$. Thus the stabilizer of a perfect triangle  is the centralizer of $\mk s$, that is $\ms Z_\mk s$.
\end{proof}

\begin{corollary}

\begin{enumerate}
	\item The map $\mk s\mapsto t_{\mk s}$ defines a ${\G}$-equivariant  homeomorphism from the space  of correct triples to the space of perfect triangles.
	\item  We have $\mc T=\mc G/\ms Z_0$ and the map $\partial:\mc G\to \mc T$ is a (right) $\ms Z_0$-principal bundle.
\end{enumerate}
	
\end{corollary}

A perfect triangle $t$, then defines a correct $\skd$-triple and thus an homomorphism denoted 
$\xi^t$ from $\sld$ to ${\G}$.

It will be convenient in the sequel to describe a tripod $\tau$ as a quadruple  $(H,t^-,t^+,t^0)$, where $t=(t^-,t^+,t^0)\eqdef \partial \tau$ is a perfect triangle and $H$ is the set of all tripods coplanar to $\tau$.
We write
$$
\partial \tau=(t^-,t^+,t^0),\ \ \partial^- \tau=t^-,\,\ \partial^+ \tau=t^+,\ \partial^0 \tau=t^0.
$$

\subsection{Structures and actions}\label{sec:act}

We have already described commuting left $\Aut ({\G})$ and right  $\Aut ({\G}_0)$ actions on $\mc G$ and in particular of $\ms G$ and $\ms G_0$.

Since $\ms Z_0$ is the centralizer of $\slt_0$, we also obtain a right action of $\sld$ on $\mc T$, as well as a left ${\G}$-action, commuting together.

We summarize the properties of the actions (and specify some notation) in the following list.

\begin{enumerate}\item Actions of ${\G}$ and ${\G}_0$\begin{enumerate}
 \item \label{act1} the transitive left ${\mathsf G}$-action on $\mc T$ is given -- in the interpretation of triangles -- by $
  g(f_1,f_2,f_3)\defeq (g(f_1),g(f_2),g(f_3))$. 
  Interpreting,  perfect triangles  as morphisms $\xi$ from $\sld$ to ${\G}$ in the class of $\rho$, then 
  $
  (g\cdotp \xi)(x)= g\cdotp\xi(x)\cdotp g^{-1}
  $.
 \item  The (right)-action of an element $b$ of ${\G}_0$ on $\mc G$  is denoted by $R_b$. 
\end{enumerate}
We have the relation 
 $
R_g\cdotp \tau= \tau(g)\cdotp\tau
$.

  \item The right $\sld$-action on $\mc G$ and $\mc T$ gives rises to a flow, an involution and an order 3 symmetry as follows;
\begin{enumerate} 
\item \label{act2} The {\em shearing flow} $\{\varphi_s\}_{s\in\mathbb R}$ \index{$\varphi_s$} is given by $\varphi_s\defeq R_{\exp(sa_0)}$ on $\mc G$ .
   -- See Figure (\ref{fig:psiT}). if we denote by $\xi$ the embedding of $\ms{SL}(2,\mathbb R)$ given by the perfect triangle $t=(t^-,t^+,t^0)$, then
   \begin{eqnarray*}
\varphi_s(H,t^-,t^+,t^0)&\defeq& \left(H,t^-,t^+,\exp(sa)\cdotp t^0)\right)\ ,\end{eqnarray*}
where $a=\T \xi(a^0)$ and $\T f$ denote the tangent map to a map $f$.
We say that $\varphi_R(\tau)$ is {\em $R$-sheared} from $\tau$ \index{Shear}.

\item \label{act4} \index{$\sigma$} The {\em reflection}\index{$\sigma$}  $\sigma:t\mapsto \overline t$ is given on $\mc GG$ by
$ \overline{\tau}=\tau\cdotp\sigma$, 
where $\sigma\in\sld$ is the involution defined by $\sigma(\infty,0,1)=(0,\infty,-1)$. For the point of view of tripods via perfect triangles
$$
\overline{(H,t^+,t^-,t^0)}=(H,t^-,t^+, s^0)\ ,
$$
where $t^-,t^-,t^0,s^0$  form a harmonic division on a circle  -- See Figure (\ref{fig:psiT}). With the same notation the involution on $\mc T$ is given by $
\overline{(t^+,t^-,t^0)}=(t^-,t^+, s^0)
$.
\item \label{act5} The  {\em rotation} $\omega$ of order 3 -- see Figure
  (\ref{fig:omT}) -- is defined on $\mc G$ by
$
\omega(\tau)=\tau\cdotp r_\omega\ .
$
where $r_\omega\in\psld$ is defined by $r_\omega(0,1,\infty)=(1,\infty,0)$. For the point of view of tripods via  perfect triangles
$$
\omega(H,t^-,t^+,t^0)=(H,t^+,t^0,t^-)\ ,
$$
Similarly the action of $\omega$ on $\ms T$ is given by $
\omega(t^-,t^+,t^0)=(t^+,t^0,t^-)
$.
\end{enumerate}

\item \label{act3} Two foliations ${\mc U}^-$ and ${\mc U}^+$ on $\mc G$ and $\mc T$ called respectively the {\em stable}
  and {\em unstable} foliations. \index{Stable and unstable foliations}\index{${\mc U}^-$}
  The leaf of ${\mc U}^\pm$ is  defined as the right orbit of respectively $\ms N^+_0$ and $\ms N^-_0$ (normalized by $Z_0$) and alternatively by 
$$
{\mc U}^\pm_\tau:=\mathsf U^\pm(\tau),
$$
where ${\mathsf U}^\pm(\tau)$ is the unipotent radical of the stabilizer of 
$\partial^\pm \tau$ under the left action of ${\G}$. We also define the {\em central stable} and {\em central
  unstable } foliations by the right actions of respectively $\ms P^\pm_0$ or alternatively by
$$
{\mc U}^{\pm,0}_\tau:={\mathsf U}^{\pm,0}(\tau),
$$ 
where ${\mathsf U}^{\pm,0}(t)$ is the stabilizer of $\partial^\pm \tau$ under the left action of ${\G}$. Observe that ${\mathsf U}^{\pm,0}(t)$ are both conjugate to $\ms P_0$.

\item\label{act-cent} A foliation, called the {\em central foliation}, $\mc L^0$ whose leaves are  the right orbits of $\ms L_0$ on $\mc G$, naturally invariant under the action of the flow $\{\varphi_s\}_{s\in\mathbb R}$. Alternatively, 
$$
{\mc L}^{0}_\tau={\mathsf L}^{0}(\tau),
$$ 
where ${\mathsf L}^0(\tau)$ is the stabilizer in ${\G}$ of $(\partial^+\tau,\partial^-\tau)$.
\end{enumerate}
\begin{figure}[h]
 \centering
\begin{subfigure}[h]{0.6\textwidth}   \begin{center}
    \includegraphics[width=\textwidth]{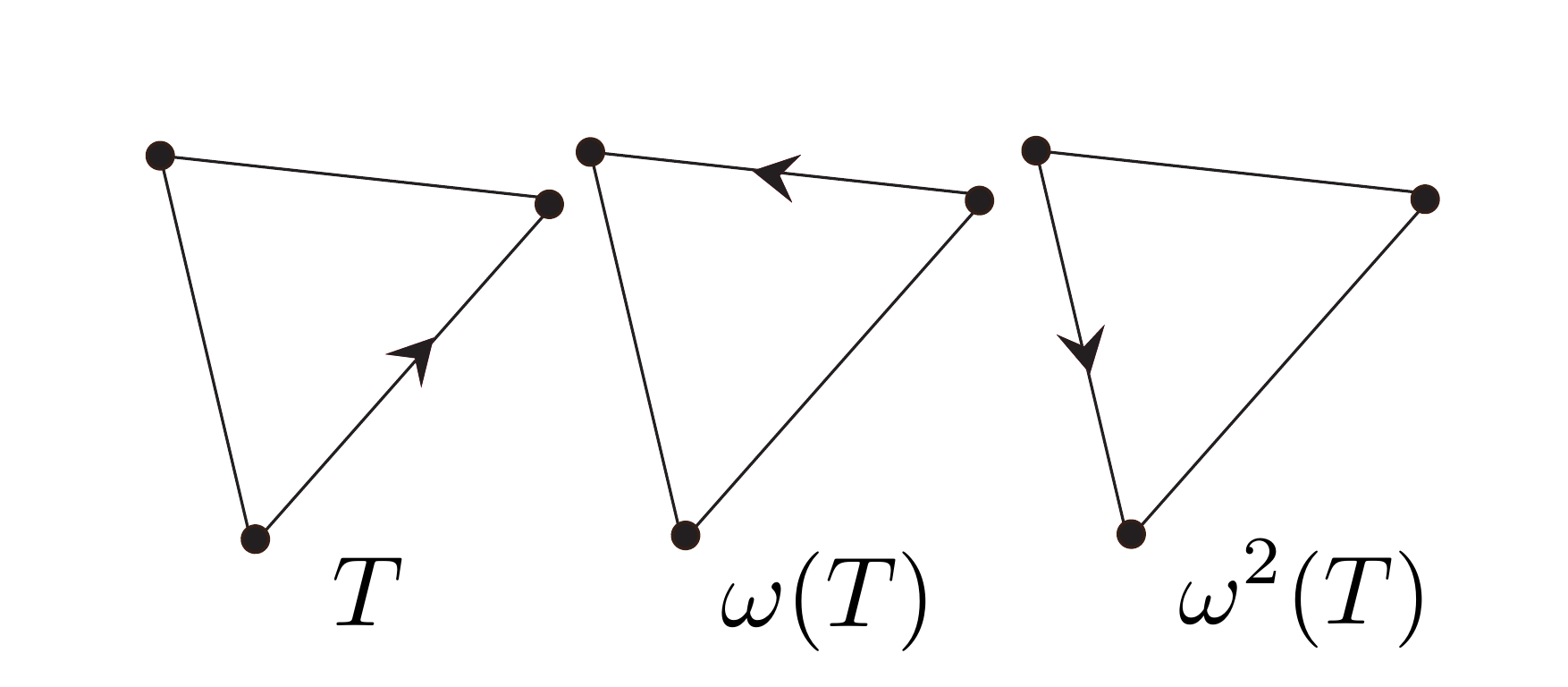}
    \caption{Permutation of order 3}\label{fig:omT}
  \end{center}
  \end{subfigure}
\quad    
 \begin{subfigure}[h]{0.3\textwidth}
 \centering
  \includegraphics[width=\textwidth]{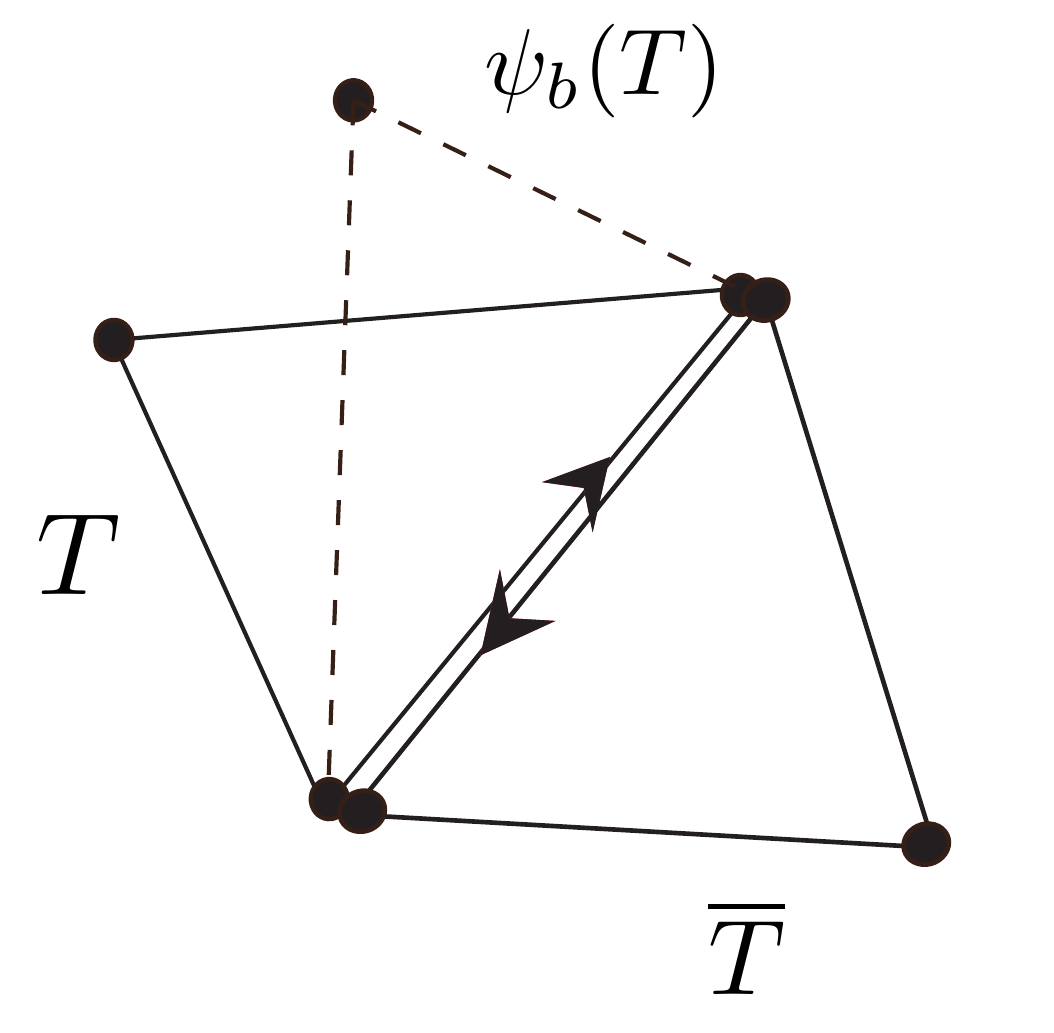}
    \caption{Flow and involution}\label{fig:psiT}
 \end{subfigure}  
 \caption{Some actions}
\end{figure}

Then we have
\begin{proposition}\label{pro:bas} \label{gpsi}\label{gomega}\label{phibar} The following properties hold:
  \begin{enumerate}[(i)]
  \item the action of ${\mathsf G}$  commutes with the flow $\{\varphi_s\}_{s\in\mathbb R}$, the involution $\sigma$ and the permutation $\omega$.
\item For any real number $s$ and tripod $\tau$,  $
    \overline{\varphi_s(\tau)}=\varphi_{-s}(\overline \tau)$. 
\item The foliations ${\mc U}^+$ and ${\mc U}^-$ are invariant
  by the left action of ${\mathsf G}$. 
\item Moreover the leaves of ${\mc U}^+$ and
  ${\mc U}^-$ on  $\mc G$ are respectively uniformly contracted (with respect to any left ${\G}$-invariant Riemannian  metric) and dilated by
  the action of $\{R_{\exp(tu)}\}_{t\in\mathbb R}$ for $u$ in the
  interior of the positive Weyl chamber and $t>0$.\label{pro:bas:cont}
\item The flow $\{\varphi_t\}_{t\in\mathbb R}$ acts by isometries along the leaves of  $\mc L_0$. \label{pro:iso-central}
\item We have $\tau\in{\mc U}^{0,-}_{\eta}$  if and only if $\partial^-\tau=\partial^-\eta$. \label{lastpoint}
\end{enumerate}
\end{proposition}

\begin{proof} The first three assertions are immediate. 

Let us  choose a tripod  $\tau$ so that $\mc G$ is  identified respectively with ${\G}=\G_0$ . If $d$ is a left invariant metric associated to a norm $\Vert\cdotp\Vert$ on $\mk g$, the image of $d$ under the right action of an element $g$ is associated to the norm $\Vert\cdotp\Vert_g$ so that  
$\Vert u\Vert_g=\Vert \operatorname{ad}(g)\cdotp u\Vert$.  The fourth and fifth assertion follow from that description. 
 
For the last assertion, ${\mc U}^{0,-}_\tau={\mc U}^{0,-}_{\sigma}$, if and only if the stabilizer of $\partial^-\tau$ and $\partial^-\sigma$ are the same. The result follows \end{proof}

\begin{corollary}\label{coro:contractleaf}{\sc[Contracting along leaves]}
	For any left invariant Riemannian metric $d$ on $\G$, there exists a constant $\bM$ only depending on $\ms G$ so that if $\epsilon$ is small enough, then for all positive $R$, the following two properties hold
\begin{eqnarray*}
		 d(u,v)\leq \epsilon\ , \ d(\varphi_R(u),\varphi_R(v))\leq \epsilon 
	&\implies& \forall t\in [0,R], \  d(\varphi_t(u),\varphi_t(v))\leq \bM\epsilon\ , \\
	\partial^-u=\partial^-v, \ d(u,v)\leq \epsilon&\implies&\forall t<0, \  d(\varphi_t(u),\varphi_t(v))\leq \bM\epsilon\ .
\end{eqnarray*}
\end{corollary}

\subsubsection{A special map}\label{mapK}
\index{$K$}
We consider the map $K$ -- see Figure (\ref{fig:mapK}) -- defined from
$\mc T$ or $\mc G$ to itself by
$$
K(x):={\omega(\overline x)}.
$$
\begin{figure}[h] 
  \begin{center}
    \includegraphics[width=1.5in]{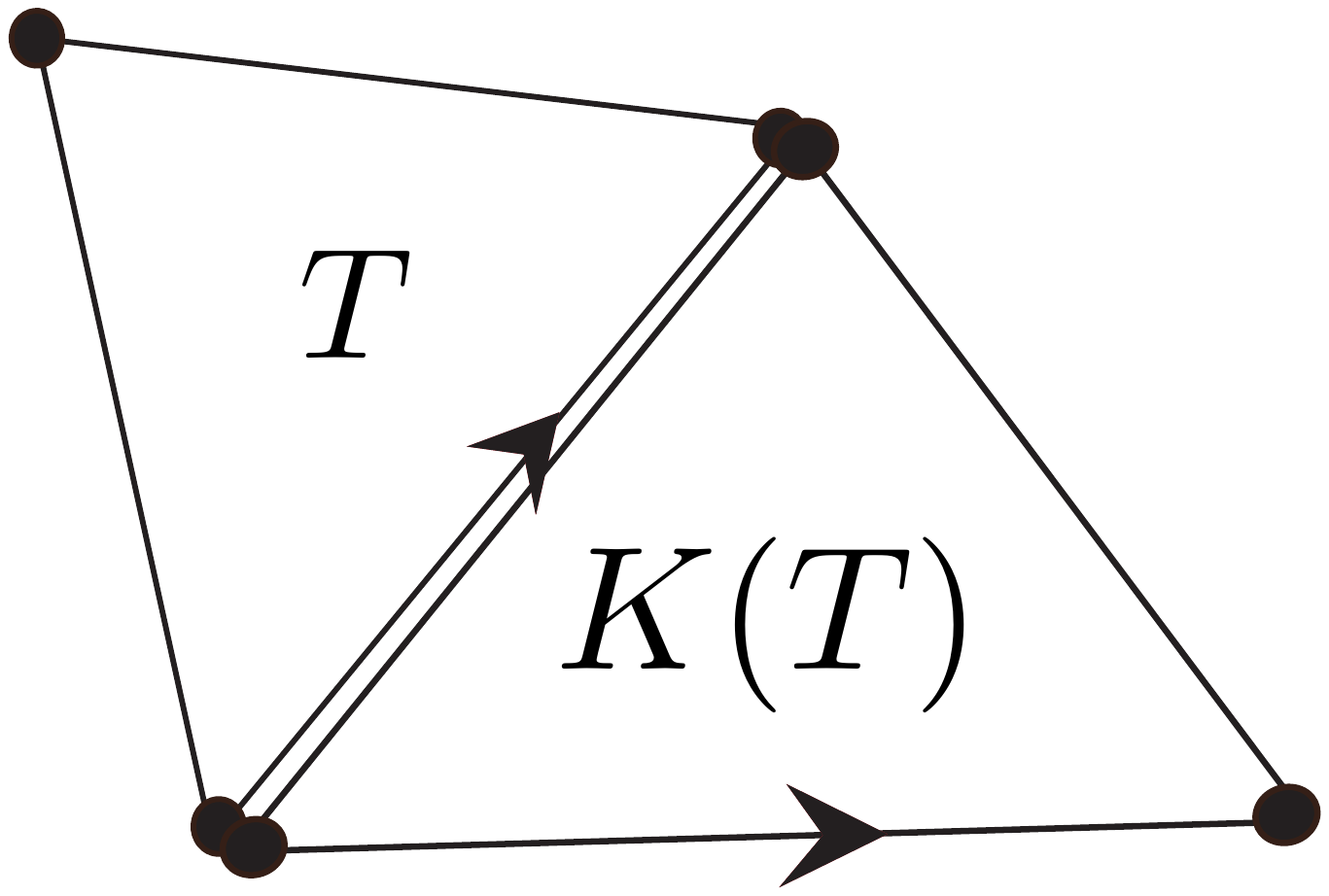}
    \caption{The map $K$}\label{fig:mapK}
\end{center}
   
\end{figure}
Later on, we shall need the following property of this map $K$.

\begin{proposition}\label{pro:Kpreserv}
  For any $(x,y,z)$ in $\mc T$, $K(x,y,z)=(x,t,y)$ for some $t$ in $\gp$.
  The map $K$ preserves each leaf of the foliation ${\mc U}^{0,-}$.
\end{proposition}

\begin{proof}
This follows from the point \eqref{lastpoint} in Proposition \ref{pro:bas}. 	
\end{proof}

\subsection{Tripods, measures  and metrics}\label{sec:tritrimet}
 Let us equip once and for all $\mc G$ with a Riemannian metric $d$ invariant under the left action of $\ms G$, as well as the action of $\omega$.  We will denote by $d_0$ the metric on $\ms G_0$ so that $d(\tau \cdotp  g, \tau \cdotp h)=d_0(g,h)$ for all tripods $\tau$ and observe that $d_0$ is left invariant. The associated Lebesgue measure is now both left invariant by $\Aut(\G)$ and  right invariant by  $\Aut(\G_0)$.

We denote by $\sg$ the symmetric space of ${\G}$ seen as the space of Cartan involutions of $\mk g$. Let us first recall some facts about the totally geodesic space $\sg$.

Let $\ms H$ be a subgroup of $\ms G$. The $\ms H$-orbit of a Cartan involution $i$, so that $i(\mk h)=\mk h$,  is a  totally geodesic subspace of $\operatorname{Sym}(\ms G)$ isometric to $\operatorname{Sym}(\ms H)$ -- we then say {\em of type $\ms H$}.  

Any  two totally geodesic spaces $H_1$ and $H_2$ of the same type are {\em parallel}: that is for all $x_i\in H_i$, $\inf(d(x_i, y)\mid y\in H_{i+1})$ is constant and equal by definition to the distance  $h(H_1, H_2)$. 

The space of parallel totally geodesic subspaces to a given one is isometric to $\operatorname{Sym}(\ms Z)$ if $\ms Z$ is the centralizer of $\ms H$, and in particular reduced to a point if $\ms Z$ is compact.

\subsubsection{Totally geodesic hyperbolic planes}

By assumption \eqref{eq:xychev}, if $\tau $ is a tripod, the Cartan involution
$$\car_\tau\defeq\tau\circ\car_0\circ\tau^{-1}$$
send the correct $\msl$-triple $(a,x,y)$ associated to the tripod $\tau$ to $(-a,y,x)$. It follows that the image of a right $\sld$-orbit gives rise to a totally geodesic embedding of the hyperbolic plane denoted $\eta_\tau$ and that we call {\em correct} and which is equivariant under the action of a correct $\sld$. 

Observe also that a totally geodesic embedding of $\hh$ in $\sg$ is the same thing as a totally geodesic hyperbolic plane $H$ in $\sg$ with three given points in the boundary at infinity in $H$.

Let us consider  $\mc H$\index{$\mc H$} the space of correct totally geodesic maps from $\hh$ to the symmetric space $\sg $.

\begin{proposition}\label{pro:HT}  The space $\mc H$ is equipped with a transitive action of $\Aut (\G)$ and  a right action of $\sld$. 

We have also have  $\sld\times{\G}$ equivariant maps 
	\begin{eqnarray}
	&\mc G\to\mc H\to \mc T,&\cr
	&\tau\mapsto \eta_\tau\mapsto \partial\tau&	
	\end{eqnarray}
	so that the composition is the map $\partial$ which associates to a tripod its vertices. Moreover if the centralizer of the correct $\skd$-triple is compact then $\mc H=\mc T$. \index{$\eta_\tau$}

\end{proposition}
\begin{proof} 

We described above that map $\tau\mapsto\eta_\tau$. By construction this map is $\sld\times{\G}$ equivariant.The map $\partial$ from $\mc G$ to $\mc T$ obviously factors through this map.  

If the centralizer of a correct $\sld$ in $\ms G$,  is compact then all correct parallel hyperbolic planes are identical. The result follows.
\end{proof}

From this point of view, a tripod $\tau$ defines 
\begin{enumerate}
	\item A totally geodesic hyperbolic plane $\hh_\tau$\index{$\hh_\tau$} in $S({\G})$, with three preferred points denoted   $\tau(0),\tau(\infty),\tau(1)$ in $\partial_\infty\hh_\tau$,
	\item An $\sld$-equivariant map $\phi^\tau$ from  $\partial_\infty\hh_\tau$ to $\gp$, so that
 $$\phi^\tau\left((\tau(0),\tau(\infty),\tau(1)\right)=\partial \tau.$$
\end{enumerate}

\subsubsection{Metrics, cones, and projection on the symmetric space}
\begin{definition}{\sc [Projection and metrics]}
	 We define the {\em projection} from $\mc G$ to $\sg$ to be the map
	 $$
	 s:\tau\mapsto s(\tau)\defeq\eta_\tau(i)\ .
	 $$\index{$s(\tau)$}
	 In other words, $s(\tau)$ is the orthogonal projection of $\tau(1)$ on the geodesic $]\tau(0),\tau(\infty)[$ -- see figure \eqref{fig:s(tau)}.
	 \begin{figure}
  \centering
  \includegraphics[width=0.4\textwidth]{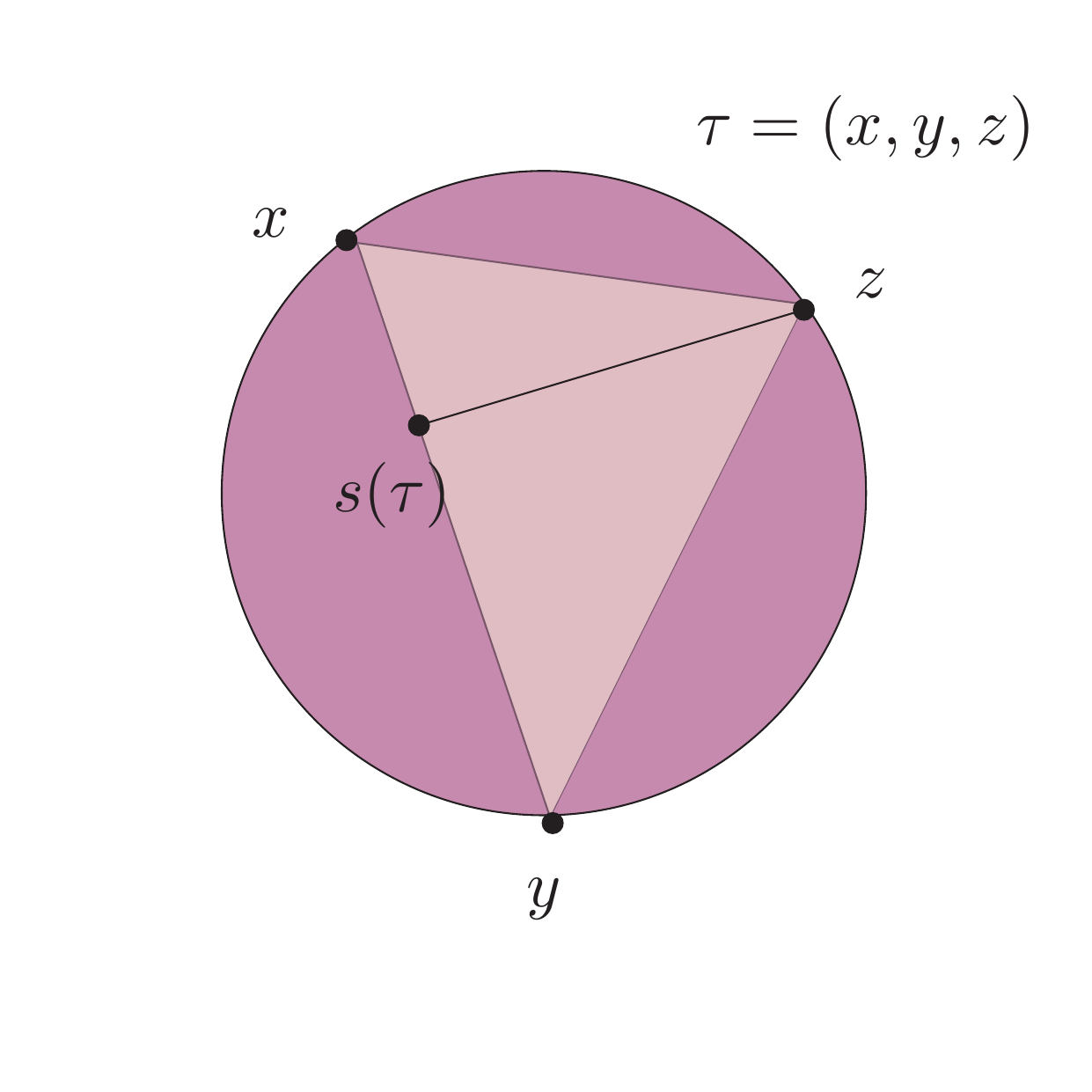}
  \caption{Projection}
  \label{fig:s(tau)}
\end{figure}
	 The metric on $\mk g$ associated to $s(\tau)$ is denoted by $d_\tau$ \index{$d_\tau$} and so are the associated metrics on $\gp$ -- seeing $\gp$ as a subset of the Grassmannian of  $\mk g$ -- and  the right invariant metric on ${\G}$ defined by \begin{equation}
d_\tau(g,h)=\sup\{d_\tau(g(x),h(x))\mid x\in\gp\}. \label{eq:dt-gp}	
\end{equation}
\end{definition}

As a particular case, a triple $\tau$ of three pairwise distinct points in  $\Rp$ defines a metric $d_\tau$ on $\Rp$ -- So that $\Rp$ is isometric to $S^1$ -- that is called the {\em visual metric} of $\tau$. 
The following properties of the assignment $\tau\mapsto d_\tau$, for $d_\tau$ a metric on $\gp$ will be crucial
\begin{enumerate}
	\item For every $g$ in ${\G}$, $d_{g\tau}(g(x),g(y))=d_{\tau}(x,y)$,\label{hyp:1-dist}
	\item The circle map associated to  any tripod $\tau$ is an isometry from $\Rp$ equipped with the visual metric of $(0,1,\infty)$ to  $\gp$ equipped with $d_\tau$. \label{hyp:2-dist}
\end{enumerate}
 
\subsubsection{Elementary properties}\label{sec:elemd}  
\begin{proposition} We have\label{all-cyc-same}\label{pro:comp-stab}
\begin{enumerate}
	\item 	For all tripod $\tau$: 
	$d_\tau=d_{\overline\tau}$.
	\item If the  stabilizer of $\mk s$ is compact, $d_\tau$ only depends on $\partial\tau$.
\end{enumerate}
\end{proposition}
\begin{proof}
	The first item comes from the fact that  $d_\tau$ only depends on $s(\tau)$.
For the second item, in that case  the map $\eta_\tau\mapsto\partial_\tau$ is an isomorphism, by Proposition \ref{pro:HT}. \end{proof}

\begin{proposition}{\sc [Metric equivalences]}\label{A-B}
For every positive numbers $A$ and $\epsilon$, 	there exists a positive number $B$   so that if $\tau,\tau'\in\mc T$ are tripods and $g\in {\G}$, then
$$
d_\tau (g,\id)\leq\epsilon \hbox{ and }d(\tau,\tau')\leq A\implies   d_\tau(g,\id)\leq B\cdotp d_{\tau'}(g,\id)\ .
$$
Similarly,  for all $u,v$ in $\gp$ and $g\in\G$
\begin{eqnarray}
d(\tau,\tau')\leq A&\implies&   d_\tau(u,v)\leq B\cdotp d_{\tau'}(u,v)\ ,\cr
d(\tau,g\tau)\leq\epsilon&\implies& d_\tau(g,\id)\leq B\cdotp d(\tau,g\tau)\ , \label{ineq:contrdtaud}
\end{eqnarray}
\end{proposition}

\begin{proof} Let  $U(\epsilon)$ be a compact neighborhood of $\id$.
The ${\G}$-equivariance of the map $d:\tau\mapsto d_\tau$ implies the continuity of $d$ seen as a map from $\mc G$ to  $C^1(U(\epsilon)\times U(\epsilon))$ -- equipped with uniform convergence. The first result follows. The second assertion follows by a similar argument. For the inequality \eqref{ineq:contrdtaud}, let us fix a tripod $\tau_0$. The metrics 
$$
(g,h)\mapsto d_{\tau_0}(g,h), \ \ (g,h)\mapsto d (h^{-1}\cdotp \tau_0,g^{-1}\cdotp \tau_0)\ ,
$$
are both right invariant Riemannian metrics on ${\G}$. In particular, they are locally bilipschitz and thus there exists some $B$ so that 
\begin{equation*}
d(\tau_0,g\tau_0)\leq\epsilon\implies d_{\tau_0}(g,\id)\leq B\cdotp d(g_0^{-1}\cdotp\tau_0,\tau_0)=B\cdotp d(\tau_0,g\cdotp\tau_0)\ .	\label{ineq:contrdtaud-1}
\end{equation*}
	We now propagate this inequality to any tripod using the equivariance:  writing $\tau=h\cdotp\tau_0$, we get that assuming  $d(\tau,g\cdotp\tau)\leq\epsilon$, then 
$$  d(\tau_0,h^{-1}gh\cdotp\tau_0)=d(h\cdotp\tau_0,gh\cdotp \tau)=d(\tau,g\cdotp\tau) \leq \epsilon\ .$$
Thus according to the previous implication,
$$ d_{\tau_0}(h^{-1}gh,\id)\leq B\cdotp d(\tau_0,h^{-1}gh\cdotp\tau_0)=B\cdotp d(\tau,g\cdotp\tau)\ .
	$$
	The result follows from the equalities
$
d_{\tau_0}(h^{-1}gh,\id)=d_{h\cdotp\tau_0}(gh,h)=d_\tau(g,\id).
$
\end{proof}
As a corollary
\begin{corollary}{\sc [$\omega$ is uniformly Lipschitz]}\label{coro:domega}
	There exists a constant $C$ so that for all $\tau$
	$$
	\frac{1}{C}d_\tau\leq d_{\omega(\tau)}\leq C\cdot d_\tau\ .
	$$
\end{corollary}
 \subsubsection{Aligning tripods}\label{sec:cone contract}
We explain a slightly more sophisticated way to control tripod distances.

Let  $\tau_0$ and $\tau_1$ be two coplanar tripods associated to  a totally geodesic hyperbolic plane $\hh$ and a circle $C$ identified with $\partial_\infty\hh$  so that  $z_1,z_0\in C$.
We say that $(z_0,\tau_0,\tau_1,z_1)$ are {\em aligned} if there exists a geodesic $\gamma$ in $\hh$, passing through $s(\tau_0)$ and $s(\tau_1)$ starting at $z_0$ and ending in $z_1$. In the generic case $s(\tau_0)\not=s(\tau_1)$, $z_1$  and $z_0$ are uniquely determined.

We first have the following  property which is standard for ${\G}=\ms{SL}(2,\mathbb R)$,
\begin{proposition}{\sc [Aligning tripods]} \label{pro:cone-contrac0}
There exist  positive constants $\bK$, $c$ and $\alpha_0$ only depending on ${\G}$  so that if  $(z_0,\tau_{0},\tau_{1},z_1)$ are aligned and associated to a circle $C\subset \gp$ the following holds:  Let  $w\in C$ satisfying $d_{\tau_{1}}\left(w,z_{1}\right)\leq  3\pi/4$, 
then we have
\begin{eqnarray}
d_{\tau_{1}}(w,u)\leq  \alpha_0\ , \ d_{\tau_{1}}(w,v)\leq  \alpha_0\implies	d_{\tau_{0}}(u,v)\leq \frac{\bK}{4} e^{-c d(\tau_0,\tau_1)}\cdot d_{\tau_{1}}(u,v)\ .\label{ineq:gross-hyp}
\end{eqnarray}

\end{proposition}

\begin{figure}
  \centering
  \includegraphics[width=0.3\textwidth]{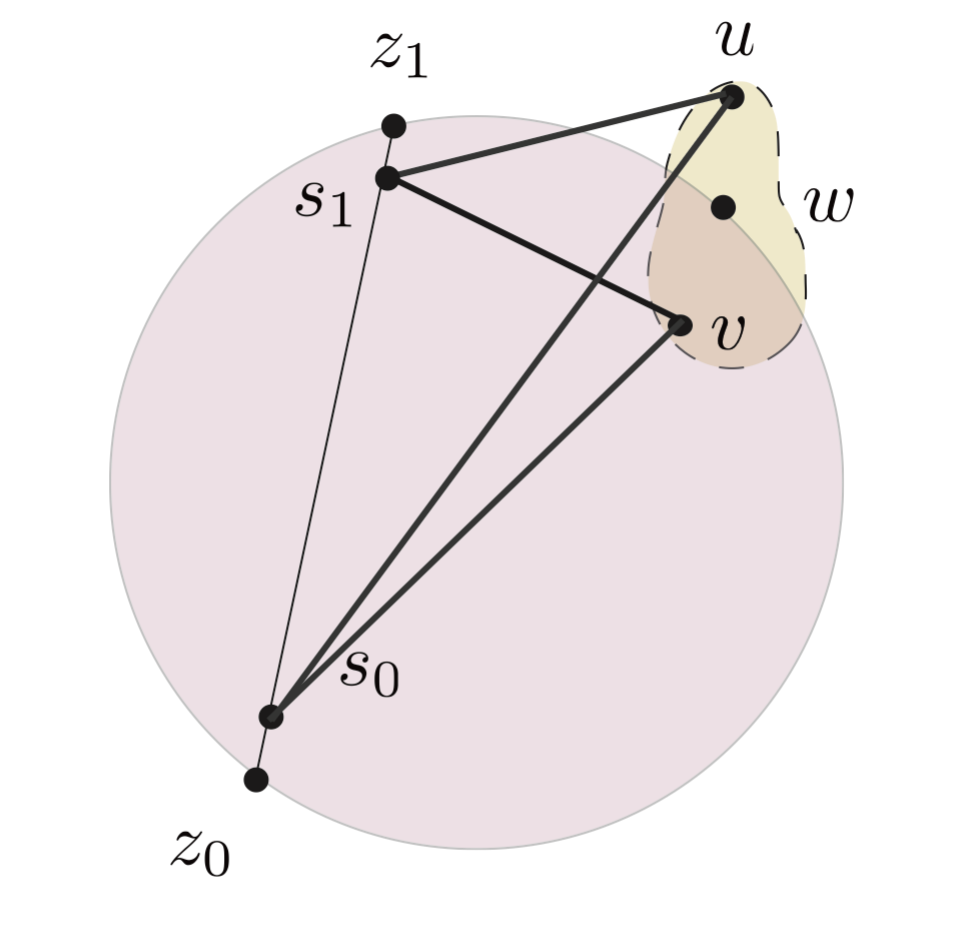}
  \caption{Aligning tripods}
  \label{fig:sszz}
\end{figure}

\begin{proof}  There exists a correct $\msl$-triple $s=(a,x,y)$ preserving  the totally geodesic plane $\hh_{\tau_0}$ so that the 1-parameter group  $\{\lambda_t\}_{t\in\mathbb R}$ generated by $a$  fixes  $C$  and has  $z_{1}$ as an attractive fixed point and $z_{0}$ as a repulsive fixed point in $\gp$. Let $t_{1}$ the positive number defined by $
\lambda_{t_{1}}(s(\tau_{0}))=s(\tau_{1})
$.

Recall  that by construction $d_{\tau}$ only depends on $s(\tau)$.
Let $B\subset C$ be the closed ball of center $z_{1}$ and radius $3\pi/4$ with respect to $d_{\tau_{1}}$. Observe that $B$ lies in the basin of attraction of $H$ and so does $U$ a closed neighborhood of $B$.  In particular, we have that the 1-parameter group $H$ converges $C^1$-uniformly to a constant on $U$. Thus, 
\begin{eqnarray}
\exists K_0, d > 0, \ \ \forall u,v\in U,\ \ \forall t\geq 0,& & \ d_{\tau_{1}}(\lambda_t(u),\lambda_t(v))\leq K_0e^{-dt}\cdotp d_{\tau_{1}}(u,v)\label{ineq:hypprop0}\ .
\end{eqnarray}
Recall that for all $u,v$ in $\gp$, since $
s\left(\lambda_{-t_1}(\tau_1)\right)=s(\tau_0)$.
\begin{equation}
d_{\tau_1}(\lambda_{t_1}(u),\lambda_{t_1}(u))=d_{\lambda_{-t_1}(\tau_1)}(u,v)=d_{\tau_0}(u,v)\label{ineq:hypprop2}
\end{equation}
Finally, there exists $\alpha>0$, only depending on ${\G}$ so that for any $w$ in $B$, the ball $B_w$ of radius $\alpha$ with respect to $d_{\tau_{1}}$ lies in $U$. 
Thus, combining \eqref{ineq:hypprop0} and \eqref{ineq:hypprop2} we get
$$
d_{\tau_{0}}(u,v)\leq K_{0}\cdotp  e^{-dt} d_{\tau_{1}}(u,v).
$$
This concludes the proof of Statement \eqref{ineq:gross-hyp} since there exists constants $B$ and $C$ so that 
$d(\tau_0,\tau_1)\leq Bt_1 + C$.
\end{proof}

\subsection{The contraction and diffusion constants}

The constant $\bK$ defined in Proposition \ref{pro:cone-contrac0} will be called the {\em diffusion constant}\index{$\bf K$}\index{Diffusion constant} and $\bk:=\bK^{-1}$ is called the {\em contraction constant}.\index{$\bk$}\index{Contraction constant}

\section{Quasi-tripods and  finite  paths of quasi-tripods}\label{sec:quas-trip}

 We now want to describe a  coarse geometry in the flag manifold; our main devices will be the following: paths of quasi-tripods and coplanar paths of  tripods. Since not all triple of points lie in a circle in $\gp$, we need to introduce a deformation of the notion of tripods. This is achieved through the definition of {\em quasi-tripod} \ref{def:qua-trip}.

 A {\em coplanar path of  tripods}\index{Coplanar path of  tripods} is just a sequence of non overlapping ideal triangles in some hyperbolic plane such that any ideal triangle have a common edge with the next one. Then a path of quasi-tripods is a deformation of that,  such a path can also be  described as a {\em model} which is deformed by a sequence of specific elements of ${\G}$. 

Our goal is the following. The common edges of a coplanar path of  tripods, considered as intervals in the boundary at infinity of the hyperbolic plane, defines a sequence of nested intervals. We want to show that in certain circumstances, the corresponding chords of the deformed path of quasi-tripods are still nested in the deformed sense that we introduced in the following sections.

One of our main result is then the Confinement Lemma \ref{lem:zigzag} which guarantees squeezing.

\subsection{Quasi-tripods} Quasi-tripods will make sense  of  the notion of a  ``deformed ideal triangle'' . Related notions are defined: swished quasi-tripods, and the  {\em foot map}.

\begin{definition}\label{def:qua-trip}{\sc[Quasi-tripods]}\index{Quasi-tripod}
	An {\em $\epsilon$-quasi tripod}\index{$\epsilon$-quasi tripod}  is a quadruple $\theta=(\dt{\theta},\theta^-,\theta^+,\theta^0)\in\mc G\times\gp^3$ so that
	 	$$
d_{\dt{\theta}}(\partial^+\dt{\theta},\theta^+))\leq\epsilon\ ,\ d_{\dt{\theta}}(\partial^-\dt{\theta},\ \ \theta^-))\leq\epsilon\ ,\ \  d_{\dt{\theta}}(\partial^0\dt{\theta},\theta^0)\leq\epsilon\ .
	$$
		
The set $\partial\theta\defeq\{\theta^+,\theta^-,\theta^0\}$ is the {\em set of vertices}\index{Vertices of an $\epsilon$-quasi tripod}
	\index{$\partial$} of $\theta$ and  $\dt{\theta}$ is the {\em interior} of $\theta$\index{$\dt{\theta}$}
	\index{Interior of an $\epsilon$-triangle}.
	An  $\epsilon$-quasi tripod $\tau$ is {\em reduced}  if $\partial^\pm\dt{\tau}=\tau^\pm$.\index{Reduced $\epsilon$-quasi tripod}
\end{definition}
Obviously
a tripod defines an  $\epsilon$-quasi tripod for all $\epsilon$.
Moreover,  some of the actions defined on tripods in paragraph \ref{sec:act} extend to $\epsilon$-quasi tripods, most notably, we have
 an action of a cyclic permutation $\omega$ of order three on the set of quasi-tripods, given by
$$
\omega(\dt{\theta},\theta^-,\theta^+,\theta^0)=(\omega(\dt{\theta}),\theta^+,\theta^0,\theta^-)\ .
$$
By Corollary \ref{coro:domega}, 
\begin{proposition}\label{pro:symqt}
There is a constant $\bM$ only depending on $\G$, such that  if $\theta$ is an	 $\epsilon$-quasi tripod then, $\omega(\theta)$ is an $\bM\epsilon$-quasi tripod 
\end{proposition}
\subsubsection{A foot map}\index{Foot map}\index{$\Psi$|see{Foot map}}
For any positive $\beta$, let us consider the following $\ms G$-stable set
$$
W_\beta\defeq \{(\tau,a^+,a^-)\mid \tau \in\mc G, a^\pm\in\gp, \,\, d_\tau(a^\pm,\partial^\pm\tau) \leq\beta\}\subset\mc G\times\gp^2.
$$

\begin{lemma}\label{lem:footmap}
There exists  positive numbers  $\beta$ and $\bM_1$, a smooth $\ms G$-equivariant map  $\Psi:W_\beta\to\mc G$, so that 
\begin{enumerate}
	\item $\partial^\pm\Psi(\tau,a^+,a^-)=a^\pm$,
	\item $d(\tau, \Psi(\tau,a^+,a^-))\leq M\cdotp\sup(d_\tau(a^\pm,\partial^\pm\tau))$.
	\item $\Psi$ is $\bM_1$-Lipschitz.
\end{enumerate}
	\end{lemma}

\begin{proof} For a transverse pair $a=(a^+,a^-)$ in $\gp$, let $\mathcal G_{a}$ be the set of tripods $\tau$ in $\mc G$ so that $\partial^\pm\tau=a^\pm$ and $\G_a$ the stabilizer of the pair $a^+$, $a^-$. 
Let us fix (in a ${\G}$-equivariant way) a small enough tubular neighborhood $N_a$ of $\mathcal G_a$ in $\mc G$  for all transverse pairs $a=(a^+,a^-)$ as well as a ${\G}_a$-equivariant projection $\Pi_a$ from $N_a$ to $\mathcal G_a$. By continuity one gets that for $\beta$ small enough, if $(\tau,a^+,a^-)\in W_\beta$ then $\tau\in N_a$. We now define
$$
\Psi(\tau,a^+,a^-)\defeq\Pi_a(\tau)\ .
$$ 
By $\G$-equivariance, $\Psi$ is uniformly Lipschitz.
\end{proof}
	\begin{definition}\label{def:footmap}{\sc [Foot map and feet]}
		A map $\Psi$ satisfying the  conclusion of the lemma is called a {\em foot map}.
For $\epsilon$  small enough, we define the {\em feet}\index{Feet of an $\epsilon$-quasi tripod} $\psi_1(\theta)$, $\psi_2(\theta)$ and $\psi_3(\theta)$ of the $\epsilon$-quasi tripod $\theta=(\dt{\theta},\theta^-,\theta^+,\theta^0)$ as the three tripods which are respectively defined by 
\begin{eqnarray*}\psi_1(\theta)\defeq\Psi\left(\dt{\theta},\theta^-,\theta^+\right)\ , \ 
\psi_2(\theta)\defeq \psi_1(\omega(\theta))\ ,\ 
\psi_3(\theta)\defeq \psi_1(\omega^2(\theta))\ .\end{eqnarray*}
Where $\Psi$ is the foot map defined in the preceding section. 
\end{definition}
By the last item of Lemma \ref{lem:footmap}, for an $\epsilon$, quasi tripod $\theta$
\begin{eqnarray}
	d\left(\psi_i(\theta),\omega^{i-1}\left(\dt{\theta}\right)\right)\leq \bM_1\epsilon\ ,
\end{eqnarray}
 Observe  also that, for $\epsilon$ small enough  there exists a constant $\bM_2$ only depending on ${\G}$, so that  for $\epsilon$ small enough if $\theta$ is an $\epsilon$-quasi tripod then \begin{eqnarray}
	d(\omega(\psi_1(\theta)),\psi_2(\theta))\leq \bM_2\epsilon\ , \  d(\omega(\psi_2(\theta)),\psi_3(\theta))\leq\bM_2\epsilon\ .\label{ineq:dist-feet}
\end{eqnarray}
Using the triangle inequality, this is a consequence of the previous inequality and  the assumption that $\omega$ is an isometry for $d$.

\subsubsection{Foot map and flow}
The following property explains how well the foot map behaves  with respect to the flow action.

\begin{proposition}{\sc [Foot and flow]}\label{pro:flow-foot}

There exists  positive constants $\beta_1$ and $\bM_3$ with the following property. Let $\epsilon\leq\beta_1$,   let $x_0$ in $\mc G$, $x_1:=\varphi_R(x_0)$ for some $R$. Let $a=(a^+,a^-)$  be a transverse pair of flags  $\gp$, so that $d_{x_i}(a^\pm,\partial^\pm x_i)\leq \epsilon$, then 
$$
d(y_1,\varphi_R(y_0))\leq \bM_3\epsilon\ , 
$$
where $y_i=\Psi(x_i,a^+,a^-)$.
\end{proposition}
\begin{proof} In the proof $M_i$ will denote a constant only depending on $\ms G$.

It is enough to prove the weaker result that there exists $z_0$, $z_1=\varphi_R(z_0)$ in $\mc G_a$ so that $d(z_i,x_i) \leq M_7\epsilon$. Indeed, it first follows that $d(z_i,y_i)\leq M_8\epsilon$ by the triangle inequality. Secondly, $\ms G_{a}$ is a central leaf of the foliation and the flow acts by isometries on it (see Property \eqref{pro:iso-central} of Proposition \ref{pro:bas}), it then follows that 
$d(y_1,\varphi_R(y_0))\leq M_9\epsilon$ and the result follows

Observe first that $d(x_i,y_i)\leq M\epsilon$ by definition of a foot map. Assume $R>0$.
	Let $x^\pm=\partial^\pm x_0=\partial^\pm x_1$. Let us first assume that $x^+=a^+$. Thus by the contraction property
	$$
	d(\varphi_R(x_0),\varphi_R(y_0))\leq M_2\epsilon\ .
	$$
	It follows by the triangle inequality that 
	$$
	d(\varphi_R(y_0),y_1)\leq M_3\epsilon\ .
	$$
	Thus this works with $z_0=y_0$, $z_1=\varphi_R(z_0)$. 
	
The same results hold symmetrically  whenever $x^-=a^-$ by taking $z_1=y_1$,  $z_0=\varphi_{-R}(z_1)$. 
	
The general case follows by considering intermediate projections. First (as a consequence of our initial argument) we find  $w_0$ and $w_1=\varphi_R(w_0)$ in $\mc G_{a^+,x^-}$ with $d(w_i,x_i)\leq M_3\epsilon$. 

Applying now the symmetric  argument with the pair $w_0$, $w_1$  and projection on $\mc G_{a^+,x^-}$ we get $z_0$ and $z_1:=\varphi_R(z_0)$ so that $d(w_i,z_i)\leq M_3\epsilon$.    

A simple combination of triangle  inequalities yield the result.
	\end{proof}

\subsubsection{Swishing quasi-tripods}

\begin{definition}{\sc[Swishing quasi-tripods]}\label{def:shear-qt}
The $\epsilon$-quasi-tripod $\theta'$ is {\em $(R,\alpha)$-swished} from the $\epsilon$-quasi tripod $\theta$ if
\begin{enumerate}
\item $\partial^\pm\theta=\partial^\mp\theta'$.
\item The tripods $\psi_1(\theta')$ and  $\overline{\varphi_R(\psi_1(\theta))}$ are  $\alpha$-close.
\end{enumerate}
\end{definition}  

Being swished is a reciprocal condition:
\begin{proposition}
If $\theta'$ is $(R,\alpha)$-swished from $\theta$, then $\theta$ is $(R,\alpha)$-swished from $\theta'$.
\end{proposition} 

\begin{proof} We have 
$
d(\sigma(\varphi_R(\theta)),\theta')=d(\varphi_R(\theta),\sigma(\theta'))
$.
Since $\partial^\pm\theta=\partial^\mp\theta'$, $\sigma(\theta')=(\theta)g$ for some $g\in L_0$. Since by Proposition \ref{pro:bas}\ \eqref{pro:iso-central},
$\varphi_R$ acts by isometries on the orbits of $\ms L_0$, we get 
$$
d(\sigma(\varphi_R(\theta)),\theta')=d(\varphi_R(\theta),\sigma(\theta'))=d\left(\theta,\varphi_{-R}\left(\sigma(\theta')\right)\right).
$$
But, by Proposition \ref{pro:bas} again,
$
\varphi_{-R}\circ\sigma=\sigma\circ\varphi_{R}
$.
The result follows.
	\end{proof}

\subsection{Paths of quasi-tripods and coplanar paths of  tripods}

\begin{figure}[h]
 \centering
 \begin{subfigure}[ht]{0.35\textwidth}
 \centering
 \includegraphics[width=\textwidth]{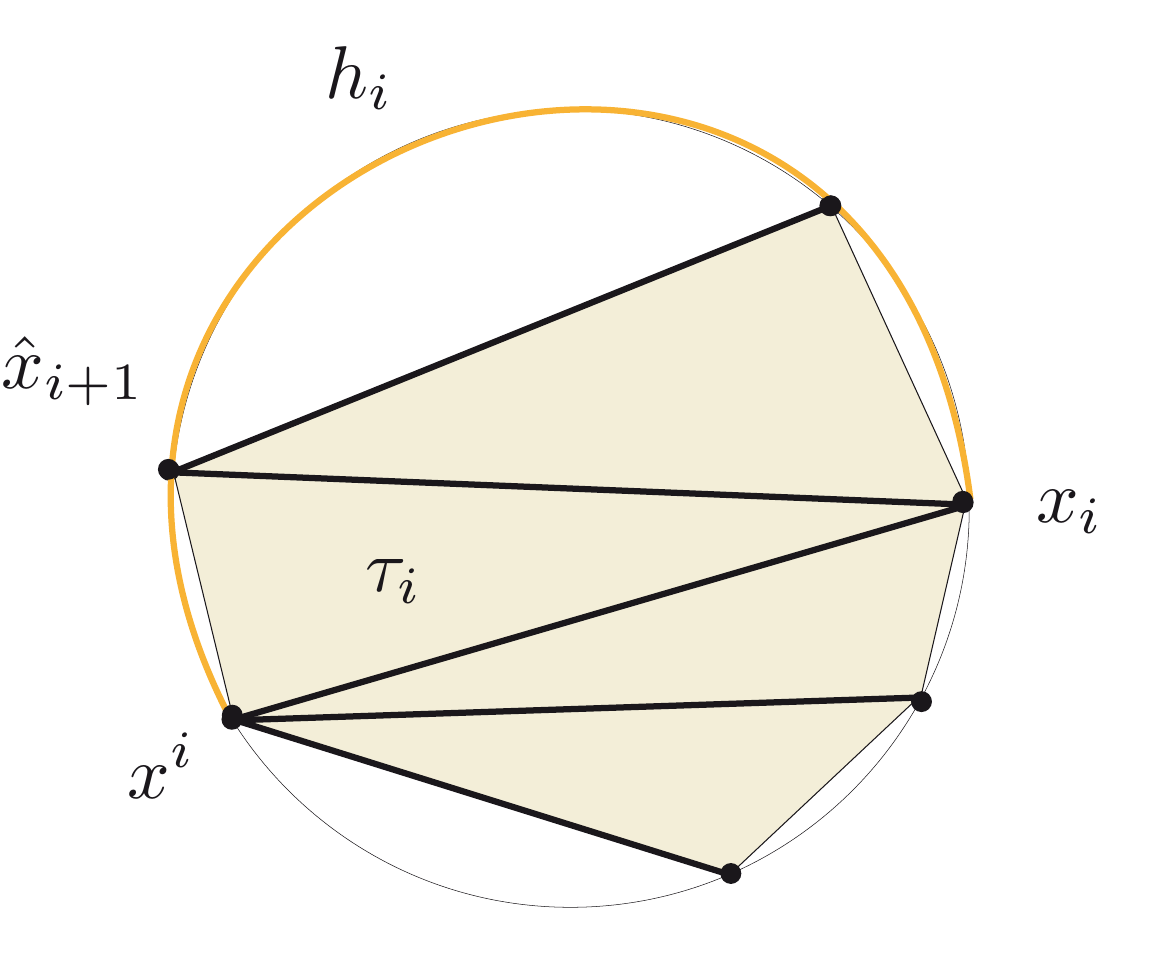}
 \caption{coplanar path of tripods}\label{fig:2a}
 \end{subfigure}
 \ \ \ \
 \begin{subfigure}[ht]{0.35\textwidth}		
 \centering
		\includegraphics[width=\textwidth]{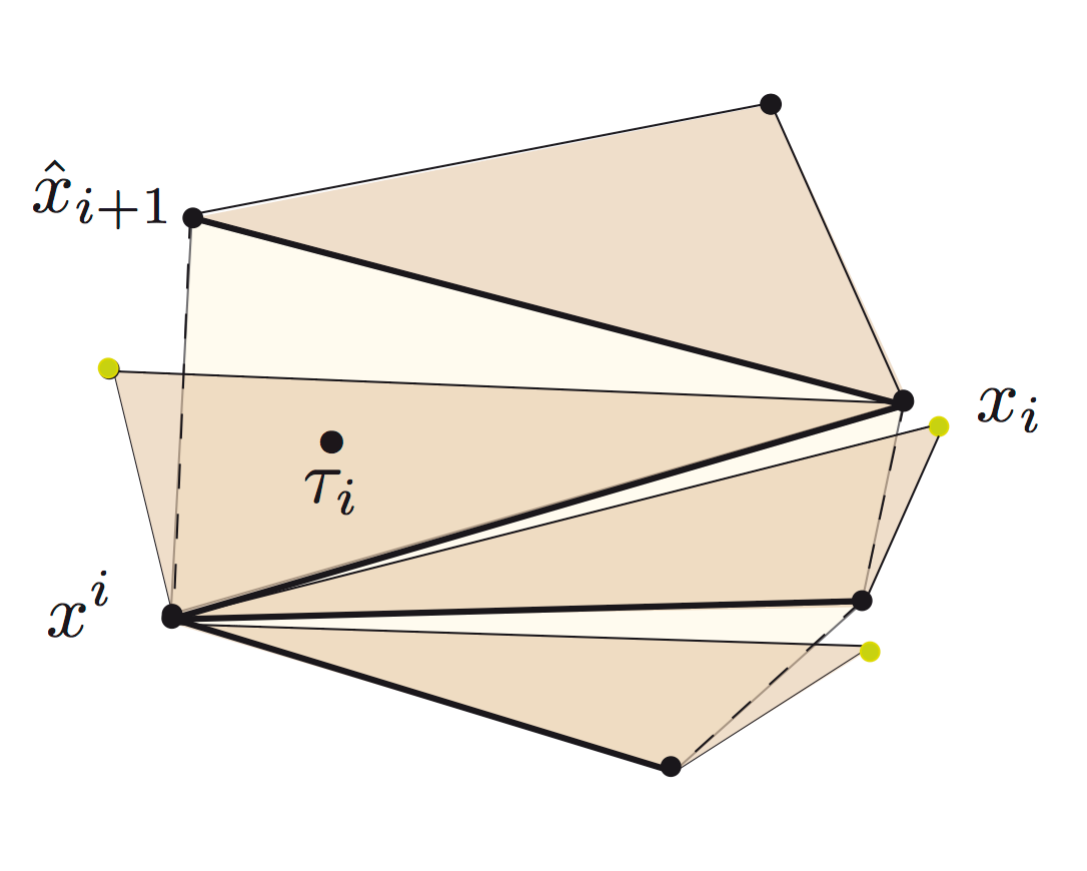}
\caption{path of reduced quasi-tripods}\label{fig:2b}
	\end{subfigure}
\caption{A deformation of a path of quasi-tripods}
 \label{fig:zdz}
\end{figure}
\subsubsection{Swished paths of quasi-tripods and their model}

Let $\vect{R}(N)=(R_0,\ldots, R_{N})$ be a finite sequence of positive numbers.

\begin{definition}{\sc [coplanar paths of  tripods]}
 An {\em $\vect{R}(N)$-swished coplanar path of  tripods} is a sequence of tripods $\vect{\tau}(N)=(\tau_0,\ldots\tau_{N})$ such that $\tau_{i+1}$ is $R_i$-swished from $\omega^{n_i}\tau_i$, where $n_i\in \{1,2\}$. The sequence $(n_1,\ldots,n_N)$ is the {\em combinatorics of the path}\index{Combinatorics of a path}.
\end{definition}
We remark that a coplanar path of  tripods consists of pairwise coplanar tripods and is totally determined up to the action of ${\G}$ by $\vect{R}(N)$ and the combinatorics. 
These coplanar paths of  tripods will represent the model situation and we need to deform them.

\begin{definition}{\sc [paths of quasi-tripods]}
 An {\em $(\vect{R}(N),\epsilon)$-swished path of quasi-tripods}\index{Path of quasi-tripods} is a sequence of  $\epsilon$-quasi tripods $\vect{\theta}(N)=(\theta_0,\ldots\theta_{N})$, 
 and such that $\theta_{i+1}$ is $(R_i,\epsilon)$-swished from $\omega^{n_i}\theta_i$, where $n_i\in \{1,2\}$. The sequence 
 $(n_1,\ldots,n_N)$ is the {\em combinatorics of the path}.

A {\em model} \index{Model of a path} of  an  $(\vect{R}(N),\epsilon)$-swished path of quasi-tripods is an $\vect{R}(N)$-swished coplanar path of  tripods with the same combinatorics.

\end{definition}
Let us introduce some notation and terminology: $\partial\theta_i$, $\partial\theta_{i+1}$  and $\partial\theta_{i-1}$ have exactly one point in common denoted $x_i$ and called the {\em pivot} of $\theta_i$\index{Pivot}.
\rmks Observe that given a path of quasi-tripods, 
\begin{enumerate}
\item There exists some constant $M$, so that any  $(\vect{R}(N),\epsilon)$-swished path of quasi-tripods give rise to an  $(\vect{R}(N), M\epsilon)$-swished path of quasi-tripods with the same vertices but which are all reduced. In the sequel, we shall mostly consider such reduced paths of quasi-tripods.
\item From the previous items, in the  case of reduced path, the sequence of triangles  $(\theta_0,\ldots,\theta_{N})$ is actually determined by the sequence of (not necessarily coplanar) tripods $(\dt{\theta_0},\ldots,\dt{\theta_{N}})$.
\end{enumerate}
One immediately have

\begin{proposition}\label{pro:model}
	Any $(\vect{R}(N),\epsilon)$-swished path of quasi-tripods admits a model which is unique up to the action of 
	${\G}$.
	\end{proposition}

\subsubsection{Coplanar paths of  tripods and sequence of chords} 

To a reduced path of quasi-tripods $\vect{\theta}(N)$ we associate a {\em path of chords}\index{Path of chords} $$\vect{h}(N)=(h_0,\ldots,h_{N})$$ such that $h_i\defeq h_{\dt{\theta_i}}$ has $x_i$ and $x^i$ as extremities. Observe, that the subsequence of triangles  $(\theta_0,\ldots,\theta_{N-1})$ is actually determined by the sequence of chords $(h_0,\ldots,h_{N})$.

In the sequel, by an abuse of language, we shall call the sequence of chords $\vect{h}(N)$ a path of quasi-tripods as well.

Observe that for a coplanar path of  tripods the associated path of chords is so that $(h_i,h_{i+1})$ is nested.

\subsubsection{Deformation of coplanar paths of  tripods}	Let $\vect{\tau}=(\tau_0,\ldots,\tau_{N})$ be a coplanar path of  tripods. 

\begin{definition}{\sc [Deformation of paths]}
	 A {\em deformation of $\vect{\tau}$}\index{Deformation fo a path} is a sequence $v=(g_0,\ldots, g_{N_1})$ with $g_i\in \ms P_{x_i}$, the stabilizer of $x_i$ in ${\G}$, where $x_i$ is the pivot of $\tau_i$.
	 The  deformation is an {\em $\epsilon$-deformation} if furthermore
$ d_{\tau_i}(g_i,\id)\leq  \epsilon$.

	Given a deformation $v=(g_0,\ldots, g_{N-1})$,  the {\em deformed path of quasi-tripods} is  the  path of quasi-tripods $\vect{\theta}^v=(\theta_0^v,\ldots,\theta_{N}^v)$ where
  \begin{eqnarray}
     \hbox{for }&i<N,&\, \theta_i^v=(b_i\tau_i, b_i\tau_i^-, b_i\tau_i^+,b_{i+1}\tau_i^0),\cr
       \hbox{for }&i=N,&\, \theta_N^v=(b_N\tau_N, b_N\tau_N^-, b_N\tau_N^+,b_N\tau_N^0)\ ,
  	\end{eqnarray}
  	where $b_0=\id$ and $b_i=g_0\circ\ldots\circ g_{i-1}$.

\end{definition}

From the point of view of sequence of chords, the sequence of chords associated to the deformed coplanar path of  tripods as above is  
$$\vect{h^v}\defeq(h_0^v,\ldots, h_N^v)\defeq (b_0\cdotp h_0,\ldots, b_N\cdotp h_N),$$
where  $(h_0,\ldots,h_N)$ is the sequence of chords associated to $\vect{\tau}$.

\subsection{Deformation of coplanar paths of  tripods and swished path of quasi-tripods} 
We want to relate our various notions and we have the following  two propositions.

\begin{proposition}\label{pro:modeldef} There exists a constant $\bM$ only depending on ${\G}$, so that 
given an $(\vect{R}(N),\epsilon)$-swished path of reduced quasi-tripods $\vect{\theta}$ with model $\vect{\tau}$, there exists a unique $\bM\epsilon$-deformation $v$ so that $\vect{\theta}=g\cdotp\vect{\tau^v}$ for some $g$ in ${\G}$.
\end{proposition}

\begin{proof}
Given any path of quasi-tripods $\vect{\theta}$. Let $x_i$ be the pivot of $\theta_i$. We know that $\theta_{i+1}$ is $(R_i,\epsilon)$-swished from $\omega^{n_i}\theta_i$.

Let then $\tau_{i}$, so that  $\psi_1(\theta_{i+1})$ is  $R_i$-swished from $\tau_i$,  and symmetrically $\tau_{i+1}$ the tripod $R_i$-swished from $\omega^{n_i}\psi_1(\theta_{i})$.

Since ${\G}$ acts transitively on the space of tripods and commutes with the right $\sld$-action, there exists a unique $g_i\in \ms P_{x_i}$ so
$$
g_i(\omega^{n_i}\psi_1(\theta_i))=\tau_i, \ \ g_i(\tau_{i+1})=\psi_1(\theta_{i+1}),
$$
We have thus recovered $\vect{\theta}$ as a $(g_0,\ldots,g_{N-1})$- deformation of its model. It remains to show that this is an $\bM\epsilon$-deformation, for some $\bM$. 
$$
d(g_i(\omega^{n_i}\psi_1(\theta_i)),\omega^{n_i}\psi_1(\theta_i))
\leq d(\tau_i, \psi_1(\omega^{n_i}\theta_i))+d(\omega^{n_i}\psi_1(\theta_i), \psi_1(\omega^{n_i}\theta_i))\ .
$$
Since $\theta_{i+1}$ is $(R_i,\epsilon)$-swished from  $\omega^{n_i}\theta_{i})$, we have $d(\tau_i, \psi_1(\omega^{n_i}\theta_i))\leq \epsilon$. Moreover, since $\theta_{i}$ is a quasi-tripod,
by inequality \eqref{ineq:dist-feet}: $d(\omega^{n_i}\psi_1(\theta_i), \psi_1(\omega^{n_i}\theta_i))\leq \bM_2\epsilon$. Thus
$$
d(g_i(\omega^{n_i}\psi_1(\theta_i),\omega^{n_i}\psi_1(\theta_i))\leq (\bM_2+1)\epsilon\ .
$$
Then inequality \eqref{ineq:contrdtaud}  and Corollary \ref{coro:domega} yields,
$$
d_{\psi_{1}(\theta_i)}(g_i,\id)\leq C ^2 d_{\omega^{n_i}\psi_{1}(\theta_i)}(g_i,\id)\leq B_0\epsilon.
$$
for some constant $B_0$ only depending on ${\G}$.
Using Proposition \ref{A-B},  this yields that there exists $\bM$ only depending on ${\G}$, so that
$$
d_{\dt{\theta_i}}(g_i,\id)\leq \bM\epsilon.
$$
This yields the result. 	
\end{proof}

\section{Cones, nested tripods and chords}\label{sec:cone}

We will describe geometric notions that generalize the inclusion of intervals in $\bf P^1(\mathbb R)$ (which corresponds to the case of $\sld$):  we will introduce {\em chords} which generalize intervals as well as the notions of squeezing and nesting which 
replace -- in a quantitative way-- the notion of being included for intervals. We will study how nesting and squeezing is invariant under perturbations. 

Our motto in this paper is that we can phrase all the geometry that we need using the notions of tripods and their associated dynamics, circles and the assignment of a metric to a tripod. These will be the basic geometric objects that we will manipulate throughout all the paper.
\subsection{Cones and nested tripods}
\begin{definition}{\sc [Cone and nested tripods]}\index{Cone}\index{$C_\alpha(\tau)$|see{Cone}}
Given a tripod  $\tau$ and a positive number $\alpha$, the {\em $\alpha$-cone} of $\tau$ is the subset of $\gp$ defined by
$$
C_\alpha(\tau)\defeq\{u\in\gp  \,|\, d_\tau(\partial^0\tau,u)\leq\alpha\}.
$$
Let $\alpha$ and $\kappa$  be positive numbers. A pair of tripods $(\tau_0,\tau_1)$ is {\em $(\alpha,\kappa)$-nested}\index{Nested tripods}\index{$(\alpha,\kappa)$-nested} if 
\begin{eqnarray}
C_\alpha(\tau_{1})&\subset& C_{\kappa\cdotp \alpha}(\tau_0)\ ,\\
 \forall 
	u,v\in C_{\alpha}(\tau_1)\ , \ \ \ 
	d_{\tau_0}(u,v)&\leq& \kappa \cdotp d_{\tau_1}(u,v)\ . \label{def:nest1}
\end{eqnarray}	
We write this symbolically as 
$
C_\alpha(\tau_{1})\prec  \kappa\cdotp  C_{\kappa\cdotp \alpha}(\tau_0)\ .
$
\end{definition}

The following immediate transitivity property justifies our symbolic notation.

\begin{lemma}{\sc[Composing cones]}\label{pro:composcone}
Assume $(\tau_0,\tau_1)$ is $(\alpha\cdotp\kappa_2,\kappa_1)$-nested and $(\tau_1,\tau_2)$ is $(\alpha, \kappa_2)$-nested,  then
$(\tau_0,\tau_2)$ is $(\alpha,\kappa_1\cdotp \kappa_2)$-nested. Or in other words
$$
C_\alpha(\tau_2)\prec\kappa_2 C_{\kappa_2\alpha}(\tau_1)\hbox{ and } C_{\kappa_2\alpha}(\tau_1)\prec\kappa_1 C_{\kappa_1\kappa_2\alpha}(\tau_0)\implies C_{\alpha}(\tau_2)\prec \kappa_1\kappa_2 C_{\kappa_1\kappa_2\alpha}(\tau_0)\ .
$$
\end{lemma}
 
\subsubsection{Convergent sequence of cones}
We say a sequence of tripods $\{\tau_i\}_{i\in \{1,\ldots,N\}}$ -- where $N$ is finite of infinite -- defines a {\em $(\alpha,\kappa)$-contracting sequence of cones}\index{Contracting sequence of cones} if for all $i$, the pair  $(\tau_i,\tau_{i+1})$ is $(\alpha,\kappa)$-nested and $\kappa<\frac{1}{2}$

As a corollary of Lemma \ref{pro:composcone} one gets, 
\begin{corollary}{\sc[Convergence Corollary]} \label{coro-conv-nest}There exists a positive constant $\alpha_3$ so that 
If $\{\tau_i\}_{i\in\mathbb N}$ defines an infinite $(\alpha,\kappa)$-contracting sequence of cones, with $\kappa< \frac{1}{2}$ and $\alpha\leq \alpha_3$, then there exists a point $x\in\gp$ called the {\em limit of the contracting sequence of cones}\index{Limit of a sequence of cone} such that
$$
\bigcap_{i=1}^\infty C_\alpha(\tau_i)=\{x\}.
$$ 
Moreover, for all $n$, for all $q$, for all $u,v$ in $C_\alpha(\tau_{n+q})$ we have
\begin{equation}
	d_{\tau_n}(u,v)\leq  \frac{1}{2^{q}}d_{\tau_{n+q}}(u,v)\leq \frac{1}{2^{q-1}}\alpha.\label{eq:convseq}
\end{equation}
 \end{corollary}

We then write $x=\lim_{i\to\infty}\tau_i$. 
\begin{proof} 
This follows at once from the fact that
$
C_\alpha(\tau_{n+p})\prec \frac{1}{2^p} C_{\frac{1}{2^p}\alpha}( \tau_{n}) 
$;
\end{proof}

\subsubsection{Deforming nested cones}
The next proposition will be very helpful in the sequel by proving the notion of being nested is stable under sufficiently small deformations.

\begin{lemma}{\sc[Deforming nested pair of tripods]}\label{pro:def-nt} There exists a constant $\beta_0$ only depending on ${\G}$, such that if $\beta\leq \beta_0$ then if 
\begin{itemize}
	\item The pair of tripod $(\tau_0,\tau_1)$ is $(\beta, \kappa/2)$-nested, with $\beta\leq \beta_0$,
	\item The element $g$ in ${\G}$ is so that  
	$d_{\tau_0}(\id  ,g)\leq \frac{\kappa\cdotp\beta}{2}$,
	\end{itemize}

Then the pair $(\tau_0,g(\tau_1))$ is   $(\beta,\kappa)$-nested
\end{lemma}

\begin{proof} Let  $z=\partial^0\tau_0$. It is equivalent to prove that $(g^{-1}(\tau_0),\tau_1)$ is $(\beta,\kappa)$-nested. Let $u\in C_\beta(\tau_1)\subset C_{\frac{\kappa\cdotp\beta}{2}}(\tau_0)$.
In particular,
$
d_{\tau_0}(u,z)\leq \frac{\kappa\cdotp\beta}{2}$.
It follows that
\begin{eqnarray}
d_{g^{-1}(\tau_0)}(u,g^{-1}(z))=d_{\tau_0}(g(u),z)
\leq d_{\tau_0}(g(u),u)+d_{\tau_0}(u,z)
\leq  \frac{\kappa\cdotp \beta}{2} +  \frac{k\cdotp \beta}{2}=k\cdotp\beta.
\end{eqnarray} Thus 
$$
C_\beta(\tau_1)\subset C_{\frac{k\cdotp\beta}{2}}(\tau_0)\subset C_{k\cdotp\beta}(g^{-1}(\tau_0)).
$$
Moreover for $\beta$ small enough, by Proposition \ref{A-B}, $d_{\tau_0}\leq 2d_{g^{-1}(\tau_0)}$ thus for all $(u,v)\in C_\beta(\tau_1)$

$$
d_{g^{-1}\tau_0}(u,v)\leq  2 d_{\tau_0}(u,v)\leq \kappa d_{\tau_1}(u,v)\ .
$$
Thus  $(g^{-1}(\tau_0), \tau_1)$ is $(\beta,k)$-nested.
\end{proof}
\subsubsection{Sliding out}

\begin{lemma}\label{coro:shcont}
		There exists constants  $k$ and $\delta_0$ depending only on the group ${\G}$, such that if   $\tau_0$ is a tripod $R$-swished from $\tau_1$ then	$$ \forall 
	u,v\in C_{\delta_0}(\tau_1), \ \ \ 
	d_{\tau_0}(u,v)\leq k\cdotp d_{\tau_1}(u,v)\ .
	$$ 
\end{lemma}

\begin{proof} This is an immediate consequence of  Proposition \ref{pro:cone-contrac0}, with  (for $R>0$) $$
z_0=\partial^-\tau_1\ , \ z_1=\partial^+\tau_1\ , \ w=\partial^0\tau_1\ .$$
The case $R<0$ being symmetric.
\end{proof}

\subsection{Chords and slivers} A {\em chord}\index{Chord} is an orbit of the shearing flow. We denote by $h_\tau$ the chord associated to a tripod $\tau$ and denote  $\check h_{\tau}\defeq h_{\sigma(\tau)}$. Observe that all pairs of tripods in $\check h_\tau\times h_\tau$ are coplanar. We also say that $h_\tau$ goes from $\partial^-\tau$ to $\partial^+\tau$ which are its {\em end points}.

The  {\em $\alpha$-sliver} of a chord $H$ is the subset of $\gp$ defined by  \index{Sliver}\index{$S_\alpha(H)$}
$$
S_\alpha(H)\defeq\bigcup_{\tau\in H}C_\alpha(\tau)\subset\gp.
$$
In particular,
$
S_0(H)=\{\partial^0\tau\mid\tau\in H\}
$.
Observe that  two  points $a$ and $b$ in the closure of  $S_0(H)$ define a unique chord $H_{ab}$ which is coplanar to $H$ so that $S_0(H_{a,b})$ is a subinterval of  $S_0(H)$ with end points $a$ and $b$.

\subsubsection{Nested, squeezed and controlled pairs of chords}\label{netsted:squeezed:controlled} We shall need the following definitions
\begin{enumerate}\index{Nested pair of chords}
	\item The pair $(H_0,H_1)$ of chords is {\em nested} if $H_0\not=H_1$, $H_0$ and $H_1$ are coplanar and  $S_0(H_1)\subset S_0(H_0)$. Given a nested pair $(H_0,H_1)$  -- with no end points in common -- the {\em projection} of $H_1$ on $H_0$  is the tripod $\tau_0\in H_0$, so that $s(\tau_0)$ is the closest point in the geodesic joining the endpoints of $H_0$, to the geodesic joining the end points of  $H_1$. Observe finally  that if $(H_0,H_1)$ is nested, then every pair of tripods in $H_0\cup H_1$ is coplanar
	\item The pair $(H_0,H_1)$ of chords is $(\alpha,k)$-{\em squeezed} if 
	$$
	\exists \tau_0\in H_0, \, \forall \tau_1\in H_1,\, \, \, (\tau_0,\tau_1) \hbox{ is $(\alpha,k)$-nested.}
	$$
	The tripod $\tau_0$ is called a {\em commanding tripod}\index{Commanding tripod} of the pair.
\item The pair $(H_0,H_1)$ of chords is $(\alpha,k)$-{\em controlled}\index{Controlled pair of chords} if 
	$$ \forall \tau_1\in H_1,\,\exists \tau_0\in H_0, \, \, (\tau_0,\tau_1) \hbox{ is $(\alpha,k)$-nested.}
	$$
\item The {\em shift} of two chords $H_0$, $H_1$ is
	$$\index{$\delta(H_0,H_1)$|see{Shift}}\index{Shift}
\delta (H_0,H_1))\defeq\inf\{d(\tau_0,\tau_1)|\tau_0 \in H_0, \tau_1\in H_1\}\ .
$$
\end{enumerate}

\begin{figure}[h]
 \centering
 \begin{subfigure}[h]{0.3\textwidth}
 \centering
  \includegraphics[width=\textwidth]{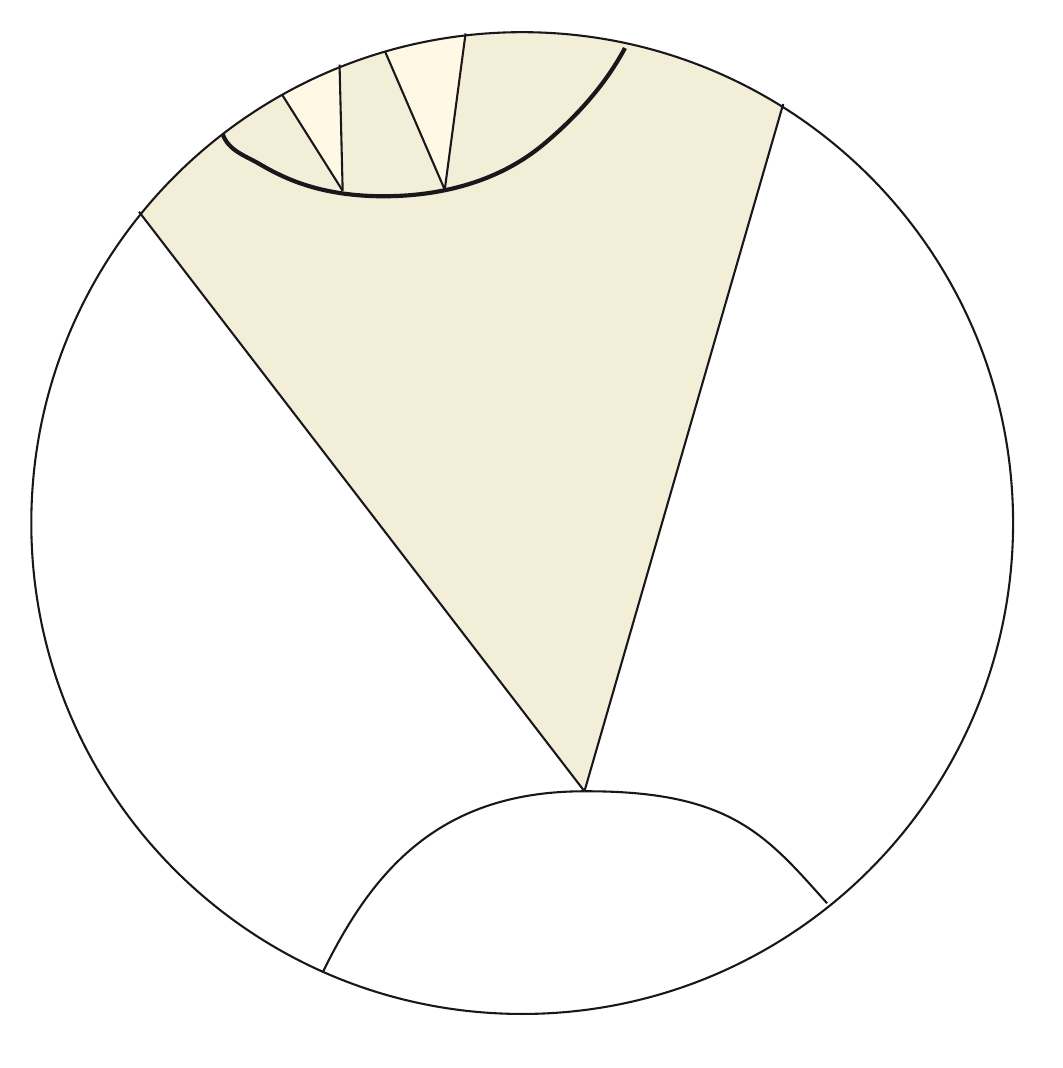}
  \caption{Squeezed chords}\label{fig:1a}
 \end{subfigure}
\quad    
 \begin{subfigure}[h]{0.3\textwidth}
 \centering
  \includegraphics[width=\textwidth]{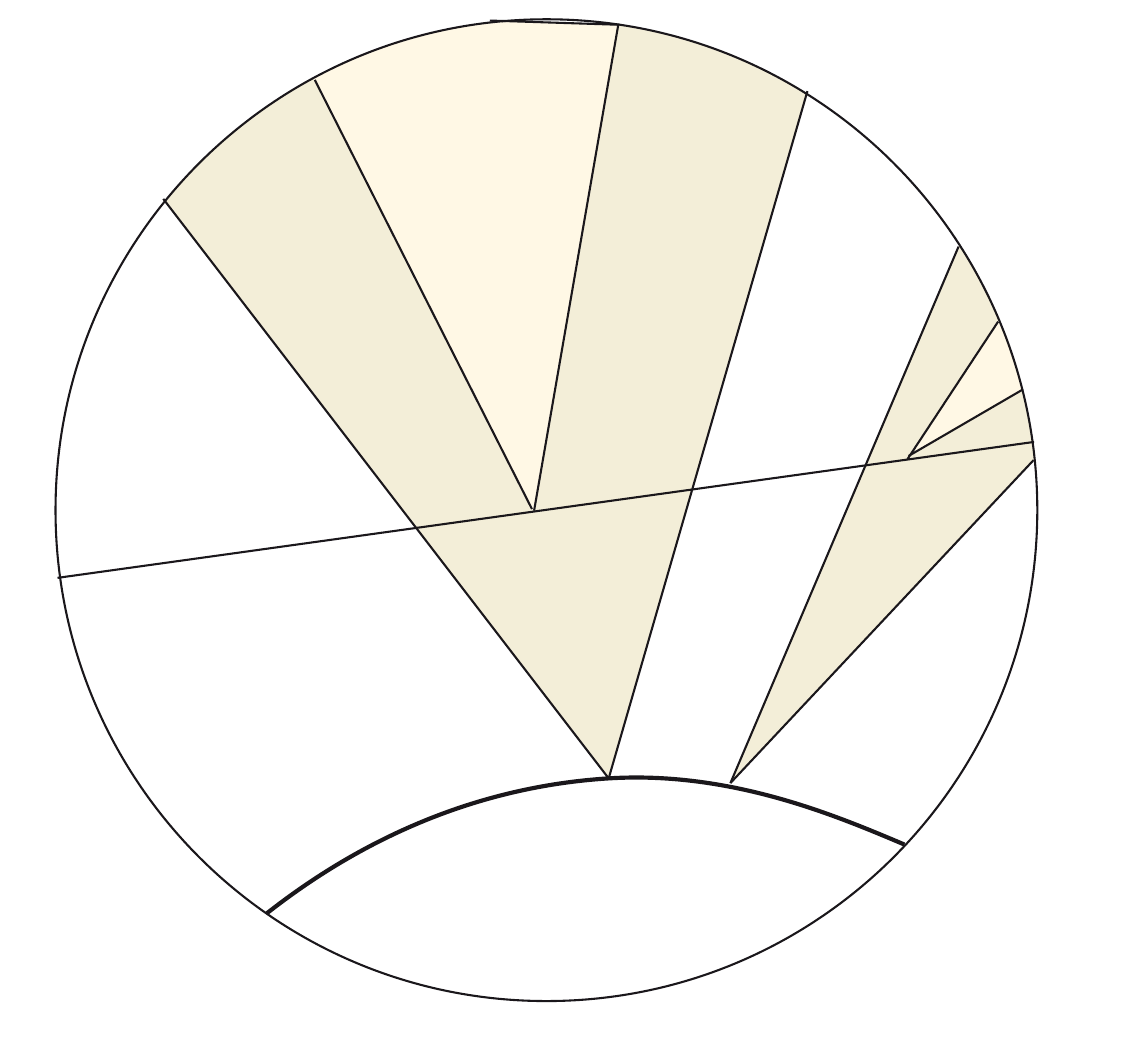}
  \caption{Controlled chords}\label{fig:1b}
 \end{subfigure}  
 \caption{Controlled and squeezed chords}
 \label{fig:chord}
\end{figure}

\subsubsection{Squeezing nested pair of chords}

In all the sequel $\bK$ is the diffusion constant defined in Proposition \ref{pro:cone-contrac0} and $\bk=\bK^{-1}$ the contraction constant
 
The following   proposition provides our first example of nested  pairs of chords in the coplanar situation.

\begin{proposition}{\sc[Nested pair of chords]}
\label{exis-cont} There exists $\beta_1$ only depending on ${\G}$, and a  decreasing function 
$$
\ell: ]0,\beta_1]\to \mathbb R,
$$
such that for any positive numbers $\beta$  with $\beta\leq\beta_1$,  any nested pair $(H_0,H_1)$ with $\delta(H_0,H_1)\geq\ell(\beta)$  is $(\bK\beta,\bk^{9})$-squeezed. The projection $\tau_0$ of $H_1$ on  $H_0$ is a  commanding tripod of $(H_0,H_1)$  \end{proposition}

Observe in particular that $S_0(H_1)\subset S_{\bK\beta}(H_1) \subset C_{\bk^{8}\beta}(\tau_0)$.
The choice of $\bk^{9}$ is rather arbitrary in this proposition but will make our life easier later on.

\begin{proof} Let $\tau_1\in H_1$. Let then  $\check\tau_0\in H_0$, with $\partial^0\check\tau_0=\partial^0\tau_1$.  Let as in paragraph 
\ref{sec:cone contract}, $s_0$, $s_1$, $z_0$ and $z_1$ be constructed from $\check\tau_0$ and $\tau_1$.  One notices that $d_{\tau_1}(z,z_1)\leq \pi/2$. Then given $\epsilon$,  for $\delta(H_0,H_1)$ large enough   the second part of  Proposition \ref{pro:cone-contrac0} yields that $(\check\tau_0,\tau_1)$ is $(\alpha,\epsilon)$-nested.

 Observe now, that for any $\beta$, there exists $\delta_1$ so that $\delta(H_0,H_1)>\delta_1$ yields $d(\tau_0,\check\tau_0)\leq \beta$, where $\tau_0$ is the projection of $H_1$ on $H_0$ Thus, using Proposition  \ref{A-B} for $\beta$ small enough, we have that the pair of tripods $(\check\tau_0,\tau_1)$ is $(\alpha,2\cdotp\epsilon)$-nested. In other words, since $\tau_0$ is independent from the choice of $\tau_1$, we have proved that the pair of chords $(H_0,H_1)$ is $(\alpha,2\cdotp\epsilon)$-squeezed for $\delta(H_0,H_1)$ large enough.

\end{proof}

\subsubsection{Controlling nested pair of chords}

Our second result about coplanar pair of chords is the following
\begin{lemma}{\sc[Controlling diffusion]}\label{exis-cont2} There exists a  positive numbers $\beta_2$, with $\beta_2\leq\beta_3$ only depending on ${\G}$, such that given a positive $\beta\leq\beta_2$,
  a nested pair $(H_0,H_1)$,
then $(H_0,H_1)$ is $(\xi\beta,\bK)$ controlled for all $\xi\leq 1$.

Assume furthermore that $\ell_0\geq\delta(H_0,H_1)$, where $\ell_0\geq\ell(\beta)$
Then, given   $\tau_1\in H_1$, there exists $H_2$ so that \begin{itemize}
	\item $(H_1,H_2)$ is nested,
	\item $0<\delta(H_0,H_2)\leq\ell_0$,
	\item $(\tau_0,\tau_1)$ is $(\bk^{2}\beta,\bK)$-nested, where $\tau_0$ is the projection of $H_2$  on  $H_0$.
\end{itemize}  
\end{lemma}

Let us first prove
\begin{proposition}\label{pro:cnc}
	There exists $\alpha_4$ with the following property. Let $(H_0,H_1)$  be two nested chords and $\tau_0\in H_0$, $\tau_1\in H_1$ so that
 for some $\alpha\leq \alpha_4$, $C_\alpha(\tau_0)\cap C_\alpha(\tau_1)\not=\emptyset$.

Then $(\tau_0,\tau_1)$ are $(\alpha,\bK)$-nested.
\end{proposition}

\begin{proof} 
Observe first that if $(z_0,\tau_0,\tau_1,z_1)$ are aligned then in that context $d_{\tau_1}(\partial^0\tau_1,z_1)\leq \pi/2$ -- see figure  \eqref{fig:hyp-weak} --.
 \begin{figure}
  \centering
  \includegraphics[width=0.35\textwidth]{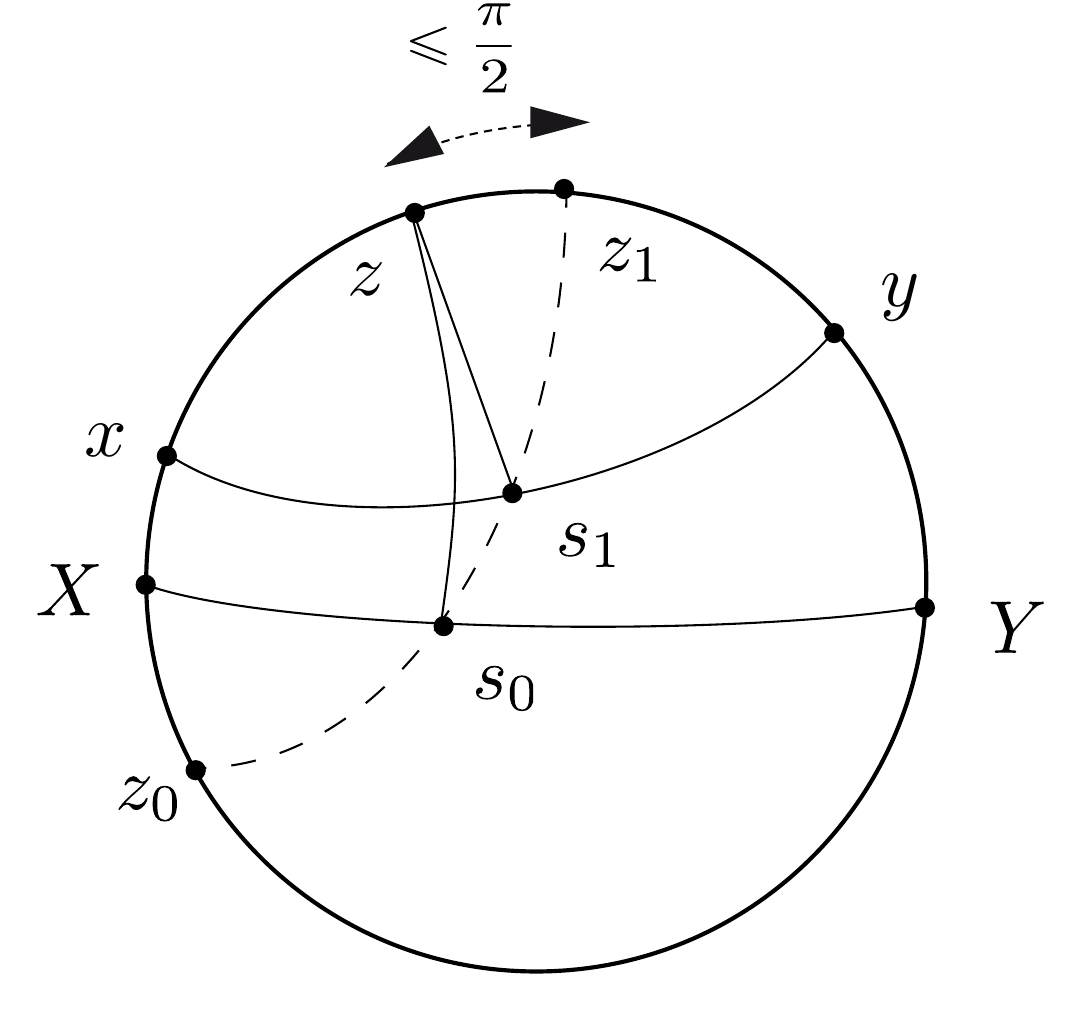}
  \caption{Aligned points and angle}
  \label{fig:hyp-weak}
\end{figure} Then
Let $u,v\in C_\alpha(\tau_1)$ and $w\in C_\alpha(\tau_0)\cap C_\alpha(\tau_1)$ then by Proposition \ref{pro:cone-contrac0}.
\begin{eqnarray}
		d_{\tau_0}(u,v)&\leq& \frac{\bK}{4}	d_{\tau_1}(u,v)\ ,\label{eq:1CNC}\\
		d_{\tau_0}(u,\partial^0\tau_0)\leq d_{\tau_0}(u,w)+d_{\tau_0}(w,\partial^0\tau_0)&\leq& \frac{\bK}{4}d_{\tau_1}(u,w)+\alpha\leq \bK\alpha\ .
\end{eqnarray}
Thus from the second equation $C_\alpha(\tau_1)\leq C_{\bK\alpha}(\tau_1)$. This concludes the proof of the proposition
\end{proof}
Let us now move to the proof of Lemma\ref{exis-cont2}:
\begin{proof} 
Let $\tau_1\in H_1$. Let $\hh$ be the associated hyperbolic plane to the coplanar pair $(H_0,H_1)$. 
Let $\tau_0\in H_0$ so that $\partial^0\tau_0=\partial^0\tau_1$. Then $\partial^0\tau^0\in C_{\xi\beta}(\tau_1)\cap C_{\xi\beta}(\tau_0)\not= \emptyset$. We conclude proof of the first assertion
 by  Proposition \ref{pro:cnc}:  
that  $(\tau_0,\tau_1)$ is $(\xi\beta, \bK)$-nested.
\vskip 0.2truecm

Assume now that $\delta(H_0,H_1)\leq\ell_0$. Let $H_3$ so that $S_0(H_3)=C_{\bk^2\beta}(\tau_1)\cap \partial_\infty \hh$. We have two cases.

\vskip 0.1 truecm
\noindent (1)\,\, If $\delta(H_0,H_3)\leq\ell_0$, we can take   $H_2=H_3$,  and  $\tau_0$  the projection of $H_2$ on $H_0$.  Thus  $\partial^0\tau^0\in C_{\bk^2\beta}(\tau_1)\cap C_{\bk^2\beta}(\tau_0)\not= \emptyset$ and    we conclude by Proposition \ref{pro:cnc}: $(\tau_0,\tau_1)$ is $(\bk^2\beta, \bK)$-nested.
\vskip 0.1 truecm
\noindent (2)\,\,  If $\delta(H_0,H_3)\geq\ell_0\geq\ell(\beta)$, a continuity argument shows the existence of $H_2$ such that the pairs $(H_1,H_2)$ and $(H_2,H_3)$ are nested and $\delta(H_0,H_2)=\ell_0$. Let $\tau_0$ be the projection of $H_2$ on $H_0$. Then we have,
$$\left(C_{\bk^2\beta}(\tau_1)\cap\hh\right)=S_0(H_3)\subset S_0(H_2)
\subset \left(C_{\bk^8\beta}(\tau_0)\cap\hh\right)\subset 
\left(C_{\bk^2\beta}(\tau_0)\cap\hh\right)\ ,
$$ 
where the first inclusion follows from the definition of $H_3$, the second by the fact of $(H_2,H_3)$ is nested, and the previous to last one  by Proposition \ref{exis-cont} since $\delta(H_0,H_2)\geq\ell(\beta)$. 
In particular $C_{\bk^2\beta}(\tau_1)\cap C_{\bk^2\beta}(\tau_0)\not=\emptyset$. Again we conclude by   Proposition \ref{pro:cnc} : $(\tau_0,\tau_1)$ is $(\bk^2\beta, \bK)$-nested.

\end{proof}

\section{The Confinement  Lemma}\label{sec:conf-lem}

The main results of this section are the Confinement  Lemma and the Weak Confinement Lemma  that guarantee that a deformed path of quasi-tripods is  squeezed or controlled, provided that the deformation is small enough.

Let us say a coplanar path of  tripods associated to a path of chords $(h_i)_{0\leq i\leq N}$ is a  {\em weak $(\ell,N)$-coplanar path of  tripods}\index{Weak coplanar path of  tripods} if 
\begin{eqnarray}
		\delta(h_0,h_i)&\leq& \ell, \hbox{ for } i< N.
		\end{eqnarray}

A coplanar path of  tripods associated to a sequence of chords $(h_i)_{0\leq i\leq N}$ is  a {\em strong $(\ell,N)$-coplanar  path of  tripods}\index{Strong coplanar path of  tripods} if  furthermore
\begin{eqnarray}
		\delta(h_0,h_N)&\geq& \ell.
\end{eqnarray}.

The main result of this section  is the following, 

\begin{lemma}{\sc[confinement]}
	\label{lem:zigzag} There exists $\beta_3$ only depending on ${\G}$, such that for every $\alpha$ with $\alpha\leq\beta_3$ then there exists $\ell_0(\alpha)$, so that for all $\ell_0\geq\ell_0(\alpha)$, there is  $\eta_0$,  so that for all $N$ 
\begin{itemize}
\item for all weak $(\ell_0,N)$-coplanar paths of  tripods $\vect{\tau}=(\tau_0,\ldots,\tau_N)$, associated to a path of chords
	$
	\vect{h}(N)=(h_0,\ldots,h_N),
	$
	\item for all $\epsilon/N$-deformation $v=\left(g_0,\ldots,g_{N-1}\right)$ with $\epsilon\leq\eta_0$ 
	\end{itemize}
The following holds:
\begin{enumerate}
	\item the pair $(h_0^v,h_N^v)$ is $(\bk^2\alpha,\bK^2)$-controlled,
	\item if furthermore $\vect{h}$ is a strong coplanar path of  tripods then $(h_0^v,h_N^v)$ is $(\alpha,\bk^7)$-squeezed. Moreover $(h_0^v,h_N^v)$ and  $(h_0,h_N)$ both have the same commanding tripod.

\item If finally, $\vect{h}$ is a strong coplanar path  with $\delta(h_0,h_N)=\ell_0$, then $\check\tau_0$, the projection of $h_N$ on $h_0$, is a commanding tripod of $(h_0^v,h_N^v)$.
\end{enumerate}  
\end{lemma}

In the sequel, we shall refer the first case as the {\em Weak Confinement  Lemma} and the second case as the {\em Strong Confinement Lemma}. 

\subsubsection{Controlling deformations from a tripod} 
We first prove a proposition that allows us to control the size of deformation from a tripod depending only on the last and first chords.

\begin{proposition}\label{lem:zigzag0}{\sc [Barrier]} For any positive $\ell_0$, there exists positive constants $k$ and $\eta_1$  so that for all integer $N$ 
\begin{itemize}
\item for all weak $(\ell_0,N)$-coplanar paths of  tripods $\vect{\tau}=(\tau_0,\ldots,\tau_N)$, associated to a path of chords
	$
	\vect{h}(N)=(h_0,\ldots,h_N), 
	$
	\item for all chord $H$ so that $(h_{N},H)$ is nested with $0<\delta(h_0,H)\leq\ell_0$, 
	\item for all $\frac{\epsilon}{N}$-deformation $v=\left(g_0,\ldots, g_{N-1}\right)$ with $\epsilon\leq\eta_1$, 
 	\end{itemize}
we have
 \begin{eqnarray}
 \, \, d_{\check\tau_0}(\id,b_N)&\leq& k\cdotp\epsilon,\label{ineq:zig}
\end{eqnarray}
where $b_N=g_0\cdots g_{N-1}$ and $\check\tau_0$ is the projection of $H$ on $h_0$.
\end{proposition}

In this proposition, the position of $h_N$ plays no role.

\subsubsection{The confinement control}

We shall use in the sequel the following proposition.

\begin{proposition}{\sc [Confinement control]}\label{pro:control}
There exists a positive  $\epsilon_0$ so that for every positive $\ell_0$,  there exists a constant $k$  with the following property;
\begin{itemize}
	\item Let $(H,h)$ be a pair of nested chords, associated to the circle  $C\subset \gp$, so that $0<\delta(h,H)\leq \ell_0$ and let  $\tau_0$ be the projection of $h$ on $H$. 
	\item Let  $(X,Y)$  and $(x,y)$ be the extremities of $H$ and $h$ respectively.
	\item  Let $u,v,w\in C\subset\gp$ be pairwise distinct  so that $(X,u,v,x,y,w,Y)$ is cyclically oriented --possibly with repetition -- in $C$ and $\tau$ be the tripod coplanar to $H$ so that $\partial\tau=(u,w,v)$
	\item  Let $g\in\ms P_w$ with $d_{\tau}(g,\id)\leq\epsilon_0\ .$
	\end{itemize}
Then
$$
d_{\tau_0}(g,\id)\leq k\cdotp d_{\tau}(g,\id).
$$  
\end{proposition}

Figure \eqref{fig:CN} illustrates the configuration of this proposition. 

\begin{figure}[h]
  \centering
  \includegraphics[width=0.35\textwidth]{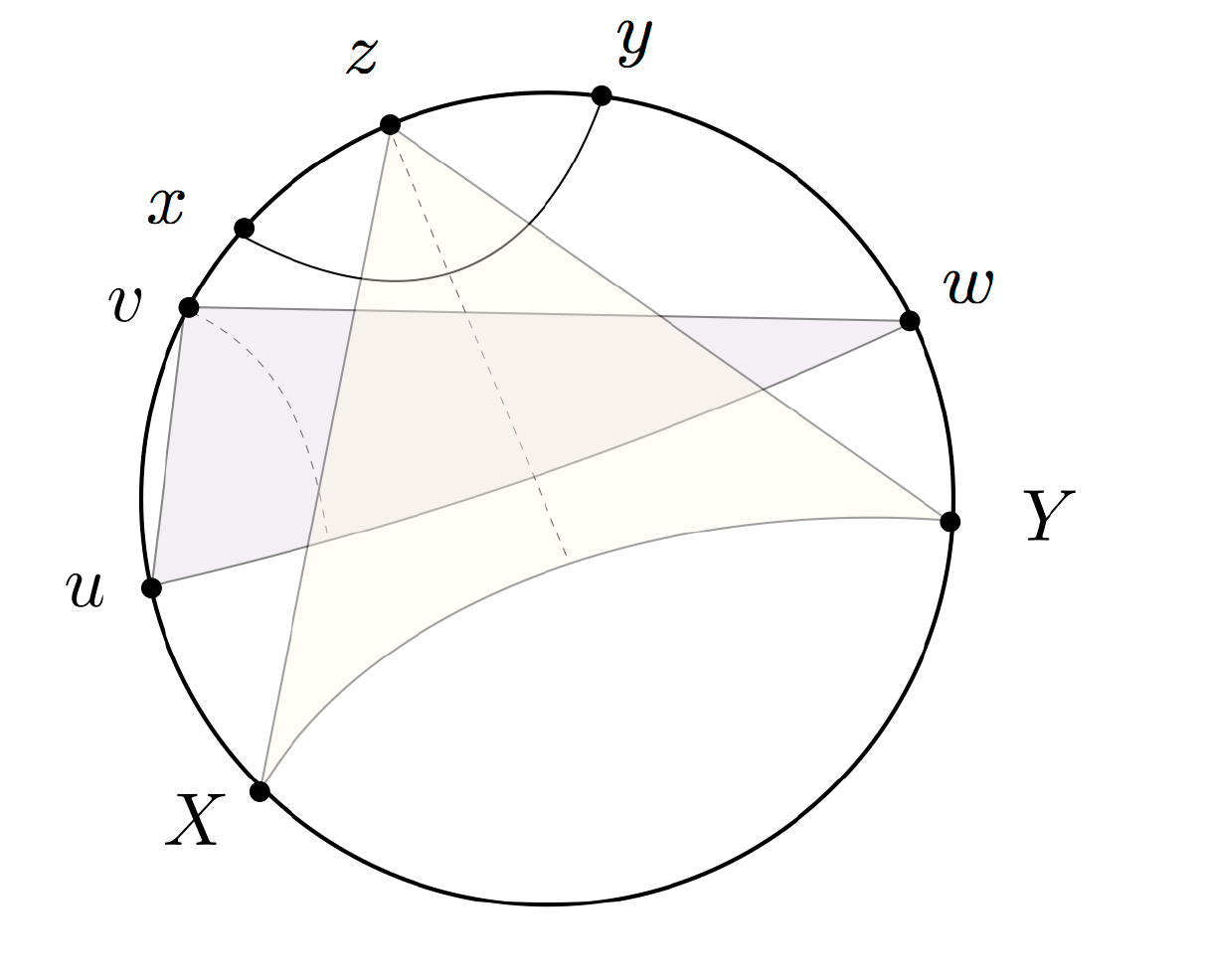}
  \caption{Confinement control}
  \label{fig:CN}
\end{figure}

\begin{proof} It is no restriction to assume    that $\delta(h,H)=\ell_0$. Let $\tau$ and $\tau_0$ be as in the statement and 
 $\tau_1$ be the tripod coplanar to $H$, so that $\partial\tau_1=(u,w,z)$. Observe that, there is a positive $t$  so that 
 $\varphi_t(\tau)= \tau_1$.  Let   $a=\T\tau (a_0)\in\mk g$, we have 
$$
d_{\varphi_t(\tau)}(g,\id)=d_{\exp(ta)(\tau)}(g,\id)=d_\tau(\exp(-ta)g\exp(ta),\id)\ .
$$
Let $\mk p_{+}$ be the Lie algebra of $\ms P_{+}\defeq \ms P_{\partial^+\tau}=P_w$, that we consider also equipped with the Euclidean norm $\Vert.\Vert_\tau$. By construction $\ms P_{+}=\tau(\ms P^+_0)$, thus
$$
\sup_{t>0}\left.\left\Vert \operatorname{ad}(\exp(-ta))\right\Vert\right\vert_{\mk p_+}<\infty.
$$ For $\epsilon$ small enough and independent of $\partial^+\tau$, $\exp$ is $k_1$-bilipschitz from the ball of radius $\epsilon$ in $\mk p_{+}$ onto its image in $\ms P_{+}$ for some constant $k_1$ independent of $\partial^+\tau$. 
Thus for $\epsilon_0$-small enough, there exists a constant $k_1$ so that
\begin{equation}
d_{\tau_1}(g,\id)\leq k_1d_\tau(g,\id).	\label{ineq:barrier1}
\end{equation}
Now the set  $K$  of tripods $\sigma$ coplanar to $\tau_0$, with  $\partial\sigma=(u,w,z)$ with $z$ fixed, $u$, $w$ as above, is compact  In particular there exists $k_2$ only depending on $\ell_0$ so that for any tripod $\sigma$  in $K$, 
$$
d(\tau_1,\tau_0)\leq k_3.
$$
Thus by Proposition \ref{A-B}, there exists $k_4$ so that
$$
d_{\tau_1}(g,\id)\leq k_4\cdotp d_{\sigma}(g,\id).
$$
The proposition now follows by combining with inequality \eqref{ineq:barrier1}. \end{proof}

\subsubsection{Proof of the Barrier Proposition \ref{lem:zigzag0} }
Let $(x_i,x^i)$ be the extremities of $h_i$ where $x_i$ is the pivot. Let $\widehat{x}_{i+1}$ the vertex of $\tau_i$ different from $x_i$ and $x^i$. 

Let $\check\tau_0$
 be the projection of $H$ on $h_0$. Observe that $x_i$ lies in one of the connected component of $h_0\setminus H$, while $x^i$ and  lie in the other (see Figure \eqref{fig:zz}).
 \begin{figure}[hbt]
\centering
  \includegraphics[width=0.4\textwidth]{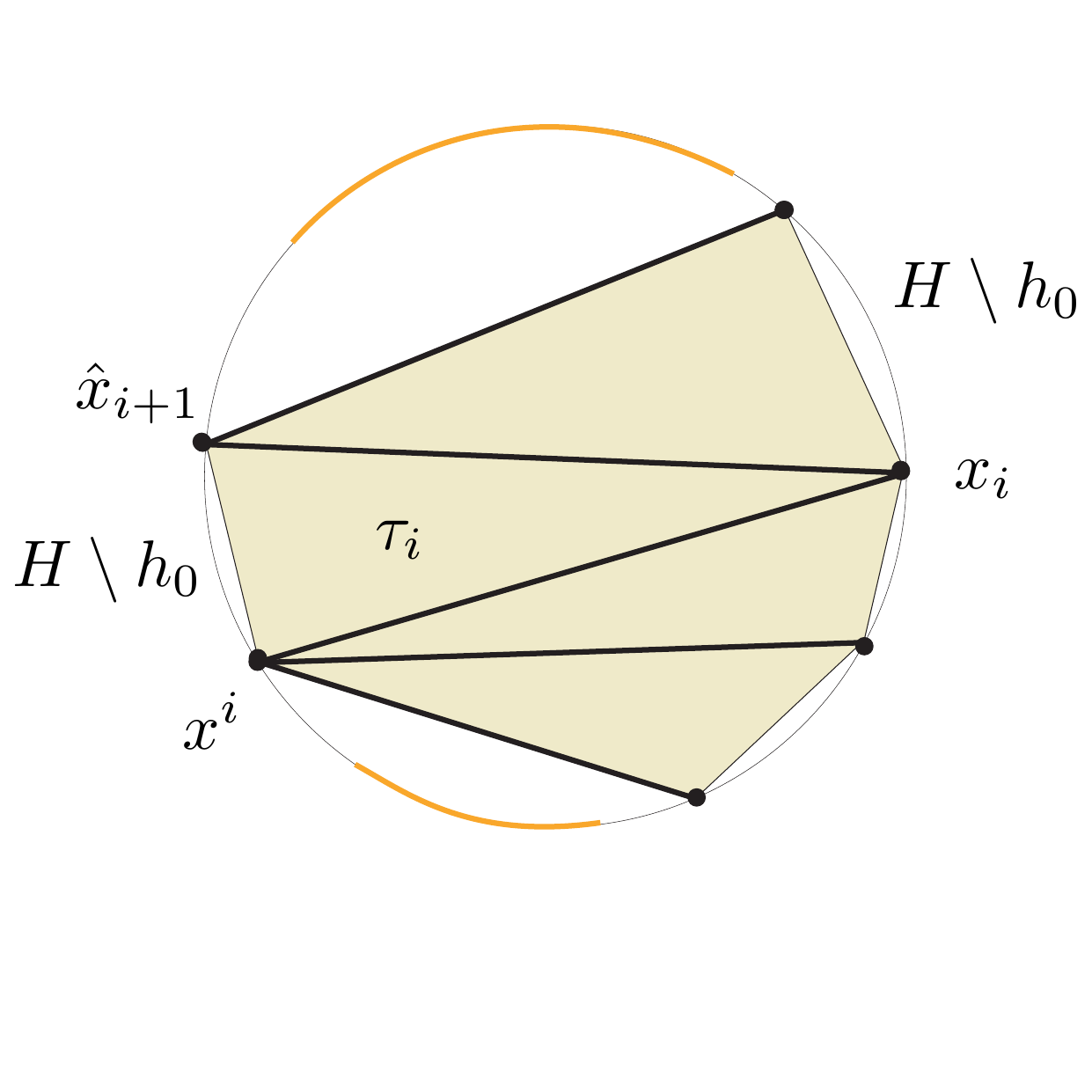}
  \caption{}\label{fig:zz}
\end{figure}

  Thus, according to Proposition \ref{pro:control} for $\epsilon$ small enough there exists $k$, only depending on $\ell$ so that 
$$
d_{\tau_0}(g_i,\id)\leq k\cdotp d_{\tau_i}(g_i,\id) \leq k\cdotp \frac{\epsilon}{N}.
$$
Thus, using the right invariance of $d_{\tau_0}$, 
$$
d_{\tau_0}\left(\id ,b_N\right)\leq \sum_{i=1}^{N}
d_{\tau_0}\left(\Pi_{j=i}^Ng_j,\Pi_{j=i+1}^N g_j\right)=\sum_{i=1}^{N}
d_{\tau_0}\left(g_i,\id\right)\leq k\cdotp\epsilon.
$$
Observe that this proves inequality \eqref{ineq:zig}
and concludes the proof of the Barrier Proposition \ref{lem:zigzag0}.

\subsubsection{Proof of the Confinement Lemma \ref{lem:zigzag}}

Let $\beta_1$ as in Proposition \ref{exis-cont}. Let then $\alpha$ with $\alpha\leq\beta_1$.  According to Proposition \ref{exis-cont}, there exists  $\ell=\ell_0(\alpha)$ so that  if  $(H_0,H_1)$ is a nested pair of chords with $\delta(H_0,H_1)\geq \ell$, then for any $\sigma_1\in H_1$,  the pair  $(\tau_0,\sigma_1)$ is $(\bK\alpha,\bk^9)$-nested, where  $\tau_0$ is the projection of $H_1$ on $H_0$. Let now fix $\ell_0\geq \ell_0(\alpha)$

\vskip 0.2truecm

\noindent{\sc First step: strong coplanar}

Consider first the case where $\delta(h_0,h_N)\geq \ell_0$. By continuity we may find a chord $\check h_N$ so that the pairs $(h_{N-1},\check h_N)$ and $(\check h_N, h_N)$ are nested and so that $\delta(\check h_N,h_0)=\ell_0$.

Let $\check\tau_0$ be the projection of $\check h_N$ on $h_0$. Then by Proposition \ref{exis-cont} for any $\sigma_1$ in $\check h_N$, $(\check \tau_0,\sigma_1)$ is $(\bK\alpha,\bk^9)$-nested

By  Lemma \ref{exis-cont2}, for any $\sigma_N$ in $h_N$, there exist $\sigma_1$ in $\check h_N$ so that $(\sigma_1,\sigma_N)$ is $(\alpha,\bK)$-nested and thus   $(\check\tau_0,\sigma_N)$ is $(\alpha,\bk^8)$-nested.

By the Barrier Proposition \ref{lem:zigzag0} applied  to $\vect{h}(N)$ and $H=\check h_N,$ we get that 
$$
d_{\check\tau_0}(\id,b_N)\leq k\cdotp\epsilon\, .
$$
for $k$ only depending on ${\G}$ and where $\ell$ and $b_N$ are defined in the Barrier Proposition.

We now furthermore assume that  $\alpha\leq\beta_0$, where $\beta_0$ comes from Proposition \ref{pro:def-nt}. For  $\epsilon$ is small enough, Proposition \ref{pro:def-nt} shows that for any $\sigma_1$ in $h_N$;  $(\check\tau_0,b_N(\sigma_1))$ is $(\alpha,2\bk^8)$-nested. Thus $(h_0,b_N(h_N))$ is  $(\alpha,2\bk^8)$-squeezed hence $(\alpha,\bk^7)$-since $2\bk\leq 1$, with $\check\tau_0$ as a commanding tripod. 

This applies of course if the deformation is trivial and we see that $(h_0,h_N)$ and $(h_0^v,h_N^v)$ both have $\check\tau_0$ as a commanding tripod.

This concludes this first step and the proof of the second item and the third item in Lemma  \ref{lem:zigzag}.

\vskip 0.2truecm

\noindent{\sc Second step}

Let us consider the remaining case when $\delta(h_0,h_N)\leq \ell$. Let us apply Proposition \ref{exis-cont2} to $(H_0,H_1)=(h_0,h_N)$ and $\tau_1$ in $h_N$. Thus there exists $H_2$ so that $(h_N,H_2)$ is nested, $0<\delta(H_0,H_2)\leq \ell$,  and $(\tau_0,\tau_1)$ is $(\bk^2\alpha,\bK)$ nested where $\tau_0$ is the projection of $H_2$ on $h_0$.

Applying the Barrier Proposition \ref{lem:zigzag0} to $h=H_2$ and $H=H_0$, yields that
$
d_{\tau_0}(\id, b_N)\leq k\cdotp\epsilon$. 
Thus for $\epsilon$ small enough, then Proposition \ref{pro:def-nt} yields that $(\tau_0,b_N(\tau_1)$ is $(\bk^2\alpha,2\bK)$ nested, hence $(\bk^2\alpha,\bK^2)$ nested.

This shows that $(h_0,b_N(h_N))$ is $(\bk^2\alpha,\bK^2)$-controlled.
This concludes the proof of Lemma \ref{lem:zigzag}.

\section{Infinite paths of quasi-tripods and their limit points} \label{sec:Morse}

The goal of this section is to make sense of the limit point of an infinite sequence of quasi-tripods and to give a condition under which such a limit point exists. The {\it ad hoc} definitions are motivated by the last section of this paper as well as by the discussion of Sullivan maps.

One mays think of our main Theorem \ref{theo:exislimi}  as a refined version of a Morse Lemma in higher rank: instead of working with quasi-geodesic paths in the symmetric space, we work with  sequence of quasi-tripods in the flag manifold; instead of making the quasi-geodesic converge to a point at infinity, we make the sequence of quasi-tripods  shrink to a point in the flag manifold. This is guaranteed by some local conditions that will allow us to use our nesting and squeezing concepts defined in the preceding section.

 Theorem \ref{theo:exislimi} is the goal of our efforts in this first part and will be used several times in the future.

\subsection{Definitions:  \texorpdfstring{$Q$}{Q}-sequences and their deformations}\label{sec:cuff}

\begin{definition} \begin{enumerate}
	\item A {\em $Q$ coplanar sequence of tripods}\index{$Q$--sequence } is an infinite sequence of tripods $\vect{T}=\seq{T}$ so that the associated sequence of coplanar chords   $\vect{c}=\{c_i\}_{i\in\mathbb N}$ satisfies: for all integers $m$ and $p$ we have
	$$\vert m-p\vert\leq Q\delta(c_m,c_p)+Q$$	
	where $\delta(\cdot,\cdot)$ is the shift defined in \ref{netsted:squeezed:controlled}.
	\item A  sequence of quasi-tripod $\vect{\tau}=\seq{\tau}$ is a {\em $(Q,\epsilon)$-sequence of quasi-tripods} if there exists a 
a  coplanar $Q$ coplanar sequence of tripods $T=\seq{T}$, so that for every $n$, $\{\tau_m\}_{m\in [0,n]}$ is an $\epsilon$-deformation of  $\{T_m\}_{m\in [0,n]}$.  
\item The associated sequence of chords to a $(Q,\epsilon)$-sequence of  quasi-tripods is called a {\em $(Q,\epsilon)$-sequence of  chords} 
\end{enumerate}
\end{definition}

\subsection{Main result: existence of a limit point}

Our main theorem asserts the existence of limit points for some deformed $(Q,\epsilon)$- sequence and their quantitative properties.

\begin{theorem}\label{theo:exislimi}{\sc [Limit point]}\label{coro:exislimi} There exist some positive constants $\bA$ and $\bQ$ only depending on ${\G}$, with $\bQ<1$,  such that for every positive number $\beta$ and $\ell_0$ with $\beta\leq \bA$, there exist a positive constant  $\epsilon>0$, so that for any $R>1$:

 For any  $(\ell_0R,\eR)$-deformed sequence of quasi tripods   $\vect{\theta}=\seq{\theta}$, with associated sequence of chords $\vect{\Gamma}=\seq{\Gamma}$ there exists some  $\delta>0$ so that
\begin{eqnarray}  \bigcap_m^\infty S_\delta(\Gamma_m)
:=\{\xi(\vect{\theta})\}\ , \hbox{ with }
\xi(\vect{\theta})=\lim_{m\to\infty}(\partial^j\theta_m) \ \ \hbox{ for } j\in\{+,-,0\} \ ,\end{eqnarray}
moreover we have the following quantitative estimates: 
\begin{enumerate}
 \item 
 for any $\tau$ in $\Gamma_0$, and $m>(\ell_0+1)^2R$, 
\begin{eqnarray}
 d_\tau(\xi(\vect{\theta}),\partial^j\theta_m))\leq \bQ^m\beta\ \ \hbox{ for } j\in\{+,-,0\}\  .	\label{ineq:limit-dist00}
\end{eqnarray} 
\item Let  $\tau$ in $\Gamma_0$. Assume $\seq{\theta}$  is the deformation of a sequence of coplanar tripods    $\vect{\tau}=\seq{\tau}$ with $\tau_0=\dt{\theta_0}$,  then
 \begin{eqnarray}
	d_{\tau}( \xi(\vect{\theta}),\xi(\vect{\tau}))\leq \beta\, . \label{ineq:limit-dist0}
\end{eqnarray}
\item  Finally, let $\seq{\theta^\prime}$ be another $(\ell_0R,\eR)$-deformed sequence of quasi tripods. Assume that $\seq{\theta^\prime}$ and $\seq{\theta}$ coincides up to the $n$-th chord with $n>(\ell_0+1)^2R$, then for all $\tau\in \Gamma_0$, 
\begin{eqnarray}
	d_{\tau}(\xi(\vect{\theta^\prime}),\xi(\vect{\theta}))\leq  \bQ^n\beta\label{ineq:limit-dist}\ .
	\end{eqnarray}
\end{enumerate}

\end{theorem}

The limit point theorem will be the consequence of a more technical one:
\begin{theorem}\label{theo:nes-cuf} {\sc [Squeezing chords]} There exists some constant $\bA$,  only depending on ${\G}$, such that for every positive number $\delta$ with $\delta\leq \bA$, there exists positive constants  $R_0$, $\ell_0$ and $\epsilon$  with the following property:

If $\vect{\Gamma}$ is an $(\ell_0R,\eR)$-deformed sequence of chords of the coplanar sequence of chords $\vect{c}$ with $R\geq R_0$, if $j>i$ are so that $\delta(c_i,c_j)\geq \ell_0$ then $(\Gamma_i,\Gamma_j)$ is  $(\delta,\bk)$-squeezed.
\end{theorem}
\subsection{Proof of the squeezing chords theorem \ref{theo:nes-cuf}}

As a preliminary, we make the choice of constants, then we cut a sequence of chords into small more manageable  pieces.  Finally we  use the Confinement  Lemma to obtain the proof.
\subsubsection{Fixing constants and choosing a threshold}
Let $\alpha_3$ as in Corollary \ref{coro-conv-nest},
let  $\beta_3$ as in the Confinement  Lemma \ref{lem:zigzag}.
We now choose $\alpha$ so that   
\begin{equation}
\alpha\leq  \inf(\beta_3,\alpha_3).
\end{equation}
Then $\ell_0=\ell_0(\alpha)$  be the {\em threshold}, and $\eta_0$   be obtained by the Confinement  Lemma \ref{lem:zigzag}. Let finally 
\begin{eqnarray}
\epsilon\leq\frac{\eta_0}{\ell_0\cdotp(\ell_0+1)}\ .
\end{eqnarray}

\subsubsection{Cutting into pieces}

Let $\vect{c}$  be a sequence of coplanar chords  admitting an $\ell_0R$-coplanar path of  tripods.
\begin{lemma}\label{lem:cut}
We can cut $\vect{\tau}$ into successive pieces $\vect{\tau}^n:=\{\tau_p\}_{p_n\leq p< p_{n+1}}$ for $n\in\{0,M\}$ so that
\begin{enumerate}
	\item for $n\in \{0,M-1\}$, $\vect{\tau}^n$ is a strong  $(\ell_0,N)$ coplanar path of  tripods 	\item $\vect{\tau}^M$ is  a weak $(\ell_0,N)$ coplanar path of  tripods.
\end{enumerate}
where in both cases, 
$N\leq L\defeq \lfloor(\ell_0 +1)(\ell_0 R)\rfloor+1$, where $\lfloor x\rfloor$ denotes the integer value of the real number $x$.

\end{lemma}

	\begin{proof} Let $\vect{c}$ be the corresponding sequence of chords. Recall that the function $q\mapsto \delta(c_p,c_q)$ is increasing for $q>p$.  Thus we can further cut into (maximal) pieces  so that 
	$$
	\delta(c_{p_n}, c_{p_{n+1}-1}))\leq \ell_0,\ \ 	\delta(c_{p_n}, c_{p_{n+1}}))\geq \ell_0.
	$$
This gives the lemma: the bound on $N$ comes from the fact that $\vect{\tau}$ is a $\ell_0 R$-sequence. In particular, since $\delta(c_{p_n},c_{p_{n+1}-1})\leq \ell_0$, then 
	$
	\vert p_{n+1}-p_n\vert -1 \leq (\ell_0 R)(\ell_0+1)
	$.\end{proof}
\subsubsection{Completing the proof} Let $\vect{\theta}$ be an  $(\ell_0R,\eR)$-sequence of quasi-tripods, with $R>R_0$.  Let $\vect{\Gamma}$ be the associated sequence of chords. Assume $\vect{\theta}$  is the deformation of an $\ell_0R$- coplanar sequence of  tripods $\vect{\tau}$, cut in smaller sub-pieces as in Lemma \ref{lem:cut}.
\begin{proposition}\label{pro:cut-control}
 for all $n$
 \begin{enumerate}
 	\item for $n<M$,  $(\Gamma_{p_n},\Gamma_{p_{n+1}})$ is $(\alpha,\bk^7)$-squeezed,
 	\item  Moreover  $(\Gamma_{p_{M}},\Gamma_{p_{M+1}})$ is $(\bk^2\alpha,\bK^2)$-controlled. 
\end{enumerate} 
\end{proposition}
\begin{proof}
 If $n<M$, $\vect{\tau^n}$ is a strong $(\ell_0,L)$-path.  Then according to the Confinement  Lemma \ref{lem:zigzag} and the choice of our constants 
	$(\Gamma_{p_n},\Gamma_{p_{n+1}})$ is $(\alpha,\bk^7)$-squeezed.
 
Since $\vect{\tau}^M$ is a weak $(\ell_0,L)$-path, it follows by our choice of constants and  the Confinement  Lemma \ref{lem:zigzag}  that $( \Gamma_{p_M}, \Gamma_{p_{M+1}})$ is $(\bk^2\alpha,\bK^2)$ controlled.\end{proof}

We now  prove the Squeezing Chord Theorem \ref{theo:nes-cuf} with $\delta=\bk^2\alpha$:
\begin{proposition}
	Assuming, $\delta(c_i,c_j)>\ell_0$ and  $j>i$,  the pair $(\Gamma_{i},\Gamma_{j})$  is  $(\bk^2\alpha,\bk)$-squeezed.
\end{proposition}
\begin{proof}
		
We will use freely the observation that $(\alpha,\bk^n)$-nesting implies $(\bk^p\alpha,\bk^q)$-nesting for $p,q\geq 0$ with $p+q\leq n$.

	Recall that thanks to the Composition Proposition \ref{pro:composcone}, if the pairs of chords $(H_0,H_1)$ and $(H_1,H_2)$   which are both $(\alpha,\bk^7)$-squeezed. (in particular 
$(H_1,H_2)$ is $(\alpha,\bk^5)$-squeezed),  then  $(H_0,H_2)$ is $(\alpha,\bk^7)$-squeezed.

We cut $\vect{\tau}$ as above in pieces and control every sub-piece using Proposition \ref{pro:cut-control}.

Thus, by induction, $(\Gamma_{p_0},\Gamma_{p_M})$ is $(\alpha,\bk^7)$-squeezed and thus $(\bK^2(\bk^2\alpha),\bk^7)$-squeezed since $\bk\bK=1$.

Finally since $(\Gamma_{p_M},\Gamma_{p_{M+1}})$ is $(\bk^2\alpha,\bK^2)$-controlled, the Composition Proposition \ref{pro:composcone} yields that $(\Gamma_{p_0},\Gamma_{p_{M+1}})$ is $(\bk^2\alpha,\bk^7\bK^2)$-squeezed and thus $(\bk^2\alpha,\bk)$-squeezed. This finishes the proof.
\end{proof}

\subsection{Proof of the existence of limit points, Theorem \ref{theo:exislimi}} Let $\seq{\theta}$  and $\seq{\tau}$ be  sequences of quasi tripods and tripods as in Theorem \ref{theo:exislimi}.

Let $\vect{\Gamma}$ be the sequence of chords associated to $\seq{\theta}$ and similarly $\vect{c}$ associated to $\seq{\tau}$ 	as in Theorem \ref{theo:exislimi}, then according to Theorem \ref{theo:nes-cuf} if $j>i$ are so that if $\delta(c_i,c_j)\geq \ell_0$ then $(\Gamma_i,\Gamma_j)$ is  $(\delta,\bk)$-squeezed. Since $\vect{c}$ is $Q_0$-controlled  (with $Q_0=\ell_0 R$) we have
$$\delta(c_i,c_j)\geq
\frac{\vert i-j\vert}{\ell_0 R}-1,$$
Thus
$$
j-i\geq L\implies \delta(c_i,c_j)\geq\ell_0, 
$$
We can summarize this discussion in the following statement
\begin{equation}
j-i\geq L\implies S_{\delta}(\Gamma_j)\subset S_{\bk\delta}(\Gamma_i)\, ,\label{inc:lacunary}
\end{equation}

\subsubsection{Convergence for lacunary subsequences}

We first prove an intermediate result.

\begin{corollary} There exists a constant $\bC$ only depending on ${\G}$, with $\bQ<1$, such that  for $\delta$ small enough, if $\vect{l}=\seq{l}$ is a sequence so that $l_{m+1}\geq l_m+L$ and $l_0=0$, then 
\begin{eqnarray}
	S_{\delta}
	(\Gamma_{l_{m+1}})\subset 	S_{\bk\delta}(\Gamma_{l_{m}}),
\end{eqnarray}
and furthermore there exists a unique point $\xi(\vect{l})\in\gp$ so that
\begin{eqnarray}
 \bigcap_{m=1}^\infty S_
   {\delta}
  \left(\Gamma_{l_m}\right)
   =\left\{\xi(\vect{l})\right\}\subset C_\delta(\check\tau_0).	
\end{eqnarray}
where $\check\tau_0$ is a commanding tripod for $(\Gamma_0,\Gamma_{l_1})$.

Finally, if  $\tau\in \Gamma_{0}$ then
 for all  $u$ in $S_{\delta}(\Gamma_{l_m})$ with $m\geq 1$ we have 
\begin{eqnarray}
d_{\tau}(u,\xi(\vect{l}))\leq 2^{-m}\bM\delta.\label{ineq:xi-cent}
\end{eqnarray}
\end{corollary}

\begin{proof}
From the squeezed condition for chords, we obtain that there exists $\check\tau_m\in \Gamma_{l_m}$ so that 
$$
S_{\delta}(\Gamma_{l_{m+1}})\subset C_{\bk\delta}(\check\tau_m)\subset S_{\bk\delta}(\Gamma_{l_{m}}).
$$
This proves the first assertion. 
As a consequence,
$
C_{\delta}(\check\tau_{m+1})\subset C_{\bk\delta}(\check\tau_m)
$.
Combining with the Convergence  Corollary \ref{coro-conv-nest}, we get the second  assertion, with
$$
\{\xi(\vect{l})\}:=\bigcap_{m=1}^\infty C_{\delta}(\check\tau_{m})=\bigcap_{m=1}^\infty S_{\delta}(\check\tau_{m})
$$

Using the second assertion of the Convergence  Corollary \ref{coro-conv-nest}, we obtain that if $u,v\in S_\delta(\Gamma_{l_m})\subset C_{\delta}(\check\tau_{m-1})$, then 
$$
d_{\tau_0}(u,v)\leq 2^{2-m}\delta.
$$
and in particular $u,\xi(\vect{l})\in C_\delta(\check\tau_0)$ and 
\begin{equation}
d_{\tau_0}(u,\xi(\vect{l}))\leq 2^{2-m}\delta.
\label{proo-exislimi12}	
\end{equation}
We now  extend the  previous inequality  when we replace $\tau_0$ by any $\tau\in C_n$.  We use Lemma \ref{coro:shcont} which produces constant $\delta_0$ and $k$ only depending on ${\G}$ so that if $\delta$ is smaller than $\delta_0$ then since $u,\xi(\vect{l})\in C_\delta(\check\tau_0)$, 
\begin{equation}
d_{\tau}(u,\xi(\vect{l}))\leq k.d_{\tau_0}(u,\xi(\vect{l})).\label{proo-exislimi11}	
\end{equation}
This concludes the proof of the corollary since we now get from inequalities \eqref{proo-exislimi11} and \eqref{proo-exislimi12}
$$
d_{\tau}(u,\xi(\vect{l}))\leq k.2^{2-m}\delta.
$$
\end{proof}
\subsubsection{Completion of the proof}

Let $\seq{l}$ and $\seq{l'}$ be two subsequences. It follows from inclusion \eqref{inc:lacunary}, that
$$
 \bigcap_{m=1}^\infty S_
   {\delta}
  \left(\Gamma_{l_m}\right)= \bigcap_{m=1}^\infty S_
   {\delta}
  \left(\Gamma_{l'_m}\right).
$$
As an immediate consequence, we get
that
$$
 \bigcap_{m=1}^\infty S_
   {\delta}
  \left(\Gamma_{m}\right)=\left\{\xi\left(\vect{\lambda}\right)\right\},
  $$
where $\vect{\lambda}=\seq{\lambda}$, with $\lambda_m=m.L$. Thus we may write $
\xi\left(\vect{\lambda}\right)\eqdef\xi\left(\vect{\Gamma}\right)
$.
The existence of $\xi(\vect{c})$ follows form the fact that $\vect{c}$ is also a $(\ell_0R,\eR)$-deformed sequence of cuffs.

By construction -- and see the second item in the Confinement Lemma \ref{lem:zigzag} -- the commanding tripod $\check\tau_0$ of $(\Gamma_0, \Gamma_L)$ is the commanding tripod of $(c_0, c_L)$.

It follows that both $\xi(\vect{l})$ and $\xi(\vect{\lambda})$ belong to $C_\delta(\check\tau_0)$. Thus using the triangle inequality and 
Lemma \ref{coro:shcont}, for all $\tau\in\Gamma_0$,
\begin{equation}
	d_\tau(\xi(\vect{l}),\xi(\vect{\lambda}))\leq k \delta\ ,\label{ineq:m=0}
\end{equation}
where $k$ only depends on $\G$.
By inequality \eqref{ineq:xi-cent}, if $\seq{l}$ is a lacunary subsequence, for any $\tau\in\gamma_0$, for $u\in S_\delta(\Gamma_{l_m})$ with $m\geq 1$,
\begin{eqnarray}
d_\tau(\xi(\vect{\lambda}),u)\leq 2^{-m}\bM\delta\ .\label{eq:comp11}	\end{eqnarray}
In particular taking $l_m=mL$, one gets
\begin{eqnarray}
d_\tau(\xi(\vect{\lambda}),\theta^j_{m.L})\leq 2^{-m}\bM\delta. \label{eq:comp12}
\end{eqnarray}
Let now $n=(m+1)L+p$, with $p\in [0,L]$.  The inclusion \eqref{inc:lacunary}, gives the first inclusion below, whereas the second is a consequence of the fact that $\bk<1$
\begin{eqnarray}
S_\delta(\Gamma_n)\subset S_{\bk\delta}(\Gamma_{m.L})\subset S_{\bk\delta}(\Gamma_{m.L})\ . \label{eq:comp13}
\end{eqnarray}
Thus combining  the previous assertion with assertion \eqref{eq:comp11}  for all $u\in S_\delta(\Gamma_n)$, with $n>L$ we have
\begin{eqnarray}
d_\tau(\xi(\vect{\lambda}),u)\leq 2^{-m}\bM\delta\leq  (2^{-\frac{1}{L}})^{n}4\bM\delta\ .
\end{eqnarray}
Taking $\bQ=2^{-\frac{1}{L}}$ and $\beta=4\bM\delta$, and $u=\theta^j_n$ yields the inequality
\begin{eqnarray}
d_\tau(\xi(\vect{\lambda}),\theta^j_n)\leq \bQ^{n}\beta\ .
\end{eqnarray} 
This completes the proof of  inequality \eqref{ineq:limit-dist00} for $n>L$. 

The second item comes from  inequality \eqref{ineq:m=0} after possibly changing $\beta$.

The third item comes form the first and the triangle inequality, again after changing $\beta$.

\section{Sullivan limit curves}\label{sec:sull}

The purpose of this section is to define and describe some properties of an analog of the Kleinian property: being a $K$-quasi-circle with $K$ close to 1. 

This is achieved in Definition \ref{def:sull-map}. We then show, under the hypothesis of a compact centralizer for the $\skd$, three main theorems of independent interest: Sullivan maps are H{\"o}lder (Theorem \ref{theo:sull-hold}), a representation with a Sullivan limit map is Anosov (Theorem \ref{theo:sull-anos}), and finally one can weaken the notion of being Sullivan under some circumstances (Theorem \ref{theo:boot}).

In this paragraph, as usual, $\G$ will be a semisimple group, $\mk s$ an $\skd$-triple, $\gp$ the associate flag manifold.
We will furthermore assume in this section  that 
\vskip 0.2truecm
\begin{center}
	\fbox{The centralizer of $\mk s$ is  compact}
\end{center}
\vskip 0.2truecm
We will comment on the case of non compact centralizer later.

Let us start with a comment on our earlier definition of circle maps \ref{def:circle}. Let $T=(x^-,x^+,x_0)$ be a triple of pairwise distinct points in $\Rp$ -- also known as a tripod for $\sld$ -- and $\tau$ a tripod in $\mc G$. Such a pair $(T,\tau)$ defines uniquely
\begin{itemize}
\item an associated circle map $\eta$ from $\Rp$ to $\gp$ so that $\eta(T)=\partial \tau$,
\item an associated {\em extended circle map}\index{Extended circle map} which is a map $\nu$ from the space of   triples of pairwise distinct points $(x,y,z)$  in $\Rp$ to $\mc G$ whose image consists of coplanar tripods and so that $$
	\partial\nu(x,y,z)=(\eta(x),\eta(y),\eta(z))\ , \ \ \ \nu(x^-,x^+,x_0)=\tau\ .
	$$ 	
\end{itemize}

\subsection{Sullivan curves: definition and main results}

\begin{definition}{\sc [Sullivan curve]}\label{def:sull-map}  We say a map $\xi$ from ${\bf P^1}(\mathbb R)$ to $\gp$ is a {\em $\zeta$-Sullivan curve}  with respect to $\mk s$ \index{Sullivan curve}  if the following property holds:

Let $T=(x^-,x^+,x^0)$ be any triple of  pairwise distinct point  in $\Rp$. Then there exists a tripod  $\tau$ -- called {\em a compatible tripod} -- a circle map $\eta$, with $\eta(T)=\partial\tau$, so that for all $y\in {\bf P^1}(\mathbb R)$,   
\begin{eqnarray}
d_\tau(\xi(y),\eta(y))&\leq& \zeta\, .
\end{eqnarray}	
\end{definition}
 Obviously if $\zeta$ is large, for instance greater than $\diam(\gp)$, the definition is pointless: every map is a  $\zeta$-Sullivan. We will however show that the definition makes sense for $\zeta$ small enough.

We also leave to the reader to check that in the case of ${\G}=\ms{PSL}_2(\mathbb C)$ -- so that $\gp$ is ${\bf P}^1(\mathbb C)$ -- the following holds: for $K>1$ and any compact interval  $C$ containing  $-1$, there exists a positive $\epsilon$ such that if $\xi$ is $\epsilon$-Sullivan, then for all $(x,y,z,t)$ in $\Rp$ so that $
[x,y,z,t]$ belongs to $C$, then 
$$
\frac{1}{K}\leq \left\vert\frac{[\xi(x),\xi(y),\xi(z),\xi(t)]}{[x,y,z,t]}\right\vert\leq K. 
$$
This readily implies that $\xi$ is a quasicircle. 
Thus in that case, an $\epsilon$-Sullivan map is quasi-symmetric for $\epsilon$-small enough.
The following results of independent interest justify our interest of  $\zeta$-Sullivan maps.

\begin{theorem}{\sc[H{\"o}lder property]}\label{theo:sull-hold}
	 There exists some positive numbers $\zeta$ and $\alpha$,  so that any  $\zeta$-Sullivan map  is $\alpha$-H{\"o}lder.
\end{theorem}
We prove a more quantitative version of this theorem with an explicit modulus of continuity  in paragraph \ref{sec:holder}.This  modulus of continuity will be needed in other proofs.

The existence of $\zeta$-Sullivan limit maps implies some strong dynamical  properties. We refer to \cite{Labourie:2006,Guichard:2012eg} for background and references on Anosov representations.

\begin{theorem} {\sc [Sullivan implies Anosov]}\label{theo:sull-anos}
There exists some positive $\zeta_1$ with the following property. Assume $S$ is a closed hyperbolic surface and $\rho$ a representation of $\pi_1(S)$ in ${\G}$. Assume  there exists a $\rho$-equivariant $\zeta_1$-Sullivan map $$
	\xi:
	\partial_\infty\pi_1(S)=\Rp\to 
	\gp\ .$$
	 Then  $\rho$ is $\ms P$-Anosov and $\xi$ is its limit curve.
\end{theorem}
Recall that $\ms P$ is the stabilizer of a point in $\gp$. We recall that a $\ms P$-Anosov representation is in particular faithful and a quasi-isometric embedding and that all its elements are loxodromic \cite{Labourie:2006,Guichard:2012eg}    . Recall also that in that context, the  parabolic is isomorphic to its opposite. We prove this theorem in paragraph \ref{sec:anos}.

During the proof we shall also prove the following lemma of independent interest 
\begin{lemma}\label{lem:sull-ano}
	Let $\rho_0$ be an Anosov representation of a Fuchsian group $\Gamma$. Assume that the limit map $\xi_0$ is $\zeta$-Sullivan, then, for any positive $\epsilon$,  any nearby (i.e., sufficiently close to $\rho_0$) representation $\rho$ is Anosov with a $(\zeta+\epsilon)$-Sullivan limit map.
\end{lemma}

The following result is  is worth stating, although we will not use it in the proof.

\begin{proposition}
		Let $\rho_n$ be a family of $\ms P$-Anosov representations of a Fuchsian group, whose limit maps are $\zeta_n$-Sullivan, with $\zeta_n$ converging to zero.  Then, after conjugation,  $\rho_n$ converges to a representation whose limit curve is a circle. 
\end{proposition}
\begin{proof}  Let $\xi_n$ be the limit curve of $\rho_n$.  By definition, there exists a tripod $\tau_n$ and and associated circle map $\eta_n$ so that
\begin{eqnarray*}
	\eta_n(0,1,\infty)&=&\partial \tau_n \\
	d_{\tau_n}(\eta_n,\xi_n)&\leq& \zeta_n\ .
\end{eqnarray*}
After conjugating $\rho_n$ by an element $g_n$, we may assume that 
$\tau_n$ and $\eta_n$ are constant equal to $\tau$ and $\eta$ respectively.
Thus we have 
\begin{eqnarray*}
	\eta(0,1,\infty)&=&\partial \tau\\
	d_{\tau}(\eta,\xi_n)&\leq& \zeta_n\ .
\end{eqnarray*}
In particular, it follows that $\xi_n$ converges uniformly to $\eta$.  Let $\rho_0$ be the Fuchsian representation associated to $\eta$.
Let now $\gamma^i$ be generators of $\Gamma$. The same argument as above show that $\xi_n\circ\gamma_i=\rho_n(\gamma^i)\circ\xi_n$ converges to $\eta\circ\gamma_i=\rho_0(\gamma_i)$. It follows that,
using the fact that the centralizer of a circle is compact, that we may extract a subsequence so that $\rho_n(\gamma_i)$ converges for all $i$ to $\rho(\gamma^i)$, where $\rho$ is a representation in $\mathsf H\times \mathsf K$ of the form $(\rho_0,\rho_1)$, where $\mathsf H$ the group isomorphic to $\sld$ associated to $\eta$, and $\mathsf K$ its centralizer..
\end{proof}

In the first paragraph of this section, we single out the consequence of the  ``compact stabilizer hypothesis''  that we shall use. 

\subsubsection{The compact stabilizer hypothesis} Our standing hypothesis will have the following consequence
\begin{lemma}\label{lem:properA} The following holds
	\begin{enumerate}
		\item	There exists a positive constant $\boldsymbol{\zeta}$, so that for every  positive real number  $M$, there exists a positive real number $N$, such that if $\xi$ is a $\boldsymbol{\zeta}$-Sullivan map, if $T_1$ and $T_2$ are two triples of distinct points in $\Rp$ with  $d(T_1,T_2)\leq M$, if $\tau_1$ and $\tau_2$ are the respective compatible tripods with respect to $\xi$, then $$
	d(\tau_1,\tau_2)\leq N.
	$$
	\item For any positive $\epsilon$ and $M$, then for $\zeta$ small enough, for any $\zeta$-Sullivan map $\xi$,   if $T_1$ and $T_2$ are two triples of distinct points in $\Rp$ with $d(T_1,T_2)\leq M$, if $\tau_1,\nu_1$ are the compatible tripods and extended circle maps of $T_1$ with respect to $\xi$,  then we may choose a compatible  tripod $\tau_2$ for $T_2$ so that
	$$
	d(\tau_2, \nu_1(T_2))\leq\epsilon.
	$$
	\end{enumerate}
\end{lemma}

Actually this lemma will be the unique consequence of our standard hypothesis which will be used in the proof. This lemma is itself a corollary of the following proposition.
\begin{proposition}\label{lem:proper}
\begin{enumerate}
    \item There exists positive constants $\boldsymbol A$ and $\boldsymbol \zeta_0$, such that if $\tau_1$ and $\tau_2$ are two tripods and  $X$  is a triple of points in $\gp$,  we have the implication
	$$
	d_{\tau_1}(X,\partial \tau_1)\leq \boldsymbol\zeta_0, \ \
	d_{\tau_2}(X,\partial \tau_2)\leq \boldsymbol\zeta_0 \implies
	d(\tau_1,\tau_2)\leq \boldsymbol{A}\ .
	$$
	\item Moreover, given $\alpha>0$, there exist $\epsilon>0$ so that
		$$
	 d_{\tau_1}(X,\partial \tau_1)\leq \epsilon, \ \  	d_{\tau_2}(X,\partial \tau_2)\leq \epsilon\implies \exists \tau_3\ ,\hbox{ with } \partial\tau_3=\partial\tau_2 \hbox{ and } d(\tau_1,\tau_3)\leq \alpha\ .
$$
\end{enumerate}
\end{proposition}
We first prove Lemma \ref{lem:properA} from   Proposition \ref{lem:proper}.
\begin{proof} Let $\boldsymbol\zeta_0$ and $\bf A$ be as in Proposition \ref{lem:proper}. 
Let first $\nu_0$ be an extended circle map with associated map $\eta_0$.
By  continuity of $\eta_0$ there exists $M$,
$$
d(T_1,T_2)\leq M\implies d_{\nu_0(T_1)}(\eta_0(T_1),\eta_0(T_2))\leq \frac{1}{2}\boldsymbol\zeta_0.
$$
The equivariance under the action of $\ms G$ then shows that the previous inequality holds for all $\nu=g\nu_0$.

Let $\boldsymbol\zeta=\frac{1}{2}\boldsymbol\zeta_0$ and 
 $\xi$ be a $\boldsymbol\zeta$-Sullivan map.
\vskip 0.1 truecm
\noindent{\em 
Proof of  the first assertion:} Let $T_1$ and $T_2$ be two tripods with $d(T_1,T_2)<M$.  Let us denote $\eta_1$ and $\eta_2$ the corresponding compatible circle maps, $\tau_1$ and $\tau_2$ the corresponding compatible tripods,  $\nu_1$ and $\nu_2$ the corresponding extended circle maps so that  $\tau_i=\nu_i(T_i)$.
Let $X=\xi(T_2)$.

Then the $\boldsymbol\zeta$-Sullivan  property  implies that 
$
d_{\tau_1}(X,\eta_1(T_2))\leq \boldsymbol\zeta$.
Then 
\begin{equation*}
d_{\tau_1}(X,\partial\tau_1)
\leq d_{\tau_1}(X,\eta_1(T_2))+d_{\nu_1(T_1)}(\eta_1(T_2),\eta_1(T_1))
\leq 2\boldsymbol{\zeta}=\boldsymbol{\zeta_0}\ .	
\end{equation*}
From the $\boldsymbol{\zeta}$-Sullivan  condition, we get$$
d_{\tau_2}(X,\partial\tau_2)=d_{\tau_2}(\xi(T_2),\nu_2(T_2))\leq\boldsymbol{\zeta}\leq\boldsymbol{\zeta_0} \ .
$$
Thus Proposition  \ref{lem:proper} implies 
$
d(\tau_2,\tau_1)\leq\boldsymbol{A}
$.
This proves  the first assertion with  $N=\boldsymbol{A}$. 
\vskip 0.1 truecm
\noindent{\em 
Proof of  the second assertion:} Let $\xi$ be a $\zeta$-Sullivan map, $T_i$, $\eta_i$ and $\tau_i$ as above.  Let  again $X=\xi(T_2)\in\gp^3$ and  $\tau_0=\eta_{1}(T_2)$. Using the definition of a $\zeta$-Sullivan map
$$
d_{\tau_1}(X,\partial\tau_0)\leq\zeta\ , \ d_{\tau_2}(X,\partial\tau_2)\leq\zeta\ .
$$ 
Moreover  
$d(\tau_0,\tau_1)=d(T_1,T_2)\leq M$.
Thus by Proposition \ref{A-B}, $d_{\tau_0}$ and $d_{\tau_1}$ are uniformly equivalent. It  follows that, for any positive $\beta$, for $\zeta$ small enough we have 
$$
d_{\tau_0}(X,\partial\tau_0)\leq\beta\ , \ d_{\tau_2}(X,\partial\tau_2)\leq\beta.
$$ 
The second part of Lemma \ref{lem:proper} guarantees us that for any positive  $\alpha$, then for $\zeta$ small enough, 
we may  choose $\tau_3$ with  the same vertices as $\tau_2$,
so that 
$$   
 d(\tau_3,\eta_1(T_2))=d(\tau_3,\tau_0)\leq\alpha\ .
$$
Observe finally that $\tau_3$ is a compatible tripod, recalling that in the case of the compact stabilizer hypothesis $d_\tau$ and the circle maps associated to $\tau$, only depends on $\partial\tau$ by Proposition \ref{pro:comp-stab}. Thus choosing $\tau_3$ concludes the proof of the lemma.
\end{proof}
Next we prove Proposition \ref{lem:proper}.
\begin{proof}  Let us first prove that  ${\G}$ acts properly on some open subset of $\gp^3$ containing the set  $V$ of vertices of tripods.

We shall use the geometry of the associated symmetric space $\ms S(\G)$. Let  $x$ be an element of  $\gp$, let $A_x$ be the family of hyperbolic elements conjugate to $a$ fixing $x$ ; observe that $A_x$ is a $\operatorname{Stab}(x)$-orbit under conjugacy.

 The family of hyperbolic elements  in  $A_x$ corresponds in the symmetric space  to an asymptotic class of geodesics at $+\infty$. Thus $A_x$ defines  a Busemann function $h_x$ well defined up to a constant. Each gradient line of $h_x$ is one of the above described geodesic. The  function $h_x$ is convex on every geodesic $\gamma$, or in other words ${\rm D}_w^2h_x(u,u)\geq 0$ for all tangent vectors  $u$. Moreover ${\rm D}_w^2h_x(u,u)=0$ if and only if the one parameter subgroup associated to the geodesic $\gamma$ in the direction of $u$ commutes with the one-parameter group associated to the gradient line of $H_x$ though the  point $w$. If now $(x,y,z)$ are three point on a circle $C$, the function $C:=h_x+h_y+h_z$ is geodesically convex. Let $\hh_C$ be the hyperbolic geodesic plane associated to the circle $C$, then $x$, $y$ and $z$ correspond to three point at infinity in $\hh_C$ and all gradient lines of $h_x$, $h_y$ and $h_z$ along $\hh_C$ are tangent to $\hh_C$. There is a unique point $M$ in $\hh_C$ which is a critical point of $H$ restricted to $\hh$. Every vector $u$ normal to $\hh$ at $M$, is  then also normal to the gradient lines of of $h_x$, $h_y$ and $h_z$ which are tangent to $\hh$, and as a consequence ${\rm D}_MH(u)=0$. Thus $M$ is a critical point of $H$. By the above discussion, ${\rm D}^2H(v,v)=0$, if and only if the one parameter subgroup generated by $u$ commutes with the $\sld$ associated to $\hh_C$.   Since, by hypothesis, this $\sld$ has a compact centralizer, $M$ is a  non degenerate critical point.

The map $G:(x,y,z)\mapsto M$ is $\ms G$ equivariant and extends continuously to some $\ms G$-invariant neighborhood $U$ of $V$ in $\gp^3$ with values in $\ms S(\G)$: to have a non degenerate minimum is an open condition on  convex functions of class $C^2$. It follows that the action of $\ms G$ on $U$ is proper since the action of $\G$ on the symmetric space $\ms S(\G)$ is proper.

We now prove the first assertion of the proposition. Let's work by contradiction, and assume that for all $n$ there exists tripods  $\tau_1^n$ and $\tau_2^n$, triple of points $X_n$ so that 
$$
d_{\tau_1^n}(X_n,\partial \tau_1^n)<\frac{1}{n},\ \ d_{\tau_2^n}(X_n,\partial \tau_2^n)<\frac{1}{n},\ \ n<d(\tau_1^n,\tau_2^n)\ .
$$
We may as well assume $\tau^1_n$ is constant and equal to $\tau$ and consider $g_n\in G$ so that $g_n(\tau^n_1)=\tau^2_n$. Thus we have,

$$
d_{\tau}(X_n,\partial \tau)\rightarrow 0,\ \ d_{\tau}(g_n(X_n),\partial \tau)\rightarrow 0,\ \ d(\tau,g_n(\tau))\rightarrow \infty.
$$	
However this last assertion contradicts the properness of the action of ${\G}$ on a neighborhood of $\partial\tau\in \gp^3$.

For the second assertion, working by contradiction again and taking limits as in the proof of the first part, we obtain two tripods $\tau_1$ and $\tau_2$ so that 
$d_{\tau_1}(\partial\tau_1,\partial\tau_2)=0$ and for all $\tau_3$ with $\partial\tau_3=\partial\tau_2$, then $d(\tau_1,\tau_3)>0$. This is obviously a contradiction. \end{proof}

\subsection{Paths of quasi tripods and Sullivan maps} \label{subsec:pathsld}

Let in this paragraph $\xi$ be a  $\zeta$-Sullivan map from a dense set $W$ of $\Rp$ to $\gp$. To make life simpler, assuming the axiom of choice, we may extend $\xi$ -- a priori non continuously -- to a $\zeta$-Sullivan map defined on all of $\Rp$: We choose for every element $z$ of $\Rp\setminus W$ a sequence $\left(w_n\right)_{n\in\mathbb N}$ in $W$ converging to $z$ so that $\xi(w_n)$ converges, and for $\xi(z)$ the limit of $\left(\xi(w_n)\right)_{n\in\mathbb N}$.

Our technical goal is, given a point $z_0$ in $\hh$ and two (possibly equal) close points $x_1$, $x_2$  with respect to $z_0$ in $\Rp$ we construct, paths of quasi-tripods ``converging'' to 
 $\xi(x_i)$. This is achieved in Proposition \ref{pro:suldef0} and its consequence Lemma \ref{lem:cuffsul}. This preliminary construction will be used  for  the main results of this section:  Theorem \ref{theo:sull-hold} and Theorem \ref{theo:sull-anos}

\subsubsection{Two paths of tripods for the hyperbolic plane}
We start with the model situation in $\hh$ and prove the following lemma which only uses hyperbolic geometry and concerns tripods for $\sld$, which in that case are triple of pairwise distinct points in $\Rp$.

\begin{lemma}\label{pro:holcuff}
	There exist universal positive constants $\bk_1$ and $\bk_2$ with the following property:

Let  $z_0$ be a point  in $\hh$,  $x_1$ and $x_2$ be two points in $\Rp$, so that $d_{z_0}(x_1,x_2)$ is small enough (and possibly zero),  then there exists two 2-sequences of tripods $\vect{T}^1$ and $\vect{T}^2$,  where $z_0$ belongs to the geodesic arc corresponding to the initial chord of both $\vect{T}^1$ and $\vect{T}^2$,  with the following properties -- see Figure \eqref{fig:2cuff}
	\begin{enumerate}
		\item we have that $\lim \vect{T}_i=x_i$.
	\item the sequences  $\vect{T}_1$ and $\vect{T}_2$ coincide for the first $n$ tripods, for $
	n$ greater than $-\bk_1\log d_{z_0}(x_1,x_2)$,
    \item Two successive tripods  $T^i_m$ and $T^i_{m+1}$ are at most $\bk_2$-swished. \label{pro:holcuff4}
    \item Defining the $\sld$-tripods $x^i_m\defeq(\partial^-T^i_m,\partial^+T^i_m, x^i)$  then $d(x^i_m,T^i_m)\leq \bk_2$. \label{pro:holcuff5},
 \end{enumerate}
\end{lemma}
 In  item \ref{pro:holcuff4} of this lemma, we use a slight abuse of language by saying $T$ and $T'$ are swished whenever actually $\omega^p\tau$ and $\omega^{q}T'$ are swished for some integers $p$ and $q$.

In the the proof of the Anosov property for equivariant Sullivan curve, we will use the  ``degenerate construction'', when $x_1=x_2\eqdef x_0$, in which  case $\vect{T}^1=\vect{T}^2\eqdef \vect{\tau}$, whereas we shall use the full case for the proof of the H{\"o}lder property.

\begin{figure}[h]
\centering
  \includegraphics[width=0.35\textwidth]{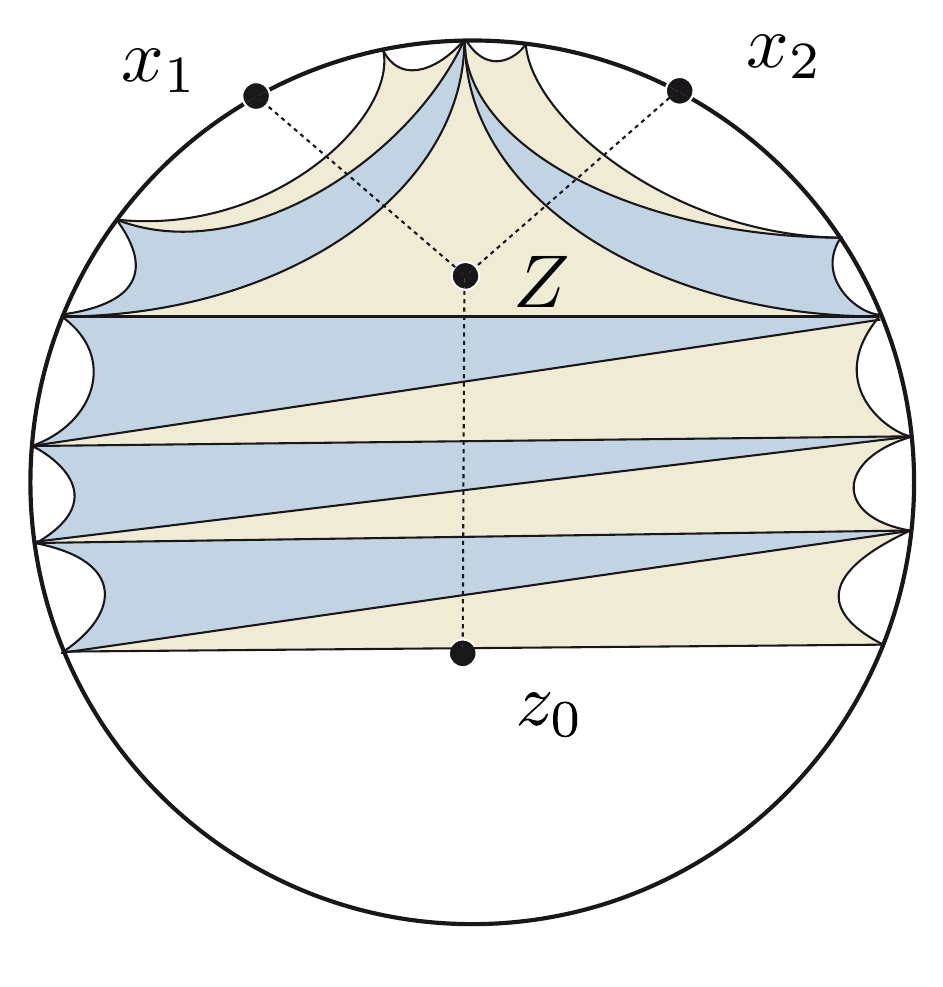}
  \caption{Two paths of tripods}\label{fig:2cuff}
\end{figure}

\begin{proof} The process is clear from Picture \eqref{fig:2cuff}. Let us make it formal. Let $x_1$ and $x_2$ be two points in $\Rp$ and  assume that $d_{z_0}(x_1,x_2)$ is small enough. If $x_1\not= x_2$, we can now find three geodesic arcs $\gamma_0$, $\gamma_1$ and $\gamma_2$ joining in a point $Z$ in $\hh$ with angles $2\pi/3$ so that their other extremities are respectively $z_0$, $x_1$ and $x_2$. The arc $\gamma_0$ is oriented from $z_0$ to $Z$, whilst the others are from $Z$ to $x_i$ respectively. The tripod $\tau^0$ orthogonal  to all three geodesic arcs $\gamma_0$, $\gamma_1$, and $\gamma_2$ will be referred in this proof as the {\em forking tripod} and the point of intersection of $\gamma_i$ with $\tau^0$ is denoted $y_i$. 

Observe now that there exists a universal positive constant $\bk_1$ so that 
\begin{equation}
{\operatorname{length}}
(\gamma_0)=d_{\hh}(z_0,Z)\geq  - 2\bk_1\log\left(d_{z_0}(x_1,x_2)\right)\ , \label{ineq:minlog}	
\end{equation}
where $d_{\hh}$ is the hyperbolic distance. We now construct a (discrete) lamination $\Gamma$ with the following properties
\begin{enumerate}
\item $\Gamma$ contains the three sides of the forking tripod, and $z_0$ is in the support of $\Gamma$.
\item All geodesics in $\Gamma$ intersect orthogonally, either $\gamma_0$,  $\gamma_1$ or $\gamma_2$. Let $X$ be the set of these intersection points.
\item The distance between any two successive points in $X$ (for the natural ordering of $\gamma_0$, $\gamma_1$ and $\gamma_2$)  is greater than 1 and less than 2.
\end{enumerate} 
We orient each geodesic in $\Gamma$ so that its intersection with $\gamma_0$, $\gamma_1$ or $\gamma_2$ is positive. We may now construct two sequences of geodesics $\Gamma^1$ and $\Gamma^2$ so that $\Gamma^i$  contains all the geodesics in $\Gamma$ that are encountered successively when going from $z_0$ to $x_i$.

For two successive geodesics  $\gamma_i$ and $\gamma_{i+1}$ -- in either $\Gamma^1$ or $\Gamma^2$ --  we consider the associated finite paths of tripods given by the following construction: 
\begin{enumerate}
\item In the case  $\gamma_i$ and $\gamma_{i+1}$ are both sides of the forking tripod:, the path consists of just one tripod: the forking tripod
\item In the other case: we consider the path of tripods with two elements  $\tau_i$ and $\hat\tau_i$ where
\begin{eqnarray*}
	\tau_i&=&(\gamma_i(-\infty), \gamma_i(+\infty), \gamma_{i+1}(+\infty))\ ,\\
\tau_{i+1}&=&(\gamma_i(-\infty), \gamma_{i+1}(+\infty), \gamma_{i+1}(-\infty))\ .
\end{eqnarray*} 
\end{enumerate}
Combining these finite paths of tripods in infinite sequences one obtains two sequences of tripods $\vect{T}^i=\{T^i_m\}_{m\in\mathbb N}$, with $i\in\{1,2\}$ which coincides up to the first $-\bk_1\log\left(d_{z_0}(x_1,x_2)\right)$ tripods.
Moreover the swish between $T_m^i$ and $T^i_{m+1}$ is bounded by a universal constant $\bk_2$ -- of which actual value we do not care, since obviously the set of configurations $(T_m^i,T^i_{m+1})$ is compact (up to the action of $\sld$).

An easy check shows that these sequences of tripods are 3-sequences. 
The last condition is immediate after possibly enlarging the value of $\bk_2$ obtained previously.
\end{proof}

\subsubsection{Sullivan curves as deformations}\label{subsec:defcuff}

Let $z_0$, $x_1$, $x_2$, $\vect{T}^1$ and $\vect{T}^2$ be as in Lemma \ref{pro:holcuff}. Let $\xi$ be an $\zeta$-Sullivan map. The main idea is that  $\xi$ will define a deformation of the sequences of tripods. Our first step is the following lemma

\begin{lemma}\label{lem:holcuf}
For every positive $\epsilon$, there exists $\zeta$, so that for every $i\in\{1,2\}$ and $m\in\mathbb N$ there exist a compatible tripod  $\tau^i_m$ for $T^i_m$ with respect to $\xi$, with associated circle maps $\eta^i_m$ and extended circle maps $\nu^i_m$,  
so that denoting by $d_m^i$ the metric $d_{\tau^i_m}$ we have
\begin{eqnarray}
\partial^\pm\tau^i_m&=&\xi(\partial^\pm T^i_m)\ ,\label{eq:lemdeftau1} \\
d_m^i(\xi,\eta^i_m)&\leq&\epsilon\ , \label{eq:lemdeftau0}\\ 
d(\tau^i_{m},\nu^i_{m-1}(T^i_{m}))&\leq&\epsilon\  .\label{eq:lemdeftau}
\label{lem:holcufreduce}
\end{eqnarray}
Moreover for all $m$ smaller than $-\bk_1\log d_{z_0}(x_1,x_2)$, we have $\tau^1_m=\tau^2_m$.
\end{lemma}

\begin{proof} Let us construct inductively the sequence $\underline{\tau}_i$ . Let us first construct $\tau^1_0=\tau^2_0$.  We first choose  a compatible tripod for $T_0$,   with associated circle maps $\eta^1_0=\eta^2_0$ and extended  circle maps $\nu^1_0=\nu^2_0$. Let  $\tau_0^i=\eta_0^i(T_0)$ 
so that denoting by $d_0^i$ the metric $d_{\tau^i_0}$, we have the inequality
\begin{equation}
d^i_0(\xi,\eta^i_0)\leq\zeta\ . \label{eq:defsul}
\end{equation}
In particular $$
d_0^i(\partial^\pm \tau^i_0,\xi(\partial^\pm T^i_0))\leq\zeta\, .
$$
we may thus slightly deform  $\eta^i_0$ (with respect to the metric $d^i_0$) so that assertion \eqref{eq:lemdeftau1} holds.   Then for $\zeta$ small enough, the relation  \eqref{eq:lemdeftau0} holds for $m=0$, where $\epsilon=2\zeta$

Assume now that we have built the sequence up to $\tau^i_{m-1}$.
Let then 
$$
\tau_1=\tau^i_{m-1}\ , \ \ T_1=T^i_{m-1}\ , \ \ T_2=T^i_{m}\ , 
$$ 
and finally $\nu_1=\nu^i_{m-1}$. Recall that by the construction of $T_1$ and  $T_2$, 
  $$d(T_1,T_2)\leq\bk_2\eqdef M\ ,\ \  d(\nu_1(T_1),\tau_1)\leq\epsilon\ . $$ We may now apply the second part of Lemma \ref{lem:properA}, which shows that given $\epsilon$ and $\zeta$ small enough, we may choose a compatible  tripod $\tau_2$ for $T_2$ with respect to $\xi$ so that,  
$$
d(\tau_2,\nu_1(T_2))\leq \frac{1}{2}\epsilon\ . 
$$
We now set $\tau^i_m\eqdef\tau_2$, possibly deforming it a little so that assertion \eqref{eq:lemdeftau1} holds. Then by the definition of compatibility  assertion \eqref{eq:lemdeftau0} holds, while 
assertion \eqref{eq:lemdeftau} is by construction.
The last part of the lemma follows from the inductive nature of our construction and some bookkeeping. \end{proof}
\subsubsection{Main result}
Let $\xi$ be a $\zeta$-Sullivan curve. We use the notation of  the two previous lemmas. Our main result is
\begin{proposition}\label{pro:suldef0}
For all positive $\epsilon$, for $\zeta$ small enough,
\begin{enumerate}
	\item the quadruples $\theta_m^i\defeq\left(\tau^i_m,\xi(\partial T^i_m)\right)$ are reduced  $\epsilon$-quasi-tripods.
	\item If $T^i_m$ and $T^i_{m+1}$ are $R^i_m$-swished then $\theta^i_m$ and $\theta^i_{m+1}$ are $(R^i_m,\epsilon)$-swished. 
	\item \label{it:eps-def} The sequences  $\underline{\theta}^1$ and $\underline{\theta}^2$  are $\epsilon$-deformations of the sequence $\nu_0(\underline T^1)$ and $\nu_0(\underline T^2)$ respectively. 
	\item The $n$ first elements  of $\underline{\theta}^1$ and $\underline{\theta}^2$  coincide for $n$ equal $-\bk_1\log(d_{z_0}(x_1,x_2))$.
 \item For all $m$, $\xi(x_i)$ belongs to the sliver $S_\epsilon(\tau^i_m)$.
\end{enumerate} 
\end{proposition}

\begin{proof}
	Equation \eqref{eq:defsul} guarantees that $\theta^i_m$ is a $\zeta$-quasi tripod and reduced  by condition \eqref{eq:lemdeftau1}. Furthermore since $T^i_m$ is at most $\bk_2$ swished from $T^i_{m-1}$ by Proposition \ref{pro:holcuff},   inequality \eqref{eq:lemdeftau} implies that  $\theta^i_{m+1}$ is at most $(R^i_m,\epsilon)$-swished from $\theta^i_{m}$, and thus $\vect{\tau}^i$ is a model for $\vect{\theta}^i$.  Statement \ref{it:eps-def} then follows from Proposition \ref{pro:modeldef}.
	  The coincidence up to $-\bk_1\log(d_{z_0}(x_1,x_2))$. follows from the last part of Lemma \ref{lem:holcuf}. Let us prove the last item in the proposition. By the $\zeta$-Sullivan condition
	$$
	d^i_m(\xi(x^i),\eta^i_m(x^i)) \leq\zeta.
	$$
	Let $x^i_m$ be the $\sld$-tripod   as in Proposition \ref{pro:holcuff},  let   $\sigma^i_m=\nu^i_m(x^i_m)$ and $
	\underline d^i_m\defeq d_{\sigma^i_m}.
	$. By construction, $\sigma^i_m$ and $\tau_i^m$ are coplanar
	By statement \ref{pro:holcuff5} of Lemma \ref{pro:holcuff}  , $d(x^i_m, T^i_m)$ is bounded by a constant $\bk_2$, thus  by Proposition \ref{A-B}  $d^i_m$ and $\underline d^i_m$ are uniformly equivalent by constants only depending on ${\G}$ and $\bk_2$. Thus for $\zeta$ small enough we have
	$$
	\underline d^i_m(\xi(x^i),\partial^0 \sigma^i_m)=\underline d^i_m(\xi(x^i),\eta^i_m(x^i)) \leq\epsilon\ .
	$$
	In other words, $\xi(x^i)$ belongs to the cone $C_\epsilon(\sigma^i_m)$ hence to the sliver $S_\epsilon(\hat\tau^i_m)$ as required, since $\sigma^i_m$ and $\tau_i^m$ are coplanar and $\partial^\pm\sigma^i_m= \partial^\pm\tau^i_m$.
	\end{proof}

\subsubsection{Limit points}\label{subsec:cufflimi}

Let then $\Gamma^i_m$ be the chords generated by the tripods $\hat{\theta}^i_{2m}$,  and let us consider the sequences of chords $\underline{\Gamma}_i\defeq\{\Gamma^i_m\}_{m\in\mathbb N}$. 
The final part of our construction is the following lemma

\begin{lemma}\label{lem:cuffsul}
	The sequence of chords $\underline\Gamma^i$ are $(1,\epsilon)$-deformed sequences of cuffs for  $\zeta$ small enough. Furthermore these two sequences coincides up to
 $N>-\bk_1d_{\tau}(z_1,z_2)$. Finally
\begin{equation}
	\bigcap_{m=0}^\infty\underline{\Gamma}^i=\{\xi(x_i)\}\label{eq:limsul}
\end{equation}
\end{lemma}

\begin{proof}
The first two items of Proposition \ref{pro:suldef0}, together with Proposition \ref{pro:modeldef} implies that for $\zeta$ small enough  the sequence $\underline{\theta}^i$ are $\epsilon$-deformations of the model sequences $\nu_0\left(\underline{T}^i\right)$.  This implies the first two  assertions. Equation \eqref{eq:limsul} follows by Theorem \ref{theo:exislimi} (taking $\ell_0=R=1$ and $\beta={\rm A}$), and the last item of Proposition \ref{pro:suldef0}. 
\end{proof}

\subsection{Sullivan curves and the H{\"o}lder property}\label{sec:holder}

We prove a more precise version of  Theorem \ref{theo:sull-hold} that we state now

\begin{theorem}\label{theo:HolQuant}{\sc [Hölder modulus of continuity]}
There exists positive constants ${\rm M}$, $\zeta_0$ and $\alpha$ so that given a $\zeta_0$-Sullivan  map $\xi$ from ${\bf P}^1(\mathbb R)$ to $\gp$, then for every tripod $T$ in ${\bf P}^1(\mathbb R)$, with  associated $\ms G$-tripod $\tau$, with respect to $\xi$, we have
$$
d_\tau(\xi(x),\xi(y))\leq {\rm M}\cdotp d_T(x,y)^{\alpha}\ ,
$$
\end{theorem}

\begin{proof} Since $d_\tau$ has uniformly bounded diameter,
It is enough to prove this inequality, for $T$ so that $d_T(x,y)$ is small enough. 
Let  then $x_1=x$, $x_2=x$ be  in $\Rp$ and $z_0=s(T)$, $\xi$ a $\zeta$-Sullivan map  (for $\zeta$ small enough) and ${\underline T}^i$, $\underline{\tau}_i$, the sequences of $\sld$-tripods and ${\G}$-tripods  constructed  in the preceding section, let  $\underline\Gamma_i$ the sequence of chords satisfying Lemma \ref{lem:cuffsul}. Let 
$$
\tau_0\defeq \tau^1_0=\tau^2_0,\ \ \nu_0\defeq \nu^1_0=\nu^2_0\ , \ T_0\defeq T^1_0=T^2_0
$$
Let $N>-\bk_1 \log(d_{T_0}(x_1,x_2))$ so that $\underline\tau^1$ and $\underline\tau^2$ coincide up to the first $N$ tripods.
By Theorem \ref{theo:exislimi} using Lemma \ref{lem:cuffsul},  we have \begin{equation}
	d_{\tau_0}(\xi(x_1),\xi(x_2))\leq {\rm q}^N\cdotp {\rm A}\leq {\rm B} \cdotp d_{z_0}(x_1,x_2)^{\boldsymbol \alpha}= {\rm B} \cdotp d_{T_0}(x_1,x_2)^{\boldsymbol \alpha}\, \label{ineq:hol1}
\end{equation}
for some positive constants ${\rm B}$ and ${\boldsymbol \alpha}$ only depending on ${\rm q}$, ${\rm A}$ and $\bk_1$.

Here $\tau_0$ is associated to $T_0$. But since $d(T_0,T)$ is uniformly bounded, by the first assertion in Lemma \ref{lem:properA}, $d(\tau_0,\tau)$ is uniformly bounded (for $\zeta$ small enough), thus by Proposition \ref{A-B}, $d_\tau$ and $d_{\tau_0}$ are uniformly equivalent. In particular,
$$
d_\tau(\xi(x),\xi(y))\leq {\rm F}\cdotp d_{\tau_0}(\xi(x),\xi(y))\leq {\rm M}\cdotp d_T(x,y)^{\alpha}\ .
$$
This concludes the proof.\end{proof}

\subsection{Sullivan curves and the Anosov property}\label{sec:anos}

In this section, let $\xi$ be a $\zeta$-Sullivan map equivariant under the action of a cocompact Fuchsian group $\Gamma$ for a representation $\rho$ of $\Gamma$ in ${\G}$.

\subsubsection{A short introduction to Anosov representations}\label{subsec:introAnos} Intuitively, a hyperbolic group is $\ms P$-Anosov if every element is $\ms P$-loxodromic, with ``contraction constant'' comparable with the word length of the the group.

Let us be more precise, let $\ms P^+$ be a parabolic and $\ms P^-$ its opposite associated to the decomposition
$$
\mk g= \mk n^+\oplus \mk l \oplus \mk n^-\ , \ \ \mk p^\pm=\mk n^\pm\oplus \mk l 
$$
For a hyperbolic surface $S$, let $\ms U S$ be its unit tangent bundle equipped with its geodesic flow $h_t$. Let $\rho$ be a representation of $\pi_1(S)$ into $\G$. Let $\mk G_\rho$ be the flat Lie algebra bundle over $S$ with monodromy $\Ad\circ\rho$. The action of $h_t$ lifts by parallel transport the action of a flow $H_t$ on $\mk g_\rho$. We say that the action is Anosov if we can find a continuous splitting into vector sub-bundles, invariant under the action of $H_t$ 
$$
\mk G_\rho=\mk N^+\oplus \mk l \oplus \mk N^- ,
$$
such that 
\begin{itemize}
	\item at each point $x\in \ms U S$, the splitting is conjugate to the splitting $\mk g=\mk p^+\oplus \mk l \oplus \mk p^-$,
	\item The action of $H_t$ is contracting towards the future on 
	$\mk N^+$ and contracting towards the past on $\mk N^-$ 
\end{itemize}
Equivalently let $\ms F^\pm_\rho$ be the associated flat bundles to $\rho$ with fibers $\ms G/\ms P^\pm$.  The action of $h_t$ lifts to an action denoted $H_t$ by parallel transport.
Then, the representation $\rho$ is Anosov, if we can find continuous $\rho$-equivariant maps  $\xi^\pm$ from $\partial_\infty\pi_1(S)$ into $\G/\ms P^\pm$ so that
\begin{itemize}
	\item for $x\not= y$, $\xi^+(x)$ is transverse to $\xi^-(y)$,
 \item the associated  sections $\Xi^\pm $ of  $\ms F^\pm$ over $\ms U S$ by $\rho$ are attracting points, respectively towards the future and the past, for the action of $H_t$ on the space of sections endowed with the uniform topology.
 \end{itemize}

\subsubsection{A preliminary lemma}

 For a tripod $\tau$, let $\tau^\perp$ be the coplanar tripod to $\tau$ so that $\tau^\perp$ is obtained after a $\pi/2$ rotation of $\tau$ with respect to $s(\tau)$. In other words, $\partial \tau^\perp=(\partial^0\tau, x,\partial^+\tau)$ where $x$ is the symmetric of $\partial^0\tau$ with respect to the geodesic whose endpoints are $\partial^-\tau$ and $\partial^+\tau$.  Observe that $s(\tau)=s(\tau^\perp)$ and thus $d_\tau=d_{\tau^\perp}$.

Our key lemma is the following
 
\begin{lemma}\label{lem:anosul} There exists $\boldsymbol{\zeta}$  with  such that if $\xi$ is a $\boldsymbol{\zeta}$-Sullivan map, then there exists positive constants $R$ and $c$ so that if $T$ is a tripod in $\sld$, then for any    $\tau$ and $\sigma$ compatible tripods (with respect to $\xi$) to $T$ and $\varphi_R(T)$ satisfying 
$$
\partial^+\sigma=\partial^+\tau=\xi(\partial^+T)\ ,
$$
we have 
$$
\forall x,y\in C_c(\sigma^\perp), \ \ d_{\tau^\perp}(x,y)\leq \frac{1}{2}\cdotp d_{\sigma^\perp}(x,y)\ .
$$
\end{lemma}

In this lemma, $\xi$ does not have to be equivariant. Observe also, that with the notation of the lemma $\partial^0(\sigma^\perp)=\xi(\partial^+T)$.

\begin{proof}  We will use the Confinement Lemma \ref{lem:zigzag}. Let then, using the notation of the Confinement Lemma,  $b\defeq\beta_3$, and $\ell_0$ an integer greater than $\ell(\beta_3)$, and $\eta_0$ as in the conclusion of the lemma. 

 	Let $z_0\defeq s(T)$ be the orthogonal projection of $\partial^0T$ on the geodesic joining $\partial^-T$ to $\partial^+T$.  Let $x_1=x_2\defeq \partial^+T$. Let us now construct, for $\epsilon\leq \frac{\eta_0}{2\ell_0}$ and  $\boldsymbol{\zeta}$ small enough as in paragraphs \ref{subsec:pathsld} and \ref{subsec:defcuff} 
\begin{itemize}
\item The sequence of $\sld$-tripods $\underline T\defeq\underline T^1=	\underline T^2$ with $T_0=T^\perp$, associated to the coplanar sequence of chords  $\underline h$, 
\item  The tripods $\tau_m\defeq\tau^1_m=\tau^2_m$, and the corresponding sequence of reduced $\epsilon$-quasi tripods $\underline\theta\defeq\underline\theta^1=\underline\theta^2$, which is an $\epsilon$-deformation of $\nu_0(\underline\tau)$ -- according to Proposition \ref{pro:suldef0} --  and associated to the deformed sequence of chords $\underline\Gamma$, 
\item we also denote by $\nu_i$ the extended circle map associated to $T_i$ that satisfies $\nu(T_i)=\tau_i$. Let us also denote by $\mu_i$ the $\sld$-tripods which is the projection of $h_{2(i+1)}$ on $h_{2i}$, and 
 $$
 \lambda_i\defeq \nu_{2i\ell_0}(\mu_{i\ell_0}).
 $$
\end{itemize}
It follows that $T_{2\ell_0m},\ldots,T_{2\ell_0(m+1)}$ is a strong $(\ell_0,2\ell_0)$-coplanar path of tripods. And thus according to the Confinement Lemma \ref{lem:zigzag} and our choice of constants, $(\Gamma_{2\ell_0m},\Gamma_{2\ell_0(m+1)})$ is $(b,\bk^7)$-squeezed and its  commanding tripod  is the projection of $\nu_{2\ell_0m}(h_{2\ell_0{m+1}})$ on  $\nu_{2\ell_0m}(h_{2\ell_0m})$ that is $\lambda_m$. In other words, since $\lambda_{m+1}\in S_0(\nu_{2\ell_0m}(h_{2\ell_0m}))$ we have for all $m$
$$
C_b(\lambda_m)\prec \bk^7 C_{\bk^7b}(\lambda_{m+1})
$$
Thus by Corollary \ref{coro-conv-nest}, using the fact that  $\beta_3\leq \alpha_3$, where $\alpha_3$ is the  constant in Proposition, we have 
\begin{equation}
	\forall u,v\in C_b(\lambda_n), \ d_{\lambda_0}(u,v)\leq \frac{1}{2^n}d_{\lambda_n}(u,v)\ .\label{ineq:lambdan}
\end{equation}

We now make the following claim
\vskip 0.2 truecm

\noindent{\em Claim1: there exists a constant  $N$ only depending in ${\G}$ so that for any tripod $\beta$ compatible with $\varphi_{2n\ell_0}(T)$ then 
\begin{equation}
	d(\beta^\perp,\lambda_n)\leq N\, .\label{ineq:lambdatau}
\end{equation}}

Elementary hyperbolic geometry first shows that there exist positive constants  $N_1$ and $M_2$ so that 
\begin{eqnarray*}
d(\lambda_n,\tau_{2n\ell_0})=d(\mu_{n\ell_0},T_{2n\ell_0})&\leq &N_1\ , \\
d(\varphi_{2n\ell_0}(T),T_{2n\ell_0})&\leq& M_2.
\end{eqnarray*}
By Lemma \ref{lem:properA},  there exists a constant $N_2$ so that 
\begin{equation*}
	d(\beta,\tau_{2n\ell_0})\leq N_2\ ,
\end{equation*}

Since there exists a constant $N_3$ so that 
$d(\beta,\beta^\perp)\leq N_3$,
The triangle  inequality  yields the claim.
\vskip 0.2 cm
inequality \eqref{ineq:lambdatau} and Proposition \ref{A-B} yields that there exists a constant $C$ so that if $\sigma_n$ is compatible with $\varphi_{n\ell_0}(T)$, then 
\begin{equation}
	\frac{1}{C} d_{\sigma^\perp_n}\leq d_{\lambda_n}\leq C d_{\sigma^\perp_n}\  . \label{ineq:eqdis}
\end{equation}
Then taking $n_0$ so that $2^{n_0-1}> C^2$, $R=n_0\ell_0$, we  get from inequality \eqref{ineq:lambdan}
$$
\forall x,y\in C_b(\lambda_{n_0}), \ \ d_{\tau^\perp}(x,y)\leq \frac{1}{2}\cdotp d_{\sigma^\perp}(x,y)\ .
$$
To conclude, it is therefore enough to prove that 

\vskip 0.2 truecm
\noindent{\em Final Claim : There exists a constant $c$ only depending on ${\G}$ so that 
$$
C_c(\sigma^\perp)\subset C_b(\lambda_{n_0})\ .
$$
}
Recall that by hypothesis, $\partial^+(\sigma)=\xi(x)$.
By the last item in Proposition \ref{pro:suldef0}, for $\zeta$ small enough $$\partial^0(\sigma^\perp)=\xi(x)\in S_{b/2}(h_n)\ ,$$
 for all $n$. By the squeezing property, it follows that $\xi(x)\in C_{b/2}(\lambda_m)$ for all $m$.

Since $d_{\lambda_{n_0}}$ and $d_{\sigma^\perp}$ are uniformly equivalent
by inequality \eqref{ineq:eqdis}, we obtain, taking$c=b(2C)^{-1}$, 
\begin{eqnarray*}C_c(\sigma^\perp)=\{u\in \gp, d_{\sigma ^\perp}(u,\xi(x))\leq c\}
\subset \{u\in \gp, d_{\lambda_{n_0}}(u,\xi(x))\leq b/2\}
=  C_{b}(\lambda_{n_0})\ .
\end{eqnarray*}
This concludes the proof of the final claim, hence of the lemma.
\end{proof}

\subsubsection{Completion of the proof of Theorem \ref{theo:sull-anos}}
The proof is now standard. Let $\rho$ be a representation of a cocompact torsion free Fuchsian group $\Gamma$. Let  $\mathcal S$ be the space of $\sld$ tripods, $U=\Gamma\backslash\mathcal S$ the space of tripods in the quotient equipped with the flow $\varphi$. The space $U$ with its flow $U$ is naturally conjugate to the geodesic flow of the underlying hyperbolic surface. Let finally $\mathcal F$ be the $\rho$-associated flat bundle over $U$ with fiber $\gp$.  This fiber bundle is equipped with a flow $\{\Phi_t\}_{t\in\mathbb R}$ lifting the flow $\{\varphi_t\}_{t\in\mathbb R}$ by parallel transport along the orbit.

Let $\xi$ be a $\rho$-equivariant $\zeta$-Sullivan map for $\zeta$ small enough so that Lemma \ref{lem:anosul} holds. Observe that $\xi$ give now rise to two transverse  $\Phi_t$-invariant sections of  $\gp$:
\begin{eqnarray*}
\sigma^+(T)\defeq \xi(\partial^+ T)\ , & &  \sigma^-(T)\defeq \xi(\partial^- T)\	
\end{eqnarray*}
These sections are transverse for $\zeta$ small enough: more precisely for $\zeta<k/2$, where $k=d_\tau(\partial^+\tau,\partial^-\tau)$ for any tripod $\tau$. 

We now choose a fiberwise metric $d$ on  $\mathcal F$ as follows: for every $T\in \mathcal S$, let $\tau(T)$ be a compatible tripod. We may choose the assignment $T\mapsto\tau(T)$ to be $\Gamma$-invariant. We define our fiberwise metric at $T$ to be $d_T\defeq d_{\tau(T)}$. This metric may not be continuous transversely to the fibers, but it is locally bounded: locally at a finite distance to a continuous metric since the set of compatible tripods has a uniformly bounded diameter by Lemma \ref{lem:properA}.

Now, lemma \ref{lem:anosul} exactly tells us that $\sigma^+$ is a attracting fixed section of $\{\phi_t\}_{t\in\mathbb R}$ towards the future, and by symmetry that  $\sigma^-$ is a attracting  fixed section of $\{\phi_t\}_{t\in\mathbb R}$ towards the past. By definition, $\rho$ is $\gp$-Anosov and this concludes the proof of Theorem \ref{theo:sull-anos}

\subsubsection{Anosov and Sullivan Lemma}

As an another relation of the Anosov property and Sullivan curves, let us prove the  following 
\begin{lemma}
	Let $\rho_0$ be an Anosov representation of a Fuchsian group $\Gamma$. Assume that the limit map $\xi_0$ is $\zeta$-Sullivan, then, for any positive $\epsilon$,  any nearby representation to $\rho$ is Anosov with a $(\zeta+\epsilon)$-Sullivan limit map.
\end{lemma}

\begin{proof} By the stability property of Anosov representations \cite{Labourie:2006,Guichard:2012eg} any nearby  representation $\rho$ is Anosov. Let  $\xi_{\rho}$ be its limit map.

	By Guichard--Wienhard \cite{Guichard:2012eg}  -- see also \cite{Bridgeman:2015ba} --  $\xi_\rho$ depends continuously on $\rho$ in the uniform topology. More precisely, for any positive $\epsilon$, for any  tripod $\tau$ for $\ms G$, there exists a neighborhood $U$ of $\rho_0$, so that for all $\rho$ in $U$, for all $x\in\partial_\infty\hh$,
	\begin{equation}
		d_\tau(\xi_\rho(x),\xi_{\rho_0}(x))\leq \epsilon\ . \label{ineq:ano01}
	\end{equation}
Instead of fixing $\tau$, we may as well assume that $\tau$ belongs to a bounded set $K$ of $\mc G$, using for instance Proposition \ref{A-B}.
	
	Let us consider a compact fundamental domain $D$ for the action of $\Gamma$ on the space of tripods with respect to $\hh$. For every tripod $T$ in $D$, we have a compatible  $\ms G$-tripod $\tau_\T$ with circle map $\nu_T$ with respect to $\xi_0$. Then by Lemma \ref{lem:properA}, the set 
	$$
	D_{\ms G}:=\{\tau_T\mid T\in D\}\ ,
	$$
	is bounded. Thus inequality \eqref{ineq:ano01} holds for all $\tau$ in $D_{\ms G}$ . It follows that for all $T\in D$,
	\begin{equation}
		d_{\tau_T}(\xi_\rho(x),\eta_T(x))\leq \zeta+\epsilon\ . \label{ineq:ano02}
	\end{equation}
	Using the equivariance under $\Gamma$, the inequality \eqref{ineq:ano02} now holds for all tripods $T$ for $\hh$.
	In other words, $\xi_\rho$ is $(\zeta+\epsilon)$-Sullivan. 	
\end{proof}

\subsection{Improving H\"older derivatives}
Our goal is to explain that under certain hypothesis  we can can promote a Sullivan curve with respect to a smaller subset to a full Sullivan curve.  We need a series of technical definitions before actually stating our theorem
\begin{enumerate}
	\item For every tripod $T$ for $\hh$, let $d_T$ be the visual distance on $\partial_\infty\hh$ associated to $T$. We say a subset $W$ of $\partial_\infty\hh$ is {\em $(a,T)$-dense} if 
\begin{eqnarray}
\forall x\in \partial_\infty\hh,\ \exists y\in W, \ d_T(x,y)\leq a
\end{eqnarray}
\item Let $a$ and $\zeta$ be a positive number, $Z$ a dense subset of $\partial_\infty\hh$.
Let us say a map $\xi$ from $\partial_\infty\hh$  to $\gp$ is {\em $(a,\zeta)$-Sullivan} if given any tripod $T$ in $\hh$, there exists \begin{itemize}
 \item 	a circle map $\xi_T$,
 \item  an $(a,T)$-dense subset $W_T$ of $Z$,
 \end{itemize}
so that, writing $\tau\defeq\xi_T(T)$, we have 
 for all $x$ in $W_T$, $
 d_\tau(\xi_T(x),\xi(x))\leq\zeta$.
 \item Let  $\Gamma$ a be cocompact Fuchsian group and $\rho$ a representation of $\Gamma$ in $\ms G$. Let $\xi$ be a $\rho$-equivariant  map from $\partial_\infty\hh$ to $F$. We say $\xi$ is {\em attractively coherent}  if given any $y$ point in $\partial_\infty\hh$, there exists a sequence  $\seq{\gamma}$ of elements of $\Gamma$ such that
\begin{itemize}
 \item the limit of  $\{\gamma^+_m\}_{m\in\mathbb N}$ is $y$,
 \item $\xi(y)$ is the limit of  $\seq{z}$, where $z_m$ is an attractive fixed point for $\rho(\gamma_m)$.	
 \end{itemize}
\end{enumerate}

Our Improvement Theorem is the following

\begin{theorem}{\sc[Improvement Theorem]}\label{theo:boot} Let $\Gamma$ be a cocompact Fuchsian group. Then there exists a positive constant $\zeta_2$ and there exists a positive  $a_0$ such that given
\begin{enumerate}
 \item 	a continuous family of  representations $\{\rho_t\}_{t\in [0,1]}$ of $\Gamma$ in $\ms G$,
 \item  an $(a_0,\zeta$)-Sullivan map $\xi_t$, with $\zeta\leq\zeta_2$, attractively coherent, and $\rho_t$ equivariant for each $t\in [0,1]$. Assume also that $\xi_0$ is $2\zeta$-Sullivan,  
 \end{enumerate}
Then for all $t$, $\xi_t$ is a $2\zeta$-Sullivan map.
	\end{theorem}

\subsubsection{Bootstrapping and the proof of Theorem \ref{theo:boot} }

Let us first start with a preliminary lemma

\begin{lemma}\label{lem:ac}
	Let $\rho$ be an Anosov representation of a Fuchsian group. Let $\xi$ be an attractively coherent map from $\partial_\infty\hh$ to $\gp$. Then $\xi$ is the limit map of $\rho$.
\end{lemma}

\begin{proof}  Let $\eta$ be the limit map of $\rho$.
Let $y\in\partial_\infty\hh$.	Let $\seq{\gamma}$ be as in the definition of attractively continuous. Since $\gamma^+_m$ is the attractive fixed point of $\gamma$, it follows that $\eta(\gamma^+_m)=z_m$. The  continuity of $\eta$ shows that $\eta(y)=\xi(y)$.
\end{proof}

We may now proceed to the proof. Let $\{\xi_t\}_{t\in[0,1]}$, $\{\rho_t\}_{t\in[0,1]}$, and $\Gamma$ as in the hypothesis of the theorem that we want to prove. Let  $\zeta_0$, ${\rm }_0$, $\alpha$ be as in  Theorem \ref{theo:HolQuant}. Let $\zeta_1$ so that Theorem \ref{theo:sull-anos} holds. Let finally $\zeta_2=\frac{1}{4}\min(\zeta_1,\zeta_0)$ and $\zeta\leq\zeta_2$.

Let us consider the subset $K$ of $[0,1]$ of those parameters $t$ so that $\xi_t$ is $2\zeta$-Sullivan. 

\begin{lemma}
	The set $K$ is closed.
\end{lemma}

\begin{proof} Let $\seq{t}$ be a sequence of elements of $K$ converging to $s$. For all $n$, $\{\xi_{t_m}\}_{m\in\mathbb N}$ forms an equicontinuous family by Theorem \ref{theo:HolQuant} since $2\zeta_2\leq \zeta_0$. We may extract a subsequence converging to a map $\hat\xi$ which is $\rho_s$ equivariant and $\zeta_2$-Sullivan. In particular since $2\zeta_2\leq \zeta_1$, it follows that $\rho_s$ is Anosov and
$\hat\xi$ is the limit map of $\rho_s$. By hypothesis, $\xi_s$ is attractively continuous and thus $\xi_s=\hat\xi$ by Lemma \ref{lem:ac}. This proves that $s\in K$.
\end{proof}
We prove that $K$ is open in two steps:
\begin{lemma} \label{lem:zetan}
Assume $\xi_t$ is $2\zeta$-Sullivan. Then there exists a neighborhood $U$ of $t$ so that for $s\in U$, $\xi_s$ is $\zeta_0$-Sullivan.
\end{lemma}
\begin{proof} Our assumptions guarantee that $\rho_t$ is Anosov and by the stability condition for Anosov representations \cite{Labourie:2006,Guichard:2012eg} the representation $\rho_s$ is Anosov for $s$ close to $t$.  Lemma \ref{lem:ac} implies that $\xi_s$ is the limit curve of $\rho_s$.  Lemma \ref{lem:sull-ano}  then shows  that for $s$ close enough to $t$, $\xi_s$ is $\zeta_0$-Sullivan since $2\zeta<2\zeta_2<\zeta_0$.
\end{proof}
We now prove a bootstrap lemma: 
\begin{lemma}{\sc [Bootstrap]}\label{lem:boot} There exists some constant $A$ so that for $a_0<A$, if  $\xi_s$ is $\zeta_0$-Sullivan,  
	then $\xi_s$ is $2\zeta$-Sullivan
	\end{lemma}
\begin{proof}
	This is an easy consequence of the triangle  inequality. Since $\xi_s$ is $(a_0,\zeta)$-Sullivan, for every tripod  $T$ for $\hh$,  there exists an $a_0$-dense subset $W$, a circle map  $\eta$ so that for all $
y \in W$, $d_\tau(\xi_s(y),\eta(y))\leq \zeta$,
where $\tau=\eta(T)$. Let then $x\in\partial_\infty\hh$ and $y\in W$ so that $d_{T}(x,y)\leq a_0$.
	Then
	\begin{eqnarray*}
		d_\tau(\xi_s(x),\eta(x))\leq d_\tau(\xi(x),\xi(y))+d_\tau(\eta(x),\eta(y))+d_\tau(\xi(y),\eta(y)
		\leq M_0 a_0^{\alpha_0}+a_0 +\zeta\ .
\end{eqnarray*}
 The last quantity is less than $2\zeta$ for $a_0$ small enough. This concludes the proof. \end{proof}

Thus $K$ is open:
let  $t\in K$, then by Lemma \ref{lem:zetan}, for any nearby $s$ in $K$, $\xi_s$ is $\zeta_0$-Sullivan hence $2\zeta$ Sullivan by the Bootstrap Lemma \ref{lem:boot}.  Since $K$ is non empty, closed and open,  $K=[0,1]$ and this concludes the proof of the theorem.

\section{Pair of pants from triangles}\label{sec:pairpants}

The purpose of this section is to define {\em almost closing  pairs of pants}.  These almost closing pairs of pants will play the role of almost Fuchsian pair of pants in \cite{Kahn:2009wh}. Section \ref{sec:equi} will reveal they are ubiquitous in $\Gamma\backslash\mc G$. 

These almost closing pair of pants are the building blocks for the construction of surfaces whose fundamental group injects. Themselves are built out of two tripods using symmetries, a construction reminiscent of building hyperbolic pair of pants using ideal triangles.

Our main results here  will be  a result  describing the structure of a pair of pantsTheorem \ref{theo:struct-pant} whose proof relies on the Closing Lemma \ref{lem:shad}; We will also prove proposition \ref{pro:boundaryloop} that gives information onf the boundary loops.

\vskip 0.2truecm
\begin{center}
	\fbox{In all this section, $\epsilon$ and $R$ are positive constants.}
\end{center}
\vskip 0.2truecm

 \subsection{Almost closing pair of pants}

Let $\Gamma$ be a subgroup of $\ms G$. We will  consider  not only the case of a discrete $\Gamma$ but also the case $\Gamma=\ms G$.

Given a tripod $\tau_0$, the {\em $R$-perfect pair of pants} associated to $\tau_0$ is the quintuple $(\alpha,\beta,\gamma,\tau_0,\tau_1)$ so that $\tau_0$ and $\tau_1$ are tripods, $\alpha$, $\beta$ and $\gamma$ are elements of $\ms G$ so that $\alpha\gamma\beta=\id$, and moreover the pairs $(\tau_0,\omega^2\tau_1)$, $(\omega(\tau_0),\omega\beta(\tau_1))$ and 
$(\omega^2(\tau_0),\alpha^{-1}\tau_1)$ are all $R$-sheared. 

We also consider alternatively an $R$-perfect pair of pants to be a quadruple $(T,S_0,S_1,S_2)$ so that 
$$
S_0=K\varphi_R T\ , S_2=\omega K\phi_R(\omega T)\ , S_1=\omega^2 K\phi_R(\omega^2 T)\ .
$$
with the {\em boundary loops} $\alpha$, $\beta$ and $\gamma$ so that $S_0=\alpha S_1$, $S_2=\beta S_0$ and $S_1=\gamma S_2$, so that $(\alpha,\beta,\gamma,T,S_0)$ is a perfect  pair of pants with respect to the previous definition.

\begin{figure}[h]
  \centering
  \includegraphics[width=0.5\textwidth]{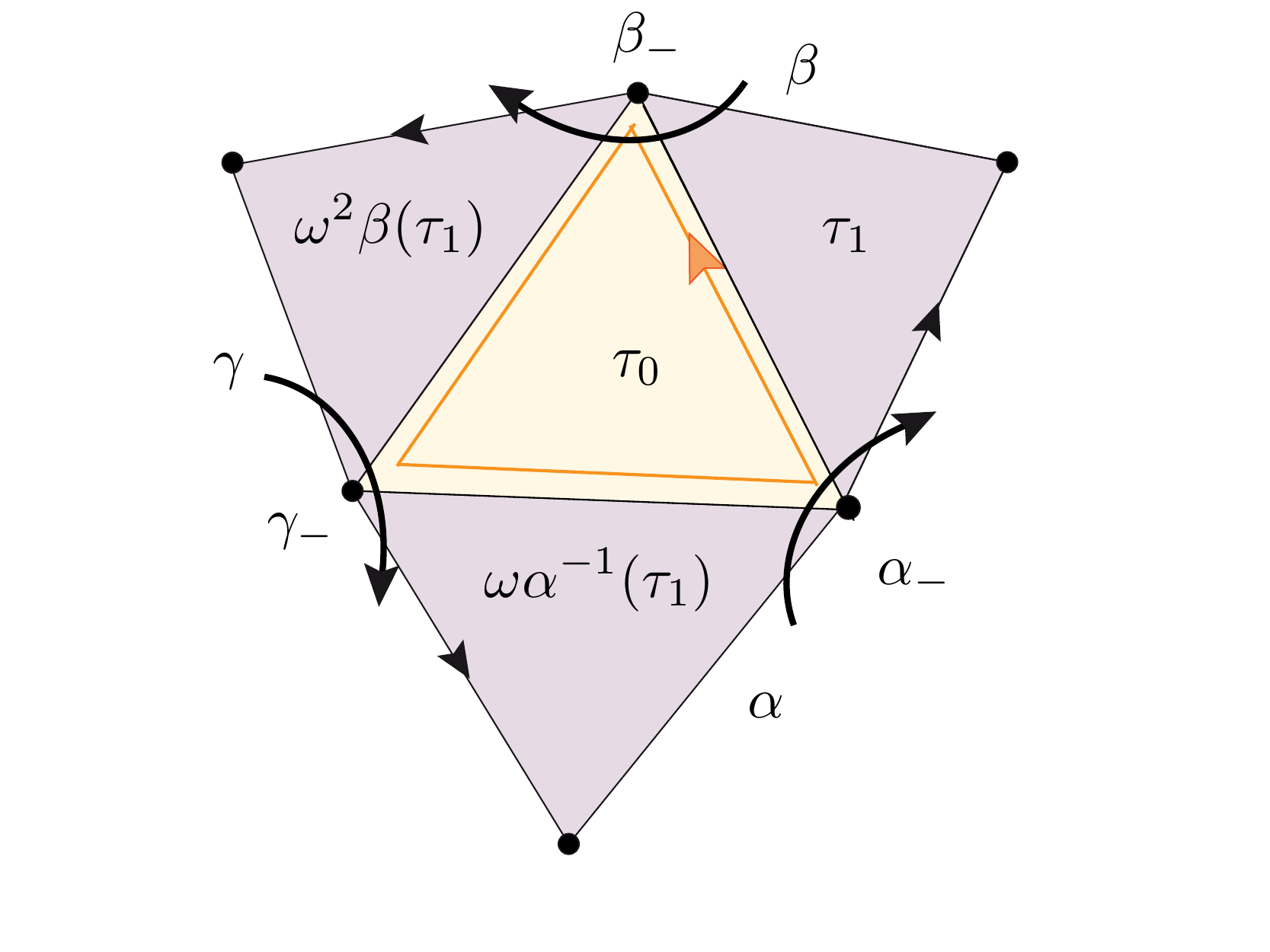}
  \caption{Pair of pants from triangles}
  \label{fig:pants}
\end{figure}

More generally we will navigate freely between quadruple of tripods $(T,S_0,S_1,S_2)$ with {\em boundary loops}  $\alpha$, $\beta$ and $\gamma$ so that $S_0=\alpha S_1$, $S_2=\beta S_0$ and $S_1=\gamma S_2$, and quintuple $(\alpha,\beta,\gamma,T,S_0)$ where $\alpha$, $\beta$ and $\gamma$ are elements of $\G$ so that $\alpha\gamma\beta=\id$ using the construction described above in the particular case of $R$-perfect pair of pants.

We now wish to deform that perfect situation.

Let $K$ be the map $x\to\omega(\overline{x})$ defined in paragraph
  \ref{mapK}.
  \begin{definition}{\sc [Almost closing]}\label{def:almost-closing}  \begin{enumerate}
	\item  Let $T$ and $S$ be two tripods in $\mc G$, $\alpha$ an element in $\ms G$ and $\mu$ a positive constant.  We say $T$, $S$ are {\em $(\mu,R)$-almost closing for $\alpha$}\index{Almost closing} if there
 exist tripods $u$ and $v$  so that
 \begin{eqnarray}
 	d(u,T)\leq \mu\ &,& \ d(v,S)\leq \mu\ ,\label{hyp:cl3}\label{hyp:cl1}\\
 	d(K\circ\varphi_R(u),S)\leq \mu\ &,&\  d(K\circ\varphi_R(v),\alpha(T))\leq \mu\ .\label{hyp:cl4} \label{hyp:cl2}
  \end{eqnarray}
\item Let $P=(\alpha,\beta,\gamma, \tau_0,\tau_1)$ so that $\alpha,\beta,\gamma\in\ms G$ with $\alpha\gamma\beta=1$, and $\tau_0,\tau_1$ are tripods, we say $P$ is  a {\em $(\mu,R)$ almost closing pair of pants}  if 
\begin{enumerate}
	\item $\tau_0,$ and $\tau_1$ are  $(\mu,R)$-almost closing for $\alpha$,
	\item $\omega^2 \tau_0,$ and $\omega \alpha^{-1}(\tau_1)$ are $(\mu,R)$-almost closing for $\gamma$, 
	\item $\omega \tau_0$ and $\omega^2 \beta(\tau_1)$ are $(\mu,R)$-almost closing for $\beta$.
 \end{enumerate}
\end{enumerate}
\end{definition}
Let us first make immediate remarks: 
\begin{proposition}\label{pro:stitc-clos}
 if  $(T,S)$ are $(\mu,R)$ almost closing for $\alpha$, then $(S,\alpha(T))$  (and then $(\alpha^{-1}(S),T)$) are also almost closing for $\alpha$.  	
\end{proposition}
\begin{proof} For the first item observe that if $(u,v)$ is the pair of tripods working in the definition for $(T,S)$, then $(v,\alpha(u))$ works for $(S,\alpha(T))$. For the second item observe that the pair $(\alpha^{-1}\tau_1,\tau_0)$ is $(\mu,R)$-almost closing for $\alpha$: we use $(\alpha_*^{-1}\tau^*_1,\tau^*_0)$ for $(u,v)$ in the definition. We apply the  first item to get that $(\tau_0,\tau_1)$ is $(\mu,R)$-almost closing. The other results for other pairs follows from symmetric considerations.
	\end{proof}
	
 We observe also the following symmetries 
 \begin{proposition}{\sc[Symmetries]}\label{pro:ac} If 	$(\alpha,\beta,\gamma,\tau_0,\tau_1)$ is an 
	$(\mu,R)$-almost closing of pants, then both
 \begin{eqnarray*}
	 & &\omega\left( \alpha,\beta,\gamma,\tau_0,\tau_1\right)\defeq\left(\beta, \gamma,\alpha, \omega(\tau_0),\omega^2(\beta\tau_1)\right)\ ,\crcr
	& &(\alpha,\beta^{-1}\alpha^{-1},\beta,\tau_1,\alpha(\tau_0))
	\end{eqnarray*}
are also $(\mu,R)$ almost closing. 
\end{proposition}
\begin{proof} We have that, using the definition of almost closing that 
\begin{enumerate}
		\item $\tau_1,$ and $\alpha(\tau_0)$ are  $(\mu,R)$-almost closing for $\alpha$, from the first item in the previous proposition.	\item After taking the image by $\alpha$,  $\omega^2 \alpha(\tau_0),$ and $\omega(\tau_1)$ are $(\mu,R)$-almost closing for $\alpha\gamma\alpha^{-1}=\beta^{-1}\alpha^{-1}$, and thus from  from the first item in the previous proposition,  $\omega(\tau_1)$ and $\omega^2 \beta^{-1}(\tau_0)$ are also $(\mu,R)$-almost closing for $\beta^{-1}\alpha^{-1}$, 
	\item $\omega \tau_0$ and $\omega^2 \beta(\tau_1)$ are $(\eR,R)$-almost closing for $\beta$ and thus from the first item in the previous proposition,  $\omega^2 \tau_1$ and  $\omega \tau_0$ are $(\eR,R)$-almost closing for $\beta$
	\end{enumerate}	
	This proves the result.
\end{proof}

\subsection{Closing Lemma for tripods} The first step in the proof of the Closing Pant Theorem is  the following lemma

 \begin{lemma}{\sc [Closing lemma]}\label{lem:shad} There exists constants $\bM_2$, $\epsilon_2$ and $R_2$, so that assuming $T$, $S$ are $(\mu,R)$ almost closing for $\alpha$ for $R>R_2$, $\mu\leq\epsilon_2$, then 
\begin{enumerate}
	\item $\alpha$ is $\ms P$-loxodromic,
	\item $d_T(T^\pm,\alpha^\pm)\leq \bM_2(\mu+\exp(-R))$ \label{it:ii-shad}
	\item Moreover, if $\tau_\alpha=\psi(T,\alpha^-,\alpha^+)$,
	$\sigma_\alpha=\Psi(S,\alpha^-,\alpha^+)$ then 
 \begin{eqnarray}
	 d(\varphi_{2R}(\tau_\alpha),\alpha(\tau_\alpha))&\leq&\bM_2(\mu+\exp(-R))\ ,\label{eq:cl1}\\ 
	 d(\varphi_{R}(\tau_\alpha),\sigma_\alpha)&\leq&\bM_2(\mu+\exp(-R))\ .\label{eq:cl2}
	\end{eqnarray}
	\item $d(T,S)\leq 2R$.
\end{enumerate}
  \end{lemma}
 In the sequel $\bM_i$, $R_i$ and $\epsilon_i$ will denote positive constants only depending on $\ms G$.
 
As an immediate  consequence and using  Proposition \ref{pro:ac}, we get the following  structure theorem for almost closing pair of pants:

\begin{theorem}{\sc[Structure of pair of pants]}\label{theo:struct-pant}
There exist positive constants  $\bM_0$, $\epsilon_0$ and $R_0$  only depending on ${\G}$ with the following property.
Let $\epsilon\leq \epsilon_0$ and $R\geq R_0$. Then for any  $(\epsilon,R)$-almost closing pair of pants
  $(\alpha,\beta,\gamma,\tau_0,\tau_1)$, we have that
	\begin{enumerate}
	\item $\alpha$, $\beta$ and $\gamma$ are all $\ms P$-loxodromic.
	\item the quadruples $
	(\tau_0,x,y,z)
	$, with $x\in\{\alpha^-,\gamma^+\}$, $y\in \{\alpha^+,\beta^-\}$ and  $z\in \{\gamma^-,\beta^+\}$  and $
	(\tau_1,u,v,w),
	$ with $u\in\{\alpha^-,\beta^+\}$, $v\in\{\alpha^+,\beta^{-1}(\gamma^-)\}$ and $w\in\{\beta^-,\beta^{-1}(\gamma^+)\}$, 
	are all $\bM_0(\epsilon+\exp(-R))$-quasi tripod,
	\item Moreover, if $\tau_\alpha=\Psi(\tau_0,\alpha^-,\alpha^+)$ and $\sigma_\alpha=\Psi(\tau_1,\alpha^-,\alpha^+)$ then 
 \begin{eqnarray}
	 d(\varphi_{2R}(\tau_\alpha),\alpha(\tau_\alpha))&\leq&\bM_0(\epsilon+\exp(-R))\ ,\label{eq:bdrqf1}\\ 
	 d(\varphi_{R}(\tau_\alpha),\sigma_\alpha)&\leq&\bM_0(\epsilon+\exp(-R))\ .\label{eq:bdrqf2}
	\end{eqnarray}
\end{enumerate}
\end{theorem}

Proposition \ref{pro:boundaryloop} will give further information on the boundary loops.

\subsection{Preliminaries}
Our first lemma is  essentially a result on hyperbolic plane geometry.

\begin{lemma}\label{lem:K2R} There exists constants $R_3$ and $\bM_3$ so that for $R\geq R_3$ the following holds.  Let $u$ be any tripod.	Then  $v\defeq\varphi_{-2R}\left(\left(K\circ\varphi_R\right)^2(u)\right)$ satisfies
\begin{eqnarray}
d(v,u)&\leq& 	\bM_3	\exp{\left(-R\right)}\,  \\
d(\varphi_R(v), K\circ\varphi_R(u))&\leq& \bM_3	\exp{\left(-R\right)}\, \\
\partial^-v&=&\partial^-u
\end{eqnarray}
\end{lemma}

\begin{proof} 
	There exist a constant $M$, so that for all $w$, 
	\begin{eqnarray}d(w,K(w))\leq M\label{ineq:6110} .\end{eqnarray}
	 Recall that $w$ and $K(w)$ are coplanar. In the upper half plane model where $\partial^-w=\partial^-K(w)=\infty$,  $K(w)$ is obtained from $w$ by an horizontal translation. Thus, for $R$ large enough,	\begin{eqnarray*}
	d(\varphi_{-R}(w),\varphi_{-R}(K(w)))\leq  \bM_3	\exp{\left(-R\right)}.
	\end{eqnarray*}
	Applying this inequality to $w=\varphi_R(K\circ\varphi_R(u))$,  gives	
	\begin{eqnarray*}
		d\left(K\circ\varphi_{R}(u),\varphi_{R}(v)\right)=d\left(K\circ\varphi_{R}(u),\varphi_{-R}(K\circ\varphi_R)^2(u)\right)\leq  \bM_3	\exp{\left(-R\right)}\ ,\label{ineq:6112}
	\end{eqnarray*}
	and thus the second assertion. Proceeding further,  for $R$ large enough, the previous inequality and inequality \eqref{ineq:6110} gives, together with the triangle  inequality
\begin{eqnarray*}
		d(\varphi_{R}(u),\varphi_{-R}(K\circ\varphi_R)^2(u))\leq 2M\ .\label{ineq:6113}
	\end{eqnarray*}
Then, for $R$ large enough, 
	\begin{eqnarray*}
		d(u,\varphi_{-2R}(K\circ\varphi_R)^2(u))\leq \bM_3	\exp{\left(-R\right)}\ .\label{ineq:6114}
	\end{eqnarray*}
This concludes the proof.
\end{proof}
The second lemma gives a way to prove an element is loxodromic
\begin{lemma}\label{lem:loxo} There exist  constants $\bM_4$, $R_4$, $\epsilon_4$  only depending on $\ms G$, so that for any $\epsilon\leq \epsilon_4$ and $R\geq R_4$, then given $\alpha\in\ms G$, assuming that there exists a tripod  $v$ so that 
$$d(\varphi_{2R}(v),\alpha(v))\leq \epsilon\ ,$$ then $\alpha$ is loxodromic.

\end{lemma}
\begin{proof} Let $\xi$ be the isomorphism from $\ms G_0$ to $\ms G$ associated to $v$, it follows that for some constant $B$ only depending on $\ms G$, by inequality \eqref{ineq:contrdtaud}, 
$$
d_0(\xi^{-1}(\alpha),\exp(2Ra_0))\leq B \epsilon.
$$
Thus $\alpha$ is $\ms P$-loxodromic and $d_v(\alpha^\pm,\partial^\pm v)\leq \epsilon$ for $R$ large enough.

\end{proof}

\subsection{Proof of Lemma \ref{lem:shad}}
We now start the proof of the Closing Lemma \ref{lem:shad}, referring to  ``$T$, $S$ are $(\mu,R)$ almost closing for $\alpha$ ''  as assumption $(*)$.

\subsubsection{A better tripod}

\begin{proposition} \label{pro:cl1}
There exist constants $\bM_2$, $\epsilon_2$ and $R_1$, so that assuming $(*)$, $\mu\leq\epsilon_1$ and $R>R_1$, then
there exist
\begin{enumerate}
	\item a tripod $u_0$ so that $u_0$, $K\circ\varphi_R(u_0)$
and $(K\circ\varphi_R)^2(u_0)$ are respectively $\bM_2\mu$-close to $T$,  $S$ and $\alpha(T)$,
\item a  tripod $u_1$ so that $u_1$, $\varphi_R(u_1)$ and $\varphi_{2R}(u_1)$ 
   are $\bM_2(\mu+\exp(-R))$-close respectively to $T$,
  $S$  and $\alpha (T)$.
\end{enumerate}

\end{proposition}
\begin{proof} Let $u$ and $v$ associated to $S$ and $T$ by assumption $(*)$.
 Recall that  that $K\circ\varphi_t$ is contracting on $\mc U^+$ for positive $t$ (large enough) -- See Proposition \ref{pro:bas}. Similarly, by Proposition \ref{pro:Kpreserv} $K$ preserves each leaf of $\mc U^{0,-}$, and thus 
     $\varphi_{-t}\circ K^{-1}$ is uniformly $\kappa$-Lipschitz (for some $\kappa$) along $\mc U^{0,-}$ for all positive  $t$. 

By hypothesis \eqref{hyp:cl2}, \eqref{hyp:cl3} and the triangle  inequality  
$$
d(K\circ\varphi_R(u),v)\leq 2\mu\ .
$$
Thus if  $\mu$ is small enough, $\mathcal
U^{0,-}_{K(\varphi_R(u))}$ intersects ${\mc U}^{+}_v$ in a unique
point $w$ which is $4\mu$-close to both $v$ and
$K\circ\varphi_R(u)$  -- Hence $5\mu$ close to $S$ --  as in Figure
(\ref{fig:shad}).\begin{figure}[htbp]   \begin{center}
    \includegraphics[width=3in]{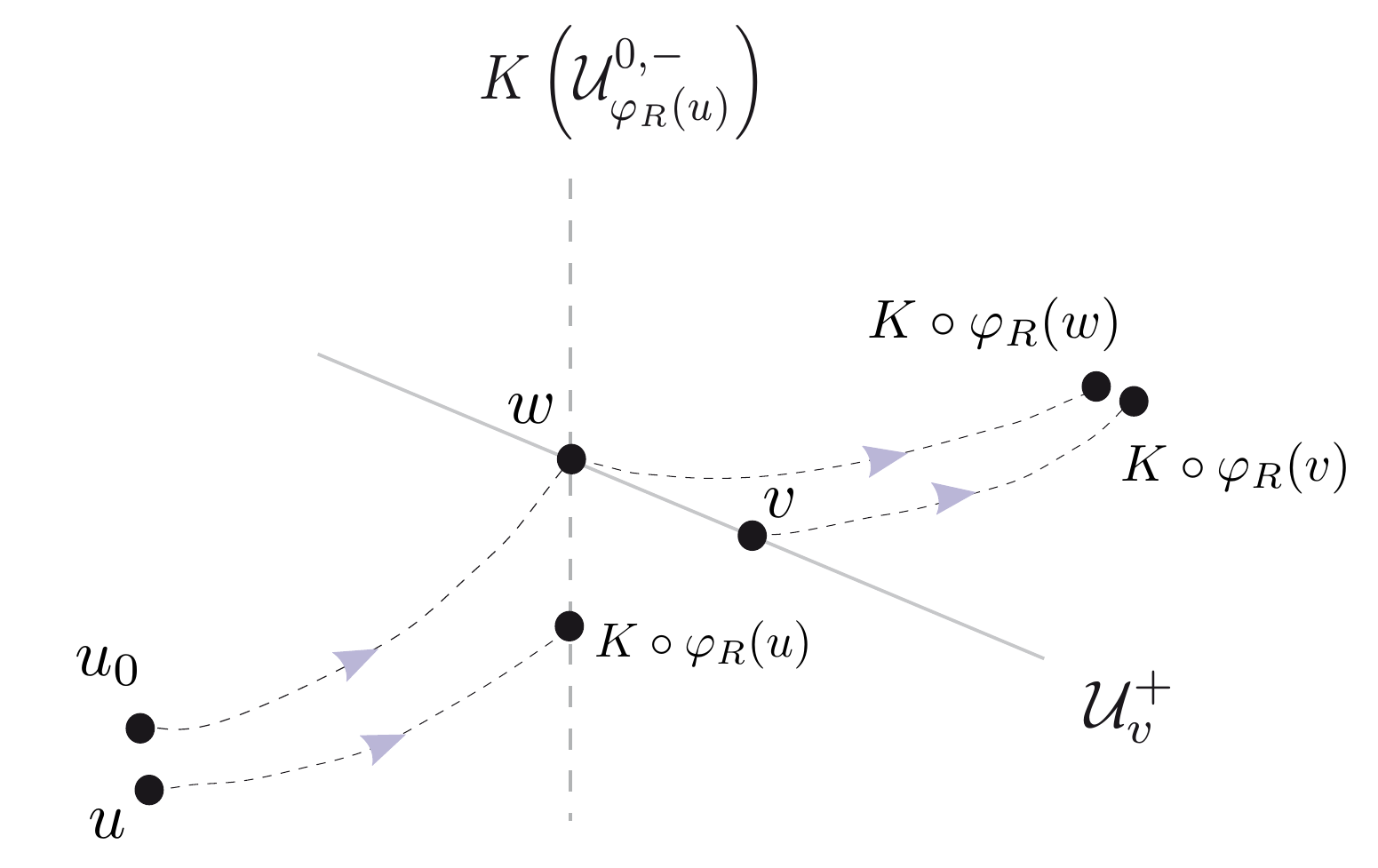}
    \caption{Closing quasi orbits}\label{fig:shad}
  \end{center}

\end{figure}

Recall that $K$ preserves each leaf of  $\mc U^{0,-}$ by Proposition \ref{pro:Kpreserv}. Thus
$${\mc U}^{0,-}_{K\varphi_R(u)}=K\left(\mathcal
U^{0,-}_{\varphi_R(u)}\right)\ .$$

Let now $u_0$ be so that $K\circ \varphi_{R}(u_0)=w$. According to our initial remark $\varphi_{-R}\circ K^{-1}$ is $\kappa$-Lipschitz, since 
\begin{eqnarray}
d(K\circ \varphi_{R}(u_0),K\circ \varphi_{R} (u))\leq 2\mu\ ,
\end{eqnarray}
we get that  
\begin{eqnarray*}
d(u_0,u)\leq \kappa\left(2\mu\right)\ &,& \
d(u_0,T)\leq (2\kappa+1)\mu\ ,	
\end{eqnarray*}
where the second inequality used hypothesis  \eqref{hyp:cl1}.

Symmetrically, using now that $K\circ\varphi_R$ is contracting for $R$ large enough along the
leaves of ${\mc U}^+$, it follows that 
\begin{eqnarray}
	d\left((K\circ\varphi_R)^2(u_0),(K\circ\varphi_R)(v)\right)&\leq&\mu\ . 
\end{eqnarray}
Combining with hypothesis \eqref{hyp:cl4}, this yields
\begin{eqnarray}
	d\left((K\circ\varphi_R)^2(u_0),\alpha(T)\right)&\leq&2\mu\ . 
\end{eqnarray}
 Thus with  $M=2\kappa+1$, we obtain a tripod $u_0$ so that 
so that $u_0$, $K\circ\varphi_R(u_0)$
  and $(K\circ\varphi_R)^2(u_0)$ are respectively $\bM_1\mu$-close to $T$, 
  $S$ and $\alpha(T)$.

Now, according to Lemma \ref{lem:K2R}, it is enough to take $u_1=\varphi_{-2R}(K\circ\varphi_R)^2(u_0)$,  then applies the triangle  inequality.
\end{proof}

\subsubsection{Proof of the closing Lemma \ref{lem:shad}}\label{sec:prooclemma}

Combining Proposition \ref{pro:cl1} and Lemma \ref{lem:loxo}, we obtain that for $\epsilon$ small enough and $R$ large enough, $\alpha$ is loxodromic and moreover
$$
d_{u_1}(\partial^-u_1,\alpha^-)\leq M_3(\mu+\exp(-R))\ .
$$
Since $u_1$ is $\bM_2(\mu+\exp(-R))$-close to $T$, applications of Proposition \ref{A-B} yields
\begin{eqnarray}
	d_{T}(\partial^-T,\alpha^-)\leq M_4(\mu+\exp(-R))\ .\label{ineq:procl1}
\end{eqnarray}

Observe that $\overline T$, $\overline S$ are $(\mu,-R)$ almost closed with respect to $\alpha$. Thus, reversing the signs in the proof, on gets symmetrically that 
$$
d_{\overline T}(\partial^+T,\alpha^+)\leq M_4(\mu+\exp(-R))\ ,
$$
and thus
\begin{eqnarray}
d_{T}(\partial^+T,\alpha^+)\leq M_5(\mu+\exp(-R))\ .\label{ineq:procl2}
\end{eqnarray}
It remains to prove the last statement in the lemma.
Since 
\begin{eqnarray}
	d(T,u_1)\leq \bM_2(\mu+\exp(-R)), \ \ d(\alpha(T),\varphi_{2R}(u_1)\leq \bM_2(\mu+\exp(-R))\ ,\label{eq:cl10}
\end{eqnarray} 
it follows that $u_1$, $\alpha^\pm$ satisfies the hypothesis of  Proposition \ref{pro:flow-foot}. Thus, setting $u_\alpha:=\Psi(u_1,\alpha^-,\alpha^+)$, 
\begin{eqnarray}
d\left(\Psi(\varphi_{2R}(u_1),\alpha^-,\alpha^+), \varphi_{2R}\left(u_\alpha)\right)\right)\leq M_6(\mu+\exp(-R)) \label{ineq:cl11}
\ .
\end{eqnarray}
Using inequalities \eqref{eq:cl10} a second time and Lemma \ref{lem:footmap}, we obtain that
\begin{eqnarray}
	d(\Psi(\varphi_{2R}(u_1),\alpha^-,\alpha^+),\alpha(\tau_\alpha))&\leq& M_6(\mu+\exp(-R))\ ,\label{ineq:cl12}\\
	d\left(\varphi_{2R}(\tau_\alpha),\varphi_{2R}(u_\alpha)\right)=d(\tau_\alpha,u_\alpha) &\leq& M_7(\mu+\exp(-R))\label{ineq:cl13}\ ,
\end{eqnarray}
where the equality in the line \eqref{ineq:cl13} comes from the fact that the flow acts by isometry on the leaves of the central foliation ({\it cf.} Property \eqref{pro:iso-central}). The triangle  inequality yields from inequalities \eqref{ineq:cl11} and \eqref{ineq:cl12}
\begin{eqnarray*}
d(\varphi_{2R}(u_\alpha),\alpha(\tau_\alpha))&\leq& M_8(\mu+\exp(-R))\ .
\end{eqnarray*}
Combining finally with \eqref{ineq:cl13}, we get 
\begin{eqnarray}
d(\varphi_{2R}(\tau_\alpha),\alpha(\tau_\alpha))&\leq&   M_9(\mu+\exp(-R))\ .	\label{ineq:procl3}
\end{eqnarray}
This proves inequality \eqref{eq:cl1}. A similar argument shows inequality \eqref{eq:cl2}. The other assertions of the lemma were proved as inequalities \eqref{ineq:procl1}, \eqref{ineq:procl2} and \eqref{ineq:procl3}.

The last statement is an obvious consequence of the previous ones.

\subsection{Boundary loops}\label{sec:boundaryloop}  We show that the boundary loops of an almost closing pair of pants are close to be perfect in a precise sense.

Let $W=(T,S_0,S_1,S_2)$ be  an $(\epsilon,R)$-almost closing pair of pants with boundary loops 
$\alpha$, $\beta$ and $\gamma$. 

Let us say a triple of tripods $(S^*_1,T^*,S^*_0)$ is {\em $R$-perfect} for $\alpha^*$. If $S^*_0=\alpha_*(S_1^*)$, $T^*=K\phi_R(S_1)$, $S_0^*=K\phi_R T^*$.

We finally say $\alpha$ is $R$-perfect if it is conjugate to $\exp(Ra_0)$.

\begin{proposition}{\sc [Boundary loop]}\label{pro:boundaryloop}  There exists a constant $C$ so that given  $\epsilon$ small enough then $R$  large enough,  and an $(\epsilon,R)$-almost closing pair of pants, there exists an $R$- perfect triple   $(S^*_1,T^*,S^*_0)$  for $\alpha^*$ so that  
$$
d(S_0^*,S_0)\leq C\frac{\epsilon}{R}\ , d(S_1^*,S_1)\leq C\frac{\epsilon}{R}\ , d(T^*,T)\leq C\frac{\epsilon}{R}\ .
$$
Moreover $\alpha=\cdotp\alpha_*\cdotp k$, with $k\in\ms L_{\alpha_*}$ and 
	\begin{eqnarray}
	\sup\left(d_{S^*_0}(k,\id), d_{T^*}(k,\id), d_{S^*_1}(k,\id)\right) &\leq&  C\eR\ \label{pro:bl2} 	\end{eqnarray} 
	If furthermore $\alpha$ is $R$-perfect then $k=\id$. Similar results holds for $\beta$ and  $\gamma$, for different $R$-perfect triples.
		\end{proposition}

In the next proofs,  $C_i$ will denote constants only depending on $\ms G$.

\begin{proof} Since $W$  is an $({\bf M}\eR,R)$-almost closing pair of pants,  let  $(S^*_0,T^*,S_1^*)$ be the perfect triple obtained by the first item in Proposition \ref{pro:cl1} and $\alpha^*$ so that $S_0^*=\alpha^* S_1$.  Let then
\begin{eqnarray*}
	\tau=\Psi(S_0,\alpha^-,\alpha^+)\ , \ & \sigma=\Psi(S_1,\alpha^-,\alpha^+)\ ,&    u=\Psi(T,\alpha^-,\alpha^+)  \\ \tau_*=\Psi(S^*_0,\alpha_*^-,\alpha_*^+)\ , \ &\sigma^*=\Psi(S_1^*,\alpha_*^-,\alpha_*^+)\ , & u^*=\Psi(T^*,\alpha_*^-,\alpha_*^+)  \ . 
\end{eqnarray*}
Recall that $\alpha(\sigma)=\tau$ and $\alpha^*(\sigma^*)=\tau^*$.  For $\epsilon$ small enough then $R$ large enough, by the Structure Theorem \ref{theo:struct-pant}, we obtain that the four points $\sigma,\sigma^*, S_1,S_1^*$ are all $C_1\eR$ close and thus $d_\sigma$, $d_{\sigma^*}$, $d_{S_1},d_{S_1^*}$ are all 2-Lipschitz equivalent, for $\eR$ small enough. The same holds for $\tau,\tau^*, S_0,S_0^*$ as well as for $u,u^*,T,T^*$.

Let $g$ in $\G$  so that $\sigma=g\cdotp\sigma^*$. Since $\sigma$ and $\sigma^*$ are $C_1\eR$ close, it follows that
\begin{equation}
	d_{S_1^*}(g,\id)\leq 2d_{\sigma}(g,\id)\leq C_2\eR\ .\label{eq:S1g}
\end{equation}
Applying the third item of Theorem \ref{theo:struct-pant} and then the first, we obtain that for some constant $C_3$ only depending on  $\G$.
\begin{eqnarray*}
	d(\varphi_{2R}(\sigma),\alpha(\sigma))\leq C_3\eR\ & , & \ \	d(\varphi_{2R}(\sigma^*),\alpha^*(\sigma^*))\leq C_3\eR\ ,\\ d(\varphi_R(\sigma),u)\leq C_3\eR\ &,& d(\varphi_R(\sigma^*),u^*)\leq C_3\eR\ .
\end{eqnarray*}
Since $\varphi_{2R}(\sigma)=g\cdotp \varphi_{2R}(\sigma^*)$, we have
\begin{eqnarray*}
		d(g\tau,\tau)&\leq& d(g(\tau),g\varphi_{2R}(\sigma))+d(g\varphi_{2R}(\sigma),\tau^*)+d(\tau^*,\tau)\cr
		&\leq& d(\tau,\varphi_{2R}(\sigma))+d(\varphi_{2R}(\sigma^*),\tau^*)+d(\tau^*,\tau)
		\leq(2C_3+C_1)\eR\ ,
\end{eqnarray*}
thus as above 
\begin{equation}
	d_{S_0^*}(g,\id)\leq 2d_{\tau}(g,\id)\leq C_4\eR\ .\label{eq:S0g}
\end{equation}
Similarly
\begin{eqnarray*}
		d(gu,u)&\leq& d(g(u),g\varphi_{R}(\sigma))+d(g\varphi_{R}(\sigma),u^*)+d(u^*,u)\cr
		&\leq& d(u,\varphi_{R}(\sigma))+d(\varphi_{R}(\sigma^*),u^*)+d(u^*,u)
		\leq(2C_3+C_1)\eR\ ,
\end{eqnarray*}
thus as above 
\begin{equation}
	d_{T^*}(g,\id)\leq 2d_{u}(g,\id)\leq C_4\eR\ .\label{eq:S0Tg}
\end{equation}
Inequalities \eqref{eq:S1g}, \eqref{eq:S0g} and \eqref{eq:S1g} prove  the inequality
$$	\sup\left(d_{S^*_0}(g,\id), d_{T^*}(g,\id), d_{S^*_1}(g,\id)\right) \leq C\eR\ $$
We can thus replace $S_0^*,T^*, S_1^*$ with  $g^{-1}S_0^*,g^{-1}T^*, g^{-1}S_1^*$ so that for this new perfect triple $g=1$.
\vskip 0.2truecm
Let us write now $k=\alpha_*^{-1}\alpha $.  Then 
\begin{eqnarray*}
d(k\sigma^*,\sigma^*)= d(\alpha(\sigma),\alpha_*(\sigma^*))=d(\tau,\tau^*)\leq \leq (2C_3+2C_1)\eR\ .
	\end{eqnarray*}
This implies  that
\begin{equation}
	d_{S_1^*}(k,\id)\leq 2d_{\sigma^*}(k,\id)\leq 2C_5\eR\ .\label{eq:S1k}
\end{equation}	
Finally since 
$\sigma^*=\sigma$, then  $\alpha_*^\pm=\alpha^\pm$. Thus $k(\alpha_*^\pm)=\alpha_*^\pm$ and in particular  $k$ commutes with $\alpha_*$.  We deduce
\begin{eqnarray}
		d_{S_1^*}(k,\id)
		=d_{S_1^*}(k\alpha_*^{-1},\alpha_*^{-1})
		=d_{S_1^*}(\alpha_*^{-1}k,\alpha_*^{-1})
		=d_{\alpha_*(S_1^*)}(k,\id)
	=d_{S_0^*}(k,\id)\leq 2C_5\eR\ .\label{eq:S0k}
\end{eqnarray}
We deduce then that 
\begin{equation*}
	d_{\sigma^*}(k,\id)\leq C_6\eR\ , d_{\tau^*}(k,\id)\leq C_6\eR\ 
\end{equation*}
Since $k$ belongs to $\ms L_{\alpha^*}$, $k$  commutes with $\alpha_*^{1/2}$. But $u^*=\alpha_*^{1/2}(tau^*)$ and thus the equation above yields
\begin{equation*}
		d_{u^*}(k,\id)\leq C_6\eR\  \ , 
\end{equation*}
From which we deduce 
\begin{equation}
		d_{T}(k,\id)\leq C_7\eR\  \ . \label{eq:Tk}
\end{equation}
Inequalities \eqref{eq:S1k},  \eqref{eq:S0k}, \eqref{eq:S0k}, prove inequalities \eqref{pro:bl2}. 
\vskip 0.1truecm
Assume finally that $\alpha$ is perfect, so that $\alpha=f\alpha^*f^{-1}$. Then,
$$
\alpha^*k=h\alpha^* f^{-1}\ .
$$
Since $k$ is small,  $\alpha^*k$ is $\ms P$-loxodromic. Since $k$ fixes $\alpha_*^\pm$ and $\alpha^*k$  has $\alpha^+$ as unique attracting fixed point and $\alpha^-$ as unique repulsive fixed point, $\alpha^*k$ also  has $\alpha^+$ as unique attracting fixed point and $\alpha^-$ as unique repulsive fixed point. It follows that  $f(\alpha_*^\pm)=\alpha_*^\pm$. Thus $f$ belongs to $\ms L_{\alpha_*}$ and as such commutes with $\alpha_*$. It follows that $k=\id$.
\end{proof}

\subsection{Negatively almost closing pair of pants}

In this section, we have only dealt with positively almost closing pair of pants. Perfectly symmetric results are obtained for negatively  almost closing pair of pants, once they have been defined correctly -- which we have not done yet. We postpone this discussion to paragraph \ref{sec:DefNegSt} after the discussion of the "inversion".

\section{Triconnected tripods and pair of pants}\label{sec:triconn}

We define  in this section triconnected pairs of tripods. These objects consist of a pair of tripods together with three homotopy classes of path between them.  One may think of them as a very loosely almost closing pair of pants.

We them define weights for these tripods, and show that when the weight of a triconnected pair of tripod is non zero, then this triconnected pair of tripods actually defines a almost closing pair of pants. 

Apart from important definitions, and in particular the inversion of tripods discussed in the last section, the main result of this section is the Closing up Tripod Theorem \ref{lem:shatri}. 

This section will make use of a discrete subgroup $\Gamma$ of  $\ms G$,  with {\em non zero injectivity radius} -- or more precisely  so that $\Gamma\backslash\sg$ has a non zero injectivity radius. When $\Gamma$ is a lattice this is equivalent to the lattice being uniform.

\subsection{Triconnected  and biconnected pair of  tripods and their lift}\label{sec:lift}

\begin{definition}{\sc[triconnected and biconnected  pair of  tripods]}
\begin{enumerate} 
\item	A {\em triconnected pair of  tripods} in $\lga\ms G$ \index{Triconnected pair of tripods}  -- see Figure
(\ref{fig:TricTrip}) -- is a quintuple
$$
W=(t,s,c_0,c_1,c_2),
$$
where $t$ and $s$ are two tripods in $\Gamma\backslash\mc G$ and
$c_0$, $c_1$ and $c_2$ are three homotopy classes of paths from $t$ to
$s$, $\omega^2(t)$ to $\omega(s)$, and  $\omega(t)$ to $\omega^2(s)$ respectively , up to loops defined in a $\ms K_0$-orbit. 
The associated   {\em
  boundary loops} are the elements of $\pi_1(\Gamma\backslash\mc G\slash\ms K_0,s)\simeq \Gamma$
\begin{eqnarray*}
  \alpha=c_0\bullet c_1^{-1}\ , \ 
  \beta=c_2\bullet c_0^{-1}\ ,\ 
  \gamma=c_1\bullet c_2^{-1}\ . 
\end{eqnarray*}
The {\em associated pair of pants} is  the triple $P=(\alpha,\beta,\gamma)$. Observe that $\alpha\cdotp\gamma\cdotp\beta=1$.
\item  A {\em triconnected pair of tripods in the universal cover} is 
 a quadruple  $(T,S_0,S_1,S_2)$ so that $T$,$S_0$, $S_1$,  and $S_2$ are tripods in the same connected component of  $\mc G$. The {\em boundary loops} of $(T,S_0,S_1,S_2)$ are the elements $\alpha$, $\beta$ and $\gamma$ of $\G$ so that $S_0=\alpha(S_1)$, $S_2=\beta(S_1)$, $S_1=\gamma(S_2)$.
\end{enumerate}
Similarly we have
\begin{enumerate} 
\item  A {\em  biconnected pair of tripods} is a quadruple $b=(t,s,c_0,c_1)$, where $t$ and $s$ are tripods and $c_0$, $c_1$ are homotopy classes of paths from $t$ to $s$  and $\omega^2 t$ to $\omega s$, respectively, in $\lga\mc G$ (up to loops in $\ms K^m_0$-orbits). Its {\em boundary loop} is $\alpha=c_0\bullet c_1^{-1}.$

 \item  A {\em biconnected pair of tripods in the universal cover} is 
 a  triple  $(T,S_0,S_1)$ so that $T$,$S_0$ and $S_1$ are tripods in the same connected component of  $\mc G$. The {\em boundary loop} of $(T,S_0,S_1)$ is the element $\alpha$ of $\G$ so that $S_0=\alpha(S_1)$.
\end{enumerate}
\end{definition}
A triconnected pair of tripods  $q=(t,s,c_0,c_1,c_2)$  defines a  triconnected pair of tripods $(T,S_0,S_1,S_2)$  in the universal cover up to the diagonal action of $\Gamma$,  called the {\em lift of a triconnected pair of tripods}\index{Lift of a triconnected pair of tripods}, 
where $T$ is a lift of $t$ in $\mc G$, and $S_0$, $S_1$, $S_2$ are the three
lifts of $s$, so that  $S_0$, $\omega S_1$, $\omega^2 S_2$ which are the end points of the paths lifting
respectively $c_0$, $ c_1$ and $ c_2$ starting respectively at $T$, $\omega^2 T$ and $\omega T$   as in Figure
(\ref{fig:TricTripUni}).  Observe  that $S_0=\alpha(S_1)$,
$S_1=\gamma(S_2)$ and $S_2=\beta(S_0)$, where $\alpha$, $\beta$ and $\gamma$ are the three boundary loops of $q$.


\begin{figure}[h]
 \centering
  \begin{subfigure}[h]{0.45\textwidth}
 \centering
  \includegraphics[width=\textwidth]{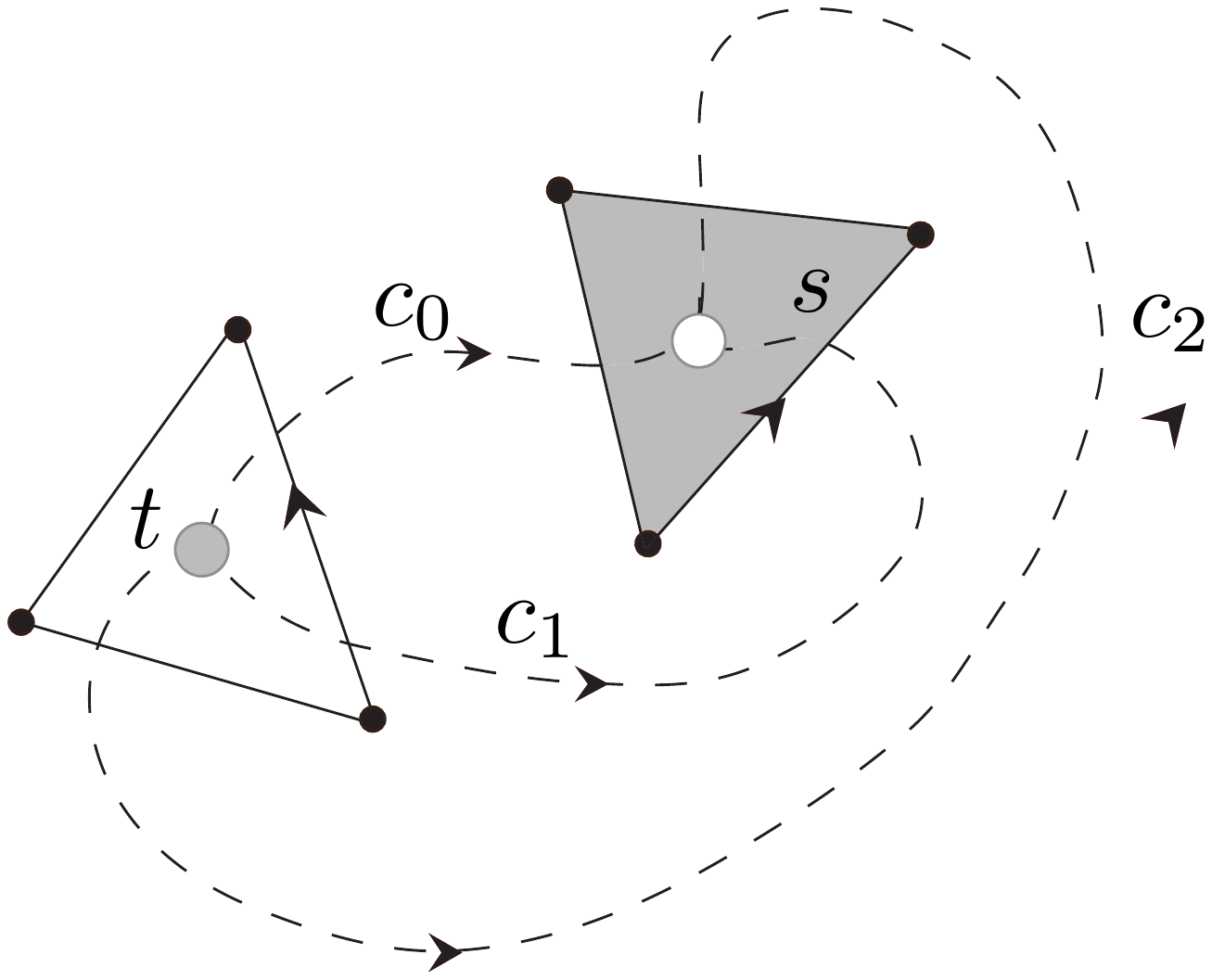}
      \caption{A triconnected pair of  tripods}\label{fig:TricTrip}
 \end{subfigure}    
 \begin{subfigure}[h]{0.45\textwidth}   \begin{center}
    \includegraphics[width=\textwidth]{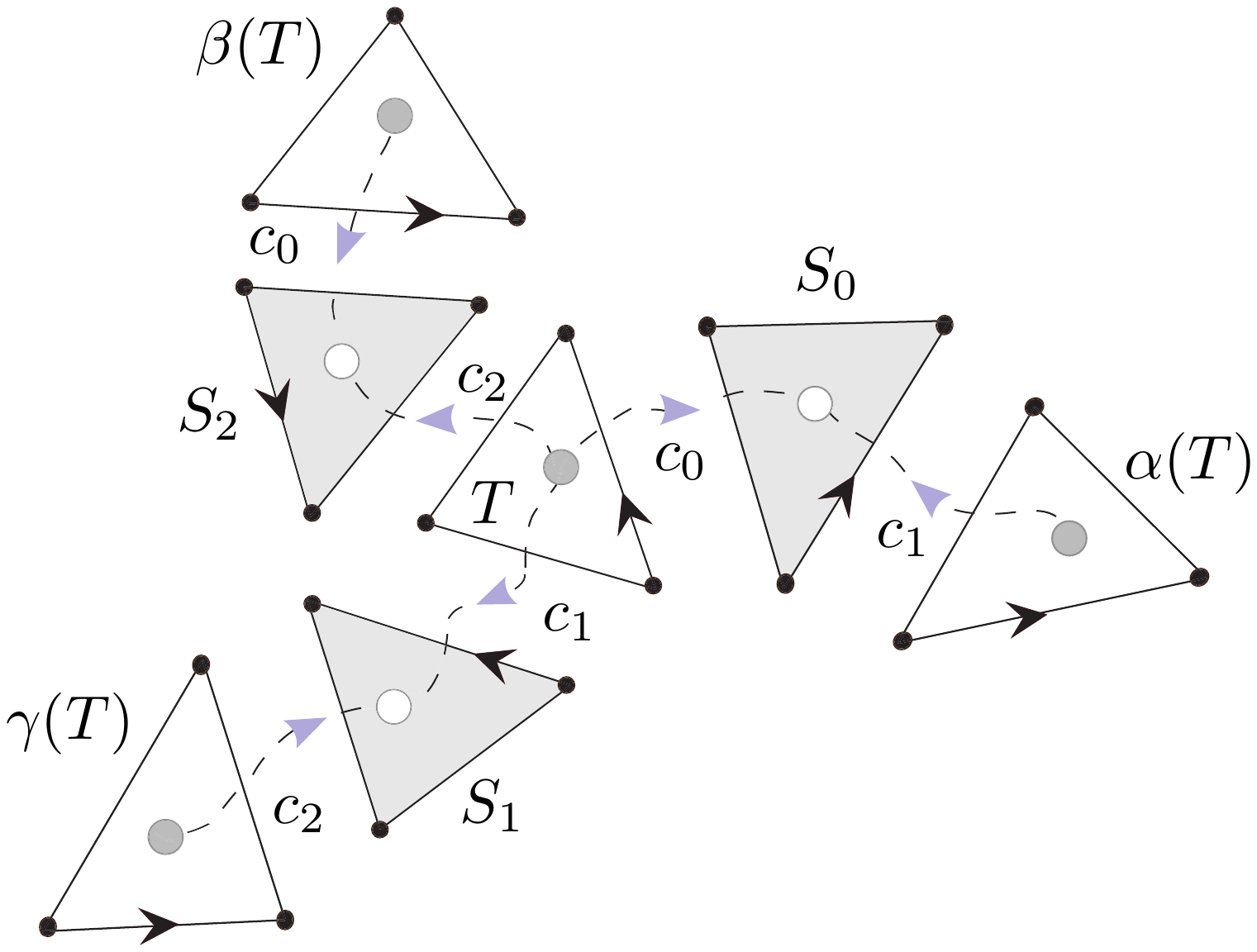}
    \caption{Triconnected tripods in the universal
      cover}\label{fig:TricTripUni}
\end{center}
  \end{subfigure}
 \caption{Triconnected tripods and their lifts}
 \label{fig:tpt}
\end{figure}

Conversely,  since $\mc G\slash \ms K_0$ is contractible, we may think of a triconnected pair of tripods as a  quadruple  of tripods $(T,S_0,S_1,S_2)$ in the same connected component of $\ms G$ well defined up to the diagonal action of $\Gamma$,  so that $S_i$ all lie in the same $\Gamma$ orbit. In particular, we define an action of $\omega$ on the space of triconnected tripod by
\begin{equation}
\omega(T,S_0,S_1,S_2)\defeq(\omega T,\omega^2 S_2,\omega^2 S_0,\omega^2 S_1)\ .	
\label{eq:om-lift}
\end{equation}
On triconnected  pair of tripods the corresponding rotation is 
$$
\omega(t,s,c_0,c_1,c_2)=(\omega(t),\omega^2(s),c_2,c_0,c_1)
$$

\subsection{Weight functions}
We  fix a positive  $\epsilon_0$ less than half the injectivity radius of $\lga\mc G$. Let us fix a smooth positive function $\Xi$ from $\mathbb R$ to $\mathbb R$, with support in $]-1,1[$ and define  for every tripod $\tau$ and real $\epsilon$ the  {\em bell function}\index{$\Theta$} $\Theta_{\tau,x}$ by

$$
\Theta_{\tau,\epsilon}(x)= \frac{1}{\int_{B(\tau,\epsilon)}\Xi\left(\frac{1}{\epsilon}d(y,\tau)\right )\ \d\mu (y)} \Xi\left(\frac{1}{\epsilon}d(x,\tau)\right )
$$
The following proposition is immediate.

\begin{proposition}\label{pro:bell}
	The function $\Theta_{\tau,\epsilon}$ has its support in an $\epsilon$ neighborhood of $\tau$, is positive and of integral 1.  For any isometry  $g$ of $\mc G$,  
	$$\Theta_{g(\tau),\epsilon}\circ g=\Theta_{\tau,\epsilon}\ .$$ Finally, there exists a constant $D$ independent of $\epsilon$ and $\tau$, so that
	\begin{equation}
	\Vert \Theta_{\tau,\epsilon} \Vert_{C^k}\leq D \epsilon^{-k-D}\label{eqn:ckT}\ .
	\end{equation}
\end{proposition}
\begin{proof} The technical part is to prove the inequality \eqref{eqn:ckT}. Let 
$G_\lambda(x)=\Xi\left(\lambda d(x,\tau)\right)$, for $\lambda>1$. An induction shows that for an auxiliary connection $\nabla$ on $\Gamma\backslash \mathcal G$ the $k^{\tiny th}$-derivatives of $G_\lambda$,  $
\nabla^{(k)}G_\lambda\ ,
$
is a  polynomial of degree at most $k$ in $\lambda$, the derivatives of $d(x,\tau)$ and the derivatives of $\Xi$, and whose coefficients only depend on $\mathcal G/\Gamma$. Since moreover, this polynomial vanishes when $d(x,\tau)>1$, we obtain that
\begin{eqnarray}
	\Vert \nabla^{(k)}G_\lambda\Vert \leq K\cdotp \lambda^k\ ,\label{ineq:bell1}
\end{eqnarray}
where $K$ only depends on $\Gamma\backslash \mathcal G$. Let us also consider the function 
\begin{eqnarray*}
f(\lambda)=\int_{\Gamma\backslash \mathcal G}\Xi\left(\frac{1}{\epsilon}d(y,\tau)\right)\ {\rm d}\mu(y)=\int_{\Gamma\backslash \mathcal G}\Xi\left(\frac{1}{\epsilon}d(y,\tau)\right){\rm d}\mu(y)\ .
\end{eqnarray*}
Taking a lower bound of the function $\Xi$ be a step function equal to $A$ on  the interval $]-C,C[$ and zero elsewhere. Then
\begin{eqnarray}
f(\lambda)\geq A \int_{B(\tau,\frac{C}{\lambda})}
{\rm d}\mu(y)
\geq 
K\left(\frac{1}{\lambda}\right)^{\dim (\mathcal G)}\ ,\label{ineq:bell2}
\end{eqnarray}
where $K$ only depends on $\mathcal G$ and $\Xi$.
The proposition now follows at once from inequalities \eqref{ineq:bell1} and \eqref{ineq:bell2}.
\end{proof}

As a consequence, the family of functions $\Theta_{\tau,\epsilon}$ also makes sense on $\lga\mc G$ and the same property holds: notice that this is the point where we make use of the fact that the lattice is uniform.

\begin{definition}{\sc[Weight Functions]}\index{Weight functions} Let  $\epsilon<\epsilon_0$ and $R$ a positive real of absolute value greater than 1. 
The {\em upstairs weight function} is defined on the space of pairs of tripods $(T,S)$ in $\mc G$ by
\begin{eqnarray}
	\wa\seR(T,S)\defeq\int_{\mc G}\Theta_{T,\frac{\epsilon}{\vert R\vert}}(x)\cdotp \Theta_{S,\frac{\epsilon}{\vert R\vert}}(K\circ \varphi_R(x))\, {\rm d}x\ .
\end{eqnarray}
The {\em downstairs weight function} is defined on the space of pairs of tripods $(t,s)$ in $\lga \mc G$ by
$$
{\rm a}\seR(t,s):=\int_{\Gamma\backslash\mc G}\Theta_{t,\frac{\epsilon}{\vert R\vert}}(x)\cdotp  \Theta_{s,\frac{\epsilon}{\vert R\vert}}(K\circ\varphi_R(x))\, {\rm d}x
$$
Let $t$ and $s$ be tripods in  $\lga \mc G$. Let $c_0$ be
a path from $t$ to $s$. The {\em connected tripod weight
  function} is defined by 
$$
{\rm a}\seR(t,s,c_0):= {\wa}\seR(T, S_0),$$
where $T$ is any lift of $t$ in $\mc G$, and $S_0$ the lift of
$s$ which is the end point of the lift of $c_0$ starting at $T$.
\end{definition}
\rmks
\begin{enumerate}
\item For  $R$  negative, we define
\begin{eqnarray}
	\wa\seR(T,S)\defeq\int_{\mc G}\Theta_{T,\frac{\epsilon}{\vert R\vert}}(x)\cdotp \Theta_{S,\frac{\epsilon}{\vert R\vert}}(K^-\circ \varphi_R(x))\, {\rm
  d}x\ ,
\end{eqnarray}
where $K^-(x)=\omega \circ K (x)=\omega^2(\bar x)$. Similarly we define $\rm a\seR$.
\item By construction, for any $g\in{\G}$ we have  $\wa\seR(gT,gS)=\wa\seR(T,S)$
  \item the value of ${\rm a}\seR(t,s,c)$ only depends  on $t$, $s$ and the homotopy class of $c$.
\item
 Let
$\pi(t,s)$ be the set of homotopy classes of paths from $t$ to $s$, then
\begin{eqnarray}
  \sum_{c\in\pi(t,s)}{\rm a}\seR(t,s,c)=\int_{\lga\mc G }\Theta_{t,\frac{\epsilon}{\vert R\vert}}(x). \Theta_{s,\frac{\epsilon}{\vert R\vert}}(K\circ\varphi_R(x)){\rm d}x={{\rm a}}\seR(t,s)\, .\label{eqn:asum}
\end{eqnarray}
\end{enumerate}

\begin{definition}\label{def:wight}{\sc[Weight of a  triconnected pair of  tripods]} Let $R$ be a positive real. Let $W=(T,S_0,S_1,S_2)$ be a  triconnected pair of  tripods in the universal cover 
 The  {\em weight} of $W$   is defined by
\begin{eqnarray}
\wb\seR(W)&=&{\wa}\seR\left(T,S_0\right)\cdotp {
 \wa}\seR\left(\omega^2T,\omega S_1\right)\cdotp {\wa}\seR\left(\omega T,\omega^2 S_2\right)\ .\label{defb}
\end{eqnarray}
Similarly, the {\em weight} of of a biconnected pair of  tripods
$B=(T,S_0,S_1)$  is defined by
\begin{eqnarray}
\wD\seR(B) \defeq {\wa}\seR\left(T,S_0\right)\cdotp \wa\seR\left(\omega^2 T,\omega S_1\right)\ .\label{defd}
\end{eqnarray}
\end{definition}
The functions $\wb\seR$ and $\wD\seR$ are $\Gamma$ invariant and thus descends to functions ${\rm b}\seR$ and ${\rm d}\seR$ for respectively triconnected tripods and biconnected tripods in $\Gamma\backslash \mc G$.
Using the definition of ${\rm b}\seR$ and equation \eqref{eq:om-lift}
\begin{eqnarray}
	{\rm b}\seR\circ\omega&=&{\rm b}\seR\ .\label{eq:bwo}\\
	\sum_{c_0,c_1,c_2}{\rm b}\seR(t,s,c_0,c_1,c_2)&=&{\rm a}\seR(t,s).{\rm a}\seR(\omega^2(t),\omega(s)).{\rm a}\seR(\omega(t),\omega^2(s))\,  \label{eq:brar}.
\end{eqnarray}
	 where the last equation used equations \eqref{defb} and \eqref{eqn:asum}, 

As an immediate consequence of the definitions of the weight functions we have 
\begin{proposition}\label{rem:st2w}
Let  $(\alpha,\beta,\gamma,\tau_0,\tau_1)$ be an $(\eR,R)$-almost closing pair of pants, let $W\defeq (\tau_0,\tau_1,\gamma(\tau_1),\beta(\tau_1))$, Then $\wb\seR(W)$ is non zero. 	
\end{proposition}

One of our main goal is to prove the converse. 

For $R<0$, for reasons that will become clear in Proposition \ref{pro:Imu}, we  define 
\begin{eqnarray}
\wb\seR(W)&\defeq &{\wa}\seR(T,\omega^2 S_1)\cdotp {\wa}\seR(\omega T, \omega S_0)\cdotp{\wa}\seR(\omega^2 T,S_2)\ ,\\
\wD\seR(W)&\defeq &{\wa}\seR(T,\omega^2 S_1)\cdotp {\wa}\seR(\omega T, \omega S_0)\ .
\end{eqnarray}

We have

\begin{proposition}{\sc [Symmetry]}
	We have for $R>0$,
\begin{eqnarray}
		\wa\seR(S,T)&=&\wa\seR(\omega^2T,\omega S)\\
	\wD\seR(B) &=&\wa\seR\left(T,S_0\right)\cdotp \wa\seR\left(S_1, T\right)\ .
\end{eqnarray}
\end{proposition}
\begin{proof} By definition, we have
\begin{eqnarray}
	\wa\seR(S,T)&=&\int_{\mc G}\Theta_{S,\frac{\epsilon}{\vert R\vert}}(x)\cdotp \Theta_{T,\frac{\epsilon}{\vert R\vert}}(K\circ \varphi_R(x))\, {\rm
  d}x\cr
  &=&\int_{\mc G}\Theta_{S,\frac{\epsilon}{\vert R\vert}}(\varphi_{-R} \circ K^{-1}(x))\cdotp \Theta_{T,\frac{\epsilon}{\vert R\vert}}(x))\, {\rm
  d}x\cr
    &=&\int_{\mc G}\Theta_{S,\frac{\epsilon}{\vert R\vert}}(\varphi_{-R} (\overline{\omega^2(x)}))\cdotp \Theta_{T,\frac{\epsilon}{\vert R\vert}}(x))\, {\rm
  d}x\cr
      &=&\int_{\mc G}\Theta_{S,\frac{\epsilon}{\vert R\vert}}(\overline{\varphi_{R}\omega^2(x)})\cdotp \Theta_{T,\frac{\epsilon}{\vert R\vert}}(x))\, {\rm
  d}x\cr
  &=&\int_{\mc G}\Theta_{S,\frac{\epsilon}{\vert R\vert}}(\overline{\varphi_{R}(x)})\cdotp \Theta_{T,\frac{\epsilon}{\vert R\vert}}(\omega(x))\, {\rm
  d}x\cr
    &=&\int_{\mc G}\Theta_{\omega S,\frac{\epsilon}{\vert R\vert}}(K\circ \varphi_{R}(x))\cdotp \Theta_{\omega^2 T,\frac{\epsilon}{\vert R\vert}}(x)\, {\rm
  d}x\cr
  &=&\wa\seR(\omega^2T,\omega S).
 \end{eqnarray}
Where we used the following  facts \begin{enumerate}
	\item in  the first equation, $K$ and $\varphi_R$ preserves the volume form,
	\item in  the second that $K(x)=\omega(\overline x)$
	\item in  the third that $\varphi_R(\overline x)=\overline{\varphi_{-R}}$
	\item in  the fourth the change of variable $y=\omega^2(x)$,
	\item for the fifth that $\omega$ is an isometry.
\end{enumerate} 
The second equation in the proposition is an immediate consequence of the first.	
\end{proof}

\subsubsection{Weight functions and mixing}

Recall that a flow $\{\varphi_t\}_{t\in\mathbb R}$ is
{\em exponential mixing} if there exists some integer $k$, positive constants $C$ and $a$ so that given two smooth $C^k$ functions $f$ and $g$, then for all positive $t$, 
\begin{equation}
	\left\vert \int_X f.g\circ \varphi_t \  {\rm d} \mu - \int_X f \  {\rm d} \mu. \int_X g  \  {\rm d} \mu\right\vert \leq C e^{-at} \Vert f\Vert_{C^k}\cdotp \Vert g\Vert_{C^k}.	
\end{equation}

In the Appendix \ref{app:mix}, we recall the fact  that the action of  $\{\varphi_{t}\}_{t\in\mathbb R}$ on
 $\Gamma\backslash \mc G$ is exponentially mixing when $\Gamma$ is a lattice. As an immediate corollary:

\begin{proposition}{\sc[Weight function and mixing]} \label{pro:expmix} Assume $\Gamma$ is a uniform lattice,
there exists a positive constant $q=q(\Gamma)$ depending only on $\Gamma$, a  positive constant $K=K(\epsilon,\Gamma)$ only depending on $\epsilon$ and $\Gamma$ so that, for $R$ large enough and every $t$, $s$ in $\Gamma\backslash\mc G$, we have 
\begin{eqnarray}
\vert {\rm a}\seR(t,s) -1\vert\leq \exp(-q\vert R\vert)\cdotp K.\label{mix1}
\end{eqnarray}

\end{proposition}
\begin{proof} This follows from the definition of exponential mixing and the definition of the function ${\rm a}\seR$ by equation \eqref{eqn:asum} and equation \eqref{eqn:ckT}.
\end{proof}
Here is an easy corollary
\begin{corollary}\label{coro:bnonzero}
	For any positive $\epsilon$ then for $R$ large enough, the function ${\rm a}\seR$ never vanishes. Moreover, given any $t$ and $s$, there exists $(c_0,c_1,c_2)$ so that ${\rm b}\seR(t,s,c_0,c_1,c_2)$ is not zero.
\end{corollary}
\begin{proof}
	The first part follows from the previous proposition (\ref{pro:expmix}), the second part from equation \eqref{eq:brar}.
\end{proof}

\subsection{Triconnected pair of tripods and almost closing pair of pants}
The main theorem of this section is  to relate triconnected tripods to a almost closing pair of pants and to prove the converse of Proposition \ref{rem:st2w}
\begin{theorem}{\sc[Closing up tripods]}\label{lem:shatri} There exists a constant $\bM$ only depending on $\ms G$, so that the following holds. 
  For any $\epsilon>0$, there exists $R_0$ so that for any
triconnected pair of  of tripods $W=(T,S_0,S_1,S_2)$ 
with boundary loops $\alpha$, $\beta$, and $\gamma$, so that ${\rm B}\seR(W)\not=0$ with $R>R_0$, then $(\alpha,\beta,\gamma,T,S_0)$ is an
$(\bM\eR,R)$--positively almost closing pair of pants.
\end{theorem}

The proof of Theorem \ref{lem:shatri} is an immediate consequence  of the following proposition.

\begin{proposition}\label{pro:clemma-step1} For $\mu$ small enough and then $R$ large enough. Assuming  $B=(T,S_0,S_1)$ is a biconnected tripod with boundary loop $\alpha$ so that that ${\wD}_{\mu,R}(B)\not=0$.  Then $T$ and $S_0$ are $(\frac{\mu}{R},R)$-almost closing for $\alpha$. \end{proposition}

\begin{proof}  We have  $S_0=\alpha(S_1)$. Since ${\rm\wa}_{R,\mu}(T,S_0)\not=0$, there exists $u$ so that
$$
\Theta_{T,\frac{\mu}{R}}(u).\Theta_{S_{0},\frac{\mu}{R}}(K\circ\varphi_R(u))\not=0.
$$
Thus, from the definition of $\Theta$,
\begin{eqnarray}
d(u,T)\leq \frac{\mu}{R}\ ,
\ \ d(K\circ\varphi_R(u), S_0)\leq \frac{\mu}{R}. \label{proo:clemma4}
\end{eqnarray}
Similarly,  since ${\rm\wa}_{R,\mu}(\omega^2(T),\omega(S_1))\not=0$, there exists a tripod $z$ so that
\begin{eqnarray} 
d(\omega(z),T)\leq d(z,\omega^2(T))\leq \frac{\mu}{R}&,&
	d(K\circ\varphi_R(z),\omega(S_1))\leq\frac{\mu}{R}\ .\label{proo:clemma1}\end{eqnarray}
	here we used that $\omega$ is an isometry for $d$ (see beginning of paragraph \ref{sec:tritrimet}).
Let
$$v\defeq\alpha(\omega^2\circ K\circ\varphi_R(z))=\alpha(\overline{\varphi_R(z)}).
$$
Then using the fact the metric on $\mc G$ is invariant by ${\G}$ and $\omega$ and using Corollary \ref{coro:domega} \begin{eqnarray} 
d(v,S_0)=d(v,\alpha(S_1))
=d(\omega^2\circ K\circ\varphi_R(z),S_1)
=d(K\circ\varphi_R(z),\omega(S_1))\leq 
 \frac{\mu}{R},\label{pro:clemma5}
\end{eqnarray}
Moreover, using the
commutations properties \eqref{gpsi}, we have
\begin{eqnarray*}
K\circ\varphi_R(v)=\omega(\overline{\varphi_R(\alpha(\overline{\varphi_R(z)})})
=
  \alpha\circ\omega(\varphi_{-R}(\varphi_R(z)))=\alpha(\omega(z)).
\end{eqnarray*}
Thus,  first inequality \eqref{proo:clemma1} and combining with inequality \eqref{pro:clemma5}
\begin{eqnarray}
	d(K\circ\varphi_R(v),\alpha(T))\leq   \frac{\mu}{R} &,& d(v,S_0)\leq \frac{\mu}{R}\ \label{pro:clemma6}\ .
\end{eqnarray}
The result now follows from inequality \eqref{pro:clemma6} and \eqref{proo:clemma4}.
\end{proof}

\subsection{Reversing orientation on triconnected and biconnected  pair of tripods}

We need an analogue of the transformation that reverse the orientation on pair of pants. Let $\bJ_0$ in $\Aut(\G_0)$ be a reflexion for $\mk s_0$ (see  \ref{sec:sld}). Let $\sigma$  be the involution $x\mapsto\overline x$  defined in paragraph \ref{sec:act}. For an even $\skd$-triple,  $\bJ_0$ and $\sigma$ commute: this follows from a direct matrix computation. Recall also that the automorphism $\bJ_0$ fixes $\ms L_0$ pointwise since by definition $\bJ_0\in \ms Z(\ms L_0)$.

\begin{definition}{\sc [Reverting orientation on $\mc G$]}\label{deg:bI0}
The  {\em reverting orientation involution} $\bI_0$ is the automorphism of $\G_0$ defined by  $\bI_0\defeq\bJ_0\circ\sigma$. 

We use the same notation to define its action on the space of tripods $\mc G=\operatorname{Hom}(\G_0,\G)$ by precomposition.
\end{definition}
\rmks
\begin{enumerate}
	\item  $\bI_0$ commutes with $\sigma$, and if  $\mk s_0=(a_0,x_0,y_0)$ is the fundamental $\msl$-triple, then 
\begin{equation}
	\bI_0(a_0,x_0,y_0)=(-a_0,y_0,x_0)\ .
\end{equation}
	\item  we have
$\bI_0\circ\varphi_R=\varphi_{-R}\circ{\bI}_0$,  similarly $
\omega\circ\bI_0=\bI_0\circ\omega^2$ and $\bI_0\circ K\circ\bI_0=\omega\circ K$ \label{eq:KI}
\item for any tripod $\tau$, $\delta^\pm\bI_0(\tau)=\delta^\mp\tau$ and $\delta^0(\bI_0(\tau))=\delta^0(\tau)$.
\item Since the action of ${\bI_0}$ commutes with the action of ${\G}$ and generates together with $\omega$ a finite group, we may assume that it preserves the left invariant metric on  $\mc G$.
\item When ${\G}$ is isomorphic to $\ms{PSL}_2(\mathbb C)$, $\bI_0$ corresponds to the symmetry $J$ with respect to a geodesic.
\end{enumerate}
\begin{definition}{\sc [Reverting orientation  and rotation on $\mc Q$]}\label{def:revI}
The {\em reverting involution} $\bI_0$ -- see Figure \ref{fig:TTI} -- on the set of triconnected pairs of tripods $\mc Q$, given by 
\begin{eqnarray}
	\bI_0(t,s,c_0,c_1,c_2)&\defeq&(\bI_0(t),\omega\bI_0(s)),\omega^2\bI_0(c_1),\omega^2\bI_0(c_0),\omega^2 \bI_0(c_2))\ .	\label{eqdef:I}
\end{eqnarray}
On the set $\mc B$ of biconnected pairs  of tripods, it is given by
\begin{equation}
\bI_0(t,s,c_0,c_1)\defeq (\bI_0(t),\omega\bI_0(s),\omega^2\bI_0(c_1),\omega^2\bI_0(c_0))\ .	\label{eqdef:IB}
\end{equation}
\end{definition}
In order for this definition  to make sense, we need to check that $\bI_0$ sends a triconnected pairs  of tripods to  a triconnected pairs  of tripods
\begin{proof}
Recall that  a  triconnected  pair of tripods  is a quintuple $(t,s,c_0,c_1,c_2)$,
\begin{itemize}
	\item $c_0$ goes from $t$ to $s$,
	\item $c_1$ goes from $\omega^2 t$ to $\omega s$
	\item $c_2$ goes from $\omega t$ to $\omega^2 s$
\end{itemize}
Let us  check that $(\bI_0(t),\omega\bI_0(s)),\omega^2\bI_0(c_1),\omega^2\bI_0(c_0),\omega \bI_0(c_2))$ is a triconnected pair of tripods.
	Indeed, denoting $u=\bI_0(t)$ and $v=\omega\bI_0(s)$, we have 
	\begin{enumerate}
		\item 	 $\omega^2 \bI_0(c_1)$ goes from $\omega^2 \bI_0 \omega^2(t)= \bI_0(t)=u$ to $\omega^2 \bI_0 \omega (s)=\omega\bI_0(s)=v$,
		\item  $\omega^2 \bI_0(c_0)$ goes from $\omega^2 \bI_0 (t)=\omega^2(u)$ to $\omega^2 \bI_0 (s)=\omega(v)$,
		\item $\omega^2 \bI_0(c_2)$ goes from $\omega^2 \bI_0 \omega(t)= \omega\bI_0(t)=\omega(u)$ to $\omega^2 \bI_0 \omega^2(s)=\bI_0(s)=\omega^2(v)$,
\end{enumerate}
Thus, indeed, $\bI_0$ sends  triconnected  pairs of tripods to  triconnected  pairs of tripods.
\end{proof}

\begin{figure}[h]
 \centering
 	\begin{subfigure}[h]{0.4\textwidth}
		\centering
\includegraphics[width=\textwidth]{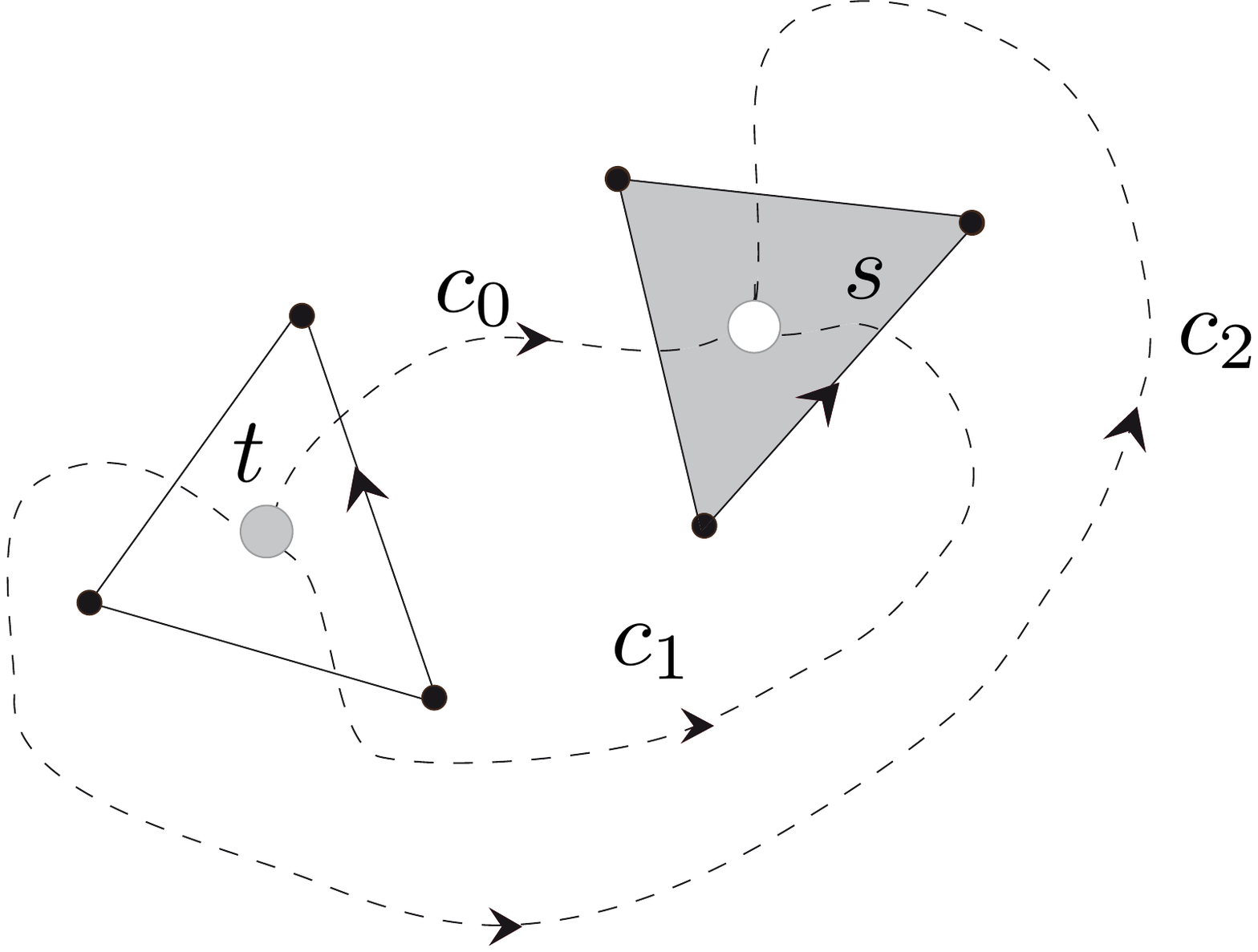}
		\caption{Before}\label{fig:3a}
	\end{subfigure}
\quad  
 	\begin{subfigure}[h]{0.4\textwidth}
		\centering
\includegraphics[width=\textwidth]{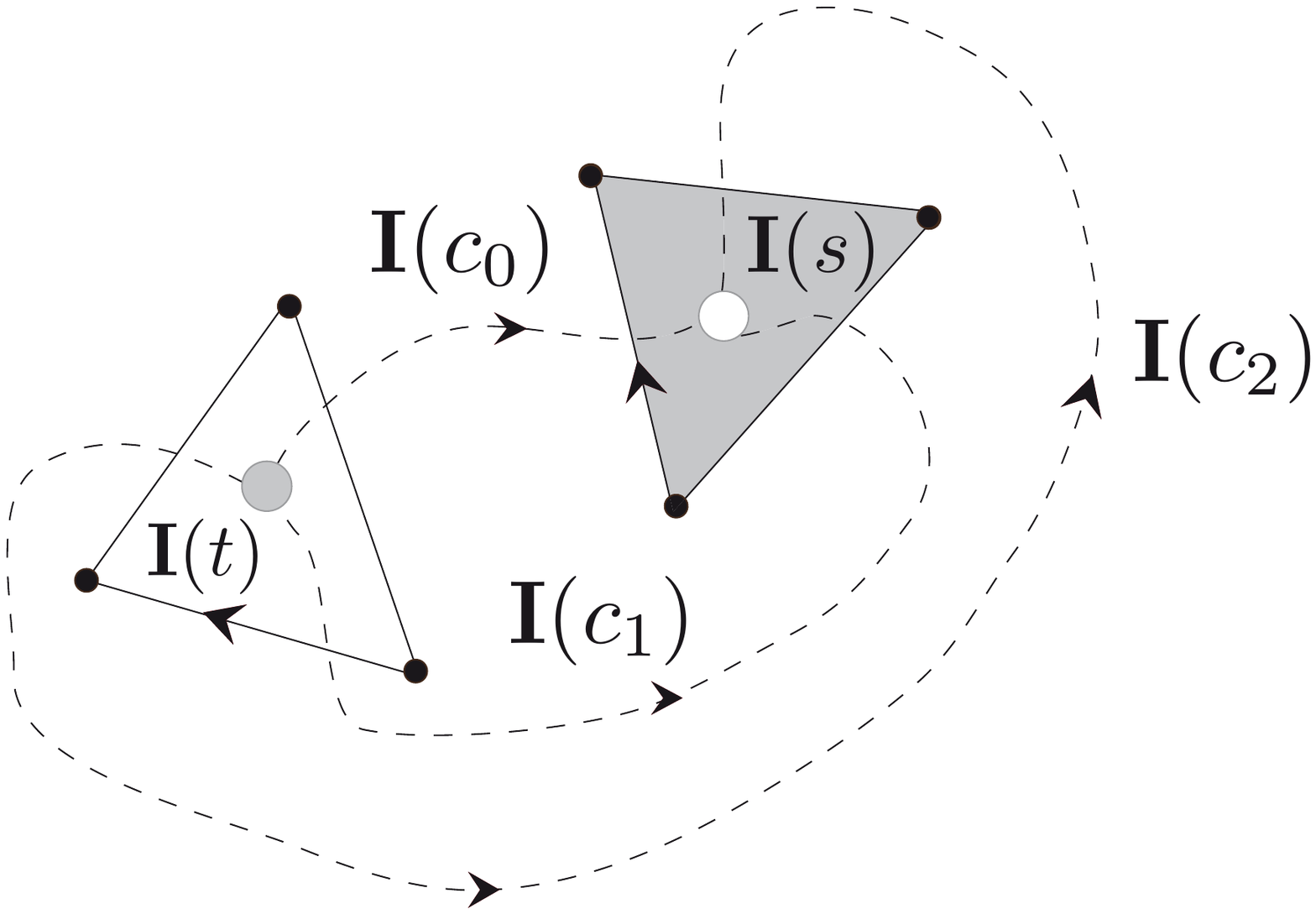}
\caption{After}\label{fig:3b}
	\end{subfigure}    
 \caption{Reverting orientation on triconnected tripods: here $\bI=\omega\bI_0$.}
 \label{fig:TTI}
\end{figure}
 Reverting orientation plays well with the weight functions and boundary loops:

\begin{proposition}\label{pro:Imu}
Let $q=(t,s,c_0,c_1,c_2)$ be a triconnected pair of tripods
\begin{enumerate}
	\item If the lift of $q$ is 
	$W=(T,S_0,S_1,S_2)$ then the lift of $\bI_0(q)$ is $$
	\bI_0(W)\defeq  \left(\bI_0(T), \omega \bI_0(S_1),\omega \bI_0 (S_0),\omega \bI_0 (S_2)\right)$$
	\item If the boundary loops of $q$ are $(\alpha,\beta,\gamma)$ then those of $\bI_0(q)$ are $(\alpha^{-1}, \gamma^{-1}, \beta^{-1})$. In particular $\bI_0$ sends $\mc Q_{[\alpha]}$ to $\mc Q_{[\alpha^{-1}]}$.
	\item Finally,	\begin{eqnarray}
	 \bI_0\circ\omega =\omega^2\circ \bI_0\ \ , \ \ 
	{\rm b}\seR\circ\bI_0 ={\rm b}_{\epsilon,-R}&,&  {\rm d}\seR\circ\bI_0 ={\rm d}_{\epsilon,-R}\  .\label{eq:rIb}
\end{eqnarray}
\end{enumerate}

\end{proposition}

\begin{proof}
This follows either from squinting at symmetries in Figure \eqref{fig:TTI} or from tedious  computations that we officiate now. 

For the first item,  observe that 
\begin{itemize}
\item 	the lift of $\omega^2\bI_0 (c_1)$ goes from $\omega^2\bI_0 \omega^2(T)=\bI_0(T)$ to $\omega^2\bI_0 \omega(S_1)=\omega\bI_0(S_1)$,  
\item  the lift of $\omega^2\bI_0 (c_0)$ goes from $\omega^2\bI_0 (T)$ to $\omega^2\bI_0(S_0)=\omega\left(\omega\bI_0(S_0)\right)$,
\item   the lift of $\omega^2\bI_0 (c_2)$ goes from $\omega^2 \bI_0 \omega (T)=\bI_0(T)$ to $\omega^2 \bI_0 \omega^2(S_2)=\omega^2\left(\omega\bI_0(S_2)\right)$.
\end{itemize}

The second item now follows from the first and the fact that $\Gamma$ commutes with $\bI_0$ and $\omega$.

Let us now check the third item.  
For the sake of convenience, we use the abusive notation $
\Theta_{t,\frac{\epsilon}{\vert R\vert}}=\Theta_t
$. Thus for $R>0$, 
\begin{eqnarray*}
\wa\seR(\bI_0 (T),\bI_0 (S))
	&=&\int_{\mc G}\Theta_{\bI_0 (T)}(x)\cdot  \Theta_{\bI_0 (S)}(K\circ \varphi_R(x))\, {\rm
  d}x\cr
  &=&\int_{\mc G}\Theta_{ T}(\bI_0 (x))\cdotp \Theta_{ S}(\bI_0 \circ K\circ \varphi_R( x))\, {\rm
  d}x\label{eq=Ia1}\\ 
  &=&\int_{\mc G}\Theta_{ T}(x)\cdotp \Theta_{ S}(\bI_0 \circ K\circ \varphi_R \circ \bI_0 (x))\, {\rm
  d}x\label{eq=Ia2}\\    &=&\int_{\mc G}\Theta_{ T}(x)\cdotp \Theta_{ S}(\omega \circ K\circ \varphi_{-R} (x))\, {\rm
  d}x\label{eq=Ia4}\\ 
  &=& \wa_{\epsilon,-R}(T,S)\ ,
\end{eqnarray*}
where for equation \eqref{eq=Ia1}  we used the equivariance of $\Theta$, for equation \eqref{eq=Ia2} a change of variables in $\mc G$ and in equation \eqref{eq=Ia4} the commuting relations \eqref{eq:KI} following definition \ref{def:revI}.
Let $q=(t,s,c_0,c_1,c_2)$ with lift $W=(T,S_0,S_1,S_2)$. 
Recall that
\begin{eqnarray*}
		{\wb}\seR(W)&=&\wa\seR(T,S_0).\wa\seR(\omega^2T,\omega S_1).\wa\seR(\omega T,\omega^2 S_2)\ ,\\
		\wb_{\epsilon -R}(W)&=&\wa\seR(\omega T, \omega S_0).\wa\seR( T,\omega ^2S_1).\wa\seR(\omega^2 T,  S_2)\ . 
\end{eqnarray*}
Thus 
\begin{eqnarray*}
	{\rm b}\seR(\bI_0(W))
	&=&{\wa}\seR(\bI_0(T),\omega\bI_0(S_1)\cdotp {\wa}\seR(\omega^2\bI_0(T),\omega^2\bI_0(S_0))\cdotp{\wa}\seR(\omega\bI_0(T),\bI_0(S_2)) \cr
	&=&{\wa}\seR(\bI_0(T),\bI_0(\omega^2 S_1))\cdotp 
	{\wa}\seR(\bI_0(\omega T),\bI_0(\omega S_0))\cdotp{\wa}\seR(\bI_0(\omega^2T),\bI_0(S_2)) \cr
		&=&{\wa}_{\epsilon,-R}(T,\omega^2 S_1)\cdotp {\wa}_{\epsilon,-R}(\omega T, \omega S_0)\cdotp{\wa}\seR(\omega^2 T,S_2)\cr
	&=&{\wb}_{\epsilon,-R}(T,S_0,S_1,S_2)\cr
	&=&{\rm b}_{\epsilon,-R}(q)\  .
	\end{eqnarray*}

The commutations properties between $\omega$ and $\bI_0$ are straightforward. And the relations for $\rm d\seR$ are obtained in the same way.
\end{proof}

\subsection{Definition of negatively almost closing pair of pants}\label{sec:DefNegSt}
Assume $\epsilon$ and $R$ are positive constants.
If $p=(\alpha,\beta,\gamma,\tau_0,\tau_1)$ is an $(\epsilon,R)$- (positively) almost closing pair of pants. Then by definition 
$$
\bI_0(p)\defeq (\alpha^{-1}, \gamma^{-1},\beta^{-1}, \bI_0(\tau_0),\omega\bI_0(\tau_1))
$$
is an $(\epsilon,-R)$- (negatively)-almost closing pair of pants. 

Then Theorem \ref{lem:shatri} holds for $R<0$, as an immediate symmetry consequence of the properties of $\bI_0$.

	\section{Spaces of biconnected tripods and triconnected tripods}\label{sec:biconn}

We present in this section  the spaces of biconnected and triconnected tripods that we shall discuss in the next sections. 
Our goal in this section are
\begin{enumerate}
	\item The definition of the various spaces involved
	\item The Equidistribution and Mixing Proposition \ref{pro:munu}
	\item Some local connectedness properties of the space of almost closing pair of pants.
\end{enumerate}
Throughout this section, 
 $\Gamma$ will be a uniform lattice in $\ms G$. Let $\alpha\in\Gamma$ be a $\ms P$-loxodromic element. Recall that (see for instance \cite[Proposition 3.5]{Gelander:2014wq}) the centralizer $$\Gamma_\alpha\defeq\ms Z_{\Gamma}(\alpha),$$
 of $\alpha$ in $\Gamma$ is a uniform lattice in the centralizer $\ms Z_\G(\alpha)$ of $\alpha$ in $\ms G$.

\subsection{Biconnected tripods}
 Let $\alpha$ be a $\ms P$-loxodromic element and $\Lambda$ be a uniform lattice in $\ms Z_\ms G(\alpha)$. We define the {\em upstairs space of biconnected tripods} $\mc B_\alpha$ as
$$
\mc B_\alpha\defeq \{(T,S_0,S_1) \hbox{ biconnected tripods in the universal cover} \mid S_0=\alpha S_1\}
$$
and the {\em downstairs space of biconnected tripods as}
$$
\mc B^\Lambda_\alpha\defeq \Lambda\backslash\mc B_\alpha\ .$$

We shall also denote by $[\Gamma]$ be the set of conjugacy classes of elements in $\Gamma$, that we also interpret as the set of free homotopy classes in $\Gamma\backslash\ms G/\ms K_0$, where $K_0$ is the maximal compact of $\ms G_0$.  

\subsubsection{An invariant measure}\label{def:nu} Observe that $\mc B_\alpha$ can be identified with $(\mc G\times \mc G)^*$ (that is the space of pairs of tripods in the same connected component). From the covering map from $\mc B^\Lambda_\alpha$ to $\Lambda\backslash\ms G\times  \Lambda\backslash\ms G$, we deduce a  invariant form $\lambda$ in the Lebesgue measure class of $\mc B_\alpha$ and $\mc B^\Lambda_\alpha$. 

Let also ${\rm D}_{\epsilon,R}$  and ${\rm d}_{\epsilon,R}$ the weight functions defined in Definition \ref{def:wight} (with respect to $\Gamma=\Lambda$)
By construction ${\rm D}_{\epsilon,R}$ is a function on $\mc B_\alpha$, while ${\rm d}_{\epsilon,R}$ is a function on  $\mc B^\Lambda_\alpha$. We now consider the measures
$$
\tilde\nu\seR={\rm D}\seR\cdotp\lambda\ , \ \nu\seR={\rm d}\seR\cdotp\lambda\ ,
$$
on $\mc B^u_\alpha$ and $\mc B^\Lambda_\alpha$ respectively.
The following are obvious
\begin{proposition}
	The measure $\nu\seR$  is locally finite and invariant under $\ms C_\alpha\defeq {\bf Z}^\circ_\G(\Lambda)$.
\end{proposition}

We finally consider
$\mc B\seR(\alpha)$ and $\mc B^\Lambda\seR(\alpha)$ the supports of the functions ${\rm D}\seR$ and ${\rm d}\seR$. It will be convenient in the sequel to distinguish between positive and negative and we introduce 
for $R>0$,  
\begin{eqnarray*}
	\mc B^+\seR(\alpha)=\{B\in \mc B_\alpha\mid {\rm  D}\seR(B)>0\} &,& \mc B^{\Lambda,+}\seR(\alpha)=\{B\in \mc B^\Lambda_\alpha\mid {\rm d}\seR(B)>0\}\ ,\\
	\mc B^-\seR(\alpha)=\{B\in \mc B_\alpha\mid {\rm  D}_{\epsilon,-R}(B)>0\}&,& \mc B^{\Lambda,-}\seR(\alpha)=\{B\in \mc B^\Lambda_\alpha\mid {\rm d}_{\epsilon,-R}(B)>0\}\ .
\end{eqnarray*}
 Recall that 
by Proposition \ref{pro:clemma-step1}, if $(T,S_0,\alpha(S_0))$ belong to $\mc B\seR(\alpha)$, then $T$ and $S_0$ are $(\eR,R)$-almost closing for $\alpha$.

\subsubsection{Biconnected tripods and lattices} Let $\Gamma$ be a uniform lattice in $\G$ and $\alpha$ a $\ms P$-loxodromic element in $\Gamma$. We may now consider the set of biconnected tripods in $\Gamma\backslash\G$ whose loop is in the homotopy class defined by $\alpha$.
$$
\mc B^\Gamma_{[\alpha]}\defeq\{(t,s,c_0,c_1)\hbox{ biconnected tripods in }\Gamma\backslash \ms G\mid c_0\bullet c_1^{-1}\in [\alpha]\}
$$
We have the following interpretation.
\begin{proposition}
The projection from $\mc B^{\Gamma_\alpha}_\alpha$ to $\mc B^\Gamma_{[\alpha]}$ is an isomorphism.
\end{proposition}
In the sequel we will use the following abuse of language: $\mc B^\Gamma_\alpha\defeq\mc B^{\Gamma_\alpha}_\alpha$.

\subsection{Triconnected  tripods} We need to give names to various spaces of triconnected tripods, including their ``boundary related'' versions.  Let as above $\Gamma$ be a uniform lattice, $\alpha$ be an element in $\Gamma$ and $\Lambda$ a lattice in $\ms Z_\G(\alpha)$. We introduce the following spaces
\begin{eqnarray*}
	\mc Q&\defeq& \{(T,S_0,S_1,S_2)\in \mc G^4\mid S_1,S_2\in\Gamma\cdotp S_0\}\ , \cr
	\mc Q^\Gamma&\defeq&\{ (t,s,c_0,c_1,c_2) \hbox{ triconnected tripods in } \Gamma\backslash\G\}\ ,\cr
		\mc Q_\alpha&\defeq& \{(T,S_0,S_1,S_2)\in \mc Q\mid S_1=\alpha S_0\}\ , \cr
		\mc Q^\Lambda_\alpha&\defeq& \Lambda\backslash \mc Q_\alpha\ , \cr
		\mc Q^\Gamma_{[\alpha]} &\defeq& \{(t,s,c_0,c_1,c_2)\in \mc Q^\Gamma\mid c_0\bullet c_1^{-1}\in [\alpha]\}\ . \end{eqnarray*}
 The following identifications are obvious
 \begin{proposition}
 	We have that $\mc Q^\Gamma$ is isomorphic to $\Gamma\backslash \mc Q$. Similarly $\mc Q^\Gamma_{[\alpha]}$ is isomorphic to $ \mc Q^{\Gamma_\alpha}_\alpha$. Finally
 	$$
 	\mc Q^\Gamma=\bigsqcup_{[\alpha]\in [\Gamma]}\mc Q^{\Gamma}_{[\alpha]}.
 	$$
 \end{proposition}
By a slight abuse of language we shall write $ \mc Q^{\Gamma}_\alpha\defeq \mc Q^{\Gamma_\alpha}_\alpha$.

\subsubsection{Triconnected tripods in $\Gamma\backslash\ms G$}Parallel to what we did for biconnected tripods, let  us introduce the  following spaces. First 
Let $\mc Q^\Gamma$ be the set of triconnected pairs of  tripods in $\Gamma\backslash\mc G$  and let $$
\mc Q^\Gamma\seR=\{w\in \mc Q^\Gamma\mid {\rm b}\seR(w)>0\}.
$$
We will assume $R>0$ and write  accordingly
 $
\mc Q^{\Gamma,+}\seR=\mc Q^\Gamma\seR
$ and $
\mc Q^{\Gamma,-}\seR=\mc Q^\Gamma_{\epsilon,-R}
$.
Let $(\lga\mc G\times\lga\mc G)^*$ be the set of pairs of points in $\lga\mc G$ in the same connected component.
We first observe
\begin{proposition}\label{pro:tri-cov}
	The (forgetting) map $p$ from  $\mc Q^\Gamma$ to $(\lga\mc G\times\lga\mc G)^*$ sending 
	$(t,s,c_0,c_1,c_2)$ to $t,s)$
	is a covering. \end{proposition}
\begin{proof} Let  ${\mc Q}_u$ be the space of quadruples $(T,S_0,S_1,S_2)$ where all $S_i$ lie in the same $\Gamma$ orbit.  The map $\pi:(T,S_0,S_1,S_2)\mapsto (T,S_0)$ is a covering. Let $\Gamma\times\Gamma$ be acting on  ${\mc Q}_u$ by $(\gamma,\eta)\cdotp(T,S_0,S_1,S_2)=(\gamma T,\eta S_0,\eta S_1,\eta S_2)$. Then 
$(\Gamma\times\Gamma)\backslash{\mc Q}_u=\mc Q$ 
and $\pi$ being equivariant gives rise to $p$. Thus $p$ is a covering.
\end{proof}

\begin{definition}{\sc [Measures]}\label{def:mu}
The {\em  Lebesgue measure} $\Lambda$ is the  locally finite measure on $\mc Q$ associated to the pullback of the $\G$-invariant volume form on $\lga\mc G$.

Given positive $R$ and $\epsilon$, the {\em   weighted measure} $\mu\seR$ on $\mc Q$ is the measure supported on  
	$\mc Q\seR$ given by 
	$
	\mu\seR\defeq{\rm b}\seR\Lambda
	$.  
\end{definition}

For the sake of convenience, we will assume that $R>0$ and write
$\mu^+\seR\defeq\mu\seR$ and $ \mu^-\seR\defeq\mu_{\epsilon,-R}
$

\begin{proposition} \label{exispant} For any positive $\epsilon$, then $R$ large enough,
	  $\mc Q^\Gamma\seR$ is non empty,  relatively compact and  $\mu\seR$ is finite.
	  Moreover 
	  \begin{eqnarray*}
	  \omega_*\mu^\pm\seR=\mu^\pm\seR\  \ , 	\ \ {\bI_0}_*\mu^\pm\seR=\mu^\mp\seR \  .
 \end{eqnarray*}\end{proposition}
\begin{proof} By  Corollary \ref{coro:bnonzero}, ${\rm b}\seR$ is not always zero, thus $\mc Q^\Gamma\seR$ is non empty.  Let $(T,S_0,S_1,S_2)$ be a lift of  a triconnected tripod $w=(t,s,c_0,c_1,c_2)$ satisfying  ${\rm b}\seR(w)\not=0$. Then, by Proposition \ref{pro:clemma-step1},  $d(T,S_i)\leq R+\epsilon$. This implies that $\mc Q\seR$ is relatively compact and thus $\mu\seR$ is finite. 
The invariance by  $\omega$ comes from the invariance of ${\rm b}\seR$  by $\omega$ (Equation \eqref{eq:bwo}) and the invariance of the metric by $\omega$.
The last assertion comes from the fact that ${\rm b}_{\epsilon,R}={\rm b}_{\epsilon,-R}\circ I$ by Proposition \ref{pro:Imu} and that $\Lambda$ is invariant by $I$, since the invariant measure on $\mc G$ is invariant by $\Aut(\G_0)$ (see beginning of paragraph \ref{sec:tritrimet}).
\end{proof}
Let's finally define 
$\mc Q^\Gamma\seR(\alpha)\defeq\mc Q^\Gamma\seR\cap \mc Q^\Gamma_\alpha$.
\subsection{Mixing: From triconnected tripods to biconnected tripods}

We have a natural forgetful map $\pi$ from $\mc Q^\Gamma_\alpha$ to $\mc B^\Gamma_\alpha$: $\pi(T,S_0,S_1,S_2)\defeq  (T,S_0,S_1)$. We then have 	
the following proposition which says that adding a third path is probabilistically independent for large $R$.

\begin{proposition}	\label{pro:munu}{\sc[Equidistribution and mixing]}
	We have the inclusion $\pi(\mc Q\seR(\alpha))\subset \mc B\seR(\alpha)$. Moreover,
there exists a function $C\seR$ depending on $R$ and $\epsilon$, a  constant $q$ and a constant  $K(\epsilon,\Gamma)$ so that
$
\pi_*\left(\mu\seR\right)=C\seR\cdotp\nu\seR
$.
The function ${\rm C}\seR$ is a smooth almost constant function: there exists a  constant $q$ and a constant  $K(\epsilon,\Gamma)$ so that
\begin{eqnarray}
\left\Vert {\rm C\seR}-1\right\Vert_{C^0}\leq K(\epsilon,\Gamma) \exp(-q\vert R\vert)\ .\label{ccons}
\end{eqnarray}
In particular, given $\epsilon$, for $R$ large enough,  the measure $\nu\seR$ is finite with relatively compact support.
\end{proposition}

\begin{proof} 
By construction for the second  equality, and assertion (\ref{eqn:asum}) for the third 
$$
\pi_*\left({\mu}\seR\right)=\pi_*\left({\rm b}\seR\Lambda\right)=\left(\sum_{c_2} {\rm a}\seR(\omega(t),\omega^2(s),c_2)\right)\cdotp{\rm d}\seR\Lambda=a\seR(\omega(t),\omega^2(s))\Lambda\ .
$$
Thus the result follows from exponential mixing: Proposition \ref{pro:expmix} and taking
$$
{\rm C}\seR(t,s)\defeq \sum_{c_2} {\rm a}\seR(\omega(t),\omega^2(s),c_2)\ = a\seR(\omega(t),\omega^2(s).
$$ 
Observe that ${\rm C}\seR$ is smooth since only finitely many terms in the sum are non zero: there are only finitely many homotopy classes of arcs of bounded length.
\end{proof}

\subsection{Perfecting pants and varying the boundary holonomies}

By a slight abuse of language we will say that $(T,S_0,S_1,S_2)$ with boundary loops $(\alpha,\beta,\gamma)$ is $(\eR,R)$-almost closing if $(\alpha,\beta,\gamma,T,S_0)$, where $S_2=\beta(\omega^2 S_0)$ and $S_0=\alpha(\omega^2 S_1)$ is $(\eR,R)$-almost closing.

Let us denote by $P\seR$ the space of $(\epsilon,R)$-almost closing pair of pants.

We say a boundary loop of a triconnected pair of pants is {\em $T$-perfect} if 
it is conjugate to $\exp(2Ta)$. Given $R$, recall that the boundary loops of an $R$-perfect pair of pants are $R$-perfect (See \cite[Proposition 5.2.2]{Labourie:2009}).

Recall that if $\alpha$ is a $\ms P$-loxodromic element, we denote by $\ms L_\alpha$  the stabilizer in $\G$ of the attractive and repulsive points $\alpha^+$ and $\alpha^-$.

Our main result is the following 

\begin{theorem}{\sc[Varying the boundary holonomies]}\label{theo:pantconn}
Let $B$ be a positive number. Then there exists positive constants $\epsilon_0$ and $C$, so that given  $\epsilon<\epsilon_0$, then for $R$ large enough the following assertion holds:

Let $W_0=(T,S_0,S_1,S_2)$ be a pair of pants in $P\seR$ with boundary loops $\alpha_0$, $\beta_0$ and $\gamma_0$. 

Let $\{k_t\}_{t\in [0,1]}$ be a smooth path in $\ms L_{\alpha_0}$ in the ball of radius of length less than  $B\eR$  with respect to $d_{T}$.

Then there exists a a continuous family $\{W_t\}_{t\in [0,1]}$ in $P_{C\epsilon,R}$ with boundary loops $\alpha_t$, $\beta_t$, $\gamma_t$ so that  $W_0=W$ and 
	\begin{enumerate}
		\item for all $t$, $\beta_t$ and $\gamma_t$ are conjugate to $\beta_0$ and  $\gamma_0$ respectively, 
		 	\item  $\alpha_t$ is conjugate to $\alpha_0\cdotp k_t$.
	\end{enumerate}
\end{theorem}

\subsubsection{Lifting holonomies}
Let $W=(T,S_0,S_1,S_2)$. For any $h$ in $\G$ let 
$$
W_h=(T,S_0,hS_1,S_2)\ .
$$
Observe that if $h$ belongs to the ball of radius $\eR$ with respect to $d_T$ then $W$ is $(M\eR,R)$ almost closing for some uniform $M$.

The boundary loops of $W_h$ are now $\alpha_h\defeq\alpha\cdotp h^{-1}$, $\beta_h=\beta$ and  $\gamma_h\defeq h\cdotp\gamma$.

Let $\G\seR$, respectively $\ms L\seR$, the ball of radius $\eR$ in $\G$, respectively $\ms L_0$.

Let us consider, for this section, 
\begin{enumerate}
	\item the map $A$ from the ball $\G\seR$ of radius $\eR$ in $\G$ to the ball $\ms L_{M\epsilon, R}$ so that $\alpha_h$ is conjugate to $\exp(Ra_0)\cdotp A(h)$. 
	\item Similarly let $C$ the map from  $\G\seR$ to  $\ms L_{M\epsilon, R}$ so that $\gamma_h$ is conjugate to $\exp(Ra_0)\cdotp C(h)$.
\end{enumerate}  
For $\epsilon$ small enough and $M$ large enough, so these maps are well defined by the Boundary Loop proposition.

We need a succession of technical result

\begin{lemma}{\sc [Horizontal distribution]} There exists constant $K_0$ and $\epsilon_0$ so that for $\epsilon$ less than $\epsilon_0$
	For any $h$ in $\G\seR$, there exists a linear subspace $H_h$ in $\T_h\G$, depending smoothly on $h$,  so that
	\begin{enumerate}
		\item $\T_h C$ is zero restricted to $H_h$
		\item $\T_h A$ is uniformly $K_0$ bilipschitz from $H_h$ to $T_{A(h)} \ms L_0$.
	\end{enumerate}
\end{lemma}
We will refer to $H$ as the {\em horizontal distribution}.
\begin{proof} Let us linearize the problem. Given a deformation $\{k_t\}_{t\in[0,1]}$, we are looking for $\{h_t\}_{t\in[0,1]}$, $\{f_t\}_{t\in[0,1]}$ and $\{g_t\}_{t\in[0,1]}$ so that  
$$
\alpha \cdotp h_t^{-1}=f_t\cdotp \alpha\cdotp k_t\cdotp f_t^{-1}\ \ , \ \ h_t\cdotp\gamma=g_t \cdotp\gamma\cdotp  g_t^{-1}\ .
$$
Writing $\dt{v}=\left.\frac{{\rm d}v}{{\rm d}t}\right\vert\mid_{t=0}$,  we get the linearized equation
\begin{equation}
\dt{h}=(\Ad(\alpha)-1)\cdotp\dt{f}+ \dt{k}\ \ , \ \ \dt{h}=(1-\Ad(\gamma))\cdotp\dt{g}\ .	\label{eq:ode}
\end{equation}
Given $\dt{k}$ in the Lie algebra $\mk l_\alpha$ of $\ms L_\alpha$, we want to find 
$\dt{h}$, $\dt{f}$, $\dt{g}$  depending linearly on $\dt{k}$ with 
$$
\frac{1}{K} \Vert \dt{k}\Vert\leq  \Vert \dt{h}\Vert \leq {\rm  K}\  \Vert \dt{k}\Vert\ ,
$$
for some constant ${\rm K}$ only depending on $\G$, $\epsilon$ and $R$, and where the norm comes from $d_{T}$ such that furthermore the choice of  $\dt{h}$ is smooth in $h$.

For $p\in \gp$, let $\ms M_p$ be the stabilizer of $p$ and $\ms N_p$ be its nilpotent radical. 

As a consequence of the Boundary Loop Proposition \ref{pro:boundaryloop} for $R$ large enough both  $\Ad(\alpha)$  contracts $\ms N_{\alpha}^-$ by a factor less than  $1/2$, dilates $\ms N_{\alpha}^-$ by a factor at least $2$.  A similar statement holds for $\Ad(\gamma)$ since the same holds for $\alpha_*$ and  $\gamma_*$

Moreover, since by the Structure Pant Theorem, $(T,\alpha^-,\alpha^+,\gamma^-)$ is an $\eR$-quasitripod, there exists  a  positive constant $C$,  for $\epsilon$ small enough and $R$ large enough, that 
$$
d_{T}(\alpha^-,\alpha^+)> C , \ \ d_{T}(\gamma^-,\alpha^+)>C , \ \ d_{T}(\gamma^-,\alpha^-)>C\ .
$$
Indeed, noting that $D:=d_\tau(\partial^+\tau,\partial^-\tau)$ is a positive constant only depending on $\G$, it follows that if $\theta^-$ is an $\mu$ quasi tripod, then for $i\not=j$
$$
d_{\dt{\theta}}(\partial^i\theta,\partial^j\theta)\geq d_{\dt{\theta}}(\partial^i\tau,\partial^j\tau)-d_{\dt{\theta}}(\partial^i\theta,\partial^i\tau)-d_{\dt{\theta}}(\partial^j\theta,\partial^j\tau)\geq D-2\mu\ .$$ 
Thus $\alpha^-$, $\alpha^+$, $\gamma^-$ are $D-2\mu$ apart using $d_{S_1}$.
Thus $
	\mk g= \mk n_{\gamma^-}+ \mk n_{\alpha^-}+\mk n_{\alpha^+}$.
Then denoting $\mk n^\circ_{\gamma^-}$ the orthogonal in $\mk n_{\gamma^-}$ to $(\mk n_{\alpha^-}\oplus \mk n_{\alpha^+})\cap \mk n_{\gamma^-}$
 , we have
\begin{equation}
	\mk g= \mk n^\circ_{\gamma^-}\oplus \mk n_{\alpha^-}\oplus \mk n_{\alpha^+}\ .\label{eq:dcg10}
\end{equation}
Observe now that $\dim {\ms L_\alpha}= \dim \mk n^\circ_{\gamma^-}$ and that the projection 
from $\ms L_\alpha$ to $\mk n^\circ_{\gamma^-}$ using the above projection is uniformly bilipschitz by a function that depends only on $h$.

We now claim that $H_h= \mk n^\circ_{\gamma^-}$ solves our problem. Indeed for $\dt{k}$ in $\G$ let us consider the decomposition of $\dt{k}$ using the  above decomposition \label{eq:dcg10} of $\mk g$ as
$$
\dt{k}=k^{\circ}_{\gamma^-}+ k_{\alpha^-}+  k_{\alpha^+}\ ,\hbox{where }\ k^{\circ}_{\gamma^-}, k_{\alpha^-}, k_{\alpha^+} \ \ \hbox{belong to }\ \ \mk n^\circ_{\gamma^-},\   
\mk n_{\alpha^-},\ \mk n_{\alpha^+}\hbox{respectively. }  
$$
We now define $\dt{f}$, $\dt{h}$ and $\dt{g}$ by
$$
\dt{f}=-(\Ad(\alpha)-1)^{-1}(k_{\alpha^-}+  k_{\alpha^+})\ , \ \ \dt{h}=k^{\circ}_{\gamma^-}\ , \ \ \dt{g}=(\Ad(\gamma)-1)^{-1}(k^{\circ}_{\gamma^-})\ .
$$
Observe that $\dt{h}$ belongs to $H_h$ and $\dt{f}$, $\dt{g}$ and $\dt{h}$ solve equation \eqref{eq:ode}. In particular $\T_hA(\dt{h})=\dt{k}$ and $T_hC(\dt{h})=0$.
\end{proof}

An horizontal distribution (here $H$) for a submersion (here $A$) provides a way to lift paths from $\ms L_0$ to $\G$ starting from any point in the fiber above the origin of the path. Using this classical idea form differential geometry we now prove the following result that complete the proof of Theorem \ref{theo:pantconn}.

\begin{lemma}{\sc [Existence of a section]}
Let $\epsilon_0$ and $K_0$ as in the previous lemma.
	For $\epsilon$ less than $\epsilon_0(2K_0)^{-2}$, then $R$, large enough,  given  $h^0$ in $\G\seR$ and  $k^0\defeq A(h^0)$,  there exists a continuous map $\Xi$ from $\ms L_{(2K_0\epsilon,R)}$ to  $\G_{(4K_0^2\epsilon,R)}$, so that
	\begin{enumerate}
		\item $\Xi(k^0)=h^0$
		\item  Letting $h=\Xi(k)$, we have $\alpha_h$ is conjugate to $\alpha\cdotp k$, while $\beta_h$ and $\gamma_h$ are respectively  conjugate to $\beta$ and $\gamma$. 
\end{enumerate}\end{lemma}

\begin{proof} Let $h^0$ in $\G\seR$. Let us write $k^0=\exp(u)$ and $k^0_t=\exp((1-t)u)$.  Observe that the path  $\chem{k^0}$ has length less than $K_0\eR$. Let $\chem{h^0}$ be the path lifting  $\chem{k^0}$ using the horizontal distribution $H$ and  starting from $h^0$. In particular $A(h^0_1)=k^0_1=\id$.  Since $\chem{h^0}$ has length less than $K_0^2(\eR)$, it follows that $h_1$ belongs to the ball of radius $2K_0^2\eR$.

For any $k$ in $\ms L_{(2K_0\epsilon,R)}$, let us write $k=\exp(v)$ and $k_t=\exp(tv)$. Let us lift $\chem{k}$ to a path $\chem{h}$ starting from $h^0_1$ and define $\Xi(k)\defeq h_1$. The map $\Xi$ satisfies all the conditions of the lemma:
\begin{enumerate}
	\item By uniqueness of the lift $\Xi(k^0)=h^0$
	\item Since all paths are tangent to the horizontal distribution,  the holonomy around $\gamma$ are all conjugate (Since the horizontal distribution lies in the kernel of $\T C$). By construction the holonomy around $\beta$ is fixed.. 
\end{enumerate}
Observe that the path $\chem{k}$ has length less than  $2K_0\eR$, thus   $\chem{h}$ has length less than $2K^2_0\eR$, and thus $\Xi(k)$ is in the ball of radius $4K_0^2\eR$
\end{proof}

\begin{theorem}{\sc [Deforming into perfect pair of pants]}\label{theo:defo-pant}
There exists a positive constant $K$, so that for $\epsilon$ small enough, then $R$ large enough the following holds.
Let $W_0=(T,S_0,S_1,S_2)$ is an $(\epsilon, R)$-almost closing pair of pants. Then there exists a deformation path  $\{W_t\}_{t\in[0,1]}$ with $W_0=W$  so that
 for each $t$, $W_t$ is a $(K\epsilon, R)$-pair of pants and $W_1$ is $R$-perfect.
 
 Moreover if one of the boundary is perfect, we may choose the deformation so that this boundary stays perfect.
 
\end{theorem}

 Let $W=(T,S_0,S_1,S_2)$ be an $(\epsilon, R)$-almost closing pair of pants. Let $W^*=(T,S^*_0,S_1^*,S_2^*)$ be the $R$-perfect pair of pants based at $T$. Let $\zeta^*_i$ be the elements of $\G$ so that 
 $$
 S^*_0=\zeta^*_0 T\ , \ \ \omega^2 S^*_2=\zeta^*_2\omega T\ , \ \ \omega S_1=\zeta^*_1\omega^2 T_2\ .
 $$
 Then the above theorem is the consequence of the following lemma
 \begin{lemma}\label{lem:defo-pant}  The quadruple $W$ is an $(\epsilon,R)$-almost closing pair of pants, if and only if there exist $f_i$ and $g_i$ in $\G$, for $i\in\{0,1,2\}$ which are $K\eR$ close to the identity   so that 
\begin{eqnarray}
	 S_0=\zeta_0 T\ , \ \ \omega^2 S_2=\zeta_2\omega T\ , \ \ \omega S_1=\zeta_1\omega^2 T_2\ , \label{eq:defor-ac}
\end{eqnarray}
where $\zeta_i=f_i\zeta_i^*g_i$.
\end{lemma}

\begin{proof}	 Say a pair $(T,S)$ is $(\mu,R)$ {\em almost closing} if there exists a tripod $u$ so that $u$ is $\mu$ close to $T$ and $K\varphi_R (u)$ is close to $S$. Thus $u=gT$ and $K\varphi_R (u)=g'S$ where $g$ and $g'$ are $\mu$ close to the identity with respect to $d_T$ and $d_S$ respectively.

Let us consider $\zeta$ and $\zeta^*$ in $\G$, so that
$$ 
S=\zeta T\ , S^*\defeq K\varphi_R T=\zeta^* T\ ,
$$
Let us consider $h$ so that $g'\zeta=\zeta h$. Then $h$ is $\mu$ close to the identity with respect to $d_T$ since
$$
d_T(h,\id)=d(T,hT)=d(\zeta T, \zeta h T)=d(S,g'S)=d_S(g,g')\ .
$$
Since 
$$
\zeta h^{-1}T = g'^{-1} \zeta  T =g'^{-1} S = K\varphi_R(u)=K\varphi(g T)=g\zeta^* T
$$
Thus
$$
\zeta= g\zeta^* h\ .
$$
It follows (using the notation above) that $(T,S)$ is almost closing if and only if
$
\zeta= g\zeta^* h
$
with $h$ and $g$,  $\mu$-close to the identity with respect to $d_T$.

Repeating this argument for all the pairs described in the lemma proves this lemma.
\end{proof}
\subsubsection{Proof of Theorem \ref{theo:defo-pant}}

Assume now that one of the boundary of $W=(T,S_0,S_1,S_2)$ --say $\alpha$-- is perfect. Let us use the Boundary Loop Proposition \ref{pro:boundaryloop} and  let $(S_1^*,T^*,S_0^*)$ be the perfect triple associated to $\alpha^*=\alpha$.

Let $g$ so that $S_1=gS_1^*$, it follows that $S_0=\alpha g\alpha^{-1} S_0^*$.  Then $g$ is close to the identity with respect to $d_{S_1}$, let then $\{g_t\}_{t\in[0,1]}$ be a path from  $g$  to $\id$, so that    $g_t$ and is close to the identity with respect to $d_{S_1}$.  Let $S_1^t=g_tS^*_1$ and $S_0^t=\alpha g_t\alpha^{-1} S^*_0=\alpha S_1^t$.
Observe that $d(S_0^t,S_0)=d(\alpha S_1^t,\alpha S_1)=d(S_1^t, S_1)=d_{S_1}(g_t,\id)$.
\vskip 0.1 truecm
As a first step in the  deformation, we consider 
$$
W^{(1)}_t=(T,S_0^t,S_1^t,S_2)  \hbox{ so that  } W^{(1)}_1=(T,S_0^*,S_1^*,S_2)\hbox{ and } W^{(1)}_0=(T,S_0,S_1,S_2)\ .
$$
We remark that the first boundary loop does not change.
\vskip 0.1 truecm
As a second step, we choose a small path $\{T^t\}_{t\in[0,1]}$ joining $T$ to $T^*$, and define
$$
W^{(2)}_t=(T^t ,S_0^*,S_1^*,S_2)\hbox{ so that  } W^{(2)}_1=(T^*,S_0^*,S_1^*,S_2)\hbox{ and } W^{(2)}_0=(T,S^*_0,S^*_1,S_2)\ .
$$
Again, we remark that the first boundary loop does not change.
\vskip 0.1 truecm
Let $S_2^*$ so that $(T^*,S_0^*,S_1^*,S^*_2)$ is perfect.
Finally,  we write as in Lemma \ref{lem:defo-pant}, $S_2=f\zeta_2^*hT^*$ with $f$ and $h$ close to the identity with respect to $d_{T^*}$, we find paths $\{f^t\}_{t\in[0,1]}$ and $\{h^t\}_{t\in[0,1]}$ joining respectively $f$ and $h$ to $\id$. We finally write $S^t_2=f^t\zeta_2^*h^tT^*$, and choose the final step as 
$$
W^{(3)}_t=(T^* ,S_0^*,S_1^*,S^t_2)\hbox{ so that  } W^{(3)}_1=(T^*,S_0^*,S_1^*,S^*_2)\hbox{ and } W^{(3)}_0=(T^*,S^*_0,S^*_1,S_2)\ .
$$
Again, we remark that the first boundary loop does not change.

\section{Cores and  feet projections}\label{sec:feet}

In this section we concentrate on discussing  the analogues of the normal bundle to closed geodesics for hyperbolic 3-manifolds  in our higher rank situation. Ultimately, in the next situation we want to show that pair of pants with having a ``boundary component'' in common are nicely distributed in this ``normal bundle''. For now we need to investigate and define the objects that we shall need for this study.

More precisely, we  define the {\em feet space} which is a  higher rank version of the normal space to a geodesic in the hyperbolic space of dimension 3. We also explain how biconnected tripods and triconnected tripods project to this feet space.

We will also introduce an important subspace of this feet space, called the {\em core}. The main result of this section is Theorem  \ref{theo:mesfeet} about measures on the feet space. 

In all this section $\alpha$ will be a semisimple ${\ms P}$-loxodromic element in $\ms G$ and $\Lambda$ a uniform lattice in $\ms Z_\ms G(\alpha)$, the centralizer of $\alpha$ in $\G$, so that  $\alpha\in\Lambda$.

\subsection{Feet spaces and their core}

\begin{definition}{\sc [Feet spaces for $\alpha$]} The {\em upstairs feet space of $\alpha$} and {\em downstairs feet space of $\alpha$} \index{Feet space} denoted  respectively $\mc F_\alpha$ \index{Feet space}\index{$\mc F_\alpha$} and $\mc F^\Lambda_\alpha$ are respectively \begin{eqnarray}
\mc F_\alpha&\defeq&\{\tau\in\mc G\mid  \partial^{\pm}\tau=\alpha^\pm\}\ ,\\
{\mc F}^\Lambda_\alpha&\defeq& {\Lambda}\backslash {\mc F}_\alpha\  .	\end{eqnarray}
We denote by ${\rm p}$ the projection from $\mathcal F_\alpha$ to ${\mc F}^\Lambda_\alpha$.
\end{definition}

If $g\in\G$, the map $F_g:\tau\mapsto g\tau$, defines a natural map $f_g$  from $\mc F_\alpha$ to $\mc F_{g\alpha g^{-1}}$ which gives rise to 
$$ f_g: \mc F^\Lambda_\alpha\to\mc F^{g\Lambda g^{-1}}_{g\alpha g^{-1}} \hbox{ so that } p\circ F_g=f_g\circ p\ ,$$ which is the identity if $g\in\Lambda$. We also introduce the groups \index{$\ms Z^{{\G}}(\alpha)$} \index{$\ms L_\alpha$} \index{${\ms Z^\Gamma}(\alpha)$}
\begin{eqnarray}
	\ms C_\alpha&\defeq&\ms Z^\circ_{\ms G}(\Lambda)\ ,\label{def:ca}\\
	\ms L_\alpha&\defeq&\{g\in\G\mid g(\alpha^\pm)=\alpha^\pm\}\ .\label{def:La}
\end{eqnarray}
Let also consider $\ms K_\alpha$ the maximal  compact factor of $\ms L_\alpha$. Below are some elementary remarks
\begin{enumerate}
\item Any tripod in ${\mc F}_\alpha$, gives an isomorphism of $\ms L_\alpha$ with  $\ms L_0$, and the space ${\mc F}_\alpha$ is a principal  left $\ms L_\alpha$ torsor, as well as a principal right $\ms L_0$ torsor. It follows that $\tau^{-1}(\exp(ta_0))$ does not depend on $\tau$ in $\mc F_\alpha$. Let then $a\in\mk g$ so that $\tau^{-1}(\exp(ta_0))=\exp(ta)$. As a consequence, for all $\tau$ in $\mc F_\alpha $, $\varphi_t(\tau)=\exp(ta)\tau$.
 \item The group  $\ms C_\alpha$ acts by isometries on $\mc F_\alpha$ and $\ms C_\alpha\subset L_\alpha$.
\end{enumerate}

\subsubsection{The lattice case} When $\Gamma$ is a lattice in $\G$, we write by a slight abuse of language $\mc F^\Gamma_\alpha\defeq \mc F^{\Gamma_\alpha}_\alpha$, where we recall that $\Gamma_\alpha=\ms Z_\Gamma(\alpha)$.  In that case, we define for $[\alpha]$ a conjugacy class in $\Gamma$, $\mc F^\Gamma_{[\alpha]}$ as the set of equivalence in $\bigsqcup_{\beta\in[\alpha]}\mc F^\Gamma_\beta$ under  the action of  $\Gamma$ given by the maps $f_g$. Since for $g\in\Gamma_\alpha$, $f_g$ gives the identity on $\mc F_\alpha^\Gamma$, the space $\mc F^\Gamma_{[\alpha]}$ is canonically identified with $\mc F^\Gamma_{\beta}$ for all $\beta\in[\alpha]$.

\subsubsection{The core of the feet space}
A (possibly empty) special subset  of the space of feet requires consideration.

\begin{definition}{\sc [Core]} Given $(\epsilon, R)$,  the {\em $(\epsilon, R)$-core of the space of feet} is
 the closed subset ${\mc X}_\alpha$ of ${\mc F}_{\alpha}$, defined by
$$
{\mc X}_\alpha=\left\{\tau\in \mc F_{\alpha}\mid  d(\varphi_{2R}(\tau),\alpha(\tau))\leq  \eR\right\},
$$
We denote by $\mc X^\Lambda_\alpha$ be the projection of ${\mc X}_\alpha$ on ${\mc F}^\Lambda_{\alpha}$
\end{definition}

We immediately have
\begin{proposition}\label{pro:procore}
	The sets $\mc X_\alpha$ and  $\mc X^\Lambda_\alpha$,  are invariant under the action of $\ms C_\alpha$. Moreover, ${\rm p}^{-1}\mc X^\Lambda_\alpha={\mc X}_\alpha$. Finally, when non empty,  $\mc X^\Lambda_\alpha$ is compact.
\end{proposition}

\begin{proof}
	The first statement  follows from the fact that  $\ms Z_{\G}(\alpha)$ acts by isometries on $\mc F_{\alpha}$ commuting both with $\alpha$ and the flow $\{\varphi_t\}_{t\in\mathbb R}$.
	The second statement comes from the fact that $\mc X_\alpha$ is in particular invariant under the action of $\ms Z_{\G}(\alpha)$.
Let us finally prove the compactness assertion, the action of the flow  $\{\varphi_t\}_{t\in\mathbb R}$on ${\mc F}_\alpha$  is given by the left action of the one parameter subgroup generated by $a$.  Indeed, let $\tau\in\mc F_\alpha$, Thus
	$$
	d(\varphi_{2R}(\tau),\alpha(\tau))=d(\tau,\exp(-2Ra)\alpha(\tau))\ .
	$$ 
	Let $\beta\defeq\exp(-2Ra)\alpha$. Let $\tau_0$ be an element of $\mc F_\alpha$. Then the core $\mc X_\alpha$ is the set of those elements $g\tau_0$, where $g\in \ms L_\alpha$ satisfies
	$$
	d(\tau_0, g^{-1}\beta g\cdotp \tau_0)\leq \eR\ .
	$$
	Since $\alpha$ is semisimple and centralizes $a$,  $\beta$ is semisimple as well.  Thus the orbit map 
	$$
	\ms L_\alpha/\ms Z_{\ms L_\alpha}(\beta)\to \mc G, \ \ g\mapsto g^{-1}\beta g\cdotp\tau_0\ ,
	$$
	is proper.  	
	It follows that the set 
	$$
	\left\{h\in {\ms L}_\alpha/\ms Z_{\ms L_\alpha}(\beta), \ \ d(\tau_0, g^{-1}\beta g \cdotp \tau_0)\leq \eR\right\}\ ,
	$$
	is compact. The result now follows from the fact that 
 $\ms Z_\G(\alpha)=\ms Z_\G (\beta)$ and $\Lambda$ is uniform lattice in $\ms Z_\G(\alpha)$ by hypothesis. \end{proof}
\subsection{Main result} The result uses the Levy--Prokhorov distance as described in the Appendix \ref{P}.
\begin{theorem}
\label{theo:mesfeet}
For $\epsilon$ small enough, then $R$ large enough, there exists a constant $\bM$ only depending on $\G$,
with the following property.
Let $\alpha$ be a loxodromic element. Let $\bmu$ be a measure supported on the core $\mc X^\Lambda_\alpha$. 

Let $\ms T_0$ be a compact torus in $\ms C_\alpha\cap \ms K_\alpha$. Let $\bnu$ be a measure on $\mc F^\Lambda_\alpha$ which is invariant under $\ms C_\alpha$  and supported on $\mc X^\Lambda_\alpha$. Assume that we have a function $C$ so that 
\begin{eqnarray}
	\bmu=C\bnu, \hbox{ with }\  \Vert C-1\Vert_\infty \leq \frac{\epsilon}{R^2} \label{eq=mix} \ .
\end{eqnarray}
Then for all ${\bf j}$ in $\ms T_0$, denoting by $d_L$ the Levy--Prokhorov distance between measures on $\mc F_\alpha$.
\begin{eqnarray}
	d_L({\bf j}_*(\varphi_1)_*(\bmu),\bmu)\leq \bM \eR\ .
\end{eqnarray}
	\end{theorem}
The difficulty here is that $\varphi_1$ is not an element of $\ms C_\alpha$, however, {\em on the core}, it behaves pretty much like an element of 	$\ms C_\alpha$ (see Lemma \ref{pro:flipw}). The fact that the core might not be connected add more difficulties.

In the specific case of $\ms{SL}(n,\mathbb C)$, or more generally complex groups, and the principal  $\skd$, this difficulty disappears: the foot space is a torus isomorphic to $\ms C_\alpha$ and the action of the flow (which is a right action) is interpreted as an  element of $\ms C_\alpha$. In this situation, this whole section becomes trivial and the theorem follows from Theorem \ref{prokho}
	
After some preliminaries, we prove this theorem in paragraph \ref{sec:prooffeet}. We then describe an example where some of the hypothesis of the theorem are satisfied in the last paragraph \ref{sec:hypmes} of this section.

\subsubsection{A 1-dimensional torus}
A critical point in the proof is to find a 1-dimensional parameter subgroup containing $\alpha$.
\begin{lemma}\label{pro:flipw}
	We assume $\epsilon$ small enough, then $R$ large enough.
There exists a constant $\bM_1$ only depending on $\G$ so that the following holds.
	Let $\alpha$ be a loxodromic element. Let $\mc A$  be a non empty connected component  of  $\mc X^\Lambda_\alpha$. Then there exists a 1-parameter subgroup $\ms T_\alpha\subset \ms Z(\ms Z(\alpha))$ containing $\alpha$ as well as an element $f\in\ms T_\alpha$  such that for any $\tau\in \mc A$
	\begin{eqnarray}
		\diam(\ms T_\alpha. \tau)&\leq& \bM_1\cdotp R\ , \\
		d(\varphi_1(\tau),f(\tau))&\leq& \bM_1\eR \ .\label{ineq:diamT}
		\end{eqnarray} 
\end{lemma}  
A first step in the proof  of this lemma is the following proposition where we use the same notation
	
\begin{proposition}\label{lem:ua} 	We assume $\epsilon$ small enough, then $R$ large enough. There exists a constant $\bM_2$ so that the following holds. Let $\alpha$ be a $\ms P$-loxodromic element. Let $\mc A^u$  be a non empty connected component  of  $\mc X_\alpha$, then there exists  $u_\alpha\in \mk g$ invariant by $\ms Z(\alpha)$,  with $\exp(2Ru_\alpha)=\alpha$ so that for all  $\tau\in\mc A^u$  
		  \begin{eqnarray}
		\forall 0\leq t\leq 2R,\ \ d(\varphi_t(\tau),\exp(tu_\alpha)(\tau))&\leq& \bM_2\cdotp\eR\, , \label{eq:tauclose3}\\
		\forall 0\leq t\leq 2R,\ \ d(\tau,\exp(tu_\alpha)(\tau))&\leq& \bM_2\cdotp R\, . \label{eq:tauclose4}
				\end{eqnarray}
\end{proposition}	
\begin{proof}  If $\tau$ belongs to the $(\epsilon, R)$ core of $\alpha$, then
$$d_0(\tau^{-1}(\alpha),\exp(2Ra_0))\leq \eR\ .$$ Since $d_0$ is right invariant, we obtain that letting $b\defeq \tau^{-1}(\alpha)\exp(-Ra_0)$\ ,
$$d_0\left(b,\id\right)\leq \eR\ .$$ Thus for $\eR$ small enough, there a exists $v_\alpha$ unique (of smallest norm) in $\mk l_0$ so that 
\begin{eqnarray}
b=\exp(2Rv_\alpha)\  &,& \ \ 
\forall t\in[0,2R], \ \ d_0(\exp(tv_\alpha),\id)\leq\eR\ .	\label{eq:btau}
\end{eqnarray}
Let $u_\alpha\defeq \T \xi_\tau(a_0+v_\alpha)$. Since $a_0$ is in the center of $\mk l_0$, we get from the first equation that 
$$
\alpha=\xi_\tau\left(\exp(2R(a_0+v_\alpha)\right)=\exp(2Ru_\alpha)\ .\
$$ 
The second inequality in assertion \eqref{eq:btau} now yields that for all $\tau$ in $\mc X_\alpha$
$$
d(\varphi_t(\tau),\exp(t u_\alpha)\tau)=d_0(\exp(tv_\alpha),\id)\leq \eR\ .
$$
This proves   inequality \eqref{eq:tauclose3}. Finally inequality \eqref{eq:tauclose4} follows from the fact that there exists a constant $A$ only depending on ${\G}$ so that $d(\varphi_t(\tau),\tau)\leq A.t$, for all $t$ and $\tau$.

If $\eR$ is small enough, $\exp$ is a diffeomorphism in the neighborhood of $v_\alpha$, hence of $u_\alpha$. It follows that $u_\alpha$ only depends on the connected component of $\mc X^u_\alpha$ containing $\tau$.

Similarly since $u_\alpha$ is a regular point of $\exp$, it commutes with the Lie algebra $\mk z(\alpha)$ of $\ms Z(\alpha)$. After complexification, it commutes with $\mk z_\mathbb C(\alpha)$, hence is fixed by  $\ms Z_\mathbb C(\alpha)=\exp(\mk z_\mathbb C(\alpha))$ (since centralizers are connected in complex semisimple groups) and in particular with $\ms Z(\alpha)$.
\end{proof}

We now prove Lemma \ref{pro:flipw} as an application:

\begin{proof}
Let $\mc A^u$ be a connected component of the lift of $\mc A$ to $\mc F_\alpha$.  The hypothesis of proposition \ref{lem:ua} are satisfied  for $\mc A$ and let $u_\alpha\in\mk l_\alpha$ as in the conclusion of this proposition. Let $\ms T_\alpha\defeq \{\exp(tu_\alpha)\}_t$. Since $u_\alpha$ is fixed by $\ms Z(\alpha)$, $\ms T_\alpha\subset \ms Z(\ms Z(\alpha))$.

Let $V_\alpha=\exp([0,2R]u_\alpha)$ be a fundamental domain for the action of $\alpha$ on $V_\alpha$.
By  inequality \eqref{eq:tauclose4}, for all $\tau\in\mc A$
 $$
 \diam\left(V_{\alpha}\tau \right)\leq \bM_2\cdotp R,
 $$
  for some constant $\bM_3$ only depending on ${\G}$.
Since $\alpha$ acts trivially on $\mc F^\Lambda_\alpha$, we obtain that 
 $$
V_{\alpha}\tau= \ms T_{\alpha}\tau\ .
$$
This concludes the  proof of  the first assertion of Proposition \ref{pro:flipw}. The second assertion follows at once from inequality \eqref{eq:tauclose3}.
 \end{proof}

\subsubsection{Averaging measures} 
Let $\mu$ and $\nu$ as in the hypothesis of Theorem \ref{theo:mesfeet}.

 Let $\{\mc X_\alpha^i\}_{i\in I}$ be the collection of connected components of $\mc X^\Lambda_\alpha$. Let us denote by ${\boldsymbol 1}_A$ the characteristic function of a subset $A$. Let \begin{eqnarray}
 \bmu_i&\defeq&{\boldsymbol 1}_{\mc X_\alpha^i}\bmu\ ,\\
 \bnu_i&\defeq&{\boldsymbol 1}_{\mc X_\alpha^i}\bnu\ , 
 	 \end{eqnarray}  
so that $\bmu=\sum_{i\in I}\bmu_i$ and $\bnu=\sum_{i\in I}\bnu_i$.
Let $\ms T^i_\alpha\defeq \ms T^0_{\alpha,\mc X_\alpha^i} $ associated to  $\mc A_0=\mc X_\alpha^i$ as a consequence of Lemma \ref{pro:flipw}. Let finally consider the tori $\ms Q^i_\alpha=\ms T_0\times \ms T^i_\alpha$. 

We first state and prove the following:

\begin{proposition}\label{pro:bm5}
	For a constant $\bM_5$ only depending on ${\G}$, and  $R$ large enough, 
	\begin{eqnarray}
		 \forall g\in \ms T_0, \ \ d_L(\bmu_i, g_*\bmu_i)&\leq& \bM_5\cdotp\eR\ , \label{ineq:gmu}\\
		d_L(\bmu_i, {\varphi_1}_*\bmu_i)&\leq& \bM_5\cdotp\eR\, \label{ineq:phmu}\ .
	\end{eqnarray}
\end{proposition}

\begin{proof} In the proof $M_i$ will be constants only depending on $\ms G$. Let $\hat\bmu_i$ be the average of $\bmu_i$ with respect to $\ms Q^i_\alpha$. By hypothesis, $\bmu_i=C \bnu_i$, where $\Vert C-1\Vert\leq  \frac{\epsilon}{R^2}$. Since $\ms Q^i_\alpha\subset \ms C_\alpha$, and $\ms C_\alpha$ preserves $\bnu_i$, it follows that
\begin{eqnarray}
	\bmu_i=D \cdotp\hat\bmu_i,
\end{eqnarray}
	where $\Vert D-1\Vert\leq  2\frac{\epsilon}{R^2}$.
We now apply Theorem \ref{prokho} to get that 
\begin{eqnarray}
	d_L(\bmu_i,\hat\bmu_i)\leq B\cdotp M_1\frac{\epsilon}{R^2}\ ,
\end{eqnarray}	
where $M_1$ only depends on the dimension of $\ms T_0$ and
$$
B:=\sup(\diam(\ms Q^i_\alpha.\tau\mid \tau \in \mc X^i_\alpha)\ .
$$
By inequality \ref{ineq:diamT}, $\diam(\ms T^i_\alpha\tau)\leq \bM_1\cdotp R$, for $\tau\in\mc X_\alpha$. Moreover since $\ms T_0\subset\ms K_0$,  we have that 
$$\diam(\ms T_0.\tau)\leq \diam \ms K_0\ .$$
 and thus $B\leq M_2 R$.
Thus 
\begin{eqnarray}
	d_L(\bmu_i,\hat\bmu_i)\leq \cdotp M_3\frac{\epsilon}{R}\ ,
\end{eqnarray}
Observe now that $g\in \ms Q^i_\alpha$ acts by isometry on $\mc F_\alpha$ and thus for any measure $\lambda_1$ and $\lambda_0$ we have
$$
d_L(g_*\lambda_0,g_*\lambda_1)=d_L(\lambda_0,\lambda_1).
$$
Using the fact that $g_*\hat\bmu_i=\hat\bmu_i$ it then follows that
\begin{eqnarray}
	d_L(\bmu_i,g_*\bmu_i)\leq d_L(\bmu_i,\hat\bmu_i) + d_L(\hat\bmu_i,g_*\bmu_i)=2d_L(\bmu_i,\hat\bmu_i)\leq 2M_3\eR\ .
\end{eqnarray}
This proves the first assertion.

For the second inequality, by inequality \eqref{eq:tauclose3}, there exists $f\in\ms Q^i_\alpha$, so that for any $\tau$ in  $\mc X^i_\alpha$ then
$$
d(f(\tau),\varphi_1(\tau))\leq \bM_1\eR.
$$
Thus from Proposition \ref{prokho1},
$$
d_L(f_*\bmu_i,(\varphi_1)_*\bmu_i)\leq \bM_0\eR.
$$
Thus 
$$
d_L(\bmu_i,(\varphi_1)_*\bmu_i)\leq d_L(f_*\bmu_i,(\varphi_1)_*\bmu_i)+ d_L(f_*\bmu_i,\bmu_i)\leq \bM_5\eR\ .
$$

The last assertion of Proposition \ref{pro:bm5} now follows.
\end{proof}

\subsubsection{Proof of Theorem \ref{theo:mesfeet}}\label{sec:prooffeet}
From Proposition \ref{pro:bm5} 
	\begin{eqnarray*}
		d_L(\bmu_i, {\varphi_1}_*\bmu_i)\leq \bM_5\cdotp\eR\, &, &
		 \forall g\in \ms T_0, \ \ d_L(\bmu_i, g_*\bmu_i)\leq \bM_5\cdotp\eR .
	\end{eqnarray*}
	Thus by  Proposition \ref{pro:mui}
	\begin{eqnarray*}
		d_L(\bmu, {\varphi_1}_*\bmu)\leq \bM_5\cdotp\eR\, &, &
		 \forall g\in \ms T_0, \ \ d_L(\bmu, g_*\bmu)\leq \bM_5\cdotp\eR .
	\end{eqnarray*}
Then, if $g\in \ms T_0$, using again that $g$ acts by isometry on $\mc F^\Lambda_\alpha$ hence on it space of measures
	\begin{eqnarray*}
d_L(\bmu, g_*{\varphi_1}_*\bmu)\leq 
d_L(g_*\bmu, g_*{\varphi_1}_*\bmu)
+d_L(g_*\bmu, \bmu)= d_L(\bmu, {\varphi_1}_*\bmu)
+d_L(g_*\bmu, \bmu)\leq 2\bM_5\eR\ .
	\end{eqnarray*}
	\subsection{Feet projection of biconnected and triconnected  tripods}\label{sec:hypmes} For $\epsilon$ small enough, then $R$-large enough, thanks to item \ref{it:ii-shad} of Lemma \ref{lem:shad}, 
we can define the feet projection $\bpsi$ from $\mc B\seR(\alpha)$ to $\mc F_\alpha$ by 
$$
\bpsi (T,S_0,\alpha(S_0))=\Psi(T,\alpha^+;\alpha^-)\ ,
$$
Similarly we define the feet projection  $\bpsi$ from $\mc Q\seR(\alpha)$ to $\mc F_\alpha$ by 
$$
\bpsi (T,S_0,S_1,S_2))=\Psi(T,\alpha^+;\alpha^-)\ ,
$$
Let then $\nu\seR$ and $\mu\seR$ defined in paragraphs \ref{def:mu} and \ref{def:nu} respectively.
$$
\bnu\seR=\bpsi^*\nu\seR\ \  ,\ \  \bmu\seR=\bpsi^*\mu\seR\ , 
$$
We summarize some properties of the projection now.
\begin{proposition}\label{pro:supfeet} First we have
$$
\bpsi\circ \bI_0=\bI_0\circ \bpsi\ .
$$
Moreover, assume $\epsilon$ small enough then $R$ large enough.
The feet projection  $\bpsi$ is proper.
The measure $\bnu\seR$ is supported on $\mc X^\Lambda_\alpha$ and is  finite. Finally $\bmu\seR=C\seR\bnu\seR$ with $\Vert 1-C\seR\Vert_\infty\leq \frac{\epsilon}{R^2}$.
\end{proposition}
\begin{proof} 
By construction 
\begin{eqnarray*}
	\bpsi\circ\bI_0 (t,s,c_0,c_1,c_2)&=&\bpsi(\bI_0(t),\bI_0(s), \bI_0(c_1),\bI_0(c_0),\bI_0(c_2))\\
	&=&\Psi(\bI_0(T),\alpha^-,\alpha^+)\\
	&=&\bI_0(\Psi(T,\alpha^+,\alpha^-))\ ,
\end{eqnarray*}
where the last equality comes from the fact that $\bI_0$ exchanges the vertices $\partial^+\tau$ and $\partial^+\tau$ of the tripod and that $\bI_0$ preserves the metric on $\mc G$.

By Lemma \ref{lem:shad} and  Proposition \ref{pro:clemma-step1}, if $B\defeq (T,S_0,\alpha(S_0))$ is in the support of ${\rm D}\seR$ and $\tau_\alpha\defeq\bpsi(B)$, then
 $d(T,\tau_\alpha)+d(S_0,\tau_\alpha)\leq M(\epsilon+R)\ ,
 $
 for some universal constant $M$. This implies the properness of $\bpsi$. 
 
 Moreover, by Proposition \ref{pro:clemma-step1}, $(T,S_0)$-is almost closing for $\alpha$. Thus by 
  the  last assertion of the Closing Lemma \ref{lem:shad}, $\tau_\alpha$ belongs to the $(M\epsilon,R)$-core $\mc X_\alpha$. Thus $\bnu\seR$ is supported on the core. Since $\mc X^\Lambda_\alpha$ is compact (Proposition \ref{pro:procore}), $\bnu\seR$ is finite.
 
 The last assertion of this proposition now follows from Proposition \ref{pro:munu}.
\end{proof}

\section{Pairs of pants are evenly distributed}\label{sec:equi} We will want to glue pairs of pants along their boundary components if their  ``foot projections''  differ by approximately a  ``Kahn-Markovic''  twist. Given a pair of pants, the existence of other pairs of pants which you can admissibly glue along a given boundary component will be  obtained by an equidistribution theorem.

Since we need to glue pair of pants along boundary data, a whole part of this section is to explain the boundary data which in this higher rank situation is more subtle than for the hyperbolic 3-space. We also need to explain what does reversing the orientation mean in this context.

The main result is the Even Distribution Theorem \ref{theo:even} which requires many definitions before being stated. The proof relies on a Margulis type argument using mixing, as well as the presence of some large centralizers of elements of $\Gamma$. This is the only part where the flip assumption -- revisited in this section -- is used. This is of course structurally modeled on the corresponding section in \cite{Kahn:2009wh}. Let us sketch the construction.
\begin{enumerate}
\item The space of triconnected tripods carries a {\em measure} $\mu^+$ coming from the weight functions defined above. Similarly we have a measure $\mu^-$  obtained while using the reverting orientation diffeomorphism on the space of tripods (See Definition \ref{def:mu}).
	\item The {\em boundary data} associated to a boundary geodesic $\alpha$ with loxodromic holonomy and  end points $\alpha^+$ and $\alpha^-$ will be the set of tripods with end points  $\alpha^+$ and $\alpha^-$, up to the action of the centralizer of $\alpha$. In the simplest case of the principal $\skd$ in a complex simple group, this space of feet is a compact torus.
	 \end{enumerate} 
We have now a projection $\bpsi$ from the space of triconnected tripods to the space of feet $\mc F\defeq\sqcup_{\alpha}\mc F_\alpha^\Gamma$, just by taking the projection of one of the defining tripods (and using Theorem \ref{theo:struct-pant}). Our goal is to establish the {\em Even Distribution Theorem} \ref{theo:even} which says that the projected measures $\bpsi_*\mu^+$ and $\bpsi_*\mu^-$ do not differ by much after a Kahn--Markovic twist. Roughly speaking the proof goes as follows.
\begin{enumerate}
	\item This projection $\bpsi$ factors through the space of  ``biconnected tripods'' (by forgetting one of the path connecting the tripods) which carries itself a weight  and a measure. The mixing argument then tells us the projected measure from triconnected tripod to biconnected tripod are approximately the same, or in other words the forgotten path is roughly probabilistically independent form the others.
	\item It is then enough to show that the projected measures from the biconnected tripod is evenly distributed.  In the simplest case of the principal $\skd$ in a complex simple group, this comes from the fact these measures are invariant under the centralizer of $\alpha$ which, in that case, acts transitively on the boundary data. The general case is more subtle (and involves the Flip assumption) since the action of the centralizer of $\alpha$ on space of feet is not transitive anymore.
\end{enumerate} 

In this section $\Gamma$ will be a uniform lattice in $\ms G$, $\alpha$ a $\ms P$-loxodromic element in $\Gamma$, $\Gamma_\alpha$ the centralizer of $\alpha$ in $\Gamma$, which is a uniform lattice in $\ms Z_\G(\alpha)$ (See \cite[Proposition 3.5]{Gelander:2014wq}).

\subsection{The main result of this section: even distribution}

We can now state the main result of this section. This is the only  part of the paper that makes uses of the flip assumption. The Theorem uses the notion of Levy--Prokhorov distance for measures on a metric space which is discussed in Appendix \ref{P}. We first need this
\begin{definition}{\sc [Kahn--Markovic twist]}
For any $\alpha\in\Gamma$, 
the {\em Kahn--Markovic twist} \index{Kahn--Markovic twist}\index{$\KMT$} $\KMT_\alpha$ is the element $\varphi_1\circ\sigma$ that we see as a diffeomorphism of the space of feet $\mathcal F^\Gamma_{\alpha}$. Similarly we consider the (global Kahn--Markovic twist) as the product map $\KMT=\prod_{\alpha\in\Gamma} \KMT_{\alpha}$ from  $\mc F$ to itself.
\end{definition}
Our main result is then
\begin{theorem}\label{theo:even}{\sc [Even distribution]} 
For any  small enough positive $\epsilon$, there exists a positive  $R_0$, such that if $R>R_0$ then $\mu^\pm\seR$ are finite non zero  and furthermore
\begin{equation}
	d_L(\bpsi^+_*\mu^+\seR, \KMT_*\bpsi^-_*\mu^-\seR )\leq M\eR \ \  ,\label{eq:lpb}
\end{equation}
where $\KMT$ is the Kahn--Markovic twist, $M$ only depends on ${\G}$, and $d_L$ is the Levy--Prokhorov distance.
\end{theorem}
The notation $\bpsi^+$ and $\bpsi^-$ in this theorem and the sequel of this section is for redundancy: 
it means $\bpsi$ as considered from $\mathcal B^+\seR(\alpha)$ and  $\mathcal B^-\seR(\alpha)$ respectively. 

The metric on $\mc F$ is the metric coming from its description as a disjoint union, not the induced metric from $\mc G$. This whole section is devoted to the proof of this theorem. We shall use the flip assumption.

\subsection{Revisiting the flip assumption} \label{sec:bjbj0}\label{def:refa} We fix in all this section a reflexion $\bJ_0$.
We will explain in this section the consequence of  the {\em flip assumption} that we shall use as well as give examples of groups satisfying the flip assumptions. Recall that for an element $\alpha$ in $\Gamma$,  we write $\Gamma_\alpha=\ms Z_\Gamma(\alpha)$.

Let $\alpha$ be a $\ms P$-loxodromic element. Let   $${\ms L}_\alpha\defeq \{g\in{\G}, g(\alpha^\pm)=\alpha^\pm\}$$ Observe first that since $\bJ_0\in\ms Z(\ms L_0)$,  then  $\bJ_\alpha\defeq\tau(\bJ_0)$ does not depend on $\tau$, for all 
$\tau$ with $\partial^\pm\tau=\alpha^\pm$, and belongs to ${\ms Z}({\ms L}_\alpha)$. Thus for all $\tau$ in $\mathcal F_\alpha$
$$
\bJ_\alpha\cdotp \tau=\tau\cdotp \bJ_0\ .
$$
The element $\bJ_\alpha$  of $\mathcal G$ is called the {\em reflexion of axis $\alpha$}.

Let also $\ms K_\alpha$ the maximal compact factor of $\ms L_\alpha$.

\begin{definition}{\sc [Weak flip assumption]}\label{def:flipw}
	We say the lattice $\Gamma$ in ${\G}$ satisfies the {\em weak flip assumption}, if  there is some integer  ${\rm M}$ only depending on $\ms G$, so that given a $\ms P$-loxodromic element $\alpha$ in $\Gamma$, then
	\begin{itemize}
	\item 	 there exists a  subgroup $\Lambda_\alpha$ of  $\Gamma_\alpha\cap  {\ms Z}^\circ_\G(\alpha)$, normalized by $\Gamma_\alpha$  with $[\Gamma_\alpha:\Lambda_\alpha]\leq {\rm M}$, 
	\item moreover $\bJ_\alpha$ belongs to  a connected compact torus  $\ms T^0_\alpha\subset \ms Z_\ms G (\Lambda_\alpha)\cap \ms K_\alpha$ , 
	\end{itemize}
\end{definition}
Denoting  by  $\ms{Z_F}(B)$ the centralizer in the group $\ms F$ of the set $B$ and ${\ms H}^\circ  $ the connected component of the identity of the group $\ms H$,
 we now   introduce the following group for a $\ms P$-loxodromic element $\alpha$ in $\Gamma$ satisfying the weak flip assumption:
\begin{eqnarray}
	\ms C_\alpha &\defeq& \left(\ms Z_{{\G}}\left(\Lambda_\alpha\right)\right)^\circ< {\ms L_\alpha}\, 
\end{eqnarray}

\subsubsection{Relating the flip assumptions}
We first  relate the flip assumptions \ref{def:flipr} and \ref{def:flip} to the weak flip assumption \ref{def:flipw}. 
\begin{proposition}
	If ${\G}$ and $\mk s_0$ satisfies the flip assumption \ref{def:flip}, or the regular flip assumption \ref{def:flipr}, then ${\G}$, $\mk s_0$ and $\Gamma$ satisfy the weak flip assumption.
\end{proposition}
\begin{proof}  Let us first make the following preliminary remark:  as an easy consequence of  a general result by John Milnor in \cite{Milnor:1964tl} the following holds: {\em Given a center-free semisimple Lie group ${\G}$, there exists a constant ${\bN}$\index{$\bN$}, so that for every semisimple $g\in{\G}$, the number of connected components of $\ms Z_{\G}(g)$ is less than ${\bN}$}.
In particular, $[\Gamma_\alpha:\Gamma_\alpha\cap \ms Z^\circ_\ms G(\alpha)])\leq {\bN}$.

We have to study the two cases of the flip and regular flip assumptions. Assume first that ${\G}$ and $\mk s_0$ satisfy the flip assumption with reflexion $\bJ_0$. Let $\alpha$ be an element of $\Gamma$ which is $\ms P$-loxodromic. Then any element  $\beta$ commuting with $\alpha$ preserves $\alpha^+$ and $\alpha^-$, thus $\Gamma_\alpha< \ms L_\alpha$. The flip assumption hypothesis thus implies that taking $\Lambda_\alpha=\Gamma_\alpha\cap\ms Z^\circ_\ms G(\alpha)$.
		$$\bJ_\alpha\in \left(\ms Z_{\G}(\ms L_\alpha)\right)^\circ\subset \left(\ms Z_{\G}(\Gamma_\alpha)\right)^\circ\subset \left(\ms Z_{\G}(\Lambda_\alpha)\right)^\circ \ . $$
		Moreover $\bJ_\alpha$ is an involution that belongs to the center of $\ms L_\alpha$ and thus to its compact factor. Thus we may choose for $\ms T_\alpha^0$  a maximal torus in $\left(\ms Z_\ms G (\Lambda_\alpha)\right)^\circ\cap \ms K_\alpha$  containing $\bJ_\alpha$
This concludes this case.

Let us move to the regular flip assumption. In that case $\ms L_0=\ms A_0\times \ms K_0$ where $\ms A _0$ is a torus without compact factor and $\ms K_0$ is a compact factor, accordingly $\ms L_\alpha=\ms A_\alpha\times \ms K_\alpha$ with the same convention.  Let $\alpha$ be a $\ms P$-loxodromic element in $\Gamma$, as above we notice that $\Gamma_\alpha\subset \ms L_\alpha$. Since $\Gamma_\alpha$ is discrete torsion free, $\Gamma_\alpha\cap \ms K_\alpha=\{e\}$. Thus, the projection of $\Gamma_\alpha$ on $\ms A_\alpha$ is injective, and $\Gamma_\alpha$ is abelian. Let  $\pi$  be the projection of $\ms L_\alpha$ on $\ms K_\alpha$,  $\ms B=\pi(\Gamma_\alpha)$  and $\ms B_1$ a maximal abelian containing $\ms B$ in $\ms K_\alpha$. Using again \cite{Milnor:1964tl}, there is a constant ${\rm M}$ only depending on $\ms G$, so that if $\ms C$ is maximal abelian in $\ms K_0$, then $[\ms C:\ms C^\circ]\leq {\rm M}$. Let 
$$
\Lambda_\alpha\defeq\pi^{-1}(\ms B_1^\circ)\cap \Gamma_\alpha\subset \ms Z^\circ_\ms G(\alpha)\  .
$$
Then $[\Gamma_\alpha:\Lambda_\alpha]\leq {\rm M}$. Moreover, setting $\ms T^0_\alpha$ to be a maximal torus containing $\ms B_1^\circ$, we have (since $\bJ$ is central in $\ms K_0)$
$$
\bJ_\alpha\in \ms T^0_\alpha\subset \ms Z_\G (\ms B^\circ_1)\cap \ms K_\alpha\subset \ms Z_\G (\Lambda_\alpha)\cap \ms K_\alpha\ .
$$
	This concludes the proof of the proposition.
\end{proof}

\subsubsection{Groups satisfying (or not) the flip assumptions}\label{sec:flip-ex}

Let us show a list of group satisfying the flip assumptions. Examples
 \ref{exa:flip ii} was pointed out to us by Fanny Kassel, who also pointed out earlier mistakes.

\begin{enumerate}
\item If $\G$ is  a complex semi-simple Lie group. Let $\mk s=(a,x,y)$ be an even $\slt$-triple. Then $\bJ_0=\exp{\frac{i\zeta a}{2}}$, for the smallest $\zeta>0$ so that  $\exp(i\zeta a)=1$ is a reflexion that satisfies the flip assumption: indeed $\exp(ita)$ for $t$ real lies in $\ms Z\left(\ms Z(a)\right)$.
\item The isometry groups of the hyperbolic space,  $\ms{SO}(1,p)$ do not satisfy the flip assumption when $p$ is even,  while they satisfy it for   $p$ odd.
\item  \label{exa:flip i} The groups $\ms{SO}^+(2,4)$  satisfies the  flip assumption for  some $\slt$-triple with compact centralizer. More precisely let us consider $\Delta$ the diagonal $\slt$ in $\ms{SO}^+(1,2)\times \ms{SO}^+(1,2)$  in $\ms{SO}^+(2,4)$. Then, on can check that $\ms{Z}(a)=\ms{GL}(2,\mathbb R)\times \ms{SO}(2)$ and, in  this decomposition $J_0=(\id , -\id)$ . Thus although $a$ is not regular, $J_0$ belongs to the connected component of the identity of $\ms{Z}\left(\ms{Z}(a)\right)$. Finally, one notices that $\ms{Z} (\Delta)$ is compact.

\item\label{exa:flip ii}The groups  $\ms{SU}(p,q)$, with $q>p>0$,  satisfy the flip assumption. We consider  $\ms H$ to be the irreducible $\sld$ in $\ms{SO}(p,p+1)$. Then we see $\ms{SO}(p,p+1)$ as subgroup of $\ms{SU}(p,q)$ in $\ms{GL}(p+q)$. Then the centralizer of $a$ in   $\ms{GL}(p+q)$ is $\mathbb C^{2p}\times\ms{GL}(q-p)$. It follows that the centralizer 
$\ms{Z}(a)$ in $\ms{SU}(p,q)$ is $\mathbb C^p\times \ms{SU}(q-p)$. Thus $a$ is regular and  $\ms{Z}(a)$ is connected. It flows that $\ms H$ satisfies the regular flip assumption. 

\end{enumerate}

On the other hand, one easily checks that the groups 
 $\ms{SL}(n,{\mathbb R})$ do not satisfy the flip assumption for the irreducible $\ms{SL}(2,\mathbb R)$.
\subsection{Proof of the Even Distribution Theorem \ref{theo:even}}
Let $\Lambda_\alpha$ the subgroup of $\Gamma_\alpha$  of index at most $\bM$ appearing in Definition \ref{def:flipw}.

Recall that $\KMT_\alpha=\varphi_1\circ\sigma=\varphi_1\circ \bJ_0\circ \bI_0=\varphi_1\circ \bJ_\alpha\circ \bI_0$, where $\bJ_\alpha$ is defined in the beginning of paragraph \ref{sec:bjbj0}.

The second equality comes from the fact that as seen in the beginning of section \ref{sec:bjbj0} the right action of $\bJ_0$ and the left action of $\bJ_\alpha$ coincide on $\mc F_\alpha$.
Using 
propositions \ref{pro:Imu}, \ref{exispant} and \ref{pro:supfeet} we have
$$
\bpsi^-_*\mu^-\seR=\bpsi^-_*{\bI_0}_*\mu^+\seR={\bI_0}_*\bpsi^+_*\mu^+\seR\ .
$$
Let then $\bmu=\bpsi^+_*\mu^+\seR$ and $\bnu=\bpsi^+_*\nu^+\seR$
Our goal is thus to prove that there exists a constant $\bM_5$ only depending on ${\G}$, so that 
\begin{eqnarray}
d_L({{\varphi_1}_*\bJ_\alpha}_*\bmu,\bmu)\leq \bM_5\eR\ ,	
\end{eqnarray}
where we consider $\bmu$ as a measure on $\mc F^{\Gamma_\alpha}_{\alpha}$. We perform a further reduction. Let ${\rm p}$ the covering from $\mc F^{\Lambda_\alpha}_\alpha$ to $\mc F^{\Gamma_\alpha}_{\alpha}$. let $\bmu^0$ and $\bnu^0$ be the preimages of $\bmu$ and $\bnu$ respectively on $\mc F^{\Lambda_\alpha}_\alpha$.  Since, ${\rm p}_*\bmu^0=q \bmu$ and   ${\rm p}_*\bnu^0=q \bnu$ where $q$ is the degree -- less than $\bM$ -- of the covering  ${\rm p}$, it is enough by Proposition \ref{pro:prokhcont} to prove that 
there exists a constant $\bM_6$ only depending on ${\G}$, so that 
\begin{equation}
d_L({{\varphi_1}_*\bJ}_*\bmu^0,\bmu^0)\leq \bM_6\eR\ , . \label{ineq:redu0}	
\end{equation}
But now this is a consequence of Theorem \ref{theo:mesfeet}, taking
$\bnu\defeq \Psi_*\nu\seR$, where $\nu\seR$ is defined in  paragraph \ref{def:nu} and its invariance under $\rm C_\alpha$ is checked in the same paragraph.  The main hypothesis \eqref{eq=mix} of  Theorem \ref{theo:mesfeet} is a consequence of Proposition \ref{pro:munu}.

\section{Building straight surfaces and gluing}\label{sec:straight}

Our aim in this section is to define {\em straight surfaces} and prove their existence in Theorem \ref{exisstraight}. Loosely speaking, a straight surface is obtained by gluing almost Fuchsian pair of pants using KM-twists.  We also explain that a straight surface comes with a fundamental group.

This section is just a rephrasing of a similar argument in \cite{Kahn:2009wh} and uses as a central argument the Even Distribution Theorem \ref{theo:even}.

\subsection{Straight surfaces} Recall that in  a graph, a {\em flag} adjacent to  vertex $v$ a  is a pair $(v,e)$ so that the edge $e$ is adjacent to the vertex $v$. The {\em link} $L(v)$ of a vertex $v$  is the set of flags adjacent to $v$.   A trivalent ribbon graph is a graph with a cyclic permutation $\omega$ of order 3, without fixed points,   on edges so that $\omega(v,e)=(v,f)$ so that every link $L(v)$ is equipped with a cyclic permutation $\omega_v$ of order 3.  If a graph is bipartite so that we can write its set of vertices as $V^-\sqcup V^+$, we denote by $e^\pm$ the vertices of an edge $e$ that belong to $V^\pm$ respectively. 

Let $\Gamma$ be a discrete  subgroup of $\ms G$.

\begin{definition}{\sc [Straight surfaces]} Let $\epsilon$ and $R$ be positive numbers.
An $(\epsilon,R)$ {\em straight surface for $\Gamma$} \index{Straight surface} is a pair $\Sigma=(\mc R,W)$ where $\mc R$ is a finite bipartite trivalent ribbon graph whose set of vertices is $V^-\sqcup V^+$, and $W$ is labeling of flags in $\mc R$ so that 
\begin{enumerate}
	\item For every flag $(v,e)$ with $v\in V^\pm$, $W(v,e)$ belongs to $Q^{\Gamma,\pm}\seR$,
    \item The labeling map  is equivariant: $W(\omega_v(v,e))=\omega(W(v,e))$.
	\item for any edge $e$, 
	\begin{equation}
		d(\bpsi^+(W(e^+,e)),\KMT \bpsi^-((W(e^-,e)))\leq\eR\ .\label{def:KMg}
	\end{equation} 
\end{enumerate}
 \end{definition}

We may now associate to a straight surface $\Sigma$  a topological surface given by the gluing of pair of pants (labeled by vertices) along their boundaries (labeled by edges), surface whose  fundamental group is  denoted $\pi_1(\Sigma)$. 
The labeling of vertices of edges will then give rise to a representation of $\pi_1(\Sigma)$ into $\Gamma$ (See section \ref{sec:fundgroup}). The main Theorem of this section is 

\begin{theorem}\label{exisstraight}{\sc[Existence of straight surfaces]}
Let $\mk s$ be an $\sld$-triple in the Lie algebra of a semisimple group $\ms G$-satisfying the flip assumption.
Let $\Gamma$ be a uniform lattice in $\ms G$. 

Then, for every $\epsilon$, there exists $R_0$  so that for any $R\geq R_0$, there exists an $(\epsilon,R)$-straight surface for $\Gamma$.\end{theorem}

\subsection{Marriage and equidistribution}
We want to prove
\begin{lemma}{\sc [Trivalent graph]}\label{trigraph}
Let $Y$ be a compact metric space. Let $\omega$ be  an order 3 symmetry  acting freely on $Y$. Let  $\mu$ be a $\omega$-invariant  finite (non zero) measure on $Y$. Let $\alpha$ be a real number. Let $f^0$ and $f^1$ 
be two uniformly Lipschitz maps from $Y$  to a metric space $Z$ such that
$
d_L(f^0_*\mu,f^1_*\mu)<\alpha
$. Then there exists a nonempty  finite  trivalent bipartite ribbon graph $\mc R$, whose set vertices are $V_0\sqcup V_1$ so that
\begin{itemize}
\item we have an $\omega$-equivariant labeling $W$ of  flags by elements of  $Y$. 
\item if $e$ is an edge from $v_0$ to $v_1$ so that $v_i\in V_i$, then $d(f^0\circ W(v_0,e),f^1\circ W (v_1,e))\leq\alpha$.
\end{itemize}
\end{lemma}
This will be an easy consequence of the following theorem.
\begin{theorem}{\sc [Measured Marriage Theorem]} \label{Hall+}
Let $Y$ be a  compact metric space equipped with a finite (non zero) measure $\mu$. Let $f$ and $g$
be two uniformly Lipschitz maps from $Y$ to a metric space $Z$ such that
$$
d_L(f_*\mu,g_*\mu)< \beta\ .
$$
Then there exists a non empty finite set $\bar Y$, a  map  $p$  from $\bar Y$ in $Y$, a bijection $\phi$ from $\bar Y$ to itself, so that
$$
d(f\circ p,g\circ p\circ \phi)\leq 2\beta.
$$
Assume moreover that we have a free action of an order 3 symmetry $\sigma$ on  $Y$  preserving the measure. Then,  there exists $\bar Y$  and $p$ as before equipped with  an order 3 symmetry $\tilde\sigma$ so that $p$ is $\tilde\sigma$-equivariant.
\end{theorem}
\begin{proof}
If $\mu$ is the counting measure and $Y$ is finite, this theorem is a rephrasing of Hall Marriage Theorem (see \cite[Theorem 3.2]{Kahn:2009wh})  We reduce to this case by the following trick:
by approximation (See Proposition \ref{pro:approx}), we can approximate with respect to the Levy Prokhorov metric $\mu$ by a finitely supported atomic measure $\nu$. 

Then  by Proposition \ref{pro:prokhcont}, $f_*\nu$ and $f_*\mu$ are very close and the same holds for $g$.  Thus
\begin{eqnarray}
	d_L(f_*(\nu),g_*(\nu))< 2\beta\ .\label{ineq:marria}
\end{eqnarray}
Since $\nu$ is atomic, we can replace $Y$ with the finite set $\operatorname{Supp}(\nu)$ of cardinal $N$. Then $\nu$ become an element of $\mathbb R^N$. The inequation \eqref{ineq:marria} now turns to be equivalent to  the following statement. Let $\mathcal E$ be the set   of  pair of subsets $(Z_0,Z_1)$ of $Y$ so that $Z_0=f^{-1}(B)$ and $Z_0=g^{-1}(B_{2\beta})$ for some $B$ in $X$.
\begin{eqnarray*}
	\hbox{ for all } (Z_0,Z_1)  \hbox{ in  } \mathcal E &\ , & \ \  \sum_{x\in Z_0}\nu\{x\}\leq \sum_{x\in Z_1}\nu\{x\}\ .
\end{eqnarray*}
In other words, into finitely many inequalities with integer coefficients and linear in the coordinates of $\nu$.

Since we have a solution with real positive coefficients of this set of inequalities, we also have a solution with positive rational coefficients, or in other words a measure $\lambda$ with rational weights so that
\begin{eqnarray*}
	d_L(f_*(\lambda),g_*(\lambda))< 2\beta\ .
\end{eqnarray*}
After multiplication we can assume  that all weights are integers. then finally we let $\bar Y$ to be the set $Y$ counted with the multiplicity given by $\lambda$, and we can conclude using the observation at the beginning of the paragraph. Finally the procedure can be made equivariant with respect to finite order symmetries.
\end{proof}

\subsubsection{Proof of Lemma \ref{trigraph}}
 Let $\bar Y$, $\tilde\sigma$ and $h$ as in Theorem \ref{Hall+}. Let us write $V=\bar Y/\langle\sigma\rangle$ and $\pi$ the projection  from $\bar Y$ to $V$. Let now $\mc R=V_0\sqcup V_1$ be the disjoint union of two copies of $V$; this will be the set of vertices of the graph. An edge is given by a point $y$ in $\tilde Y$, that we consider joining the vertex $v_0\defeq \pi(y)$ to $v_1\defeq \pi(\phi(y))$. The labeling is given by $W=p$.

\subsection[Proof of the existence of straight surfaces]{Existence of straight surfaces: Proof of Theorem \ref{exisstraight}}  

We apply Lemma \ref{trigraph} to the set $Y\defeq\mc Q^{\Gamma,+}\seR$ (which is non empty by Proposition \ref{exispant}), the measure  $\mu\defeq\mu^+\seR$ (which is $\omega$-invariant and satisfy $I_*\mu^\pm\seR=\mu^\mp\seR$ by Proposition \ref{exispant}) and the functions $f^0\defeq\bpsi^+$, $f^1\defeq \KMT\circ\bpsi^+=\KMT\circ\bpsi^-\circ\bI_0$. Observe that we label vertices of $V_0$ by $W(v,e)$, while we label vertices of $V_1$ by $I(W(v,e))$. For $\epsilon$ small enough, then $R$ large enough (depending on $\epsilon$)  we have that 
\begin{itemize}
	\item the set   $\mc Q^+\seR$ is non empty by Proposition \ref{exispant}.
\item  by Theorem \ref{theo:even}, we have the inequality $
d(f^0_*\mu,f^1_*\mu)\leq M\eR
$, using the fact that $I_*\mu^\pm\seR=\mu^\mp\seR$.
\end{itemize}
 Theorem \ref{exisstraight} is now a rephrasing of the Trivalent Graph Lemma \ref{trigraph}.
 
 \section{The perfect lamination}\label{sec:lam}

In this section, we concentrate on plane hyperbolic geometry. We present some results  of \cite{Kahn:2009wh} concerning the $R$-perfect lamination. This perfect lamination is associated to a tiling by hexagons.

We also introduce a new concept: {\em accessible points} from a given hexagons. Apart from the definition, the most important result is the Accessibility Lemma \ref{lem:access} which guarantees accessible points are almost (in a quantitative way) dense.

\subsection{The \texorpdfstring{$R$}{R}-perfect lamination and the hexagonal tiling}
Let us consider two ideal triangles in the (oriented) hyperbolic plane, glued to each other by a swish of length $R$ (with $R>0$) to obtain a pair of pants $P_0$, called the {\em positive  $R$-perfect pair of pant}.  Symmetrically, the {\em negative $R$-perfect pair of pants} $P_1$ is obtained by a swish of length $-R$. Both perfect pair of pants come by construction  with ideal triangulations and orientations.

The {\em $R$-perfect surface}\index{Perfect surface}  $S_R$ \index{$S_R$} is the genus 2  oriented surface obtained by gluing the two pairs of pants  $P_0$ and $P_1$  with a swish of value 1. The surface $S_R$  possesses three {\em cuffs} which are the three geodesic boundaries of the initial pairs of pants. These cuffs are oriented, where the  orientation comes from the orientation on $P_0$.

Let $\Lambda_R$\index{$\Lambda_R$} be the Fuchsian group so that $\hh/\Lambda_R=S_R$.

The {\em $R$-perfect lamination}\index{Perfect Lamination} $\mathcal L_R$ of ${\hh}$ is the lift  of the cuffs of $S_R$ in  $\hh$.

Observe that each leaf of $\mathcal L_R$ carries a natural orientation. Connected components of the complement of $\mathcal L_R$ are {\em even} or {\em odd} whenever they cover respectively a copy of $P_0$ or  $P_1$.

We denote by $\mathcal L^\infty_R$ the set of endpoints of $\mathcal L_R$ in $\partial_\infty\hh$.

\subsubsection{Length, intersection and diameter} We collect here important facts about the $R$-perfect lamination from Kahn--Markovic paper \cite{Kahn:2009wh}. 
\begin{lemma}{\sc[Length control]\  }\cite[Lemma 2.3]{Kahn:2009wh},
There exist a constant  $K$, so that for $R$ large enough,  for all geodesic segments $\gamma$ in $\hh$ of length  
$\ell$, we have
$
\sharp(\gamma\cap\mc L_r)\leq K\cdotp R\cdotp\ell
$ .
\end{lemma}

\begin{lemma}{\sc [Uniformly bounded diameter]\ } \cite[Lemma 2.7]{Kahn:2009wh}\label{lem:diamSR}
There exists a constant $M$ independent of $R$, such that for all $R$,
$
	\diam\left(S_R\right)\leq M$.
\end{lemma}
As a corollary of the first Lemma, using the language of section \ref{sec:cuff}, we have

\begin{corollary}\label{cor:KMbd}
There exists a constant $K$, so that for $R$ large enough, any coplanar sequence of cuffs whose underlying geodesic lamination is a subset of $\mathcal L_R$ is a $KR$-sequence of cuffs.
\end{corollary}

\subsubsection{Tilings: connected components, tiling hexagons and  tripods}
Let $C$ be a connected component of ${\hh}\setminus \mc L_R$. 

Observe that $C$  is  tiled by right-angled {\em tiling hexagons} coming from the decomposition in pair of pants of $S_R$.  Each such hexagon $H$ is described by a triple of geodesics $(a,b,c)$ in $\mathcal L_R$, whose ends points (with respect to the orientation) are respectively $(a^-,a^+)$, $(b^-,b^+)$ and $(c^-,c^+)$ so that the sextuple $(a^-,a^+,b^-,b^+, c^-,c^+)$ is positively oriented. Let us then define three disjoint intervals, called {\em sides at infinity} in $\partial_\infty\hh$ by $\partial_aH:=[b^+,c^-]$, $\partial_cH:=[a^+,b^-]$,    and 
 $\partial_bH:=[c^+,a^-]$. 	Each such side corresponds to the edge of the hexagon connecting the two corresponding cuffs.
 
 \begin{definition}
 \begin{enumerate}
 	\item The {\em successor} of an hexagon $H=(a,b,c)$ is the unique hexagon of the form $\Sc(H)=(a,d,b)$. 
 	\item The {\em opposite} of an hexagon $H=(a,b,c)$ is the hexagon  $\Op(H)=(a,b',c')$, so that $H$ and $\Op(H)$ meet along a geodesic segment of length $R-1$.
 	\item Given a tiling hexagon $H$, an {\em admissible tripod}\index{Admissible tripod} with respect to $H$ is given by three points $(x,y,z)$ in $\partial_aH\times
 \partial_bH\times\partial_cH$.
 \end{enumerate}
\end{definition}  
We remark that 
$\Op\circ\Sc\circ\Op\circ\Sc={\rm Id}$. We can furthermore color hexagons:

\begin{proposition}\label{pro:hcolor}
	There exists a labeling of hexagons by two colors (black and white) so that $H$ and $\Op(H)$ have the same color, while $H$ and $\Sc(H)$ have different colors.
\end{proposition}

We denote by $T_R(H)$ the set of admissible tripods with respect to a given hexagon $H$ and $T_R$ the set of all admissible tripods. Elementary hyperbolic geometry yields

\begin{proposition}\label{pro:hexa}
	There exists a universal constant $K$, so that for $R$ large enough
	\begin{enumerate}
		\item the diameter of each tiling hexagon is less than $R+K$.
		\item each hexagon has long edges (along cuffs) of length $R$, and short edges of length $\ell$ where   $$
\frac{e^\ell+1}{e^\ell-1}=\sqrt{\frac{1+e^{3R}}{e^R+e^{2R}}}, \ \ \ \ \ \lim_{R\to\infty}e^{\frac{R}{2}\ell}=1\ .
$$
\item the  distance between any  two admissible tripods with respect to the same hexagon is at most $2e^{-\frac{R}{2}}$.
\end{enumerate} 
 \end{proposition}

\subsubsection{Cuff groups and graphs}\label{sec:perfect}
The {\em cuff elements}\index{Cuff elements, cuff group} are those  elements of the Fuchsian group $\Lambda_R$ whose axis are cuffs, a {\em cuff group} $\Lambda$ is a finite  index subgroup of $\Lambda_R$ containing all the primitive cuff elements: equivalently, $\Lambda\backslash\hh$ is obtained by gluing  $R$-perfect pair of pants by swishes of length 1. We will identify oriented cuffs with primitive cuff elements.

To a cuff group $\Lambda$, we can associate   a ribbon graph $\mathcal R$. Observe that $S\defeq \Lambda\backslash\hh$ is tiled by hexagons. We consider the graph $\mathcal R$ whose vertices are hexagons in the above tiling of  $S$, up to cyclic symmetry, and edges corresponding to pair of  hexagons who lift to opposite hexagons.

Observe $\mathcal R$ is the covering of the corresponding graph for $S_R$ and has thus two connected components which correspond respectively to the two coloring in black and white hexagons. The distinction between odd and even components (and thus between odd and even hexagons) gives to $\mathcal R$ the structure of a bipartite graph.

 Hexagons in $S$ correspond to links of $\mathcal R$. By construction each hexagon $H$ is associated to a perfect triconnected pair of tripods $W_0(H)$ with respect to $\sld$, in other words an element in $\mc Q_{0,R}$. We have thus associated to each cuff group $\Lambda$ a $(0,R)$-straight surface $\Sigma(\Lambda)\defeq(\mc R,W_0)$ -- which actually has two connected components. One easily checks that every connected $(0,R)$-straight surface $\Sigma$ is obtained from a well defined cuff group $\Lambda$, as a connected component of $\Sigma(\Lambda)$.

\subsection{Good sequence of cuffs and accessible points}

Let us start with a definition associated to a positive number $K$.

\begin{definition} A pair $(c_1,c_2)$ of cuffs is {\em $K$-acceptable} if 
\begin{enumerate}
	\item There is no cuffs between $c_1$ and $c_2$,
	\item Moreover $d(c_1,c_2)\leq K$.
\end{enumerate}

A triple of cuffs $(c_1,c_2,c_3)$ of cuffs is {\em $K$-acceptable} if 
\begin{enumerate}
     	\item we have  $d(c_1,c_3)\leq K$.
     	\item $c_2$ is the unique cuff between $c_1$ and $c_3$
\end{enumerate}
\end{definition}

Observe that if  $(c_1,c_2,c_3)$ of cuffs is  $K$-acceptable, then both $(c_1,c_2)$ and $(c_2,c_3)$ are $K$-acceptable

\begin{definition}
\begin{enumerate}
	\item A {\em $K$-good sequence of cuffs } is a sequence of cuffs   $\{c_m\}_{1\leq m\leq p}$ such that for every $m$, 
whenever it makes sense, $(c_m,c_{m+1},c_{m+2})$ is $K$-acceptable.
\item  An {\em accessible point} with respect to an tiling hexagon $H$  is a point in $\partial_\infty\hh$	 which is a limit of subsequences of end points of the cuffs of $K$-good sequence of cuffs, where $c_1$ and $c_2$ contains long segments of the boundary of $H$.
\end{enumerate}
\end{definition}

Observe that we have an associated nested sequence of chords, where the chord is defined by the geodesic $c_n$ and the half space containing $c_{n+1}$ or not containing $c_{n-1}$. For a point $x$ in $\hh$, we denote by $W^R_x(K)$ the set of  $K$-accessible  points from an hexagon containing $x$ (with respect to the lamination $\mathcal L_R$).

The main result of this section is the following lemma

\begin{lemma}{\sc [Accessibility]}\label{lem:access}
Let $K_0$ be a positive constant large enough.
	There exists some function $R\mapsto a(R)$ converging to zero as $R\to \infty$, so that $W_H^R(K_0)$ is $a(R)$-dense.
\end{lemma}

\subsection{Preliminary on acceptable pairs and triples}
We need first to understand $K$-acceptable pairs

\begin{proposition}\label{pro:Kpairs}
For $R$ large enough, Let $(c_1,c_2)$ be a $K$-acceptable pair
	\begin{enumerate}
		\item then $ \frac{1}{2}e^{-\frac{R}{2}}\leq d(c_1,c_2)\leq 2e^{-\frac{R}{2}}$
		\item There exists exactly two hexagons $(H_1,H_2)$ whose sides are $c_1$ and $c_2$. Moreover $H_2=\Sc(H_1)$.
		\item if $(c_1,\eta)$ is $K$-acceptable, and furthermore  $\eta$ and $c_2$ lie in the same connected component of $\hh\setminus c_1$,  then there exists $\gamma\in\Lambda_R$ preserving $c_1$ so that $\eta=\gamma\cdotp c_2$.
		\end{enumerate}
\end{proposition}

We have also a proposition on $K$-acceptable triples
\begin{proposition}\label{pro:Ktrip} There exists $K_0$ so that if  $(c_1,c_2)$ is a $K$-acceptable pair with $K>K_0$, then
	\begin{enumerate}
		 \item there exists exactly three $K$-acceptable triples starting with $c_1$ and $c_2$. Fixing an orientation of $c_2$, we can describe the last geodesic in the triple  as $c_3^+:=<c_1,c_2>^+$,  and similarly $c_3^0$ and $c_3^-$, where if $x^i$ is the projection of $c^i_3$ on $c_2$, then $(x^-,x^0,x^+)$ is oriented.
		 \item  If $(c_1,c_2,c_3)$ is a $K$-acceptable triple, then $d(c_1,c_3)\leq K_0$ and moreover if $x_i$ is the point in $c_2$ closest to $c_i$, then $d(x_1,x_2)\leq 3R$.
		 \item \label{it:hexatrip} Moreover if $(H_1,\Sc(H_1))$ and $(H_2,\Sc(H_2))$ are the pairs of hexagons bounded respectively by $(c_1,c_2)$ and $(c_2,c_3)$, then 
		 $$
		 H_2=\gamma^p\Op(H_1)\ ,
		 $$
		 where $\gamma$ is the cuff element associated to $c_2$ and $p\in\{-1,0,1\}$.
		\item  if $c$ is a geodesic non intersecting $c_1$ and $c_2$, so that $c_2$ is between $c_1$ and $c$ and so that $d(c,c_1)<K$, then  there is a cuff $c_3$ so that $(c_1,c_2,c_3)$ is a $K$-acceptable triple and \begin{itemize}
			\item either $c_3$ do not not intersect $c$ and \begin{itemize}
			\item  $c_3$ lies between $c$ and $c_2$, 
			\item or $c$ lies between $c_3$ and $c_2$,
		 \end{itemize}
			\item or $c_3$ intersects $c$.
		\end{itemize} 
	\end{enumerate}
\end{proposition}

These two propositions have  immediate consequences summarized in the following corollary:
\begin{corollary}\label{coro:K10}
\begin{enumerate}
	\item 	For all positive $K_1$ and $K_2$ greater than $K_0$, there exists $R_0$ so that for all $R>R_0$, $W_x^R(K_1)=W_x^R(K_2)$.
	\item Any finite $K$-good  sequence of cuffs $\{c_1,\ldots,c_p\}$ can be extended to an infinite $K$-good sequence $\seq{c}$.
\end{enumerate}
\end{corollary}

\subsubsection{Proof of Proposition \ref{pro:Kpairs}}

If there is no cuffs between $c_1$ and $\eta$, then $c_1$ and $\eta$ are common bounds of the universal cover of one pair of pants. Then for $R$ large enough
\begin{itemize}
\item either $d(c_1,\eta)>R/2$,
\item Or they bounds two hexagons with a common short edge that joins $c_1$ to $c_2$. Then by construction of the shear coordinates, the pair of pants obtained by gluing to ideal triangles using an $R$-swish has $2R$ as length of its boundaries. Thus the two hexagons have opposite long sides of  length $R$ and short side of length approximately $e^{-\frac{R}{2}}$  by the last item of Proposition \ref{pro:hexa}. The result now follows

\end{itemize}
			
This shows the first assertion.

 Finally  all $K$-acceptable pairs $(c_1,\eta)$ -- if $\eta$ and $c_2$ are in the same connected component of $\hh\setminus c_1$ --  are equivalent under the action of $\Lambda_R$, the first item follows.

\subsubsection{Controlling distances to geodesics}
We will denote in general  by $[c,d]$ the geodesic arc passing between $c$ and $d$ where $c$ and $d$ could be at infinity
We first need a statement from elementary hyperbolic geometry
			\begin{proposition}\label{pro:quad}
				If $a$ and $b$ are two non intersecting geodesics, if $x$ is the closest point on  $a$ to $b$, if $y$ is a point on $a$ so that $d(x,y)>R_0$, then
			$$
			 d(y,b)\geq \inf\left(\frac{1}{10}d(a,b)e^{\frac{3}{4}d(x,y)}, \frac{1}{4}d(x,y)-d(a,b)\right)\ .
			$$
			\end{proposition}
			
			\begin{proof} Let $w$ and $z$ be the projections of $x$ and $y$ on $b$. 	 Let $A:=d(x,y)$.
						\begin{enumerate}
				\item Assume first $d(z,w)\leq \frac{3A}{4}$. Then 
				\begin{eqnarray*}
			d(y,z)
					\geq  d(y,x) - d(x,w)- d(w,z)\geq A- d(a,b)-\frac{3}{4}A
					\geq  \frac{A}{4}-d(a,b)\ ,
				\end{eqnarray*}
			\item If now $d(z,w)\geq \frac{3A}{4}$, then 
				 $d(y,z) \geq  \frac{1}{10} d(x,w)\  e^{\frac{3A}{4}}$.
			\end{enumerate}
			This concludes the proof of the inequalities	
			\end{proof}\subsubsection{Proof of Proposition \ref{pro:Ktrip}}

\begin{figure}[h]
\centering
\includegraphics[width=0.4\textwidth]{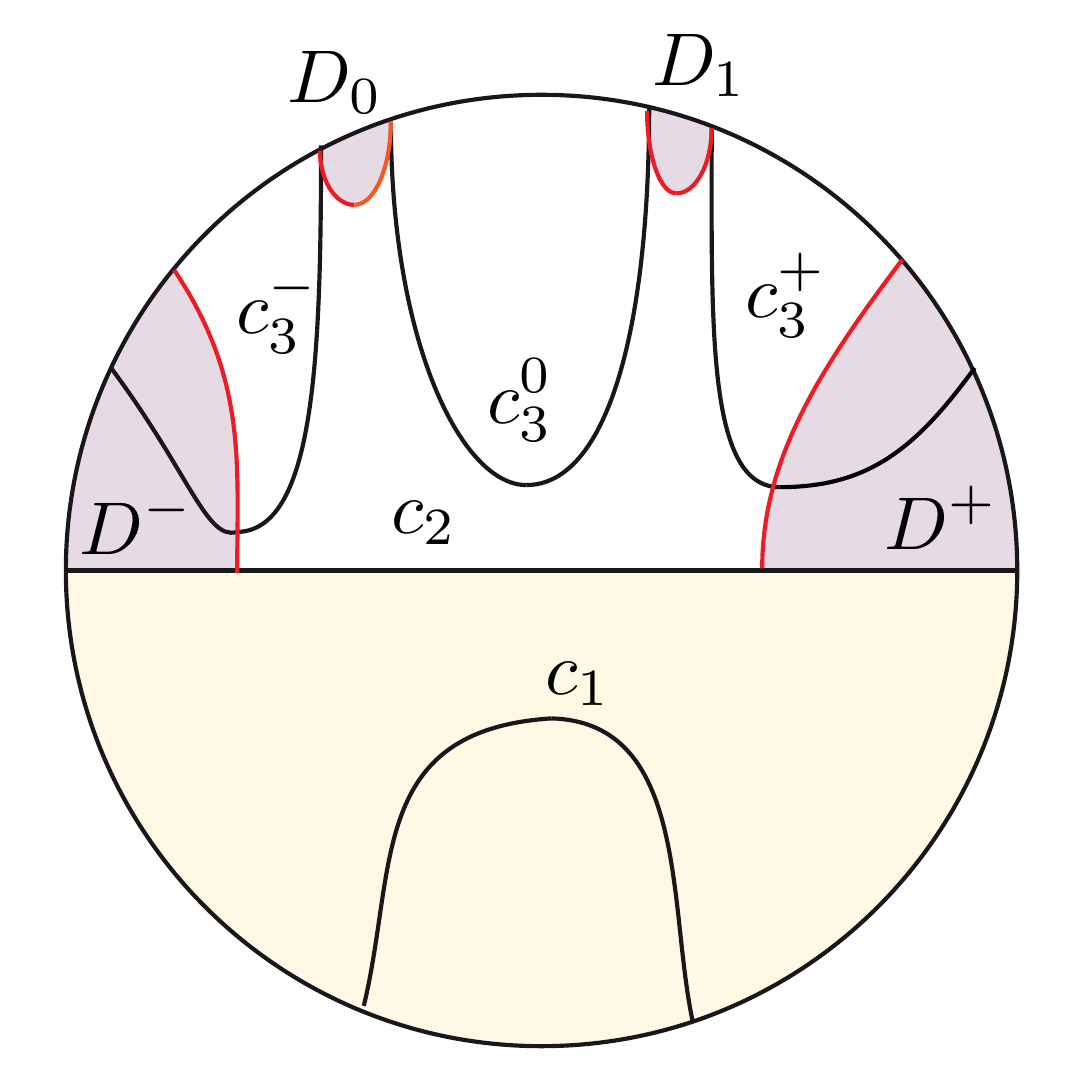}
  \caption{$K$-acceptable triples}\label{fig:Kaccep}
\end{figure}
Let $(c_1,c_2)$ be a $K$-acceptable pair. Let $c^0_3$ be the unique cuff so that $(c_2,c^0_3)$ are $K$-acceptable and if $z$ is the projection of $c_3^0$ on $c_2$, $y$ is the projection of $c_1$ on $c_2$, then $d(z,y)=1$. 

Let $\gamma$ the primitive element of $\Lambda_R$ preserving $c_2$, $p\in\mathbb Z$ and 
$$c^p_3=\gamma^{\frac{p}{2}}\left(c^0_3\right)\ , \ \ \ z^p=\gamma^{\frac{p}{2}}\left(z\right)\ .$$
Observe that $z^p$ is the projection of $c^p_3$ on $c_2$ and that $d(z,z^p)=pR$. 

Obviously $(c_1,c_2,c_3^0)$ is $K$-acceptable since $d(c_1,c_3^0)\leq 2$ for $R$ large enough.

Observe now that the configuration of five geodesics given by 
$c^0_3,c_3^1,c_2,c_1,\gamma(c_1)$ converges to a pair of ideal triangles swished by 1. Thus, there exists a universal constant $K_0$ so that, for $R$-large enough
\begin{eqnarray}
d(c_1,c_3^1)&\leq& K_0\ ,\label{ineq:00000}\\	
d(c_1,c_3^{-1})&\leq& K_0\ \label{ineq:00001}\ . 
\end{eqnarray}
where the second inequality is obtained by a similar argument

 As a consequence for $K>2$,  $(c_1,c_2,c_3^p)$ is $K$-acceptable for  $p=+1,0,-1$ and $K\geq K_0$.
We want to show that these are the only ones. Let us write to simplify $c^{\pm}_3=c^{\pm1}_3$, $z^\pm=z^{\pm 1}$.
 \begin{itemize}
	\item let $D_2$ be the connected component of $\hh\setminus c_2$ not containing $c_1$,
	\item Let $\eta^\pm$ be the geodesic arc orthogonal to $c_2$ passing though $z^\pm$ and lying inside $D_2$. 
	\item Let $D^\pm$ be the convex set bounded by $\eta^\pm$ and the geodesic arc $[z^\pm,c_2(\pm\infty)]$,
\end{itemize} 
Observe that 
\begin{enumerate}
	\item for all $p>3$, $c^p_3\subset D^+$,
	\item The closest point $m$ to $c_1$ in $D^\pm$ lies on $c_2$ (geodesic arcs orthogonal to $\eta^\pm$ never intersect $c_2$ and $c_1$).
\end{enumerate}
It follows that for all $p>1$
$$
d(c^p_3, c_1)\geq d(D^\pm,c_1)=d(m,c_1) \geq \inf\left(\frac{1}{10}d(c_1,c_2)e^{\frac{3}{4}A}, \frac{1}{4}A-d(c_1,c_2)\right)
$$
where $A=d(m, y)$ and 
where the last inequality comes from Proposition \ref{pro:quad}.  
Observe that 
$$
d(m, y)\geq d(z^\pm,y)\geq d(z^\pm,z)-d(z,y)=R-1.
$$
Since  $d(c_1,c_2)\geq \frac{1}{2}e^{-\frac{1}{2}R}$, we obtain from the previous inequality that
$$
d(c^p_3, c_1) \geq d(D^\pm,c_1)\geq \inf\left(\frac{1}{1000}e^{\frac{R}{4}-1}, \frac{1}{4}R-2\right).
$$
Thus for $R$ large enough, 
$$
d(c^p_3, c_1) \geq d(D^\pm,c_1)\geq  \frac{1}{8}R\ .
$$
It follows that $(c_1,c_2,c_3^p)$ is not $K$-acceptable for $R$ large enough and $p>1$ (and a symmetric argument yields the case $p<1$).  This finishes the proof of the first point.

The second point follows from inequality \eqref{ineq:00000}, \eqref{ineq:00001}). The third point is  an immediate consequence of the previous construction and more precisely the restriction on $p$ appearing.

We use the notation of the previous paragraph to prove the last point. Let $c$ so that $d(c_1,c)\leq K$.
Since $d(D^\pm,c_1)\geq \frac{1}{8}R$, it follows that 
$$
c\not\subset D^+\sqcup D^-.
$$

Let furthermore $D_0$ (respectively $D_1$) be the hyperbolic half plane not containing $c_1$ bounded by $[c_3^-(+\infty),c_3^0(-\infty)]$ ( and respectively by $[c_3^0(+\infty),c_3^+(-\infty)]$). Observe that $d(D_0,c_2)\geq R$ and $d(D_1,c_2)\geq R$. Thus 
$$
c\not\subset D_0\sqcup D_1\ .
$$
Thus the result now follows from the examination of Figure \ref{fig:Kaccep}.

\subsection{Preliminary on accessible points}

The following proposition is obvious and summarizes some properties of accessible points

\begin{proposition}
	A $K$-good sequence of cuffs $\seq{\gamma}$ admits a unique accessible point which is also the Hausdorff limit of $\seq{\gamma}$ in the compactification of $\hh$, as well as the limit of the nested sequence of associated chords.
	\end{proposition}

We can explain our first construction of accessible points

\begin{proposition}\label{pro:access} There exists a function $\alpha(R)$ converging to zero as $R$ goes to infinity with the following property.
Given $K$ there exists $R_0$ so that for all $R>R_0$ the following holds: let $(c_1,c_2)$ be a $K$-acceptable pair, let $a$ be an extremity at infinity of $c_2$, Then there exists an accessible point $\beta$ in $\partial_\infty\hh$,  so that for all $x$ on $c_1$,
$$
d_x(\beta,a)\leq\alpha(R).
$$	
\end{proposition}
\begin{proof}
It is enough to prove this inequality whenever $x$ is the projection of $a$ on $c_1$. Let us consider the $K$-good sequence $\seq{c}$, starting with $c_1$, $c_2$, characterized by the following induction procedure:

First we choose an orientation on $c_2$ so that $a=c_2(+\infty)$, let also $b=c_1(-\infty)$ when $c_1$ inherits the orientation form $c_2$.

Assume $\{c_1,\ldots,c_p\}$ is defined. We choose the orientation on $c_i$ compatible with $c_2$. Then we choose $c_{p+1}\defeq\braket{c_{p-1},c_p}^+$, where the notation is from Proposition \ref{pro:Ktrip}. 

Let $\beta$ be the accessible point from this sequence. We will now show that 
$$
\lim_{R\to\infty}d_x(\beta,a)=0\ .
$$
This will prove the result setting $\alpha(R)\eqdef d_x(\beta,a)$.
Let us start by the following construction and observations
\begin{itemize}
\item let $z$ the projection of $c_3$ on $c_2$, 
\item let $\eta$ the geodesic arc orthogonal to $c_2$ starting at $z$ and intersecting $c_3$. 
\item $D$ be the convex set bounded by $\eta$ and $[y,a]$
 \end{itemize}
Observe that for all $p>3$, $c_p\subset D$.  It is therefore enough to prove that $D$ converges to $\{a\}$ whenever $R$ goes to infinity. Since 
$$
d(x,D)=d(x,z) \ .
$$
it will be enough to prove that $d(x,z)$ converges to $\infty$. 
 Then let $y$ be the projection of $c_1$ on $c_2$. We then know that 
 $$
A:=d(y,z)\geq R-1\ .
 $$ 
 It the follows from Proposition \ref{pro:quad} that
 $$
 d(x,z)\geq d(c_1,z)\geq \inf\left(\frac{1}{2}d(c_1,c_2)e^{\frac{3}{4}A}, \frac{1}{4}A-d(c_1,c_2)\right)\ .
 $$  
Since $d(c_1,c_2)\geq \frac{1}{2}e^{-\frac{R}{2}}$  for $R$ large enough,  it follows,  again for $R$ large enough, that
$$
d(x,z)\geq \frac{1}{8}R\ .
$$  
In particular $\lim_{R\to\infty}d(x,z)=\infty$. This concludes the proof.
\end{proof}

\subsection{Proof of the accessibility Lemma \ref{lem:access}}
 Let us work by contradiction. Then there exists $\beta>0$, and for all $R$,  an interval $I_R$ in $\partial_\infty\hh$ of visual length with respect to $H$  greater than $2\beta$ so that  $W^R_H$ does not intersect $I_R$. As a consequence, there exist a non empty closed interval $I$ of length $\beta$,  a subsequence $\seq{R}$ going to infinity so that $W_m(K):=W^{R_m}_{H}(K)$ never intersects $I$.

Let $\gamma$ be the geodesic connecting the extremity of $I$ and $D_0$ the closed geodesic half-plane whose boundary is $\gamma$ and boundary at infinity $I$. We may as well assume -- at the price of taking a smaller $\beta$ -- that $D_0$ does not intersect $H$. Let $x$ be the center of mass of $H$.

Let then $K$ be the distance form $\gamma$ to $x$.  Assume $m$ is large enough (that is $R_m$ is large enough) so that $W_m(K)=W_m(K_0)$. Let also $\eta$ be a geodesic inside $D_0$, so that $d(\eta,x)=2K$ and $d(\eta,\gamma)=K$. Let $D_1\subset D_0$ bounded by $\eta$. Let also $\zeta$ be the geodesic segment joining $x$ to $\eta$. This segment intersects finitely many cuffs and let $c$ be the closest cuff to $\eta$, non intersecting $D_0$. Let us consider all the cuffs 
$\{c_1,\ldots,c_p\}$ intersecting $\zeta$ between $x$ and $c=c_p$. Then $\{c_1,\ldots,c_p\}$ is a $K$-good sequence of cuffs.

We can now work out the contradiction. According to the last item of Proposition \ref{pro:Ktrip}, there exists a cuff $c_{p+1}$ so that $(c_{p-1},c_p,c_{p+1})$ is a $K$-acceptable triple and  either

\begin{itemize}
\item $\gamma$ intersects $c_{p+1}$, \label{case1}
	\item $\gamma$ is between $c_p$ and $c_{p+1}$. \label{case2}
\end{itemize}
Indeed, $c_{p+1}$ cannot be between $c_p$ and $\gamma$, by the construction of $c_p$.

In both cases, $c_{p+1}$ has an extremity -- call it $a$ -- inside $D_1$.
 Then according to Proposition \ref{pro:access}, we can find an accessible point with respect to a sequence starting with $(c_p,c_{p+1})$ -- hence starting with $(c_1,c_2)$ -- so that the corresponding accessible point $y$ satisfies for any $\epsilon$ and $R$ large enough
$$d_z(y,a)\leq  \epsilon,$$
where $z$ is the intersection of $c_p$ with $\zeta$. Hence, since $a$ lies in $D_0$, 
$$d_x(y,a)\leq \epsilon\ .$$
But this implies that $y\in D_0$ for $\epsilon$ small enough and thus the contradiction.

\section{Straight surfaces and limit maps}\label{sec:ss-lm}

We finally make the connection with the first part of the paper and the path of quasi tripods. Our starting object in this section will be a straight surface as discussed in the previous section, or more generally an {\em equivariant straight surface}: see Definition \ref{def:eq-sf}. Such an equivariant straight surface comes with a monodromy $\rho$ and our main result, Theorem \ref{theo:sull-straight}, shows that there exists a $\rho$ equivariant limit curve which is furthermore Sullivan. This implies the Anosov property and in particular the fact that the representation is faithful.

The proof involves introducing another object: unfolding a straight surface gives rise to a labeling of each hexagons of the fundamental tiling of the hyperbolic plane by tripods, satisfying some coherence relations  -- see Proposition \ref{pro:label-straight}.

Then we show that accessible points with respect to a given hexagon can be reached through nice paths of tripods. The labeling of hexagons gives  deformations of these paths into path of quasi-tripods. We can now use the Limit Point Theorem \ref{theo:exislimi} and thus  associate to an accessible point, a point in $\gp$: the limit point of the sequence of quasi-tripods. 

Using finally the Improvement Theorem \ref{theo:boot} and the explicit control on limit points in Theorem \ref{theo:exislimi}, we show that we can define an actual Sullivan limit map.

\subsection{Equivariant straight surfaces}\label{sec:fundgroup}
We extend the definition of straight surfaces (which require a discrete subgroup of $\G$) to that of an equivariant straight surface that you may think as of a ``local system'' in our setting, similar in spirit to the definition of  positive representations in \cite{Fock:2006a}.

Recall that a almost closing pair of pants for  $\G$ is a quintuple $T=(\alpha,\beta,\gamma, T_0,T_1)$ so that $\alpha,\beta,\gamma$ are $\ms P$-loxodromic elements in $\G$ and $T_0,T_1$ are tripods, satisfying the conditions of Definition \ref{def:almost-closing}. We denoted by $P^\pm\seR$ the space of $(\eR,\pm R)$-almost closing pair of pants.

Then we defined if $T=(\alpha,\beta,\gamma,T_0,T_1)$ is an $(\eR,R)$-almost closing pair of pants, 
$$
\bpsi(T)=\Psi(T_0,\alpha^+,\alpha^-)\ .
$$
where $\Psi$ is the foot map for quasi-tripods defined in Definition \ref{def:footmap}. We now define

\begin{definition}{\sc[Configuration spaces and gluing]}\label{def:conf-glu}
	\begin{enumerate}
\item the {\em configuration space of  pair of pants  }is defined as $\mc P^\pm\seR\defeq{\ms G}\backslash P^\pm\seR$\ .
  \item the {\em configuration space of  gluing} is
   $\mathcal Z\seR\defeq {\ms G}\backslash Z\seR$ where $Z\seR$ is the set of pairs $(T^+,T^-)\in P^+\seR\times 
  P^-\seR$ so that
  \begin{eqnarray*}
  T^\pm=(\alpha_\pm,\beta_\pm,\gamma_\pm, T_0^\pm,T_1^\pm)\ ,\ \ 
   \alpha_+=(\alpha_-)^{-1}\ ,\  \ d(\bpsi (T^+), \KMT\bpsi (T^-))\leq\eR\ . \end{eqnarray*}

 Observe that we have obvious $\G$-equivariant projections
$\pi^\pm: Z\seR\mapsto P^\pm\seR$, and also denote by $\pi^\pm$ the resulting projection $\mathcal Z\seR\mapsto \mathcal P^\pm\seR$.
An element of $\mathcal Z_{\epsilon,R}$ is called a {\em gluing}.
\item A {\em perfect gluing} based at an element $q$ of $\mathcal P^\pm\seR$, is given by the class of a pair $(w, \varphi_1 \bJ_\alpha \bI_0 w)$, where  
$w=(\alpha,\beta,\gamma,t,s)$ is a representative in $P^\pm\seR$  of the class   $q$  and $\bJ_\alpha$
is the reflexion of axis $\alpha$ defined in the beginning of paragraph \ref{def:refa}.
	\end{enumerate}
\end{definition}
With this definition at hand, we can now introduce the main object of this section:
\begin{definition}{\sc [Equivariant straight surfaces]}\label{def:eq-sf}
 Let $\epsilon$ and $R$ be positive numbers.
An $(\epsilon,R)$ {\em  equivariant straight surface} \index{equivariant straight surface} is a pair $(\mc R,Z)$ where $\mc R$ is a bipartite trivalent ribbon graph whose set of vertices is $V^-\sqcup V^-$, so that 
\begin{enumerate}
\item Every edge $e$, is labeled by an element $Z(e)$ of $\mathcal Z\seR$. For convenience, we define the corresponding label of flag 
	\begin{equation}
		W(e^\pm,e)\defeq
		\pi^\pm Z(e)\in\mathcal P^\pm\seR .\label{def:KMg-eq}
\end{equation}By an abuse of notation, we will talk about equivariant straight surfaces as triples
$(\mathcal R, Z,W)$ even though $W$ is redundant.
    \item The labeling map from the link of a vertex $v$ is equivariant with respect to the order 3 symmetries:
     $$W(\omega_v(v,e))=\omega(W(v,e))\ .$$
\end{enumerate}
 \end{definition}

 Given a discrete subgroup $\Gamma$, $\epsilon$ small enough and then  $R$ large enough, a straight surface for $\Gamma$ gives rise to an equivariant straight surface, whose underlying graph is finite.
 
Observe that, given a bipartite trivalent graph  $\mathcal R$ there is just one $(0,R)$-equivariant straight surface, that we call the {\em perfect surface} for $\mathcal R$.

\subsubsection{Monodromy of an  equivariant straight surface}
The {\em fundamental group} (as a graph of groups) $\pi_1(\Sigma)$ of $\Sigma=(\mathcal R,Z,W)$ is described as follows.

First let $\mathcal R^u$ be the universal cover of the trivalent graph  $\mathcal R$ and $\pi_1(\Sigma^u)$ be the group \begin{enumerate}
	\item with one  generator for every oriented edge, so that the element associated to the opposite edge is the inverse. 
	\item one relation for every vertex: the product of the three generators corresponding to the edges is one, using the orientation at each vertex.
\end{enumerate}  
Observe that the fundamental group of $\mathcal R$ acts by automorphisms on $\pi_1(\Sigma^u)$. We now define $$\pi_1(\Sigma)\defeq\pi_1(\Sigma^u)\rtimes \pi_1(\mathcal R)\ .$$

Then 
\begin{proposition} When the underlying graph is finite, 
	the group $\pi_1(\Sigma)$ is isomorphic to the fundamental group of a surface whose Euler characteristic is the (opposite) of  the number of vertices of $\mathcal R$.
	\end{proposition}

Let denote by $[(v,e)]$ the flag in $\mathcal R$ which is the projection of the flag $(v,e)$ in $\mathcal R^u$. The following follows at once from the fact that $\ms G$ acts freely on the space of almost closing  pair of pants.

\begin{proposition}
	There exists a map $Z^u$ from the set of oriented of edges  of $\mathcal R^u$ into $Z^\pm\seR$, a map $W^u$ from the set of flags edges of $\mathcal R^u$ with values in $P^\pm\seR$,  so that \begin{eqnarray}
		[W^u(v,e)]=W([v,e])&,&\pi^\pm Z^u(e)=W^u(e^\pm,e)\ .
		\end{eqnarray}
	Moreover $(W^u, Z^u)$ is unique up to the action of $\G$. 
\end{proposition}

A pair $(\mathcal R^u, Z^u)$ is called a {\em lift} of $(\mathcal R,Z)$
As a corollary of this construction we obtain from $W^u$ a representation $\rho^u$ of $\pi_1(\Sigma^u)$ in $\G$, where the image by $\rho^u$ of the element represented by the flag $(v,e)$ is $\alpha$, when $W^u(v,e)=(\alpha,\beta,\gamma,T_0,T_1)$. Moreover by uniqueness of $W^u$ up to the action of $\ms G$, we obtain also a representation $\rho_0$ of $\pi_1(\mathcal R)$ into $\G$ so that if $a\in \pi_1(\Sigma^u)$, $b\in \pi_1(\mathcal R)$ then 
$$
\rho^u(b\cdotp a \cdotp b^{-1})=\rho_0(b)\cdotp \rho^u(a)\cdotp\rho_0(b^{-1})\ .
$$
\begin{definition}{\sc [Monodromy--Cuff elements]}
	The {\em monodromy} of $\Sigma=(\mathcal R,Z,W)$ is the unique  morphism $\rho$ from $\pi_1(\Sigma)$ to   $\G$  extending both $\rho^u$ and $\rho_0$. 
	The {\em cuff limit map}\index{Cuff limit map} is the map $\xi^\prime$ which for every cuff element $a$,  associates to the attractive point $a^+$ of $a$, the attracting fixed point $\xi^\prime(a^+)\defeq\rho(a)^+$ in $\gp$ of the $\ms P$-loxodromic element $\rho(a)$. 
\end{definition}

\subsection{Doubling and deforming equivariant straight surfaces}

Let $\mathcal R$ be a bipartite trivalent tree, $v$ a vertex in $\mathcal R$, $N$ an integer. The  {\em $N$-doubling graph at $v$}   $\mathcal R^{(2)}$ is the finite bipartite trivalent graph obtained by the following procedure: We consider the ball $B$of radius $N$ based at a lift of $v$ in the universal cover of $\mathcal R$, then $\mathcal R^{(2)}$ is the graph obtained by taking two copies of $B^0$ and $B^1$ and adding two edges between the identified extreme points of $B$. Moreover the cyclic order on  the edges   of $B^1$ is reversed from the cyclic order for $B^1$.


Let now $\Sigma=(\mathcal R,Z)$ be an $(\epsilon,R)$-equivariant straight surface  over a tree and $v$ a vertex of $\mathcal R$. The  {\em $N$-doubling equivariant straight surface  at $v$} is the surface $\Sigma{^(2)}=(\mathcal R^{(2)},Z^{(2)}$ whose underlying graph is the $N$-doubling graph at $v$. Let us now describe the labeling $Z^{(2)}$. Let us denote by $F$ the bijection between $B^0(v,N)$ and $B^1(v,N)$, and $\pi$ the projection from $B^0(v,N)$ to $\mathcal R$.
\begin{enumerate}
	\item If $w$ is a vertex of $B^0(N,v)$, we define 
	$$Z^{(2)}(w)=Z(\pi(w)\ .$$
	\item  If $w$ is a vertex of $B^1(N,v)$, we define 
	$$Z^{(2)}(w)=\bI_0 Z(\pi(F(w)))\ .$$
	\item if $e$ is an edge of $B^0(N,v)$, 
we define 
	$$Z^{(2)}(e)=Z(\pi(e)\ .$$
		\item if $e$ is an edge of $B^1(N,v)$, 
we define 
	$$Z^{(2)}(e)=Z(\pi(e)\ .$$
	\item if $e$ is an edge joining $w$ in $B^0$ to $F(w)$ in $B^1$,  so that $e=\omega(f)$ where $f$ is an edge joining $w$ to a point in $B^0$, we define  $Z^{(2)}(e)$ to be the perfect gluing based at $\omega W(w,f)$,
	\item Finally, if  if $e$ is an edge joining $w$ in $B^0$ to $F(w)$ in $B^1$,  so that $e=\omega^2(f)$ where $f$ is an edge joining $w$ to a point in $B^0$, we define  $Z^{(2)}(e)$ to be the perfect gluing based at $\omega^2 W(w,f)$.
\end{enumerate}

We may deform equivariant straight surfaces. Let us say a family $\Sigma_t=(\mathcal R,Z_t,W_t)$, with $t\in[0,1]$, of $(\epsilon,R)$-equivariant straight surface
is {\em continuous} if ,  $W_t$ is continuous in $t$. The corresponding family of representations is then continuous as well. Our main result in this section is the following

\begin{proposition}\label{pro:deflabel}{\sc[Deforming the double]}
	There exists $\epsilon_0$ and $R_0$ and a constant $\rm C$ only depending on $\G$, so that if $\epsilon<\epsilon_0$ and $R>R_0$ the following holds:  
	Let  $\Sigma=(\mathcal R,Z)$ be an  equivariant $(\epsilon,R)$-straight surface, whose underlying graph is a tree, , let $v$ be a vertex of $\mathcal R$ and $N$ a positive  integer. 
	
	Then there exists a continuous  family $\Sigma^{(2)}_t$ of  $(C\epsilon,R)$-equivariant straight surfaces, with $t\in[0,1]$, so that  $\Sigma^{(2)}_1$ is the $(v,N)$-doubling of $\Sigma$  and $\Sigma^{(2)}_0$ is the perfect surface  for $\mathcal R^{(2)}$.
\end{proposition}

\begin{proof} Thanks to the doubling procedure, it is equivalent to show that we can have a deformation of any   labeling of the ball of radius $N$ in a trivalent tree. 

In this proof $B_i$ will denote constants only depending on $\ms G$.

We first  prove that the fibers of the projection
$$
\pi:=(\pi^+,\pi^-): \ \ {\mc Z}\seR\to {\mc P}^+\seR\times {\mc P}^-\seR \ .
$$
are connected for $\epsilon$ small enough and $R$ large enough. These fibers are described in the way: given $p=(p^0,p^1)\in P^+\times P^-$, where $p^*=(\alpha_*,\beta_*,\gamma_*, \tau^*_0,\tau^*_1)$, then there exist a unique element $g\in \ms Z_{\ms G}(\alpha)$ so that 
 $$
\bpsi(\tau^0_0,\alpha_0^+,\alpha_0^-)=g \KMT  \bpsi(\tau^1_0,\alpha_1^+,\alpha^-_1)\ .
 $$
 This element $g$ which is uniquely defined will be referred as the {\em gluing parameter}.
 
 Observe now that $g$ is $B_1\eR$-close to the identity with respect to the metric $d_{\tau^+_0}$. This follows from the definition of $Z\seR$ and  assertion \eqref{ineq:contrdtaud} in Proposition \ref{A-B}. 
 
 Our first step is to deform the gluing parameters to the identity so that all gluing are perfect.
 
 Let us use the tripod $\tau^+_0$ as an identification of $\ms G$ with $\ms G_0$. Let then $g_0=\tau^+_0(g)$ and $\alpha_0=\tau_0^+(\alpha)$. Then $g_0$ is $B_2\eR$ close to the identity with respect to $d_0$.   Thus  write for $\epsilon$ small enough,  $g=\exp(u)$ with
  $u\in\mk l_{0}$ or norm less than $B_3\eR$. Moreover $u$ is the unique such vector of norm less than $B_4$. We now prove that $u\in\mk z_\G(\alpha_0)$. Observe that 
$$
\exp(\operatorname{ad}(\alpha_0)\cdotp u)=\alpha_0\cdotp g_0 \cdotp\alpha_0^{-1}=g_0\ .
$$
By assertion \eqref{eq:bdrqf1},  $\alpha_0=\exp(2R\cdotp a_0)h$ where $h$ is  $\bM\eR$ close to the identity for some constant $\bM$. Thus for $\epsilon$ small, the linear operator  $\operatorname{ad}(\alpha_0)$ -- acting on $\mk l_0$ -- is close to the identity and has norm less than 2. Thus   
$$\Vert \operatorname{ad}(\alpha_0)\cdotp u\Vert \leq B_5\eR\ .$$
It follows by uniqueness of $u$ that  $\operatorname{ad}(\alpha_0)\cdotp u=u$. Thus  $u\in\mk z_{\ms G}(\alpha_0)$. Then we can deform $g$ to the identity through elements $B_6\eR$-close to identity with respect to $d_{\tau_0^+}$.

Then as a first step of our deformation, we deform each gluing parameter for every edge  to the identity. 

Observe now that an equivariant straight surface
with trivial gluing parameters is the same thing  as  a labeling of each vertex with an element of $\mc P_{\seR}$ with the constraints that boundary loops corresponding to opposite edges are conjugate. More formally, if $(v,e)$ is a flag corresponding to the oriented edge $e$ in the graph, and $\bar e$ is the edge with the opposite orientation, if $W(v,e)=(\alpha,\beta,\gamma,\tau_0,\tau_1)$ and we write $\alpha(e)=[\alpha]$, the constraint is that $\alpha(e)=\alpha(\bar e)$.
From now on, we keep the gluing parameters trivial.

Let us describe the  second step of the deformation. Let $V=(T,S_0,S_1,S_2)$ be the pair of pants labeling $v$ with boundary loops  $(\alpha_0,\beta_0,\gamma_0)$.  Using Theorem \ref{theo:defo-pant}, we consider a deformation  of pair of pants $V_t=(T^t,S^t_0,S^t_1,S^t_2)$ with boundary loops  $(\alpha_0,\beta_0,\gamma_0)$ to an $R$-perfect pair of pants; observe that the conjugacy classes $(\alpha_t,\beta_t,\gamma_t)$ of the boundaries change. However, we may write that $\alpha_t=g_t\alpha_0 k_t g_t^{-1}$ where $k_t$ belongs to ${\ms L}_{\alpha_0}$ and the length of the corresponding curves is small with respect to .

For the third step, let $w$ be a vertex at distance 1 from $v$ whose common boundary with $v$ is --say -- $\alpha$. We use Theorem \ref{theo:pantconn} to deform the pair of pants associated to $w$, following the deformation of the (unique) common boundary between $v$ and $w$, while keeping the other boundary loops of $w$ in the same conjugacy class. As a result the ``inner" boundary loop of $w$ is $R$-perfect.

As a fourth step, we deform the pair of pants labeling $w$ to a perfect pair of pants keeping the inner boundary loop of $w$ $R$-perfect.

Then we repeat inductively this procedure, namely step 3 and 4.

We have thus deformed our doubled  equivariant straight surfaces to an equivariant straight surface with perfect pair of pants for each vertex and perfect gluing parameters, or in other words a perfect surface.
 \end{proof}

\subsection{Main result}
Our main result is the  following  theorem that shows the existence of Sullivan curves.

\begin{theorem}{\sc [Sullivan and straight surfaces]}\label{theo:sull-straight}
We assume $\mk s$ has a compact centralizer. 

For any positive $\zeta$, there exists positive numbers $\epsilon_0$ so that for $\epsilon<\epsilon_0$,  there exists $R_0$ so that if $R>R_0$, if $\Sigma$ is an $(\epsilon,R)$ equivariant straight surface with monodromy $\rho$ and cuff limit map $\xi^\prime$, then there exists a unique $\rho$-equivariant $\zeta$-Sullivan map $\xi$ from $\partial_\infty\pi_1(S_R)$ to $\gp$ extending $\xi^\prime$.
\end{theorem}

In this section, we will always precise when we use the hypothesis that $\mk s$ has a compact centralizer.

\subsection{Hexagons and tripods}

We need to connect our notion of equivariant straight surfaces to the picture of tiling by hexagons.

\subsubsection{labeling hexagons by tripods}

Let us consider $\Sigma_0$ the perfect surface for $\mathcal R$, that is the unique $(0,R)$-straight surface of the form $(\mathcal R, Z_0)$. Gluing perfect $R$-pair of pants associated to the vertices of $\mathcal R$ along sides corresponding to edges of $\mathcal R$ by an $1$-swish, we obtain a covering  $S$ of the perfect surface $S_R$.
 We now consider $\rho$ as a representation of the cuff group $\Lambda$ which is so that $\Lambda\backslash \hh=S$.
  
 We recall (see paragraph \ref{sec:perfect}) that conversely $\mathcal R$ is obtained  as the adjacency graph of the tiling  of $S$ by (let us say) white hexagons.

Taking the universal cover of this perfect surface, one obtains a map 
 $\pi$ from the set of tiling hexagons to the flags of $\mathcal R$,  so that  $\pi(\Sc(H))=\pi(H)$, if $\pi(\Op(H))$ is the opposite flag to $\pi(H)$.

\begin{proposition}{\sc[Straight surfaces and equivariant labeling]}\label{pro:label-straight}
	Let $\Sigma=(\mathcal R,Z,W)$ be an equivariant  $(\epsilon,R)$-straight surface, with monodromy $\rho$ and cuff limit map $\xi^\prime$. Then there exists a labeling $\tau$ of tiling hexagons by tripods so that \begin{enumerate}
	\item $\tau(a,b,c)=\omega(\tau(b,c,a))$
	\item If $H=(a,b,c)$, then $P(H)\defeq (\tau(H),\tau(\Sc(H)),\rho(a),\rho(b),\rho(c))$ is an $(\epsilon,R)$-almost closing pair of pants,   
		\item for a white hexagon $W(\pi(H))=[P(H)]$,
		\item for all $\gamma\in\Lambda$, $\tau(\gamma(H))=\rho(\gamma)\cdotp\tau(H)$.
	\end{enumerate}
\end{proposition}

We will refer to $\tau$, $P$ as {\em equivariant labelings} associated to the straight surface $\Sigma$.

\begin{proof}  From the definition of $\Sigma=(\mathcal R, Z, W)$ we have a map from the set of white hexagons to ${\mathcal P}^+\seR$ given by $H\mapsto W(\pi(H))$.

We are now going to lift $W\circ\pi$ to a map $P$ with values in $P^\pm\seR$: Let us choose a white hexagon $H_0$ and fix a lift  $P(H_0)=(\alpha_0,\beta_0,\gamma_0,\tau(H_0),\tau_1(H_0))$ of $W\circ\pi$. For any white hexagon $H$, let us lift $W\circ\pi(H)$ to   $P(H)=(\alpha_H,\beta_H,\Gamma_H,\tau(H),\tau_1(H))$ by using the following rules. 
\begin{enumerate}
	\item If $H'=\Op(H)$ then $P(H')$ is uniquely defined from $P(H)$ so that $(P(H),P(H'))$ is a lift of $Z(e)$ in $Z\seR$ where $e$ is the edge in $\mc R$ associated to the pair $(H,H')$.
	\item If $H'=\Sc^2(H)$, then $P(H')=\alpha_H P(H)$.
\end{enumerate} 
We leave to the reader to check that these rules are coherent. 
We finally choose a labeling of the black hexagons using the following rule: if $H'=\Sc(H)$ and $H$ is labeled by  $P(H)=(\alpha_H,\beta_H,\Gamma_H,\tau(H),\tau_1(H))$  then the labeling of $H'$ is given by 	$$P(H')=(\alpha_H,\beta_H^{-1}\gamma_H\beta_H, \beta_H, \tau_1(H), \alpha_H\tau(H))$$

Our label by tripods is finally given by the maps $\tau:H\mapsto\tau(H)$, where $P(H)=(\alpha_H,\beta_h,\gamma_H,\tau(H),\tau_1(H))$. \end{proof}

\subsection{A first step: extending to accessible points}
Our first step will not use the assumption that $\mk s$ has a compact centralizer and will be used to show a weaker version of the surface subgroup theorem in that context.

Let us denote by $W^R_H$ the set of accessible points from a tiling hexagon $H$ and let us define the set of accessible points as
$$
 W^R:=\bigcup_{H}W^R_H
$$ 
 the union set of of all accessible points with respect to any hexagon. Observe that $W^R$ is $\pi_1(S)$ invariant and thus dense. 

Our main result in this paragraph is the next lemma that unlike Theorem \ref{theo:sull-straight} will not use the compact stabilizer hypothesis.

\begin{lemma}{\sc [Extension]}\label{lem:wlimit}
For any positive $\zeta$, there exist positive numbers $(\epsilon_0,R_0)$ so that for $\epsilon<\epsilon_0$,there exists $R_\epsilon$ so that for $R>R_\epsilon$, the following holds.

Let $\Sigma$ is an $(\epsilon,R)$ equivariant straight surface with monodromy $\rho$ and cuff limit map $\xi^\prime$.  Then there exist a unique $\rho$-equivariant  map $\xi$ from the set of accessible points  $W^R$ to $\gp$, so that
if $\seq{c}$  is a   nested sequence of cuffs  converging to an accessible point $y\in W^R_H$, then
 $$
\lim_{m\to\infty}(\xi^\prime(c_m^\pm))=\xi(y)
$$
Moreover, if $\eta$  is the circle map associated to $\tau_0=\tau(H)$, then
 for any $\tau$ coplanar to $\tau(H)$ so that $\tau^\pm=\tau_0^\pm$\begin{eqnarray*}
 	d_\tau(\xi(y),\eta(y))\leq\zeta\ .\label{ineq:limit-sul}
 \end{eqnarray*}
\end{lemma}

We furthermore show that the dependence of $\xi$ on the straight surface is continuous

\begin{corollary}\label{coro:wlimit}
	Let $\{\Sigma_t\}_{t\in\mathbb R}$ be a continuous family of $(\epsilon,R)$--equivariant straight surfaces, and $\{\xi_t\}_{t\in\mathbb R}$ the family of maps produced as above, then for every accessible $z$, the map $\xi_t(z)$ is continuous as a function of $t$.
\end{corollary}
 
We first construct a sequence of quasi-tripods associated to an accessible point and an equivariant labeling, then show that this sequence of quasi-tripods converges and complete the proof of the Extension Lemma \ref{lem:wlimit}

\subsubsection{A sequence of quasi tripods for an accessible point}

 Let $\Sigma$ be an equivariant straight surface with monodromy $\rho$ and cuff limit map $\xi'$, let $\tau$ be an equivariant labeling obtained by Proposition \ref{pro:label-straight}.

Given $K$, choose $R_0$ so that Proposition \ref{pro:Ktrip} holds and $R>R_0$
Let $y$ be an accessible point which is the limit of a sequence of cuffs $\seq{c}$.

As a first step we associate to $\seq{c}$ a sequence of coplanar tripods $\seq{T}$ {\em associated} to a $K$-good sequence of cuffs $\seq{c}$: first we orient each cuff so that $c_{m+1}$ is on the right of $c_m$, then  we associate to every $K$-acceptable pair $(c_m,c_{m+1})$ the pair of tripods $(T_{2m-1},T_{2m})$ defined by
$$
T_{2m-1}=(c^-_m, c^+_m,c^-_{m+1}, ),\ \ T_{2m}=(c^-_{m+1}, c^+_{m}, c^+_{m+1},).
$$
Let then $A_m$ be the swish between $T_m$ and $T_{m+1}$.

Our second step is to associate our data a sequence of quasi-tripods. Recall that $c_m$, $c_{m+1}$ are the common edges of exactly two hexagons $H_{2m-1}=(c_{m+1},\overline{c_m},b_m)$ and $H_{2m}=(c_{m+1},d_m, \overline{c_{m}})=\Sc(H_{2m-1})$, where we denote by $\overline{c}$ the cuff $c$ with the opposite orientation. 

Let us consider the sequence $\seq{\theta}$ of quadruples given by $\theta_m=(\tau(H_m),\xi'(T_m))$. Then it follows by the second item of  Theorem \ref{theo:struct-pant} that for $R_\epsilon$ only depending on $\epsilon$ and $R>R_\epsilon$, $\theta_m$ is an $\bM_0(\eR)$-quasi-tripod, for some constant $\bM_0$ only depending on $\ms G$.

We can now prove
\begin{proposition} \label{pro:accQT} There exists a positive constant $\bM_1$ only depending on $\ms G$ so that
	the sequence $\seq{\theta}$ is an $(\seq{A},\bM_1\eR)$ swished sequence of quasi tripods whose model is $\seq{T}$. 
\end{proposition}

\begin{proof}
	Let us first consider the pair $(\theta_{2m-1},\theta_{2m})$,   From the definition, the $\eR$-quasi tripod 
	$$\beta_{2m-1}:=\left(\tau_{2m-1}, \xi'(c_{m+1}^-),\xi'(c_{m}^+),\xi'(b_m^-)\right)$$
is $(R,\eR)$ swished from $$\omega(\beta_{2m}):=
\left(\omega(\tau_{2m}), \xi'(c_{m}^+),\xi'(c_{m+1}^-),\xi'(d_m^-)\right)
	$$ for $m$ odd and $(-R,\eR)$ swished for $m$ even. Since 
 by construction,  \begin{eqnarray*}
		\beta_{2m}^\pm=\theta_{2m}^\pm\ , \ \
\omega(\beta_{2m-1})^\pm=\omega(\theta_{2m-1})^\pm\ , \ \
\dt{\beta}_m=\dt{\theta}_m\ . 
	\end{eqnarray*} it follows that $\theta_{2m}$ is $(R,\eR)$ swished from $\omega(\theta_{2m-1})$  for $m$ odd and $(-R,\eR)$ swished for $m$ even.
.Then since \begin{itemize}
	\item $\omega(T_{2m})$ is $2\eR$ close to  $t_{2m}=(c_{m}^+,c_{m+1}^-,d_m^-)$ by Proposition \ref{pro:hexa} and similarly
	\item $T_{2m-1}$ is $2\eR$ close to  $t_{2m-1}=(c_{m+1}^-,c_{m}^+,b_m^-)$, 
\end{itemize}
it follows that $A_m$ is $2\eR$ close to $R$, for $m$ odd and to $-R$ for $m$-even
Thus $\theta_{2m}$ is $(A_m,\bM_2\eR)$ swished from $\omega(\theta_{2m-1})$ for some constant $\bM_2$.
	
Let us consider now the pair $(\theta_{2m-1},\theta_{2m})$. Since $(c_{m-1},c_m,c_{m+1})$ is a $K$-acceptable triple, it follows by item \ref{it:hexatrip} of Proposition \ref{pro:Ktrip} that 
	$$
	H_{2m+1}=\eta_m\Op(H_{2m-1})\ ,
	$$
	where $\eta_m=\gamma_m^p$, $\gamma_m$ is the cuff element associated to $c_m$ and $p\in\{-1,0,1\}$.

By the definition of a labeling, $\eta_m^{-1}(\theta_{2m})$ is $(1,\eR)$-swished from $\theta_{2m-1}$. By construction (see Proposition \ref{pro:label-straight})
 $P(H_{2m-1})$ is  an $(\eR,R)$-almost closing pair of pants associated to  and thus by the last item of Theorem \ref{theo:struct-pant}
$$
d(\eta(\theta_{2m}),\varphi_{2R}(\theta_{2m}))\leq \bM_3\eR\ ,
$$
for some constant $\bM_3$ only depending on $\ms G$.

It follows that $\theta_{2m}$ is $(1+pR,\bM_4\eR)$-swished from $\theta_{2m-1}$ for a constant $\bM_4$ only depending on $\ms G$. Since $A_m=1+pR$, the quasi-tripod $\theta_{2m}$ is $(A_m,\bM_4\eR)$-swished from $\theta_{2m-1}$ for a constant $\bM_4$ only depending on $\ms G$.

This concludes the proof of the proposition. \end{proof}

\subsubsection{Proof of Lemma \ref{lem:wlimit} and its corollary}

We first prove the following result which is the key argument in the proof.

\begin{proposition}{\sc [Extension]}\label{lem:wlimit1}
For any positive $\zeta$ and $K$, there exists positive numbers $\epsilon_0$, $\bQ<1$, $\beta$, $L$,  so that for $\epsilon<\epsilon_0$,  there exists $R_\epsilon$ so that for  $R>R_\epsilon$, \begin{itemize}
	\item if $\Sigma$ is an $(\epsilon,R)$ straight surface,
	\item  if $\seq{c}$  is a   nested sequence of cuffs  converging to an accessible point $y$ with respect to a tiling hexagon $H$ for $\Sigma$, 
	\end{itemize}  
Then $\{\xi(c_m^\pm)\}_{\in\mathbb N}$ converges to a point $Y$ so that  for any $\tau$ coplanar to $\tau(H)$ so that $\tau^\pm=\tau_0^\pm$ and $m>L$
\begin{eqnarray}
	d_\tau(Y, \xi(c_m^\pm))\leq {\bQ}^m\beta \label{ineq:wlimi1}
\end{eqnarray}
Moreover, if $\eta$  is the circle map associated to $\tau_0=\tau(H)$, then
 for any $\tau$ coplanar to $\tau(H)$ so that $\tau^\pm=\tau_0^\pm$\begin{eqnarray}
 	d_\tau(Y,\eta(y))\leq\zeta\ .\label{ineq:limit-sul1}
 \end{eqnarray}
\end{proposition}

\begin{proof}
 Let $\zeta$ be a positive constant. 
 The sequence of tripods $\seq{T}$ is a $2KR$-sequence of tripods by Corollary \ref{cor:KMbd}.  From Proposition \ref{pro:accQT}, it follows that $\seq{\theta}$ is a $(KR,\eR)$-deformed sequence of quasi tripods. In particular, using Theorem \ref{theo:exislimi} with $\beta=\zeta$, $\{\xi(c_m^+)\}_{m\in\mathbb N}$ and $\{\xi(c_m^-)\}_{m\in\mathbb N}$ both converge to a point $y(\theta)=:Y$ in $\gp$. 
 
 Then inequality \eqref{ineq:wlimi1} is a consequence of \eqref{ineq:limit-dist00}.

Since $y(\tau)=\eta(y)$, inequality \eqref{ineq:limit-sul1} also follows from Theorem \ref{theo:exislimi}.
  \end{proof} 

The proof of  Lemma \ref{lem:wlimit} now follows immediately. The proof of Corollary \ref{coro:wlimit} follows from that fact thanks to inequality \eqref{ineq:wlimi1} the convergence of $\{\xi(\theta^j_m)\}_{m\in\mathbb N}$ is uniform.

\subsection{Proof of  Theorem \ref{theo:sull-straight}} We now make use of the compact stabilizer hypothesis using in particular the  Improvement Theorem \ref{theo:boot}.

It is enough (by eventually passing to the universal conver) to prove the theorem when the underlying graph of $\Sigma$ is a tree and that is what we do now.

Let us start with an  observation. Let $\tau$ be any tripod in $\hh$. Since the diameter of the hyperbolic surfaces $S_R$ is bounded independently of $R$ (Lemma \ref{lem:diamSR}). It follows that there exists some constant $C_0$, so that  given any tripod $\tau$, we can find a tiling hexagon $H$ so that 
\begin{eqnarray}
	d(\tau,\tau_H)\leq C_0\ ,\label{ineq:tth}
\end{eqnarray}
where $\tau_H$ is a  admissible tripod in $\hh$ for $H$. It follows that there exists a universal constant $C_1$ so that  for any extended circle map $\eta$
\begin{eqnarray}
	d_{\eta(\tau)}\leq  C_1\cdotp d_{\eta(\tau_H)}\ ,
\end{eqnarray}

Given a positive number  $\zeta$, let us let fix   fix $\epsilon$, and $R_0=R_\epsilon$ so that Lemma \ref{lem:wlimit} holds.  Let then $R>R_0$

Let  $\Sigma=(\mathcal R,Z)$ be an $(\epsilon,R)$ equivariant straight surface  with monodromy $\rho$ and cuff limit map $\xi^\prime$ so that $\mathcal R$ is a tree.  

Then according to Proposition \ref{pro:deflabel}, for any vertex $v$ in $\mathcal R$ and integer $N$,  we can find a continuous family $\{\Sigma^{(2)}_t\}_{t\in[0,1]}$ of $(\epsilon,R)$ equivariant labeling under $\{\rho^{(2)}_t\}_{t\in[0,1]}$ deforming the $(v,N)$ double $\Sigma^{(2)}$ of $\Sigma$.

 It follows by the Extension Lemma \ref{lem:wlimit} and Corollary \ref{coro:wlimit} that we can find a continuous family $\{\xi^{(2)}_t\}_{t\in[0,1]}$ defined on the dense set of accessible points $W^R$ so that
\begin{enumerate}
	\item $\xi^{(2)}_t$ is equivariant under $\rho^{(2)}_t$,
	\item $\xi^{(2)}_0$ is a circle map,
	\item For any tiling hexagon $H$, for all $y$ in $W^R_H$\begin{eqnarray}
		d_{\tau_t}(\eta_t(y),\xi^{(2)}_t(y))\leq\frac{\zeta}{C_1}\ .\label{ineq:wsul}
		\end{eqnarray}
		where $\eta^H_t$ is the circle map so that 
		$\eta^H_t(\tau_H)=\tau_t:=\tau_t(H)$.
\end{enumerate}
Remark now that
\begin{enumerate}
	\item by Theorem \ref{lem:wlimit}, $\xi^{(2)}_t$ is attractively continuous: for all $y\in W$ which is the limit of  elements $c^+_m$, $\xi^{(2)}_t(y)$ is the limit of $\xi^{(2)}_t(c^+_m)$ as $m$ goes to infinity,  where $\xi^{(2)}_t(c^+_m)$ is the attractive element of the cuff element $\rho^{(2)}_t(c_m)$; Applying this for $t=1$, we get that $\xi^{(2)}$ extends the cuff maps $\xi^{\prime (2)}$.
	\item by the Accessibility Lemma \ref{lem:access}, $W^R_H$ is $a(R)$-dense, where $a(R)$ goes to zero when $R$ goes to $\infty$.
\end{enumerate}
We thus now choose $R_0$ so that for all $R$ greater than $R_0$, $a(R)<a_0$ where $a_0$ is given from $\zeta$ by Theorem \ref{theo:boot}.

Using the initial observation, we now have that for any tripod $\tau$, and any $t\in[0,1]$, we can find a circle map $\eta_t=\eta_t^H$ so that for any $y$ in some $a_0$-dense set 
\begin{eqnarray*}
	d_{\eta_t(\tau)}(\eta_t,\xi^{(2)}_t(y))\leq\zeta\ ,
\end{eqnarray*}
where we have used both inequalities \eqref{ineq:6110} and \eqref{ineq:wlimi1}. In other words, $\xi^{(2}_t$ is $(a_0,\zeta)$-Sullivan

We are now in a position to apply the Improvement Theorem \ref{theo:boot}. This shows that $\xi^{(2)}_t$ -- and in particular $\xi^{(2)}$-- is $2\zeta$-Sullivan. By construction $\xi^{(2)}$ extends $\xi^{\prime(2)}$. 

Remember now that the doubling construction depends on the choice of a parameter $N$ and we now write $\xi_N^{(2)}$ and  $\xi_N^{\prime(2)}$ to mark the dependency in $N$.

Let us consider the universal cover of both the original surface $\Sigma$ and the double   $\Sigma^{(2)}$. By construction, the labeling are identical on the large ball $B(v,N)$. This large ball corresponds to a free subgroup of $\pi_1(\Sigma_R)$ with limit set $\Lambda(N)$. This for any cuff element $c_m$ with end points in $\Lambda(N)$,
\begin{equation}
	\xi_N^{\prime(2)}(c_m^+)=\xi^{\prime}(c_m^+)\ .\label{eq:improv00}
\end{equation}
It follows that if $M>N$gluing
\begin{equation}
\left.\xi_M^{(2)}\right\vert_{\Lambda(N)}=\left.\xi_N^{(2)}\right\vert_{\Lambda(N)}\ .\label{eq:improv01}
\end{equation}
Recall now that all limit maps  $\xi_M^{(2)}$ being $\zeta$-Sullivan admits a modulus of continuity by Theorem \ref{theo:HolQuant}. We  may thus extract form the sequence $\{\xi^{(2)}_M\}_{M\in\mathbb N}$ a uniformly converging subsequence to a $\zeta$-Sullivan map $\xi$.

Let finally $\Lambda\defeq\bigcup_N\Lambda(N)$ and observe that $\Lambda$ is dense and contains the end points of all cuff elements. Then, the $\zeta$-Sullivan map $\xi$ coincide with $\xi^\prime$ on $\Lambda$ by equations \eqref{eq:improv00} and \eqref{eq:improv01}

This completes the proof of Theorem \ref{theo:sull-straight}.

\section{Wrap up: proof of the main results}
This section is just the  wrap up of the proof  of the main Theorems obtained by combing the various theorems obtained in this paper.

\begin{theorem}
Let $\ms G$ be a semisimple Lie group of Lie algebra $\mk g$ without compact factors. Let $\mk s=(a,x,y)$ be an $\sld$-triple in $\mk g$.  Assume that $\mk s$ satisfies the flip assumption and that
$\mk s$ has a compact centralizer. 

Let $\Gamma$ be a uniform lattice in $\ms G$. Let $\epsilon$ be a positive real number. Then there exists a closed hyperbolic surface $S_\epsilon$, a faithful $(\ms G,\ms P)$ Anosov representation $\rho_\epsilon$ of $\pi_1(S_\epsilon)$ in $\Gamma$, whose limit curve is $\epsilon$-Sullivan with respect to $\mk s$, where $\ms P$ is  the parabolic associated to $a$.

\end{theorem}

As a corollary, considering the case of the principal $\sld$ in a complex semisimple Lie group,  we obtain
\begin{theorem}
	Let $\ms G$ be a complex semisimple group, let $\Gamma$ be a uniform lattice in $\ms G$, then there exists a closed Anosov surface subgroup in $\Gamma$.
\end{theorem}

\begin{proof}
	From Theorem \ref{exisstraight}, for any positive $\epsilon$, there exists $R_0$, so that for any $R>R_0$, there exists an $(\epsilon,R)$- straight surface $\Sigma$ in $\Gamma$ associated to $\mk s$. This straight surface is equivariant under a representation $\rho$ of a surface group $\Gamma_0$ in $\Gamma$.

By Theorem \ref{theo:sull-straight}, for any $\zeta$, for  $\epsilon$ small enough, there exists $R_0$ so that for $R>R_0$,  an  $(\epsilon,R)$- straight surface equivariant under a representation $\rho$ of a surface group $\Gamma_0$ in $\Gamma$, is so that we can find a $\zeta$-Sullivan $\rho$-equivariant Sullivan map from $\partial_\infty\Gamma_0$ to $\gp$. By Theorem \ref{theo:sull-anos}, for $\zeta$-small enough the corresponding representation is Anosov and in particular faithful.
\end{proof}
\subsection{The case of the non compact stabilizer} In that context we obtain a less satisfying result. Recall that we denote by $c^+$ the attractive point in $\partial_\infty\pi_1(S)$ of a non trivial element $c$ of $\pi_1(S)$.

Let $\ms G_1,\ldots,\ms G_n$ be semisimple Lie groups without compact factors. Let $\ms G=\prod_{i=1}^n\ms G_i$ with  Lie algebra $\mk g$   Let $\Gamma$ a uniform lattice in $\ms G$ so that (up to finite cover) its projection on $\ms G_i$ is an irreducible lattice.  Let $(a,x,y)$ be an $\sld$-triple in $\mk g$ so that 
\begin{itemize}
	\item $\mk s$ satisfies the flip assumption,
	\item the projections on all factors $\mk g_i$ are non trivial,
\end{itemize}
Let $\ms P$ the parabolic associated to $a$. Let $\Gamma$ be a uniform lattice in $\ms G$. 

\begin{theorem}
Let $\epsilon$ be a positive real. Then there exists some $R$  and  \begin{itemize}
	\item  a faithful representation $\rho$ of $\Gamma_R=\pi_1(S_R)$ in $\Gamma$, so that the image of every cuff element of $\Gamma_R$ has an attractive fixed point in $\gp$.
	\item a  $\rho$-equivariant $\xi$ from $\partial_\infty\Gamma_R$ to $\gp$ so that
	\begin{itemize}
	\item For a cuff element $c$,  $\xi(c^+)$ is the attractive fixed point of $\rho(c)$,
	\item If $\seq{c}$ is a sequence of nested cuff elements so that $\{c^+_m\}_{m\in\mathbb N}$  converges to $y$, then $\{\xi(c^+_m)\}_{m\in\mathbb N}$ converges to $\xi(y)$.
\end{itemize}	
\end{itemize} 
\end{theorem}
\begin{proof} The proof runs as before except that we replace the use of the Theorem \ref{theo:sull-straight} by Lemma \ref{lem:wlimit}, from which we obtain the existence and properties of the application $\xi$ which is equivariant under a representation $\rho$. Let us now show that $\rho$ is injective. We already know that the image of any cuff element in $\Gamma_R$ is non trivial. Let $\gamma$ so that $\rho(\gamma)$ is the identity and assume by contradiction that $\gamma$ is not the identity and $\gamma^+$ its attracting point in $\partial_\infty \Gamma_R$ and $\gamma^-$. Let $c^+_m$ the end point of a cuff. Then $\xi(\gamma^n c^+_m)=\xi(c^+_m)$. It follows, by taking the limit when $n$ goes to infinity, that $\xi(c^+_m)=\xi(\gamma^+)$ since $c^+_m$ is different from $\gamma^-$. Thus $\xi$ would be constant but this is a contradiction: $\xi(c_m^+)$ is different from $\xi(c_m^-)$.
\end{proof}

\section{Appendix: L\'evy--Prokhorov distance}\label{P}

Let $\mu$ and $\nu$ be two finite measures of the same mass on a metric space $X$ with metric $d$. For any subset $A$ in $X$, let $A_\epsilon$ be its $\epsilon$-neighborhood. Then we define
$$
d_L(\mu,\nu)=\inf \{\epsilon >0 \mid \forall A\subset X,\ \nu(A_\epsilon)\geq  \mu(A)\}.
$$
This function $d_L$ is actually a distance (see \cite[Paragraph 3.3]{Kahn:2009wh}) related to both  the {\em L\'evy--Prokhorov distance} \index{Levy--Prokhorov} and the {\em Wasserstein-$\infty$ distance}. By a slight abuse of language, we call still call this distance the {\em L\'evy--Prokhorov distance}.

We want to prove the following result which is an extension of a result proved in \cite{Kahn:2009wh} for connected 2-dimensional tori. The proof uses different ideas. 
\begin{theorem}\label{prokho}
Let $X$ be a  manifold. Assume that a connected compact torus $\ms T$ -- with Haar measure $\nu$ -- of dimension $n$ acts freely on $X$ preserving a  a bi-invariant Riemannian metric $d$ and measure $\mu$.  Let $\phi$ be a positive function on $X$. Let $\overline\phi:=\int_\ms T\phi\circ g .{\rm d}\nu(g)$ be its $\ms T$-average. Assume that
$(1-\kappa){\overline \phi}\leq \phi \leq (1+\kappa){\overline \phi}$. Then 
$$
d_L(\phi.\mu,\overline{\phi}.\mu)\leq 4n. \kappa.\sup_{x\in X}\diam(\ms T.x)\ .
$$
\end{theorem}

\subsubsection{Elementary properties}
The  following  properties of the L\'evy--Prokhorov distance will be used in the proof.
\begin{proposition}\label{pro:mui}
Let $\{\mu_n\}_{n\in\mathbb N}$ and  $\{\nu_n\}_{n\in\mathbb N}$ be two families of measures so that $\mu=\sum_{n=1}^\infty$ and $\nu=\sum_{n=1}^\infty\nu_n$ are finite measures.  Assume that for all $i$, $d_L(\mu_i,\nu_i) < \epsilon$, then
$d_L(\mu,\nu) < \epsilon$. 
	\end{proposition} 
	\begin{proof} Assume $\epsilon > d_L(\mu_i,v_i)$.
Then for all $i$ and for all $A\subset X$,
	$\nu_i(A_\epsilon) \geq  \mu_i(A)$. Thus 
	$\nu(A_\epsilon) \geq  \mu(A)$. It follows that $\eta\geq d(\mu,\nu)$.	
	\end{proof}

\begin{proposition}\label{pro:approx}
	Let $\mu$ be a  finite measure on a compact metric space $X$.  Then for all positive $\epsilon$, there exists an atomic measure $\mu_\epsilon$   with finite support so that 
	$$
	d_L(f\mu,\mu_\epsilon)\leq \epsilon.
	$$
	If $\mu$ is invariant by a finite group $H$, then we may choose $f$ and $\mu_\epsilon$ invariant by $H$.
\end{proposition}

\begin{proof} One can find a  finite partition of $X$ by sets $U^i,\ldots,U^n$ together with a finite set of points $x_1,\ldots, x_n$ so that $x_i\in U_i\subset B(x_i,\epsilon)$. 

We then choose the atomic measure $\mu_\epsilon:=\sum_{i=1}^n \mu(U_i)\delta_{x_i}$, so that $\mu_\epsilon(U^i)=\mu(U^i)$. Let $A\subset X$ and $A^i=A\cap U^i$. Let $I$ be the set of $i$ so that $A^i$ is non empty, then for $i\in I$, 
$$
U^i\subset B(x_i,\epsilon)\subset A^i_{2\epsilon}\ .$$
Thus,
$$
A\subset \bigsqcup_{i\in I} U^i=\bigsqcup_{i\in I} \left(U^i\cap A^i_{2\epsilon}\right)\subset A_{2\epsilon}
$$
It follows that for all subset $A$, 
$$
\mu(A)\leq \mu\left(\bigsqcup_{i\in I} U^i\right)=\mu_\epsilon\left(\bigsqcup_{i\in I} U^i\right)=\mu_\epsilon\left(\bigsqcup_{i\in I}  \left(U^i\cap A^i_{2\epsilon}\right)\right)\leq\mu_{\epsilon}(A_{2\epsilon})\ .
$$
in particular $d(\mu,\mu_{\epsilon})\leq 2\epsilon$. 

To obtain the invariance by the finite group $H$, one just averages by $H$, using proposition \ref{pro:mui}.

\end{proof}

	\begin{proposition}\label{prokho1}
Let $f$ and $g$ be two maps from a measured space $(Y,\nu)$ to a metric space $X$. Assume that for all $y$ in $Y$,
$
d(f(y),g(y))\leq \kappa.
$
Then
$$
d_L(f_*\nu,g_*\nu)\leq\kappa.
$$
\end{proposition}
\begin{proof} Observe that by hypothesis, for any subset $B$ of $Y$,  $f(B)\subset\left( g(B)\right)_{\kappa}$.
Let $A$ be a subset of $X$, $C=f^{-1}(A)$ and $D=g^{-1}(A)$. Then $f(D) \subset A_\kappa$. It follows that 
$$f_*\mu(A_\kappa)\geq f_*\mu (f(D))=\mu(f^{-1}(f(D))\geq \mu(D)=g_*\mu(A). 
$$
The assertion follows.
\end{proof}

	\begin{proposition}\label{pro:prokhcont}
		Let $\pi$ be a $K$-Lipschitz  map from $X$ to $Y$. Let $\mu$ and $\nu$ be measures on $X$, then
		$$
		d_L(\pi_*(\mu),\pi_*(\nu))\leq K. d(\mu,\nu).
		$$
	\end{proposition}
	We will actually apply this proposition when $\pi:X\to Y$ is a finite covering.	\begin{proof} By renormalizing the distance, we can assume the map $\pi$ is contracting.
		Let $\epsilon\geq d(\mu,\nu)$. Let $B\subset Y$, observe that
	$\pi^{-1}(B)_\epsilon \subset\pi^{-1}(B_\epsilon)$. Then,
	$$
	\pi_*\mu (B_\epsilon)=\mu(\pi^{-1}(B_\epsilon))\geq \mu \left(\pi^{-1}(B)_\epsilon\right)\geq \nu (\pi^{-1}(B))=\pi_*\nu(B).
	$$
	Then by definition, $\epsilon\geq d(\pi_*(\mu),\pi_*(\nu))$ and 
	the result follows.
	\end{proof}
	 
\subsubsection{Some lemmas}

We need the following lemmas.
\begin{lemma}\label{prokho3}
Let $X$ be a metric space equipped some metric $d$. Let $\pi:X\to X_0$ be a fibration. Let $d_x$ be the restriction of $d$ to the fiber $\pi^{-1}\{x\}$.   Let $\nu$ and $\mu$ be two measures on $X$ so that $\pi_*\mu=\pi_*\nu=\lambda$.
For every $x$ in $X_0$, let $\mu_x$ --respectively $\nu_x$-- be the disintegrated measure on $\pi^{-1}(x)$  coming from $\mu$ and $\nu$ respectively.
Then
\begin{eqnarray}
d_L(\mu,\nu)\leq \sup_{x\in X_0}d_x(\mu_x,\nu_x).\label{propro}
\end{eqnarray}

\end{lemma}
\begin{proof}
Let $A$ be a subset of $X$ and $A^x:=A\cap\pi^{-1}\{x\}$.  Let $(A^x)_\kappa$ be the $\kappa$ neighborhood of $A^x$ in $\pi{-1}(x)$. By construction $(A^x)_\kappa\subset (A_\kappa)^x$. Thus, for any set $A$, if $\kappa\geq d_x(\nu_x,\mu_x)$ for all $x$, we have
\begin{eqnarray*}
\nu (A_\kappa)=\int_{X_0} \nu_x\left((A_\kappa)^x\right){\rm d}\lambda (x)
\geq\int_{X_0} \nu_x\left((A^x)_\kappa\right){\rm d}\lambda (x)
\geq\int_{X_0} \mu_x\left(A^x\right){\rm d}\lambda (x)
\geq\mu_0(A)\ .
\end{eqnarray*}
Thus, $\kappa\geq d(\mu,\nu)$. inequality  (\ref{propro}) follows. 
\end{proof}
\begin{lemma}\label{prokho2}
Let $\ms T^1$ be the connected compact torus of dimension $1$ equipped with a bi-invariant metric $d$ and  Haar measure $\mu$.  Let $\phi$ be a positive function on $\ms T^1$. Let $\overline\phi:= \int_{\ms T^1} \phi\circ g \ {\rm d}\mu(g)$ be its $\ms T^1$-average, that we see as a constant function. Assume that
$\exp(-\kappa){\overline \phi}\leq \phi$. Then 
$$
d(\phi.\mu,\overline{\phi}.\mu)\leq \kappa\diam\left(\ms T^1\right).
$$
\end{lemma}
\begin{proof} We can as well assume after multiplying the distance by a constant that $\diam(\ms T^1)=1$. Let $A$ be any subset in $\ms T^1$. Assume first that $A_{\kappa}$ is a strict subset of $\ms T^1$ (and thus  $\kappa< 1/2$). Then
\begin{eqnarray}
(\phi\cdotp\mu)(A_{\kappa})\geq  \exp(-\kappa)\int_{A_\kappa} \overline{\phi} .{\rm d}\mu
\geq \exp(-\kappa)(\mu(A)+2\kappa)\overline{\phi}
\end{eqnarray}
Next observe that $\mu(A)\leq 1-2\kappa$. Hence
\begin{eqnarray}
(\phi\cdotp\mu)(A_{\kappa})\geq \exp(-\kappa)\left(1+\frac{2\kappa}{1-2\kappa}\right)\overline{\phi}.\mu(A)\geq\left( \frac{\exp(-\kappa)}{1-2\kappa}\right)\overline{\phi}.\mu(A).
\end{eqnarray}
Thus if $A_{\kappa}$ is a strict subset of $\ms T^1$: 
$
\phi\cdotp\mu(A_{\kappa})\geq  \overline{\phi}\cdotp\mu (A)$.
Finally if $A_\kappa=\ms T^1$, 
$$
\phi\cdotp\mu(A_\kappa)=\int_{\ms T^1}\phi \cdotp{\rm d}\mu=\overline{\phi}\geq\overline{\phi}\cdotp\mu(A).
$$
This concludes the proof of the statement.
\end{proof}

These two lemmas have the following immediate consequence
\begin{corollary}\label{proco}
Let $X:=\ms T^1\times X_0$. Let $d$  -- respectively $\mu$ -- be a $\ell_1$ product metric -- respectively a measure -- on $X$ invariant by $\ms T^1$. Let $\phi$ be a function on $X$.  Let $\overline\phi:=\int_{\ms T^1} \phi\circ g \ {\rm d}\nu(g)$ be its $\ms T^1$-average. Assume that
$(1-\kappa){\overline \phi}\leq \phi \leq (1+\kappa){\overline \phi}$. Then 
$$
d(\phi.\mu,\overline{\phi}.\mu)\leq \kappa\diam(\ms T^1).
$$
\end{corollary}
\subsubsection{Proof of Theorem \ref{prokho}}
We first treat the case of $X=\ms T=(\ms T^1)^n$ with the $\ell_1$ product metric $d_1$ which is of diameter 1 on each factor.  Note first that if $\tilde\phi$ is its average along one of the $\ms T^1$ factor, then
$$
\exp(-2.\kappa).\tilde\phi\leq\phi\leq \exp(2.\kappa).\tilde\phi.
$$
Applying Corollary \ref{proco} to all the factors of $\ms T$, we get after an induction procedure that for the corresponding L\'evy--Prokhorov distance 
$$
d_1(\phi.\mu,\overline{\phi}.\mu)\leq 2n.\kappa.
$$
We can conclude. 

We still consider the case $X=\ms T^n$ but  now equipped with a bi-invariant Riemannian metric $d$. Observe that  $\pi_1(X)$ can be generated by translations of length smaller than $2\diam(X)$. Thus there exists a bi-invariant $\ell_1$ product metric $d_1$ on this torus whose factors have diameter 1, so that
$$
d\leq 2.\diam(\ms T).d_1.
$$ 
The statement in that case follows from the following observation: let $d_1$, $d_2$ be two metrics whose corresponding L\'evy--Prokhorov distances are  respectively $\delta_1$ and $\delta_2$. Assume  that $d_2\leq K. d_1$. Then 
$
\delta_2\leq K.\delta_1
$.
Finally,  we apply Lemma \ref{prokho3} to conclude for the general case.

\section{Appendix B: Exponential Mixing}\label{app:mix}

The following lemma is well  known to experts as a combination of various deep results. However, it is difficult to track it precisely in the literature.  We thank Bachir Bekka and Nicolas Bergeron for their help on that matter.

\begin{lemma}\label{lem:mix}
Let $\ms G$ be a semi-simple Lie group without compact factor and $\Gamma$ be an irreducible lattice in $\ms G$, then the action of any non trivial hyperbolic element is exponentially mixing.

When the lattice is not irreducible, we have to impose furthermore that the projection of the hyperbolic element  to all  irreducible factors is non trivial.
\end{lemma}

\begin{proof} The extension to non irreducible factors follow from simple considerations. Thus let just prove the first statement. Let 
 $\ms G_1,\dots, \ms G_n$ be the simple factors of $\ms G$. Let 
 $\pi$ be the unitary representation of  $G$ in $L^2_0(\ms G/\Gamma)$, the orthogonal to the constant function in   $L^2(\ms G/\Gamma).$

By  Kleinbock--Margulis \cite[Corollary 4.5]{Kleinbock:1999tk} we have to show that the restriction $\pi_i$ of $\pi$ on $\ms G_i$ has a spectral gap (see also  Katok--Spatzier \cite[Corollary 3.2]{Katok:1994wr}) 

In the simplest case is when $\ms G$  is simple and $\Gamma$ uniform, this follows by standard arguments, for instance see Bekka's survey \cite[Proposition 8.1]{Bekka:2016tx}

When $\ms G$ is still simple, but $\Gamma$ non uniform, this now follows from Bekka \cite[Lemma 4.1]{Bekka:1998wg}.

When finally  $\ms G$ is a actually a product, by Margulis Arithmeticity Theorem \cite{Margulis:1984tk}, $\Gamma$ is arithmetic.  
For $\Gamma$  uniform, the spectral gap follows from Burger--Sarnak \cite{Burger:1991tp} and Clozel \cite{Clozel:2003wd}.  For $\Gamma$ non uniform, this is  due to Kleinbock--Margulis \cite[Theorem 1.12]{Kleinbock:1999tk}.
\end{proof}

\bibliographystyle{amsplain}
\providecommand{\bysame}{\leavevmode\hbox to3em{\hrulefill}\thinspace}
\providecommand{\MR}{\relax\ifhmode\unskip\space\fi MR }
\providecommand{\MRhref}[2]{
  \href{http://www.ams.org/mathscinet-getitem?mr=#1}{#2}
}
\providecommand{\href}[2]{#2}

\printindex

\end{document}